\documentclass[11pt]{amsart}

\allowbreak

\numberwithin{equation}{section}

\usepackage{amsfonts,amssymb,amscd,amsmath,latexsym,amsbsy}
\usepackage{graphicx}
\usepackage{subfig}
\usepackage{amssymb}
\usepackage{amsfonts}
\usepackage{latexsym}
\usepackage{mathrsfs}
\usepackage{mathtools}
\usepackage{amsthm}
\usepackage{comment}
\usepackage{listings}
\usepackage{tikz-cd}
\usepackage{hyperref}
\usepackage{amsmath}
\usepackage{leftidx}
\usepackage[toc,page]{appendix}
\usepackage[letterpaper,top=2cm,bottom=2cm,left=2cm,right=2cm,marginparwidth=1.75cm]{geometry}
\usepackage{upgreek}

\allowdisplaybreaks
\newtheorem{thm}{Theorem}[section]

\newtheorem{lmm}[thm]{Lemma}
\newtheorem{prop}[thm]{Proposition} 
\newtheorem{coro}[thm]{Corollary}

\theoremstyle{definition}
\newtheorem{defn}[thm]{Definition}
\newtheorem{remark}[thm]{Remark}
\newtheorem*{remark*}{Remark} 
\newtheorem{example}[thm]{Example}

\numberwithin{equation}{section}

\newcommand{\R}{\ensuremath{\mathbb{R}}}
\newcommand{\RR}{\ensuremath{\mathbb{R}}}
\newcommand{\Z}{\ensuremath{\mathbb{Z}}}
\newcommand{\N}{\ensuremath{\mathbb{N}}}
\newcommand{\C}{\ensuremath{\mathbb{C}}}

\newcommand{\WF}{\ensuremath{\mathrm{WF}}}
\newcommand{\Ell}{\ensuremath{\mathrm{Ell}}}
\newcommand{\supp}{\ensuremath{\mathrm{supp}}}

\newcommand{\mk}{\ensuremath{\mathfrak}}

\newcommand{\msf}{\ensuremath{\mathsf}}

\newcommand{\la}{\ensuremath{\langle}}
\newcommand{\ra}{\ensuremath{\rangle}}

\newcommand{\sgn}{\ensuremath{\mathrm{sgn}}}
\newcommand{\Id}{\ensuremath{\mathrm{Id}}}
\newcommand{\oc}{\ensuremath{\mathrm{1c}}}

\newcommand{\ps}{\ensuremath{\mathrm{ps}}}
\newcommand{\full}{\ensuremath{\mathrm{full}}}
\newcommand{\sub}{\ensuremath{\mathrm{sub}}}

\newcommand{\ff}{\ensuremath{\mathrm{ff}}}

\newcommand\SL{\mathcal{L}}
\newcommand{\SM}{\mathcal M}

\newcommand{\SR}{\mathcal R}

\newcommand{\SI}{\mathcal I}

\newcommand\xoc{{x_{\oc}}}
\newcommand\yoc{{y_{\oc}}}
\newcommand\etaoc{{\eta_{\oc}}}
\newcommand\xioc{{\xi_{\oc}}}
\newcommand\zetaoc{{\zeta_{\oc}}}
\newcommand\rhops{{\rho_{\ps}}}
\newcommand\ffocps{{\ff_{\oc-\ps}}}
\newcommand\ffpsoc{{\ff_{\ps-\oc}}} 
\newcommand\ffococ{{\ff_{\oc-\oc}}} 
\newcommand\xoco{x_{\oc, 1}}
\newcommand\yoco{y_{\oc, 1}}
\newcommand\xioco{\xi_{\oc, 1}}
\newcommand\etaoco{\eta_{\oc, 1}}
\newcommand\xoct{x_{\oc, 2}}
\newcommand\yoct{y_{\oc, 2}}
\newcommand\xioct{\xi_{\oc, 2}}
\newcommand\etaoct{\eta_{\oc, 2}}

\newcommand\Poi{\mathcal{P}}
\newcommand\Poim{\mathcal{P}_-}
\newcommand\Poip{\mathcal{P}_+}
\newcommand\Poiz{\mathcal{P}_0}
\newcommand\Poipm{\mathcal{P}_\pm}

\newcommand{\Rp}{\mathcal{R}_+}
\newcommand{\Rm}{\mathcal{R}_-}
\newcommand{\Rpm}{\mathcal{R}_\pm}
\newcommand\Legps{L}

\newcommand\Lagps{\Lambda}

\newcommand\ang[1]{\langle #1 \rangle}
\renewcommand\sgn{\operatorname{sgn}}
\newcommand\FBR{\mathrm{FBR}}
\newcommand\BBR{\mathrm{BBR}}
\newcommand\FSR{\overline{\Lambda_-'}}
\newcommand\BSR{\overline{\Lambda_+'}}
\newcommand\ocphase{\overline{{}^{\oc} T^* \R^n}}
\newcommand\psphase{\overline{{}^{\ps} T^* \R^{n+1}}}
\newcommand\Diag{\operatorname{Diag}}
\newcommand\SjR{\mathrm{SR}}

\newcommand\Schw{\mathcal{S}}
\newcommand\dvolh{d\mathrm{vol}^{1/2}}

\newcommand{\ococparaone}{v}
\newcommand{\ococparatwo}{w}

\newcommand{\cl}{\ensuremath{\mathrm{cl}}}

\newcommand{\Char}{\ensuremath{\mathrm{Char}}}

\newcommand{\Sg}{\ensuremath{\mathfrak{Cl}_g}}

\newcommand{\csmp}{\ensuremath{\tilde{S}}}
\newcommand{\ococb}{\ensuremath{ \leftidx{^{ \mathsf{L}\oc}}{ T^*X^2_b} } }

\newcommand{\rHp}{\ensuremath{ H^{m',0}_p } }

\title{The scattering map for the Schr\"odinger operator on curved spaces}

\author{Andrew Hassell, Qiuye Jia}
\date{\today}

\thanks{We acknowledge the support of the Australian Research Council through grant FL220100072}

\begin{document}

\begin{abstract}
 Let $P$ be a Schr\"odinger operator $D_t+\Delta_g$ with metric and potential perturbation that are compactly supported in spacetime $\R^{n+1}$. 
 Here $D_t = -i \partial_t$ and $\Delta_g$ is the positive Laplacian.
 %a suitable metric and potential perturbation of the free Schr\"odinger operator $P_0 = D_t + \Delta_0$ acting on distributions on $\R^{n+1}$, where $D_t = -i \partial_t$ and $\Delta_0$ is the positive Laplacian. Our assumption is that both the metric and potential perturbations are compactly supported in spacetime. 
 We consider the scattering map $S$ defined previously by the first author with Gell-Redman and Gomes \cite{gell2022propagation}, which relates the asymptotic data, as $t \to \pm \infty$, of global solutions $u$ to $Pu = 0$. We show that $S$ is a `1-cusp' Fourier integral operator, where `1-cusp' refers to a pseudodifferential calculus introduced by Vasy and Zachos \cite{zachos2022inverting} in the completely different setting of inverse problems on asymptotically conic manifolds. Our viewpoint is that 1-cusp geometry is the natural setting for studying the asymptotic data of solutions to Schr\"odinger's equation. 
\end{abstract}
\maketitle
%%%%%%%%%%%%%%%%%%%%%%%%

\tableofcontents

\section{Introduction}\label{sec:intro}

\subsection{The Poisson operator and the scattering map}
\label{subsec:1-1}
We consider the Schr\"odinger operator $P$ on $\R^{n+1}_{z, t}$, where $z \in \R^n$, $t \in \R$, 
\begin{equation} \label{eq:def-Schrodinger-op}
P = D_t + \Delta_{g(t)} + V(z, t), 
\end{equation}
where $D_t = -i \partial_t$, $\Delta_{g(t)}$ is the positive Laplace operator with respect to a smooth family of metrics $g(t)$ on $\R^n_z$, and $V$ is a smooth complex-valued potential function. In addition, we make the strong assumption that $g(t)$ is a compactly supported, in spacetime, perturbation of the flat metric, and $V$ is also compactly supported in spacetime. Thus, there exist large constants $R$, $T$ such that if either $|z| \geq R$, or if $|t| \geq T$, then $g(t) = \sum_{i,j} g_{ij}(z,t) dz_i dz_j$ coincides with the flat metric $g_0 = \sum_i dz_i^2$, and $V$ vanishes identically. To put this another way, near spacetime infinity, $P$ coincides with the `free' Schr\"odinger operator $P_0 := D_t + \Delta_0$, where $\Delta_0$ is the standard (positive) Laplacian on $\RR^n$. We assume, in addition, that each $g(t)$ is non-trapping. The support condition on $g(t) - g_0$ and $V$ is a severe restriction, and a more natural assumption would be a decay rate on the perturbation. However, even this strong assumption leads to an interesting and novel result, with a (relatively) transparent proof. For expository clarity, the authors have chosen to explain their result in this simplest possible setting. 

Our interest is in global solutions to $Pu = 0$, and in the asymptotic data, otherwise known as final state data or scattering data. Every  global solution $u$ has an asymptotic expansion of the form 
\begin{equation}\label{eq:u expansion}
u \sim (4\pi it)^{-n/2} e^{i|z|^2/4t} f_\pm\big( \frac{z}{2t} \big) + O(|t|^{-n/2 - \epsilon}), \quad t \to \pm \infty
\end{equation}
for large positive or negative times, where this holds in a pointwise sense if $f_\pm$ are sufficiently regular, and distributionally in general. The functions $f_\pm$ are by definition the asymptotic data of the solution $u$. 

We may also view the $f_\pm$ as `Cauchy data'. In fact, this is a useful analogy that we will pursue further. Consider a classical elliptic problem such as the Poisson problem, that is, solving the flat Laplace equation $\Delta v = 0$ in the unit ball $B$ of $\R^n$, with given boundary data. The Cauchy data for this problem consists of the function $v$ restricted to the boundary $\partial B$, together with the normal derivative of $v$, denoted $\partial_\nu v$, restricted to $\partial B$. Given a smooth solution to $\Delta v = 0$ --- or even just $v$ such that $\Delta v$ vanishes to infinite order at $\partial B$ --- the full Taylor series of $v$ at $\partial B$ is determined by the Cauchy data. 

In a similar way, we can represent a Schr\"odinger solution $u$ (so $Pu = 0$) for large $|t|$ (in which case, due to our assumption on $g(t)$, the operator is locally the free Schr\"odinger operator) in the form 
\begin{equation}
u(z, t) = (4\pi it)^{-n/2} e^{i|z|^2/4t} F_\pm\big( \frac{z}{2t}, t \big).
\end{equation}
Writing $F_\pm$ as a function of 
\begin{equation}\label{eq:Z defn}
Z = \frac{z}{2t}, \quad s = \frac1{t}, 
\end{equation} 
we find that under the assumption that $F_\pm$ is a smooth function of these variables, then the full Taylor series of $F_\pm$ at $s=0$ is determined by its value at $s=0$, that is, by the asymptotic data $f_\pm$. (In fact, as a function of $s$ and $Z$, $F_\pm$ itself solves a Schr\"odinger equation thanks to the pseudoconformal symmetry of Schr\"odinger's equation, which certainly shows that the Taylor series of $F_\pm$ at $s=0$ being determined by its leading term. However, the general point that we are making, that the value of $f_\pm$ determines the full Taylor series of $F_\pm$, follows from a Taylor series analysis in $s$ and does not require any such symmetry.) That means $u$ itself is determined up to a rapidly decreasing error.  In this sense, the $f_\pm$ justify the name `Cauchy data'. 

As is well known for the Poisson problem, one cannot specify \emph{both} sets of Cauchy data if one wants to solve the global problem $\Delta v = 0$ on $B$. Indeed, Green's formula for two such solutions shows that 
\begin{equation}
\int_{\partial B} \Big( v_1 (\partial_\nu v_2) - (\partial_\nu v_1) v_2 \Big) \, d\sigma = 0.
\label{eq:pairing Laplace}\end{equation}
This can be interpreted as saying that the Cauchy data $(v, \partial_\nu v)$ for \emph{global} solutions lie in an isotropic subspace for a symplectic pairing of two sets of Cauchy data. Because of this, it can be at most `half-dimensional'. Therefore, it makes heuristic sense to expect that one can prescribe one piece of Cauchy data, not both. As is well known, prescribing either one of the two pieces of Cauchy data leads to well-posed problems, the Dirichlet problem (prescribing $v$ at $\partial B$) and the Neumann problem (prescribing $\partial_\nu v$ at $\partial B$), although in the latter case there is a one-dimensional obstruction to solvability, that is, the Neumann data must have mean zero for a solution to exist, and then the solution is unique up to an arbitrary constant. 

A similar phenomenon is true for the asymptotic data of a solution to $Pu = 0$. In fact, there is a similar pairing formula as in \eqref{eq:pairing Laplace}, except that $P$ is not formally self-adjoint (on $L^2(\R^{n+1})$ with respect to either the standard measure $dt \, dz$ or the metric measure $dt \, d\mathrm{vol}(g(t))$) , so the pairing formula involves a global solution $u$ to $Pu = 0$ and a global solution $w$ to $P^* w = 0$. If the asymptotic data of $u$ and $w$ are $f_\pm$ and $g_\pm$ respectively, then we obtain 
\begin{equation}\label{eq: pairing}
\int_{\R^n} \Big( f_+(Z) \overline{g_+(Z)} - f_-(Z) \overline{g_-(Z)} \Big) \, dZ = 0. 
\end{equation}
This again shows that the Cauchy data of $P$ and $P^*$ annihilate each other with respect to a symplectic pairing. Under the reasonable assumption that the solvability theory of $P$ and $P^*$ are similar (they differ only by a smooth, compactly supported multiplication operator), this shows that again the Cauchy data of global solutions are at most `half-dimensional' and supports the heuristic that to solve the $Pu = 0$ one should specify only one piece of Cauchy data. This heuristic was verified in \cite{gell2022propagation} where it was shown that such a `final-state problem' has a unique solution, given arbitrary distributional data $f_-$. 

Returning once more to the Poisson problem, given that the Dirichlet data alone determines a solution to $\Delta v = 0$ in the ball $B$, one can study the map that takes the Dirichlet data to the Neumann data of the same solution, that is, the Dirichlet-to-Neumann map $\Lambda$. From a microlocal point of view this is a relatively simple operator, due to the ellipticity of $\Delta$. Elliptic regularity implies that the Neumann data cannot have wavefront set larger than that of the Dirichlet data, and indeed $\Lambda$ is an (elliptic) pseudodifferential operator, in particular preserving the wavefront set. 

Similarly, for the Schr\"odinger operator $P$, one can study the map that takes the asymptotic data $f_-$ to the global solution $u$ to $Pu = 0$ with asymptotic data $f_-$ as $t \to -\infty$ (which we call the Poisson operator $\Poim$, i.e. $u = \Poim f_-$), as well as the map that takes $f_-$ to  asymptotic data $f_+$ for a solution to $Pu = 0$, since $f_-$ alone determines $u$, and $u$ determines $f_+$. This is called the scattering map in \cite{gell2022propagation}, and denoted $S$, thus $Sf_- = f_+$. 
However, \emph{unlike} the Laplacian, the Schr\"odinger operator is non-elliptic and, in the parabolic calculus introduced by Lascar \cite{lascar} and developed in \cite{gell2022propagation}, there is microlocal propagation along bicharacteristics. We will describe this more fully shortly, but here we just observe that the heuristic analogue of the statement that $\Lambda$ is a pseudodifferential operator would be a hypothesis that $S$ is a Fourier integral operator. Indeed, this is the main result of the present article --- see Theorem~\ref{thm: main}. An analogous result has been obtained in many similar situations, for a scattering-type map relating different pieces of `Cauchy data'. The first time this was noticed seems to be in Guillemin \cite{guillemin1976sojourn}. It was shown in a discrete setting by Zelditch \cite{zelditch1997quantizedcontact}, for Helmholtz scattering on asymptotically conic manifolds by Melrose and Zworski \cite{melrose1996scattering}, in a semiclassical setting by Alexandrova \cite{alexandrova2005, alexandrova2006}, by the first author with Wunsch \cite{hassell-wunsch2008}, and by Vasy for the wave operator on a Lorentzian spacetime \cite{vasy2010deSitter}. Thus, it is almost a `meta-theorem' that such a result should hold, and in that context, our result is unsurprising.

However, there is an unexpected aspect to our result, which is that $S$ is an FIO of a previously undescribed type, which we call a 1-cusp FIO as it is related to the Lie Algebra of 1-cusp vector fields as introduced by Vasy and Zachos \cite{zachos2022inverting}. To explain this, we first describe in more detail the geometry of the bicharacteristic (or Hamilton) flow of $P$. This flow takes place,  in the microlocal setup of \cite{gell2022propagation}, on a \emph{compactification} of the natural phase space $\R^{n+1}_{t,z} \times \R^{n+1}_{\tau, \zeta}$. 
%In this compactification,  spacetime $\R^{n+1}_{t,z}$ is compactified radially, but the cotangent fibres $\R^{n+1}_{\tau, \zeta}$. That means, for example, that $z/t$ is a coordinate on the open northern and ,
 The characteristic variety $\Sigma(P)$ is the zero set of the symbol of $P$ on this compactified phase space, 
 and the set $\Char(P)$, where $\Sigma(P)$ meets the boundary of the compactification, is the stage on which microlocal propagation takes place. (We remind the reader that microlocal propagation takes place along the Hamilton vector field of the symbol $p$ of $P$ within the characteristic set. Integral curves of the Hamilton vector field within the characteristic set are called bicharacteristics.) The characteristic set $\Char(P)$ consists of an open part $\Char^\circ(P)$, where the (suitably rescaled) bicharacteristic flow is non-vanishing and two \emph{radial sets} $\SR_\pm$ defined in \eqref{eq:radial defn}, where the bicharacteristic  flow vanishes.  
 The dynamics of bicharacteristic flow are rather simple, on account of the non-trapping assumption on the metrics $g(t)$: $\Char^\circ(P)$ is foliated by bicharacteristics, all of which converge to $\SR_-$ in the backward direction and $\SR_+$ in the forward direction. In fact, $\SR_-$ is a source, and $\SR_+$ a sink, of the Hamilton flow on $\Sigma(P)$. 
 
 The dimension of $\Sigma(P)$ is $2n+1$, while the dimension of each $\SR_\pm$ is $n$. Thus, for each point $q \in \SR_\pm$ there is an $n$-dimensional family of bicharacteristics that tends to $q$ (along either the forward flow, corresponding to the $+$ sign, or the backward flow if $-$). It turns out, and this is explained in detail in Section~\ref{subsec: 1c arises}, that bicharacteristic flow determines a canonical symplectic map between $T^* \SR_-^\circ$ and $T^* \SR_+^\circ$, where $\SR_\pm^\circ$ denotes the interior of $\SR_\pm$. (The interior $\SR_\pm^\circ$ is diffeomorphic to $\R^n_Z$, where $Z$ is as in \eqref{eq:Z defn}, and $\SR_\pm$ is diffeomorphic to its radial compactification --- see Section~\ref{subsec:Hvf radial}.) This map extends continuously, and therefore also canonically, to a diffeomorphism between the `1-cusp' cotangent bundles $^{\oc}T^* \SR_-$ and $^{\oc}T^* \SR_+$; this is a particular compactification of the cotangent bundle over the interior taking account of the so-called 1-cusp structure on $\SR_\pm$. It does \emph{not} typically extend to a map between the usual cotangent bundle or the scattering cotangent bundle; \emph{the 1-cusp geometry is intrinsically associated to the Schr\"odinger equation}. 
 
 We describe the 1-cusp structure in Section~\ref{sec: 1c PsiDO}. Here we give an informal description near the boundary of the radial compactification of $\R^n_Z$, which could represent either of $\SR_\pm$. We use polar coordinates $r = |Z|$ and $y_i$, $1 \leq i \leq n-1$ being angular coordinates (i.e. homogeneous of degree zero); for example, in a conic neighbourhood of a given direction, without loss of generality given by $(0, 0, \dots, 1)$, i.e the coordinate axis in the $Z_n$ direction, we could take $y_i =Z_i/Z_n$. Then a basis of sections of the 1-cusp cotangent bundle, which is both smooth and uniformly linearly independent as $r \to \infty$, is given by $r dr$ and $r dy_i$. Compare this to a smooth basis of the scattering cotangent bundle, which would take the form $dr$ and $r dy_i$, or a smooth basis of the usual cotangent bundle, which would take the form $d(1/r)$ and $dy_i$. 
 
 We can now give a partial explanation for why it is the 1-cusp structure that is the correct structure for understanding $S$ as a Fourier integral operator. From \cite{gell2022propagation}, we have a formula for $S$ of the form\footnote{There is a sign error in \cite{gell2022propagation} that we have corrected here. The sign can be checked by noting that $S$ is the identity when $P = P_0$.}
 \begin{equation}
 S = i(2\pi)^n \mathcal{P}_+^* [P,Q_+] \mathcal{P}_-,
\label{eq:S formula} \end{equation}
 where $\mathcal{P}_\pm$ are the Poisson operators that map asymptotic data $f_\pm$ to the global solution $u$ with that asymptotic data as $t \to \pm \infty$, and $Q_+$ is a microlocal cutoff, such that $[P, Q_+]$ is microsupported away from the radial sets. Let us consider, in this introduction, just the free Poisson operator $\mathcal{P}_0$, which maps the asymptotic data $f_-(Z)$, where $Z$ is as in \eqref{eq:Z defn}, to the free Schr\"odinger solution with this asymptotic data at $t=-\infty$. In the free case, we have $\Poim = \Poip = \Poi_0$. Using Fourier analysis we can write this down explicitly:
 \begin{equation}\label{eq: Poisson zero}
\mathcal{P}_0f(t, z) = (2\pi)^{-n} \int e^{i(z \cdot Z - t |Z|^2)} f(Z) \, dZ.
\end{equation}
This can be interpreted as a Fourier integral operator, and as usual, its canonical relation is given by 
\begin{equation}
\mathcal{C} = \{ (z, t, \zeta, \tau; Z, \mathfrak{Z}) \mid \tau = - |\zeta|^2, \ \zeta = Z, \ \mathfrak{Z} = -z + 2t Z \}.
\label{eq:canrel}\end{equation}
Here $\zeta, \tau$ are the frequency coordinates dual to $z, t$, while $\mathfrak{Z}$ is dual to $Z$. We interpret this canonical relation geometrically, over a compact set in $(z,t)$-space. We note that, as per the discussion above, the coordinate $\mathfrak{Z}$ can be interpreted as an element of the 1-cusp cotangent bundle over $\overline{\R^n_Z}$. Indeed, writing $z = \hat Z z^\parallel + z^\perp$ in components parallel and orthogonal to the direction $\hat Z$, the coefficient of $r dr$, $r = |Z|$, is equal to $2t - r^{-1} z^\parallel$, while the angular component of $\mathfrak{Z}$ is given by $- z^\perp$, the component of $z$ orthogonal to $Z$.

We now consider which points $(z, t, \zeta, \tau)$ are related to $(Z, \mk{Z})$. Notice that the canonical relation has dimension $2n+1$, while $(Z, \mk{Z})$ is $2n$-dimensional. Since $P \Poiz = 0$, one might guess that the set of $(z, t, \zeta, \tau)$ related to a fixed $(Z, \mk{Z})$ is a bicharacteristic, and we shall show that this is the case. 

We begin by writing the equations of bicharacteristic flow. Outside the compact set in spacetime where the metric perturbation is supported, the flow is simply given by the free flow which is very simple:
\begin{equation}
\dot t = 1, \quad \dot z = 2 \zeta, \quad \dot \tau = 0, \quad \dot \zeta = 0. 
\end{equation}
Combined with $\zeta = Z$ from \eqref{eq:canrel}, we obtain 
\begin{equation}
\dot t = 1, \quad \dot z = 2 Z. 
\end{equation}
We deduce from this that if $t_0$ is the time when $z$ is orthogonal to $Z$, say $z(t_0) = z_0^\perp$, then 
\begin{equation}
z = z_0^\perp + 2(t-t_0) Z    . 
\end{equation}
The parameters $(z_0^\perp, t_0)$ parametrize all the bicharacteristics with a fixed spatial frequency $\zeta = Z$. 

On the other hand, we also have $z = -\mk{Z} + 2tZ$ from \eqref{eq:canrel}, and comparing these equations we see that indeed, the $(z, t, \zeta, \tau)$ related to $(Z, \mk{Z})$ comprise the bicharacteristic of $P$ with $(z_0^\perp, t_0)$ determined by $\mk{Z}$ by 
\begin{equation}
    \mk{Z} = - z_0^\perp + 2 t_0 Z. 
\end{equation}
Therefore, $-z_0^\perp$ and $2t_0$ are precisely the perpendicular, resp. parallel 1-cusp components of the 1c-frequency variable $\mk{Z}$. Moreover, $\mathfrak{Z}$ is large \emph{as a 1c-frequency} if and only if the corresponding bicharacteristic has $(z_0^\perp, t_0)$ large, i.e. if and only if the bicharacteristic avoids a large ball in spacetime.

\begin{comment}
Thus, at fibre-infinity, $t$ is fixed --- that is, the flow happens at one instant of time --- and for a given value of $\hat\zeta$ all such bicharacteristics converge to a single point on the boundary of $\SR_\pm$ given by $t/|z| = 0$ and $\hat z = \pm \hat \zeta$, or equivalently $|Z| = \infty, \hat Z = \pm \zeta$. They form a pencil of parallel straight lines in $z$-space with fixed direction $\hat\zeta$.   The bicharacteristics are parametrized by the level set of $t$ in which they lie, together with the component of $z$ orthogonal to $\hat\zeta$ --- in other words, exactly by the 1-cusp coordinate $\mathfrak{Z}$  identified above in the discussion of the canonical relation \eqref{eq:canrel}. To summarize, the 1-cusp frequency variable $\mathfrak{Z} \in {}^{\oc} T^* _{\hat Z} \overline{\R^n}$ in \eqref{eq:canrel}, $\hat Z \in \partial \SR_\pm$, determines a unique bicharacteristic with direction $\hat \zeta = \pm \hat Z$. 
\end{comment}

Now returning to \eqref{eq:S formula}, the perturbed Poisson operators $\mathcal{P}_\pm$  map $L^2(\R^n_\zeta)$ into the null space of $P$. Their canonical relations are therefore invariant, not under the free bicharacteristic flow, but the perturbed bicharacteristic flow. That means that, at time $t$, they follow geodesics of the metric $g(t)$ in the $(z, \xi)$ variables. We view \eqref{eq:S formula} as a composition of FIOs, and the first thing to check is that the canonical relations compose sensibly. \emph{This is the case provided one uses the 1-cusp cotangent bundle for $\SR_\pm$.}  One checks that the overall canonical relation of $S$ is a canonical graph that relates the initial point of each bicharacteristic, viewed as a point of ${}^{\oc} T^* \SR_-$, with its final point viewed as a point in ${}^{\oc} T^* \SR_+$. The metric perturbation ensures that this is (at least typically) a nontrivial perturbation of the identity map. Then it is straightforward to check that $S$ is indeed a 1-cusp FIO, provided of course that this class of operators has been correctly defined. The setting up of FIO theory in the 1-cusp setting occupies the greater part of this article.

\subsection{Main results}

As discussed above, bicharacteristic flow leads to a map from ${}^{\oc} T^* \SR_-$ to ${}^{\oc} T^* \SR_+$, relating the initial and final points of each (compactified) bicharacteristic. We call this map the \emph{classical scattering map}, and denote it $\Sg$.  The main result of this article is 

\begin{thm}\label{thm: main}
The scattering map $S$ is a elliptic 1-cusp Fourier integral operator of order zero. In the notation of Section~\ref{sec: 1c-1c Lagrangian distribution} (see Remark~\ref{rem:main theorem notation} for a brief explanation), 
\begin{align}\label{eq:S 1c-1c FIO}
S \in I^0_{\oc-\oc}(X_b^2,\beta_b^*(\mathrm{Gr}(\Sg)')), \quad X = \overline{\R^n}. 
\end{align}
The scattering map $S$ acts as the identity microlocally on functions (asymptotic data) supported in a compact subset of $\R^n$, or are supported microlocally near frequency-infinity in the 1-cusp sense. 
\end{thm}

\begin{remark}\label{rem:main theorem notation}
In \eqref{eq:S 1c-1c FIO}, $X$ is the radial compactification of $\R^n = \R^n_Z$, $Z = z/(2t)$ as in \eqref{eq:Z defn}, and is identified with both the incoming and outgoing radial sets as explained in Section~\ref{subsec:Hvf radial}. The scattering map $S$ acts on functions of $Z$, thus its Schwartz kernel is \emph{a priori} a distribution on $\R^n_Z \times \R^n_Z$. For technical reasons we view the Schwartz kernel of $S$ as living on the b-double space $X^2_b$, that is, $X^2$ with the corner $(\partial X)^2$ blown up. We use $\beta_b$ to denote the blow-down map from $X^2_b$ to $X^2$. Next,  
$\mathrm{Gr}(\Sg)$ denotes the graph of the classical scattering map; 
$\mathrm{Gr}(\Sg)'$ is a Lagrangian submanifold of the 1c-1c phase space $\ococb$ in the sense of Definition~\ref{defn: admissible 1c-1c Lagrangian submanifold and 1c-1c fibred-Legendre submanifold} and $I^0_{\oc-\oc}(X_b^2,\beta_b^*(\mathrm{Gr}(\Sg)'))$ is the class of 1c-1c FIO associated to it of order zero defined as in Definition~\ref{def:1c-1c-FIO}. 
The prime in $\mathrm{Gr}(\Sg)'$ indicates negation of the 1c-fibre variable, as is usual in FIO theory: $\mathrm{Gr}(\Sg)$ is a canonical relation, and $\mathrm{Gr}(\Sg)'$ is the corresponding Lagrangian submanifold. 
\end{remark}

\begin{remark} The scattering map $S$ may not be unitary, as we did not assume that $P$ is self-adjoint. In the self-adjoint case, for which the imaginary part of the potential term in \eqref{eq:def-Schrodinger-op} is determined by the family of metrics $g(t)$, $S$ will be unitary thanks to the pairing formula \eqref{eq: pairing}. 
\end{remark}

We record some corollaries of Theorem~\ref{thm: main}. The brief proofs will be provided in Section~\ref{subsec:proof of corollaries}. 

\begin{coro}\label{coro:bounded}
The scattering map $S$ is a bounded map on all 1c-Sobolev spaces $H_{\oc}^{s, l}(\R^n)$, for all $s, l \in \R$.     
\end{coro}

This was previously proved for $s \in \mathbb{N}$ and $l=s$ in \cite{gell2022propagation}, where the spaces were denoted $\mathcal{W}^s$. 

\begin{coro}\label{coro:wavefront}
The scattering map acts on 1c-wavefront set according to the classical scattering map: if $q_- \in \partial \ocphase$ and $q_+ = \Sg(q_-)$ then 
\begin{equation}
q_+ \in \WF_{\oc}(Sf) \Longleftrightarrow q_- \in \WF_{\oc}(f). 
\end{equation}
\end{coro}

The next corollary states that, whenever the classical scattering map is nontrivial, $S$ is a nontrivial perturbation of the identity operator. 

\begin{coro} \label{coro:noncompact-difference}
When $\mathrm{Gr}(\Sg) \neq \mathrm{Gr}(\mathfrak{Cl}_{g_0})$ (which is just the graph of the identity map), for all $s,l \in \R$, $S-\Id$ is not a compact operator from $H_{\oc}^{s,l}(\R^n)$ to itself. In particular, for $s=l=0$, $S - \Id$ is not compact on $L^2(\R^n)$. 
\end{coro}

\begin{remark}
    One may have a more general statement for $g_1,g_2$, which are two metrics in the class we are considering. Let $S_i$ be the scattering map associated to them respectively, if the classical scattering maps of $g_1,g_2$ do not coincide, then the difference $S_1-S_2$ will not be compact on any $H^{s,l}_{\oc}(\R^n)$.
    This gives rise to a inverse problem in the reversed direction that will be addressed elsewhere in the future: 
    suppose $S_1-S_2$ is compact on $H^{s,l}_{\oc}(\R^n)$, or even more restrictively suppose $S_1=S_2$, then can we say that $g_1$ and $g_2$ are the same modulo certain natural obstructions? 
\end{remark}

%There are a number of other potential applications of Theorem~\ref{thm: main}, which 
%although they fall outside the scope of this article. 
%First, it suggests a number of interesting inverse problems. 
%For example, this might be related to the work of Uhlmann et al on invisibility.

%Second, in the case of the semiclassical Schr\"odinger operator $h^2 \Delta + V$ on $\R^n$, with $V \in C_c^\infty(\R^n)$ real, it has been shown that, under a dynamical condition on the classical scattering map, the scattering matrix (which is unitary) has discrete spectrum on the unit circle away from the point $1$, which is asymptotically equidistributed around the unit circle (excluding an arbitrarily small neighbourhood of $1$) as $h \to 0$. Potentially a similar result could be proved here, in the self-adjoint case, for a judiciously chosen family of operators depending on a parameter $h \to 0$ that approximates $P$ suitably. 

\subsection{Relation with previous literature}

We briefly review the relation with previous works in this subsection. 
Since the topic of scattering matrix or scattering map has been a vast subject, we can only mention a very limited part of it that is closely related. Moreover, we completely omit any discussion of the vast literature on dispersive equations.

The theory of Lagrangian distributions Fourier integral operators arose in the 1960s through work of Maslov \cite{Maslov1967}, Egorov \cite{Egorov1971} and others, building on earlier studies of WKB-type expansions in quantum physics, and geometrical optics expansions of Lax, Keller, Arnold and others. It was systematized and put on a rigorous foundation (for example, giving a rigorous, invariant definition of the principal symbol) by H\"ormander \cite{FIO1} and Duistermaat-H\"ormander \cite{FIOII}. Duistermaat-Guillemin \cite{duistermaat-guillemin1975spectrum} extended the theory to encompass clean phase functions. Melrose-Zworski \cite{melrose1996scattering} defined Legendre distributions in close analogy with Lagrangian distributions and used this to show that the scattering matrix for the Laplacian on an asymptotically conic space (manifold with scattering metric) is a Fourier integral operator. They did this by first constructing a microlocal parametrix for the Poisson operator (in the context of the Helmholtz operator) using a subtle symbolic construction for pairs of Legendre distributions, one of which has a conic singularity that is resolved by blowup. This strategy was simplified by Vasy \cite{vasy1998geometric} who showed how to avoid working near the conic singularity. 

Our strategy in the present paper follows the approach of Melrose-Zworski and Vasy. 
In particular, the formula \eqref{eq:S formula} (which comes from \cite{gell2022propagation}) and strategy of viewing the scattering map as a composition of certain types of Fourier integral operators is adapted from Vasy's work. This strategy is also closely related to formulae previously obtained by Isozaki-Kitada \cite{isozakikitada1985, isozakikitada1986} and Yafaev \cite{yafaevbook2000} for analyzing long-range scattering; indeed, they used an operator that (in our language) microlocalizes away from one of the radial sets, and hence plays a similar role to the $Q_+$ operator in \eqref{eq:S formula}. 

In the semiclassical setting, Guillemin \cite{guillemin1976sojourn} showed in several instances the scattering matrix has a semiclassical expansion involving the `sojourn time' of classical rays, and hypothesized that this should be a general phenomenon. This hypothesis was confirmed by Alexandrova \cite{alexandrova2005, alexandrova2006} and by the first author and Wunsch in \cite{hassell-wunsch2008}, where it was pointed out that the `sojourn time' is really a Lagrangian submanifold (called the sojourn relation in \cite{hassell-wunsch2008}) rather than a function. The sojourn relation shows up in exactly the same way in the present article (see Section~\ref{sec: forward and backward sojourn relations}). 

In the present article, we define `fibred-Legendre distributions' which is the analogue of the class of Legendre (or Lagrangian) distributions in our geometric setup. The fibred aspect is reminiscent of spaces of distributions described in \cite[Part I]{hassell-wunsch2008}, and has similarities to the fibrations at the boundary of the Legendre distributions in \cite{hassell2001resolvent}. 
As well as a parametrix-type construction of the Poisson operator as a fibred-Legendre distribution,  we also use results  obtained by Gell-Redman, Gomes and the first author in \cite{gell2022propagation} using a non-elliptic Fredholm framework. In addition, the module regularity used by the first author, Gell-Redman and Gomes in \cite{gell2022propagation, gell2023scattering} corresponds to the 1c-regularity of the final state data.
This correspondence is implicit in \cite{gell2022propagation} but was made explicit in \cite{hassell2024final}. As already mentioned, the 1-cusp structure was first introduced by Vasy and Zachos in \cite{zachos2022inverting}.

\subsection{Outline of this article} In sections~\ref{sec: parabolic sc PsiDO} and \ref{sec: 1c PsiDO} we outline the two pseudodifferential calculi that we require, the parabolic scattering (abbreviated to ps) calculus on $\R^{n+1}_{z,t}$, where the operator $P$ acts, and the 1-cusp (abbreviated to 1c) calculus on $\R^n$, on which the scattering map $S$ acts. However, we also need to understand the Poisson operators $\Poipm$ microlocally, which map functions on $\R^n$ to functions on $\R^{n+1}$. In Sections~\ref{sec:1c-ps geometry} and \ref{sec: 1c-ps FIO}, we discuss 1c-ps Fourier integral operators, which act in precisely this way. We show that suitably microlocalized Poisson operators belong to this class of operators. Then, in Sections~\ref{sec:1c-1c geometry} and \ref{sec: 1c-1c Lagrangian distribution}, we discuss 1c-1c Fourier integral operators. In these four technical sections, the theory broadly follows the standard theory of Fourier integral operators as developed by H\"ormander  \cite{FIO1} and and Duistermaat-H\"ormander \cite{FIOII}, but there are interesting changes due to the anisotropies in both the ps- and 1c-calculi. These manifest in the structure of the boundary of Lagrangian submanifolds at the boundary of phase space. In contrast to the usual setting, where the boundary is a contact manifold, in both the ps and 1c settings, the boundary fibres over a Legendrian curve with fibres that are symplectic for a family of symplectic structures on the boundary. This geometry is explained in detail in Sections~\ref{subsec:symplectic structures 1cps} and \ref{subsec: symplectic structures 1c 1c}.  In Section~\ref{sec:composition} we analyze the composition of Poisson-type operators in a setting that covers (after suitable microlocalization) the formula \eqref{eq:S formula}. The main theorem and corollaries are proved in Section~\ref{sec: scattering map}. The appendices cover technical material that is standard in the theory of Fourier integral operators but requires some minor adaptation to the current setting. 

%%%%%%%%%%%%%%%%%%%%%%%%%%%%%%%%%%%%%%%%%%%%%%%%%%%%%%%%%%%%%
%%%%%%%%%%%%%%%%%%%%%%%%%%%%%%%%%%%%%%%%%%%%%%%%%%%%%%%%%%%%%
%%%%%%%%%%%%%%%%%%%%%%%%%%%%%%%%%%%%%%%%%%%%%%%%%%%%%%%%%%%%%
%%%%%%%%%%%%%%%%%%%%%%%%%%%%%%%%%%%%%%%%%%%%%%%%%%%%%%%%%%%%%
%%%%%%%%%%%%%%%%%%%%%%%%%%%%%%%%%%%%%%%%%%%%%%%%%%%%%%%%%%%%%

\section{Parabolic scattering calculus and the Schr\"odinger operator}
\label{sec: parabolic sc PsiDO}
%choose parabolic sc instead of parabolic classical for future purpose
In this section we give a rather brief introduction of the parabolic scattering pseudodifferential algebra, and the geometry and dynamics of the Schr\"odinger operator within this calculus.  We refer readers to \cite{lascar}, \cite[Section~2]{gell2022propagation} and references therein for more details.

\subsection{Parabolic scattering algebra}
As the name indicates, it is the `parabolic version' of the scattering pseudodifferential algebra introduced by Melrose \cite{melrose1994spectral}, where `parabolic' refers to the way to compactify the fiber infinity. More precisely, let 
\begin{align}
(t,z,\tau,\zeta)
\end{align}
be coordinates on $T^*\R^{n+1}$, with $\tau dt + \zeta \cdot dz$ being the canonical form.
Then the (compactified) parabolic scattering cotangent bundle, denoted by $\overline{^{\ps}T^*\R^{n+1}}$ , is defined by compactifying $T^*\R^{n+1}$ in following way.
The `base' $\R^{n+1}$ is compactified radially to be a ball with boundary defining function
\begin{align}\label{eq:xps defn}
x_{\ps} = (1+t^2+|z|^2)^{-1/2},
\end{align}
while each fiber is compactified `parabolically' to be a ball with boundary defining function
\begin{align}\label{eq:rhops defn}
\rho_{\ps} = (1+\tau^2+|\zeta|^4)^{-1/4},
\end{align}
with $\ps$ standing for `parabolic scattering'. 

\begin{defn}
The symbol class $S_{\ps}^{m,l}(\R^{n+1})$, with $m,l$ differential and decay (in fact growth) order respectively, we use, is defined to be the space of $a \in C^\infty(T^*\R^{n+1})$ (in fact also their extension to $\overline{^{\ps}T^*\R^{n+1}}$) such that
\begin{align*}
\|a\|_{ S_{\ps,N}^{m,l}} 
:= \sum_{|\alpha|+k+|\beta|+j \leq N} \sup_{T^*\R^{n+1}} |x_{\ps}^{l-|\alpha|-k} \rho_{\ps}^{m-\beta-2j} \partial_z^\alpha \partial_t^k \partial_\zeta^\beta \partial_\tau^j a(z,t,\zeta,\tau) | < \infty,
\end{align*}
for any $N \in \N$. And these norms give $S_{\ps}^{m,l}(\R^{n+1})$ a structure of Fr\'echet space. In addition, when $\rho_{\ps}^m x_{\ps}^l a$ extends to a smooth function on $\overline{^{\ps}T^*\R^{n+1}}$, then we say $a$ is classical, and the corresponding symbol class is denoted by $S_{\ps,\cl}^{m,l}(\R^{n+1})$.
\end{defn}
\begin{remark}
In the definition above, each $\partial_\tau$ improves the decay at fiber infinity by 2 order whereas $\partial_\zeta$ only improves one order. This reflects the parabolic nature of this symbol class.

Geometrically, the difference between the large symbol class $S_{\ps}^{m,l}(\R^{n+1})$ and the `small' symbol class $S_{\ps,\cl}^{m,l}(\R^{n+1})$ is that the membership of the former means only conormal property on $\overline{^{\ps}T^*\R^{n+1}}$, while the membership of the latter means smoothness on $\overline{^{\ps}T^*\R^{n+1}}$ (in particular, down to its boundaries).
\end{remark}

Then the corresponding parabolic scattering pseudodifferential operators 
$\Psi_{\ps}^{m,l}(\R^{n+1})$ are operators that are quantizations of symbols in
$S_{\ps}^{m,l}(\R^{n+1})$, which means they have Schwartz kernels of the form
\begin{align*}
q_{L,\ps}(a) = (2\pi)^{-(n+1)} \int e^{ i (t-t')\tau+(z-z')\cdot \zeta}
a(t,z,\tau,\zeta) d\zeta d\tau,
\end{align*}
in the distributional sense.

The principal symbol of $A = q_{L,\ps}(a) \in \Psi_{\ps}^{m,l}(\R^{n+1})$ is defined to be the equivalence class of $a$ in $S_{\ps}^{m,l}(\R^{n+1})$ quotient by $S_{\ps}^{m-1,l-1}(\R^{n+1})$:
\begin{align*}
[a] \in S_{\ps}^{m,l}(\R^{n+1})/S_{\ps}^{m-1,l-1}(\R^{n+1}).
\end{align*}

A symbol $a \in S_{\ps}^{m,l}(\R^{n+1})$ (and corresponding operator $A=q_{L,\ps}(a)$) is said to be elliptic if it satisfies:
\begin{align*}
|\rho_{\ps}^m x_{\ps}^{l}a| \geq C, \text{ when } \rho_{\ps} \leq \epsilon 
\text{ or } x_{\ps} \leq \epsilon, 
\end{align*}
for some $\epsilon>0,C>0$. And we say that $a$ (and corresponding operator $A=q_{L,\ps}(a)$) is elliptic at $q \in \partial \overline{^{\ps}T^*\R^{n+1}}$ if there is a neighborhood of $q$ on which above inequality is satisfied. And the set of all such $q \in \partial \overline{^{\ps}T^*\R^{n+1}}$ is denoted by $\Ell_{\ps}^{m,l}(a)$ (or $\Ell_{\ps}^{m,l}(A)$).

Let $A \in \Psi_{\ps}^{m,l}(\R^{n+1})$ be classical, i.e. $A = q_{L, \ps}(a)$ where the symbol $a$ is classical. Thus, by assumption, $\rho_{\ps}^{m} x_{\ps}^{l}  a$ extends smoothly to the boundary of $\overline{^{\ps}T^*\R^{n+1}}$. 
We define $\Sigma(A)$ to be the set where the normalized symbol $\rho_{\ps}^{m} x_{\ps}^{l}  a$ vanishes on $\overline{^{\ps}T^*\R^{n+1}}$, and define the 
 characteristic set $\Char_{\ps}^{m,l}(A)$ (or $\Char_{\ps}^{m,l}(a)$) to be the intersection of $\Sigma(A)$ with the boundary of phase space: 
\begin{align}
\Char_{\ps}^{m,l}(A) = \{ q \in \partial \overline{^{\ps}T^*\R^{n+1}} \ | 
\  (\rho_{\ps}^m x_{\ps}^{l}a)(q) = 0 \}.
\end{align}
Notice that $\Char_{\ps}^{m,l}(A)$ (but not $\Sigma(A)$) depends only on the principal symbol of $A$. Microlocal propagation takes place at the boundary of phase space, therefore strictly speaking only $\Char_{\ps}^{m,l}(A)$ is relevant, but we will sometimes find it convenient to discuss propagation through the interior of phase space, that is, in $\Sigma(A) \setminus \Char_{\ps}^{m,l}(A)$, as well. (Actually only the jet of $\Sigma(A)$ at the boundary of $\overline{^{\ps}T^*\R^{n+1}}$ is relevant.) 

%Standard facts about composition, wavefront sets, definition of Sobolev spaces are almost identical to those of scattering pseudodifferential algbera, which we omit here.
%The parabolic scattering Sobolev spaces are defined to be  
%%%%%%%%%%%%%%we don't need ps Sobolev spaces?

One of the key ingredient of our analysis is the global source-sink structure of the rescaled Hamilton flow associated to $P$, which we introduce below.

Using the symplectic structure on $T^*\R^{n+1}$ we have a local expression for the Hamilton vector field
\begin{align}
H_a = \partial_\tau a \partial_t - \partial_ta \partial_\tau
+ \partial_\zeta a \cdot \partial_z - \partial_z a \cdot \partial_\zeta.
\end{align}
When $a \in S_{\ps,\cl}^{m,l}(\R^{n+1})$, we define the rescaled Hamilton vector field to be
\begin{align}\label{eq:rescaled Hvf}
\mathsf{H}^{m,l}_a = \rho_{\ps}^{m-1} x_{\ps}^{l-1} H_a.
\end{align}
Due to the parabolic nature, $\partial_\tau a \partial_t - \partial_ta \partial_\tau$ is one more order lower than other parts of the expression at fiber infinity:
\begin{align}
\msf{H}^{m,l}_a|_{\rho_{\ps}=0}
= \rho_{\ps}^{m-1} x_{\ps}^{l-1}(\partial_\zeta a \cdot \partial_z - \partial_z a \cdot \partial_\zeta).
\end{align}
As shown in \cite[Section~2]{gell2022propagation}, $\msf{H}^{m,l}_a$ extends to a smooth vector field on $\overline{^{\ps}T^*\R^{n+1}}$ that is tangent to the boundary. In addition, it is tangent to $\Char_{\ps}^{m,l}(a)$, and the vector field restricted to the characteristic set depends only on the principal symbol. 

The rescaled bicharacteristic flows of time dependent Schr\"odinger equations that we are going to investigate have a global source-sink structure (see Section \ref{sec: forward and backward sojourn relations} for more discussion) with following `radial sets' being the source and sink.
\begin{defn} \label{defn: radial sets}
The radial set (of $A$) $\mathcal{R}$ is defined to be 
\begin{align}
\{ q \in \Char(A): \msf{H}^{m,l}_a \text{ vanishes at } q \}.
\end{align}
\end{defn}

\begin{remark} Since we are only interested in the integral curves of $\msf{H}^{m,l}_a$, when working locally it is harmless to replace the global boundary defining functions $x_{\ps}$ and $\rho_{\ps}$ by any other locally defined boundary defining function. This means in particular that when working away from spacetime infinity, $x_{\ps}$ can be replaced (locally) by $1$, i.e. the rescaling in spacetime can be disregarded. Similarly when working away from fibre-infinity, $\rho_{\ps}$ can be replaced by $1$. Moreover it is enough to work with (locally defined) boundary defining functions that are valid in a neighbourhood of $\Sigma(A)$. In an abuse of notation, we will denote any of these locally-defined rescalings of $H_a$ by $H_a^{m,l}$, even though they are not literally equal. Similarly, to ease the notational burden, we will use the notation for $x_{\ps}$ and $\rho_{\ps}$ for locally defined boundary defining functions of spacetime infinity and fibre-infinity respectively, even though these are not the same as the globally defined quantities \eqref{eq:xps defn} and \eqref{eq:rhops defn} respectively. 
See the calculation of the Hamilton vector field of $P_0$ below for examples where this is convenient. 
\end{remark}

Finally we recall the concept of (parabolic) wavefront sets, which is also called the micro-support.
For $A = q_{L,\ps}(a)$, its parabolic scattering operator wavefront set $WF'_{\ps}(A)$ is defined, as a subset of $\partial \overline{^{\ps}T^*\R^{n+1}}$, as follows.
For $q \in \partial \overline{^{\ps}T^*\R^{n+1}}$, we say
$q \notin \WF'_{\ps}(A)$ if and only if there is $\chi \in C^\infty(\overline{^{\ps}T^*\R^{n+1}})$ with $\chi(q) = 1$ such that $\chi a \in \mathcal{S}(\overline{^{\ps}T^*\R^{n+1}})$.
In particular, $\WF'_{\ps} (A) = \emptyset$ if and only if $A \in \Psi_{\ps}^{-\infty,-\infty}(\R^{n+1})$.

Similarly, for $u \in \mathcal{S}'(\R^{n+1})$ its $\ps$-wavefront set of order $s,r$, denoted by $\WF_{\ps}^{s,r}(u)$, is defined to be the (microlocal) loci of where $u$ fails to lie in $H_{\ps}^{s,r}(\R^{n+1})$.
Concretely, for $q \in \partial \overline{^{\ps}T^*\R^{n+1}}$, we have
\begin{equation}
    q \notin \WF_{\ps}^{s,r}(u) \iff \text{ there exists } A \in \Psi_{\ps}^{s,r} \text{ that is elliptic at } q, \, \text{ such that } Au \in L^2(\R^{n+1}).
\end{equation}
Then the loci of the (microlocal) loci of where $u$ fails to lie in $\mathcal{S}(\R^{n+1})$ is denoted by
\begin{equation}
    \WF_{\ps}(u) = \bigcup_{s,r} \WF_{\ps}^{s,r}(u).
\end{equation}

\subsection{\texorpdfstring{Hamilton vector field and radial sets of $P_0$}{Hamilton vector field and radial sets of P0}}\label{subsec:Hvf radial}
For example, we compute the radial set of the free Schr\"odinger operator $P_0 = D_t + \Delta_0$, largely following \cite[Section 3]{gell2022propagation}. The Hamilton vector field is 
\begin{equation}
H_p = 2\zeta \cdot \partial_z + \partial_t.
\end{equation}
Let us rescale this vector field as follows: first we work over a compact region of spacetime, near fibre-infinity. Then we can take $x_{\ps} = 1$ and $\rho_{\ps} = |\zeta|^{-1}$, as this is a boundary defining function for fibre-infinity in a neighbourhood of $\Char_{\ps}^{2,0}(P_0)$. We obtain the rescaled vector field 
\begin{equation}
H_p^{2,0} = \frac{2\zeta}{|\zeta|} \cdot \partial_z + \frac1{|\zeta|} \partial_t.
\end{equation}
This vector field is never vanishing, as the coefficient of $\partial_z$ always has length $2$. Thus, the radial set of $P_0$ lives over spacetime infinity. We now investigate spacetime infinity. We do this in two regions: first, when $|t| \geq 10 |z|$, and second when $|z| \geq 10|t|$. In the first region we can take as spacetime boundary defining function $x_{\ps} = 1/|t|$ and the remaining spacetime coordinates can be $Z = z/(2t)$ as in \eqref{eq:Z defn}. We divide by $\ang{\zeta}$ instead of $|\zeta|$ as we want to consider finite, even zero, frequencies at spacetime infinity. Further, we define (away from fibre-infinity) $w = Z - \zeta$. We obtain then 
\begin{equation}\label{eq:Hp polar}
H_p^{2,0} = (-\sgn t) \Big(\frac{\zeta}{\ang{\zeta}} \cdot \partial_Z + \frac1{\ang{\zeta}} \big(x_{\ps} \partial_{x_{\ps}} - Z \cdot \partial_Z)\Big)  =  
\frac{(-\sgn t)}{\ang{\zeta}}  \Big( w \cdot \partial_{w} + x_{\ps} \partial_{x_{\ps}} \Big). 
\end{equation}
The first expression shows that $H_p^{2,0}$ is non-vanishing in a neighbourhood of fibre-infinity in this region. Restricting to finite frequencies, then the $w$ coordinates are valid, and the second expression shows that the rescaled Hamilton vector field vanishes precisely where 
\begin{align}\label{eq:rad.interior}
\begin{split}
\Rpm \supset \{ x_{\ps} = 0, \ w = 0, \tau & = - |\zeta|^2, \pm t \geq 0 \} \cap \{ \frac{|z|}{|t|} \leq 10 \} 
\\ & =  \{ x_{\ps}  = 0, \ \zeta = Z, \tau = - |Z|^2, \pm t \geq 0 \} \cap \{  \frac{|z|}{|t|} \leq 10 \}.
\end{split}
\end{align}
The condition $\tau = - |\zeta|^2$ enforces containment in $\Sigma(P)$.  We also note that the set in \eqref{eq:rad.interior} has two components, depending on the sign of $t$, and  this rescaled Hamilton vector field is a source at $\Rm$, when $t < 0$, and a sink at $\Rp$, when $t > 0$. Notice that $Z$ is a coordinate on $\Rpm$ in this region. 

We move to the region $|z| \geq 10|t|$. In this region,  we may suppose without loss of generality that the first spatial coordinate $z_1$ is a locally dominant variable, i.e. $z_1 \geq 1/2 \max_i |z_i|$. Then we can take $x_{\ps} =  1/z_1$. We use projective coordinates $s = t/z_1$ and $v_j = z_j/z_1$ for $j \geq 2$ for the remaining spacetime coordinates. First working in a region where frequency is bounded, we rescale the Hamilton vector field by dividing by $x_{\ps}$, that is, multiplying by $z_1$. Using coordinates $(x_{\ps}, s, v_j, \zeta, \tau)$, we have in this region 
$$
H_{p_0}^{2,0} = \big( 1 - 2s \zeta_1 \big) \partial_s + \sum_{j \geq 2} \big( 2 \zeta_j - 2 \zeta_1 v_j \big) \partial_{v_j} - 2 \zeta_1 x_{\ps} \partial_{x_{\ps}}.  
$$
We are interested in the region where $s$ is small, otherwise the previous calculation applies. In order for this vector field to vanish, we see from the $\partial_s$ coefficient that when $s$ is small, necessarily $|\zeta_1|$ is large, and either positive or negative depending on the sign of $s$. Thus to find zeroes of $H_p^{2,0}$ it suffices to work in a neighbourhood of fibre-infinity, and we may assume that $\zeta_1$ is a dominant variable, i.e. $|\zeta_1| \geq c |\zeta|$ (as we will soon show, if $z_1$ is dominant and $\zeta_1$ is not dominant, we would be disjoint from $\Sigma(P_0)$).  First taking the case that $\zeta_1$ is large and positive, i.e. the same sign as $z_1$, we  use fibre boundary defining function $\rho_{\ps} = 1/\zeta_1$ and projective coordinates $\omega_j = \zeta_j/\zeta_1$ and $\sigma = \tau/|\zeta|^2$. The rescaled vector field to be $x_{\ps}^{-1} \rho_{\ps} H_p$ takes the form in these coordinates 
$$
z_1 \zeta_1^{-1} H_{p_0} = \big( \rho_{\ps} - 2s \big) \partial_s + \sum_{j \geq 2} \big( 2 \omega_j - 2  v_j \big) \partial_{v_j} - 2  x_{\ps} \partial_{x_{\ps}}.
$$
Changing variables to $\tilde s = 2s - \rho_{\ps}$, $\tilde v_j = v_j - \omega_j$, this vector field expressed  in coordinates $(\tilde s, \tilde v_j, x_{\ps}, \rho_{\ps}, \omega_j, \sigma)$ takes the form 
\begin{equation}\label{eq:Hvf Rplus equator}
-2\tilde s \partial_{\tilde s} -2  \sum_{j \geq 2} \tilde v_j \partial_{\tilde v_j} - 2  x_{\ps} \partial_{x_{\ps}}.
\end{equation}
We see that in this region,  the radial set is given by $\{ x_{\ps} = 0, \tilde s = 0, \tilde v_j = 0, \sigma = -1 \}$.  These equations define a submanifold of dimension $n$ inside spacetime infinity, which meets fibre infinity transversely at $s = 0$, that is at the `equator'.

\begin{figure}
    \centering
     \subfloat{{\includegraphics[width=6cm]{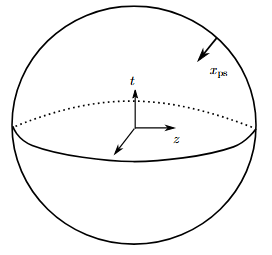} }}%
      \qquad
\subfloat{{\includegraphics[width=5.7cm]{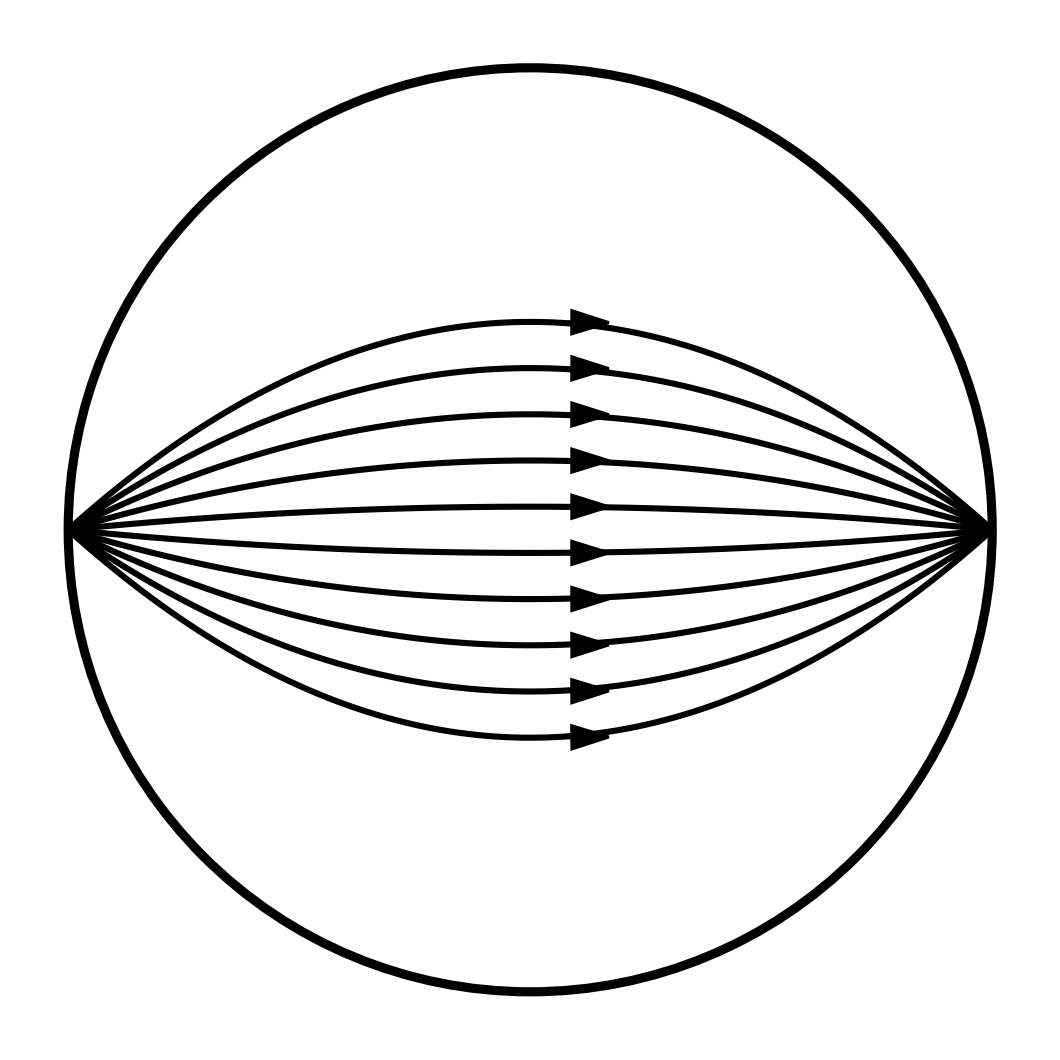} }}%

    \caption{ On the left: the radially compactified spacetime $\overline{\R^{n+1}_{z,t}}$. The function $x_{\ps} = (1+t^2+|z|^2)^{-1/2}$ defines (i.e. vanishes simply at) spacetime infinity. The radial sets live in phase space $\psphase$ over the boundary of $\overline{\R^{n+1}_{z,t}}$; $\SR_+$ lies over the top, or `northern' hemisphere, while the radial set $\SR_-$ lies over the bottom or `southern' hemisphere. The intersection of the northern and southern hemispheres is referred to as the equator. 
    \\ On the right: Projection to $\overline{\R^{n+1}}$ of bicharacteristics starting from the boundary of $\mathcal{R}_-$, which lies over the equator. Each bicharacteristic lies in a single level set of $t$.}
    \label{fig:flow-equator}
\end{figure}

It is not hard to check that where $s$ is sufficiently positive, this set coincides with the set $\mathcal{R}_+$ defined in \eqref{eq:rad.interior}. We write the equations for $\mathcal{R}_+$ using more natural coordinates in this region as 
\begin{equation}\label{eq:rad.corner.p}
\mathcal{R}_+ \supset \{ x_{\ps} = 0, \ s = \rho_{\ps}/2, \ \hat \zeta = \hat z, \ \sigma = -1 \} \cap \{  \frac{|t|}{|z|} \leq 10 \}
\end{equation}
or, equivalently, 
\begin{equation}\label{eq:rad.corner.p2}
\mathcal{R}_+ \supset \{ x_{\ps} = 0, \ 1/|\zeta| = 1/|Z|, \ \hat \zeta = \hat Z, \ \sigma = -1 \} \cap \{  \frac{|t|}{|z|} \leq 10 \}
\end{equation}
In particular, this shows that the identification of $\Rpm^\circ$ with $\R^n_Z$ in \eqref{eq:rad.interior}, \eqref{eq:rad.corner.p2} extends to an identification of $\Rpm$ with the radial compactification of $\R^n_Z$. 

The rescaled Hamilton vector field in \eqref{eq:Hvf Rplus equator} is clearly a sink at $\mathcal{R}_+$ in this region. Notice that the projection to spacetime infinity gives the closed `northern hemisphere', since $s = \rho_{\ps}/2 \geq 0$ in \eqref{eq:rad.corner.p}. Wherever $s$ is strictly positive, $\mathcal{R}_+$ is a graph, but it fails to be so at the boundary $s=0$: the graph `turns vertical' at the equator, and has a boundary at fibre-infinity. We can repeat the analysis assuming that $\zeta_1 < 0$, that is, when the signs of $z_1$ and $\zeta_1$ are opposite (and this is a good exercise for the reader); then we find zeroes of $H_p^{2,0}$ at 
\begin{equation}\label{eq:rad.corner.m}
\mathcal{R}_- \supset \{ x_{\ps} = 0, \ s = -\rho_{\ps}/2, \ \hat \zeta = -\hat z, \ \tau/|\zeta|^2 = -1 \} \cap \{  \frac{|t|}{|z|} \leq 10 \},
\end{equation}
where the Hamilton vector field takes the form (with now $\tilde s = 2s + \rho_{\ps}$)
\begin{equation}\label{eq:Hvf Rminus equator}
2\tilde s \partial_{\tilde s} + 2  \sum_{j \geq 2} \tilde v_j \partial_{\tilde v_j} + 2  x_{\ps} \partial_{x_{\ps}}.
\end{equation}
Now we are over the `southern hemisphere' since $s \leq 0$. We see that $\Rm$ is a graph over the southern hemisphere that `turns vertical' over the equator, and that the Hamilton vector field is a source at $\Rm$.

%%%%%%%%%%%%%%%%%%%%%%%%%%%%%%%%%%%%%%%%%%%%%%%%%%%%%%%%%%
%%%%%%%%%%%%%%%%%%%%%%%%%%%%%%%%%%%%%%%%%%%%%%%%%%%%%%%%%%
%%%%%%%%%%%%%%%%%%%%%%%%%%%%%%%%%%%%%%%%%%%%%%%%%%%%%%%%%%
%%%%%%%%%%%%%%%%%%%%%%%%%%%%%%%%%%%%%%%%%%%%%%%%%%%%%%%%%%
%%%%%%%%%%%%%%%%%%%%%%%%%%%%%%%%%%%%%%%%%%%%%%%%%%%%%%%%%%

\section{The 1-cusp pseudodifferential algebra} \label{sec: 1c PsiDO}

In this section we briefly introduce the 1-cusp pseudodifferential algebra, and refer readers to \cite{Zachos:Thesis}\cite[Section~2]{zachos2022inverting}\cite[Section~2]{jia2022tensorial} for more details.

\subsection{1-cusp tangent and cotangent bundles}\label{subsec: 1c bundles}
Let $M$ be an $m-$dimensional manifold with boundary, with boundary defining function $x_{\oc}$ that is distinguished up to quadratic terms; in other words, the section $d\xoc$ of the conormal bundle at $\partial M = \{ \xoc = 0 \}$ is distinguished. Let $y=(y_1,..,y_{m-1})$ be a coordinate system of $\partial M$, which together with $x_{\oc}$ form a local coordinate system of $M$ in a collar neighbourhood of $\partial M$. Then the space of 1-cusp vector fields, denoted by $\mathcal{V}_{\oc}$, is locally spanned over $C^\infty(M)$ by
\begin{align} \label{eq: 1-cusp vector fields, local}
x_{\oc}^3\partial_{x_{\oc}}, x_{\oc}\partial_{y_j}, j=1,2,...,m-1.
\end{align}
$\mathcal{V}_{\oc}$ give rise to a vector bundle with \eqref{eq: 1-cusp vector fields, local} being its local frame. This is called the 1-cusp tangent bundle, and denoted by $^{\oc}TM$.

The 1-cusp differential operators of order (at most) $k$, denoted by $\mathrm{Diff}_{\oc}^k(M)$, consists of polynomials of these vector fields of degree at most $k$:
\begin{align}  \label{eq: Diff 1c definition}
\begin{split}
\mathrm{Diff}_{\oc}^k(M) 
= \{ \sum_{\alpha+|\beta| \leq k}  a_{\alpha\beta}(x_{\oc},y)(x_{\oc}^3\partial_{x_{\oc}})^\alpha(x_{\oc}\partial_{y})^\beta : 
\\ a_{\alpha \beta} \in C^\infty(M), \alpha \in \N, \beta \in \N^{m-1} \}.
\end{split}
\end{align}

The 1-cusp cotangent bundle, denoted by $^{\oc}T^*M$ is the dual bundle of $^{\oc}TM$. It is locally spanned over $C^\infty(M)$ by 
\begin{align} 
\frac{dx_{\oc}}{x_{\oc}^3}, \frac{dy_j}{x_{\oc}}, j=1,2,...,m-1.
\end{align}
Then $T^*M$ embeds into $^{\oc}T^*M$ canonically in the interior of $M$, giving it a symplectic structure naturally. In particular, one may write the canonical one form as
\begin{align}  \label{eq: 1c- canonical form}
\xi_{\oc} \frac{dx_{\oc}}{x_{\oc}^3} + \eta_{\oc} \cdot \frac{dy}{x_{\oc}},
\end{align}
where $(\xi_{\oc},\eta_{\oc})$ are coordinates of fibers on $^{\oc}T^*M$. 
We use $\overline{^{\oc}T^*M}$ to denote the compactification of $^{\oc}T^*M$, obtained by compactifying each fiber radially to be a ball
with boundary defining function
\begin{align}
\rho_{\oc} = (1+\xi_{\oc}^2+|\eta_{\oc}|^2)^{-1/2}.
\end{align}

Notice the minor difference with $\overline{^{\ps}T^*\R^{n+1}}$, which is compactified on both base and fiber level, is because the 1-cusp construction is happening on the manifold with boundary $M$, which is already `compactified in a priori'.

\subsection{1-cusp pseudodifferential algebra}
The 1-cusp symbol class of differential order $m$ and decay (in fact growth) order $l$, denoted by $S_{\oc}^{m,l}(M)$ ,is defined to be smooth functions on ${}^{\oc}T^*M$ satisfying for any $j,\alpha,k,\beta$ there exists a constant $C_{j\alpha k \beta}$ such that
\begin{align}
 |x_{\oc}^l \rho_{\oc}^{m-k-|\beta|} (x_{\oc}\partial_{x_{\oc}})^j\partial_y^\alpha\partial_{\xi_{\oc}}^k\partial_{\eta_{\oc}}^\beta a(x_{\oc},y,\xi_{\oc},\eta_{\oc})|
 \leq C_{j\alpha k \beta}.
\end{align}
And locally they quantize to be operators acting by
%[TODO: xDx or Dx?]
\begin{align}
(q_{L,\oc}(a)u)(x_{\oc},y) = (2\pi)^{-n}
\int e^{ i \xi_{\oc}\frac{x_{\oc}-x_{\oc}'}{x_{\oc}^3}+\eta_{\oc} \cdot \frac{y-y'}{x_{\oc}}}
a(x_{\oc},y,\xi_{\oc},\eta_{\oc}) u(x_{\oc}',y') d\xi_{\oc}d\eta_{\oc} \frac{dx_{\oc}'dy'}{(x_{\oc}')^{n+2}}.
\end{align}
The collection of all such operators, are called 1-cusp pseudodifferential operators with differential order $m$ and decay order $l$, and denoted by $\Psi_{\oc}^{m,l}(M)$.

%%%%%%%%%%%%%%%%%%%%%%%%%%%%%%%%%%%%%%
Next we define the ellipticity of symbols and operators.
\begin{defn}
A symbol $a \in S^{m,l}_{\oc}(M)$ is called elliptic if
\begin{align*}
|a(x_{\oc},y,\xi_{\oc},\eta_{\oc})| \geq cx_{\oc}^{-l} \la (\xi_{\oc},\eta_{\mathrm{1c}}) \ra^m, \quad c>0 \text{ when }  |(\xi_{\oc},\eta_{\oc})|^{-1} \leq \epsilon \text{ or } x_{\oc} \leq \epsilon,
\end{align*}
for some $\epsilon>0,C>0$, and its quantization $A$ is also called elliptic in this case.
\end{defn}
Under this condition, see \cite[Section~2.5]{zachos2022inverting}, its quantization $A$ has a parametrix $B \in \Psi_{\oc}^{-m,-l}(M)$ such that
\begin{align*}
AB-\Id, \; BA-\Id \in \Psi_{\oc}^{-\infty,-\infty}(M).
\end{align*}

One can now define 1-cusp Sobolev spaces
$H^{s,r}_{\oc}(M)$ (see \cite[Section~2.5]{zachos2022inverting}) for $s \geq 0$ by choosing $A \in \Psi_{\oc}^{s,0}(M)$ elliptic, and demanding
$$
u\in H^{s,r}_{\oc}(M) \Leftrightarrow u\in x_{\oc}^r L_{\oc}^2(M)\ \text{and}\ Au\in x_{\oc}^r L_{\oc}^2(M);
$$
here $L_{\oc}^2(M)$ is the $L^2$ space relative to the 1-cusp density $\frac{dx\,dy}{x_{\oc}^{n+2}}$.
%a fixed polynomially weighted density, 
%which in the geometric context of asymptotically conic spaces is natural to take to be the metric density, which is
They are equipped with norms:
$$
\|u\|^2_{H^{s,r}_{\mathrm{1c},h}}=\|x_{\oc}^{-r}u\|_{L_{\oc}^2}^2+\sum_{j+|\alpha|\leq s}\|(hx_{\oc}^3D_{x_{\oc}})^j(h^{1/2}x_{\oc}D_y)^\alpha\|_{L_{\oc}^2}^2.
$$
The 1-cusp Sobolev spaces for other $s$ are defined via interpolation and duality.

With this definition, 1-cusp pseudodifferential operators are bounded on 1-cusp Sobolev spaces:
\begin{equation}\label{eq:1c boundedness}
\text{ For all $B \in \Psi_{\oc}^{m,l}(M)$ and all $s,r$, $B : H^{s,r}_{\oc} \to H^{s-m,r-l}_{\oc}$ is bounded.}    
\end{equation}

%, namely for $A \in \Psi_{\oc}^{m,l}(M)$ and all $s,r$, $A$ is bounded linear operator from $H^{s,r}_{\oc}$ to $H^{s-m,r-l}_{\oc}$.

%%%%%%%%%%%%%%%%%%%%%%%%%%%%%%%%%%%%%%%%%%%%
The 1-cusp wavefront set, which characterizes the location of singularities in the phase space, is defined by:
\begin{defn} \label{defn: 1c WF, both decay and smoothness}
\begin{align} \label{eq: definition of 1c WF, both decay and smoothness}
\begin{split}
\WF_{\oc}(u) =&  \Big(\{ q \in \partial \overline{{}^\oc{T^*M}}: \text{There exists } A \in \Psi_{\oc}^{0,0}(M) \text{ such that } 
 \\& A \text{ is elliptic at } q, \,Au \in \mathcal{S}(M) \} \Big)^c.
\end{split}
\end{align}
\end{defn}
And we have a refined version which captures the order of the singularities:
\begin{defn} \label{defn: 1c WF, both decay and smoothness, with order}
\begin{align} \label{eq: definition of 1c WF, both decay and smoothness, with order}
\begin{split}
\WF_{\oc}^{s,r}(u) =&  \Big(\{ q \in \partial \overline{{}^\oc{T^*M}}: \text{There exists } A \in \Psi_{\oc}^{s,r}(M) \text{ such that } 
 \\& A \text{ is elliptic at } q, \,Au \in L^2(M)  \}\Big)^c.
\end{split}
\end{align}
\end{defn}

The residual one, for example $\WF_{\oc}^{-\infty,r}(u)$ is defined by 
\begin{align}
\WF_{\oc}^{-\infty,r}(u) = \bigcap_{s \in \Z} \WF_{\oc}^{s,r}(u). 
\end{align}
In the opposite direction, we define
\begin{align}
\WF_{\oc}^{\infty,r}(u) = \overline{\bigcup_{s \in \Z} \WF_{\oc}^{s,r}(u)}. 
\end{align}

\subsection{\texorpdfstring{How the $\oc$-cotangent bundle arises in the geometry of the Schr\"odinger operator}{How the 1c-cotangent bundle arises in the geometry of the Schr\"odinger operator}}
\label{subsec: 1c arises}
We explain here how the $\oc$-cotangent bundle naturally arises from the geometry of the Schr\"odinger operator, as developed in \cite{gell2022propagation}. 

We first make some remarks on the blowup of a submanifold $T$ of a manifold with boundary $M$, such that $T \subset \partial M$. The blowup $[M; T]$ is a manifold with corners of codimension two. The new face created by blowing up $T$, which we call the front face and denote by $\ff$, is fibred over $T$. On the interior of this new front face, we have projective coordinates. Choose local coordinates $(x, y, z)$ such that $x$ is a boundary defining function for $\partial M$, $x = y = 0$ defines $T$, and $z$ are local coordinates when restricted to $T$. Then the projective coordinates are 
\begin{equation}\label{eq:Yj}
Y_j = \frac{y_j}{x}.
\end{equation}
How do these coordinates change under smooth changes of coordinates $(x, y, z) \mapsto (\tilde x, \tilde y, \tilde z)$ still having the properties listed above? It is easy to check that the corresponding $\tilde Y_j$ are related to $Y_k$ by an affine change of variables. (A good exercise for the reader!) Thus, the interior of the fibres of $\ff \to T$ have a natural affine structure; that is, the `fibre-interior' $\ff^\circ$ is\footnote{By the fibre-interior $\ff^\circ$ of $\ff$, we mean the union, over all $p \in T$ (including $p \in \partial T$) of the interior of the fibre over $p$.} an affine bundle over $T$.

If we are provided with slightly more structure, then we can view $\ff^\circ$ as a vector bundle instead. What we need is a section $V$ of the normal bundle of $T$ that is transverse to $\partial M$. It can be viewed as the `1-jet' of an extension of the submanifold $T$ to a submanifold of dimension one greater that is transverse to $\partial M$. The point of $\ff^\circ$ lying over $z \in T$ that corresponds to a curve reaching $z$ arriving with direction $V$ can then be deemed to be the origin of the fibre over $z$, giving the interior of each fibre the structure of a vector space, not just an affine space. More explicitly, the functions $y_j$ can be taken to vanish over some extension of $T$ that is tangent to $V$,  and then the functions $Y_j$ in \eqref{eq:Yj} are linear functions on the interior of each fibre. 

We now turn to the particular case where $T$ is one of the radial sets $\Rpm$ of a Schr\"odinger operator $P$ (see Section~\ref{subsec:Hvf radial}) of the form specified in the Introduction. Recall that the assumption on the metric perturbation $g(t)$ is that $g(t) - g_0$ (with $g_0$ being the Euclidean metric on $\R^n$) is a smooth perturbation compactly supported in spacetime, and that $g(t)$ is non-trapping for each $t$. Because the radial set lives at spacetime infinity, it is not affected by the perturbation at all, so we can assume that $P = P_0$ for the purposes of this discussion. Then each radial set, in its respective hemisphere $\pm t \geq 0$, is defined by 
\begin{equation}\label{eq:radial defn}
\SR_{\pm} = \{ x_{\ps} = 0, \quad \zeta = \frac{z}{2t}, \quad \tau + |\zeta|^2 = 0 \}.
\end{equation}
We can extend $\SR_{\pm}$ into the interior by making them conically invariant in $(z, t)$. (We ignore issues of smoothness near the origin in spacetime, as this is not relevant to us. Actually all that is relevant is the 1-jet of the extension at the spacetime boundary, as discussed earlier.) We denote these by $\hat \SR_{\pm}$. Thus, 
\begin{equation}
\hat \SR_{\pm} = \{  \zeta = \frac{z}{2t}, \quad \tau + |\zeta|^2 = 0 \}.
\end{equation}
As noted in \cite[Section 3]{gell2022propagation}, these are Lagrangian submanifolds of $T^* \RR^{n+1}$ with its canonical symplectic structure. We shall use this Lagrangian property to prove the following lemma. In preparation, let $\Sigma_0$ be the (interior) characteristic variety of $P_0$, that is, the codimension 1 submanifold
\begin{equation}
\Sigma = \{ \tau + |\zeta|^2 = 0 \} \subset 
\overline{{}^{\ps} T^*\R^{n+1} }.
%Y, \quad Y = \overline{\RR^{n+1}}. 
\end{equation}
Recall that $\SR_{\pm}$ are $n$-dimensional submanifolds of $\Sigma$, which has dimension $2n+1$. Let $W_\pm$ be the front face of the blowup 
$$
[\Sigma; \SR_{\pm}].
$$

\begin{lmm} Consider the radial sets $\SR_{\pm}$ endowed with the distinguished extension $\hat \SR_{\pm}$ into the interior of ${}^{\ps} T^* Y$. Then the fibre-interior $W_\pm^\circ$ of $W_\pm$ is canonically isomorphic to the one-cusp cotangent bundle ${}^{\oc} T^* \SR_{\pm}$ and $W_\pm$ itself is canonically isomorphic to the radially compactified one-cusp cotangent bundle. If the radial set is not endowed with a distinguished extension (or just a 1-jet of extension) then $W_{\pm}^\circ$ is canonically isomorphic to ${}^{\oc} T^* \SR_{\pm}$ as an affine bundle rather than a vector bundle. 
\end{lmm}

\begin{remark}  \label{remark:Wpm-dynamics}
Before turning to the proof, we remark on the geometric significance of $W_\pm$. We recall the general procedure of symplectic reduction: given a Hamiltonian $p$ (here, the principal symbol of $P$ or $P_0$) on a symplectic manifold (here, the cotangent bundle of $\RR^{n+1}$), there is an induced symplectic structure on the set of null bicharacteristics, i.e. on the characteristic variety $\Sigma$ quotiented by the flow. We can quotient by the flow concretely by intersecting $\Sigma$ with a cross section, that is, a codimension 1 submanifold of $\Sigma$ that is everywhere transverse to the flow. 

Now recall that, due to the nontrapping hypothesis, every flow line of $H_p^{2,0}$, the Hamilton vector field of $p$, converges to $\SR_+$ in the forward direction and $\SR_-$ in the backward direction. Moreover, these radial sets are sources ($\SR_-$) or sinks ($\SR_+$) for the flow, and the Hamilton vector field vanishes simply there. Therefore, when $\SR_\pm$ are blown up inside $\Sigma$, the different flowlines (bicharacteristics) are resolved: there is a one-to-one correspondence between the points of $W_\pm$ lying over $p \in \SR_\pm$, and the flowlines of the (rescaled) Hamilton vector field inside $\Sigma$ that converge to $p$. That is, the $W_\pm$ serve as natural cross-sections to the flow. We should thus expect that there is an induced symplectic structure on $W_\pm$. Since the interior of $W_\pm$ is an $n$-dimensional bundle over the $n$-dimensional radial sets $\SR_\pm$, with a canonical symplectic structure, it is natural to expect that it will turn out to be the cotangent bundle of $\Rpm$, at least away from the boundary of $\SR_\pm$, and we will shortly see that this is the case. What is less obvious is the behaviour near the boundary of $\SR_\pm$. The main point of the lemma is that the \emph{canonical diffeomorphism between $W_\pm^\circ$ and the cotangent bundle over the interior of $\SR_\pm$ extends to the boundary of $\Rpm$ as a diffeomorphism to the one-cusp cotangent bundle}. 
\end{remark}

\begin{proof}
We first work away from the boundary of $\SR_\pm$, that is, in the region where $(\zeta, \tau)$ are finite. We claim that, given a distinguished extension of the radial sets into the interior of ${}^{\ps} T^* \R^{n+1}$ --- we will take the conic extension $\hat \SR_\pm$ as above --- the fibre-interior of $W_\pm$ is \emph{canonically} diffeomorphic to the cotangent bundle of $\SR_\pm$. To emphasize, we will initially work locally over the interior of $\SR_\pm$; the analysis near the boundary of $\SR_\pm$, that is, near fibre-infinity of the compactified phase space, will be carried out later. 

We have already observed above that each fibre of $W_\pm^\circ$ has a natural affine structure, and this is upgraded to a natural linear structure given a distinguished extension of $\SR_\pm$ into the interior (or its 1-jet at $\SR_\pm$). 
Our strategy is to show that every linear function on the fibres of $W_\pm^\circ$ corresponds canonically to a vector field on $\Rpm$. This proves the claim, since a linear function on a vector bundle is equivalent to an element of the dual space. We will show that, over the interior of $\Rpm$, the dual space of the bundle $W_\pm^\circ$ is canonically the tangent bundle of $\Rpm$, which shows that $W_\pm^\circ$ itself is canonically the cotangent bundle of $\Rpm$. 

Recall from the discussion above that linear functions on the fibres of $W_\pm^\circ$ correspond precisely to functions $Y = y/\rho$, where $y$ vanishes at $\Rpm$ and $\rho$ is a boundary defining function. We can take $\rho = 1/t$ in this region, and we can assume without loss of generality that $y$ is homogeneous of degree zero in $(z, t)$, and thus extends smoothly to the spacetime boundary, where $\Rpm$ lives. Notice that $Y$ then is homogeneous of degree 1 in $(z, t)$ and vanishes on $\hat \SR_\pm$. Now we exploit the Lagrangian nature of $\hat \SR_\pm$, as proved in \cite[Section 3]{gell2022propagation}. Indeed, for any Lagrangian submanifold $\Lambda$ of a symplectic manifold, the symplectic form $\omega$ determines an isomorphism between the conormal bundle $N^*_q \Lambda$, and the tangent space $T_q \Lambda$ at each point $q \in \Lambda$. Moreover, the correspondence is given explicitly as follows: if $f$ is a function vanishing at $\Lambda$, then the corresponding vector is the Hamilton vector field $H_f$, which is tangent to $\Lambda$. In our specific situation where $\Lambda = \hat \SR_\pm$, and $f = Y = y/\rho$, we obtain a vector field $H_Y$ tangent to $\hat \SR_\pm$. Moreover, since $Y$ is homogeneous of degree 1 in $(z,t)$, $H_Y$ is homogeneous of degree zero in $(z, t)$ and thus determines a vector field on $\Rpm$. 

While this abstract argument suffices to prove the claim, let us be more explicit. We can take $y = 2(\zeta_j - Z_j)$, where $Z = z/(2t)$. Thus, the corresponding $Y$ is $2t \zeta_j - z_j$. The Hamilton vector field is given by 
\begin{equation}\label{eq:SR tangent vectorfields}
V_j = 2t D_{z_j} - 2 \zeta_j D_\tau + D_{\zeta_j}
\end{equation}
and it is easy to check that, in the local coordinates $Z$ on $\Rpm$, this vector field is precisely $\partial_{Z_j}$. This calculation shows explicitly that the linear functions on $W_\pm^\circ$ correspond diffeomorphically to points in the tangent bundle.

Since $W_\pm^\circ$ is canonically identified with $T^* \SR_\pm$, it has an intrinsic symplectic structure. 
We next show that this intrinsic symplectic structure on $W_\pm$ agrees with the symplectic structure obtained from symplectic reduction. The symplectic form as a cotangent bundle is given by 
$$
\sum_{j=1}^n dc_j \wedge dZ_j ,
$$
where $c_j$ are the dual coordinates to $Z_j$, where the $Z_j$ are taken as local coordinates on $\SR_\pm$ (away from the boundary of $\SR_\pm$). The $c_j$ are precisely the linear functions corresponding to the tangent vector fields $D_{Z_j}$, which are exactly the vector fields \eqref{eq:SR tangent vectorfields} written in the local coordinates $Z_j$ on $\SR_\pm$. Thus $c_j$ is the restriction of $2t \zeta_j - z_j$ to $W_\pm$, while $Z_j = \zeta_j$ on $\Rpm$, and this gives an expression for the symplectic form induced by the structure of $W_\pm$ as a cotangent bundle:
\begin{equation}\label{eq:symplectic form W cotangent bundle}
d(2t \zeta_j - z_j) \wedge d\zeta_j =  dt \wedge d|\zeta|^2 - \sum_{j=1}^n  dz_j \wedge d\zeta_j .
\end{equation}
This is not a convenient expression for the symplectic form however, as the two summands $dt \wedge d|\zeta|^2$ and $- dz \wedge d\zeta$ do not individually have a finite restriction; only the sum of the two does. Instead, we write this symplectic form in more geometrically natural coordinates, using coordinates on $\Sigma$ that parametrize the bicharacteristics. Assuming that we are dealing with the free operator $P_0$, bicharacteristics take the form 
\begin{equation}
z = z_0 + 2s \zeta, \quad t = t_0 + s, \quad \tau = -|\zeta|^2.
\end{equation}
If we assume that $z_0$ is orthogonal to $\zeta$ --- let us then write it $z_0^\perp$ to emphasize this orthogonality --- then this gives unique parameters $(t_0, z_0^\perp)$ for each bicharacteristic. This parametrization is not valid uniformly as $|\zeta| \to \infty$, but it becomes so if we reparametrize, 
replacing $s$ by $\tilde s = s |\zeta|$. (This is equivalent to the rescaling of the Hamilton vector field by a factor of $|\zeta|^{-1}$ in order to have a limit at fibre-infinity, as in \eqref{eq:rescaled Hvf} or  \cite[Section 3, above (3.6)]{gell2022propagation}.) Then we get 
\begin{equation}\label{eq:z0t0}
z = z_0 + 2\tilde s \hat \zeta, \quad t = t_0 + \frac{\tilde s}{|\zeta|}, \quad \tau = -|\zeta|^2,
\end{equation}
which has a limit as $|\zeta| \to \infty$, in which $t$ is constant in the limit and $t_0$ is this fixed value. So the coordinates $(t_0, z_0^\perp)$ are valid everywhere where $\zeta \neq 0$, and are valid uniformly up to fibre-infinity. Using the coordinates $(\zeta, z_0^\perp, t_0, \tilde s)$ on $\Sigma$, we find from \eqref{eq:symplectic form W cotangent bundle} that the symplectic form as a cotangent bundle is given by 
\begin{equation}
\label{eq:symplectic form W cotangent bundle 2} 
\sum_j d\Big( 2(t_0 + \frac{\tilde s}{|\zeta|} ) \zeta_j - (z_0^\perp + 2 \tilde s \hat \zeta)_j  \Big) \wedge d\zeta_j = dt_0 \wedge d|\zeta|^2 - dz_0^\perp \wedge d\zeta. 
\end{equation}
On the other hand, using symplectic reduction, we start with the symplectic form $\omega = d\tau \wedge dt + d\zeta \wedge dz$ on $T^* \RR^{n+1}$, restrict to $\Sigma$ by setting $\tau = -|\zeta|^2$ and then restricting to (or more precisely taking a limit as) $\tilde s \to \infty$. Using the same coordinates $(\zeta, z_0^\perp, t_0, \tilde s)$ as above on $\Sigma$, we find that $\omega$ restricts to 
\begin{equation}\label{eq:sympl red}
 dt \wedge d|\zeta|^2 - dz \wedge d\zeta = 
 d|\zeta|^2 \wedge d(t_0 + \frac{\tilde s}{|\zeta|})
 - d(z_0^\perp + 2\tilde s \hat \zeta) \wedge d\zeta   = dt_0 \wedge d|\zeta|^2 - dz_0^\perp \wedge d\zeta,
\end{equation}
which agrees with \eqref{eq:symplectic form W cotangent bundle 2}. 

Having established this canonical diffeomorphism, we proceed to investigate its behaviour as approaching the boundary of $\SR_\pm$. To do so, we consider the equality of the canonical one form
\begin{equation}\label{eq:can 1form}
   (2t\zeta-z) \cdot d\zeta = \xi_{\oc} \frac{dx_{\oc}}{x_{\oc}^3} + \eta_{\oc} \cdot \frac{dy_{\oc}}{x_{\oc}}.
\end{equation}
This identity allows us to identify the 1c-frequencies $\xioc$, $\etaoc$ in terms of the variables $(z_0^\perp, t_0, \tilde s, \hat \zeta, 1/|\zeta|)$ that are valid on $\Sigma$ near fibre-infinity, that is, for small values of $1/|\zeta|$. 

Without loss of generality, we consider the region where $\zeta_1$ dominates other variables and define
\begin{align}
x_{\oc} = \frac{1}{|\zeta|}, \; y_{\oc} = (\zeta_2/|\zeta|,...,\zeta_n/|\zeta|).
\end{align}
We denote $y_{\oc,j} = \zeta_j/|\zeta|$, $1 \leq j \leq n$. We also use coordinates $z_0' = ((z_0)_2, \dots, (z_0)_n)$ for $z_0^\perp$; since $z_0^\perp$ is by definition orthogonal to $\zeta$, we can compute that the first component of $z_0^{\perp}$ is 
$$
(z_0^\perp)_1 = - \frac{z_0' \cdot \yoc}{\sqrt{1 - |\yoc|^2}}.
$$
Then we can express the LHS of \eqref{eq:can 1form} using \eqref{eq:z0t0} by 
\begin{equation}\begin{aligned}
-(z_0^\perp &+ 2\tilde s \hat \zeta) \cdot d\zeta + (t_0 + \tilde s |\zeta|^{-1}) d|\zeta|^2  \\
&= -z_0^\perp \cdot d\zeta + t_0 d|\zeta|^2 \\
&=  \frac{z_0' \cdot \yoc}{\sqrt{1 - |\yoc|^2}} d\big( \frac{\sqrt{1 - |\yoc|^2}}{\xoc}\big) - z_0' \cdot d\big( \frac{\yoc}{\xoc}\big) - 2 t_0 \frac{d\xoc}{\xoc^3} \\
&= \big(-2t_0 + O(\xoc)\big) \frac{d\xoc}{\xoc^3} - \big(z_0' + \frac{(z_0' \cdot \yoc)\yoc}{1 - |\yoc|^2} \big)\frac{d\yoc}{\xoc}
\end{aligned}
\end{equation}
where the $O(\xoc)$ term is smooth. We see that 
\begin{equation}
\xioc = -2t_0 + O(\xoc), \quad \etaoc = -z_0' - \frac{(z_0' \cdot \yoc)\yoc}{1 - |\yoc|^2}.
\end{equation}
It is clear that the map from $(z_0', t_0)$ to $(\xioc, \etaoc)$ is invertible provided that $|\yoc|$ is small. 
Since we noted above that the coordinates $(t_0, z_0^\perp)$ on the fibre-interior of $W_\pm$ are valid linear coordinates uniformly up to the boundary of $\SR_\pm$, and likewise $(\xioc, \etaoc)$ are valid linear coordinates on ${}^{\oc} T^* \SR_\pm$ uniformly to the boundary of $\SR_\pm$, we see that the canonical diffeomorphism between $({}^{\oc} T^* \SR_\pm)^\circ$ and $(W_\pm)^\circ$ extends to the boundary. It follows that $W_\pm$ is canonically diffeomorphic (as a vector bundle if $\SR_\pm$ is endowed with its conic extension, and as an affine bundle if not) to ${}^{\oc} T^* \SR_\pm = {}^{\oc} T^* \overline{\R^n}$. 
\end{proof}

Using this identification, we can define the classical scattering map $\Sg$ as follows.
As pointed out in Remark~\ref{remark:Wpm-dynamics}, for each point $q \in {}^{\oc} T^* \SR_-$, under the identification with $W_-$ above, there is a unique bicharacteristic that tends to it in the backward direction.
This bicharacteristic will tends to a point $q' \in {}^{\oc} T^* \SR_+$ in the forward direction, again after identified with $W_+$ as above.
Then we define
\begin{equation} \label{eq:def-classical-sc-map}
    \Sg(q) = q'.
\end{equation}
As we will discuss in more detail in Section~\ref{sec: forward and backward sojourn relations}, this flow reaches $W_\pm$ transversely in finite time in both directions if we further renormalize the vector field as in \eqref{eq:further-rescaled-Hp}. In particular, $\Sg$ is a smooth map.
%%%%%%%%%%%%%%%%%%%%%%%%%%%%%%%%%%%%%%%%%%%%%%%%%%%%%%%%%%
%%%%%%%%%%%%%%%%%%%%%%%%%%%%%%%%%%%%%%%%%%%%%%%%%%%%%%%%%%
%%%%%%%%%%%%%%%%%%%%%%%%%%%%%%%%%%%%%%%%%%%%%%%%%%%%%%%%%%
%%%%%%%%%%%%%%%%%%%%%%%%%%%%%%%%%%%%%%%%%%%%%%%%%%%%%%%%%%
%%%%%%%%%%%%%%%%%%%%%%%%%%%%%%%%%%%%%%%%%%%%%%%%%%%%%%%%%%

\section{Geometry of the 1c-ps phase space}
\label{sec:1c-ps geometry}
\subsection{The 1c-ps cotangent bundle}
%The most natural choice of the cotangent bundle 
Recall that the Poisson operators $\Poipm$ maps a function $f_\pm$ on $\Rpm$ to the unique solution $u : \RR^{n+1} \to \C$ to $Pu = 0$ with the given final state data $f_\pm$. 
The phase space for 1c-ps Lagrangian distributions (a class of operators that we shall show includes suitably microlocalized Poisson operators) is obtained by blowing up the corner of
\begin{align} \label{eq:calM0-def-product-bundle}
\mathcal{M}_0 =  \overline{^\ps{T^*\RR^{n+1}}} \times \overline{^\oc{T^*\RR^n}} 
\end{align}
at base infinity of $\overline{^\oc{T^*\RR^n}}$ and fiber-infinity of $\overline{^\ps{T^*Y}}$. Here the $\RR^n$ factor represents the interior of either $\Rp$ or $\Rm$.  We refer readers to \cite{melrose1993atiyah}\cite{melrose1994spectral} for more details about blow ups. Concretely, we define
\begin{align} \label{eq: 1c-ps cotangent bundle definition}
\mathcal{M}:= [ \mathcal{M}_0 ; \{ \rho_{\ps} = 0, x_{\oc} =0 \} ], \quad \rhops = \big( \sum_{j,k} g^{jk}(z, t) \zeta_j \zeta_k \big)^{-1/2}, 
\end{align}
and denote the blow down map by 
\begin{align}
\beta_{\oc-\ps}: \mathcal{M} \rightarrow \mathcal{M}_0.
\end{align}
The front face created by this blow-up shall be denoted $\ffocps$. 
The new smooth coordinate on $\ffocps$ introduced by the blow up is 
\begin{equation}\label{eq:sigma}
\sigma = \frac{x_{\oc}}{\rho_{\ps}},
\end{equation}
or its reciprocal. In the interior of $\ffocps$, either $\xoc$ or $\rhops$ can be taken as a boundary defining function. 

The manifolds $\SM_0$ and $\SM$ have codimension 4 corners. However, we shall only be interested in a neighbourhood of a compact subset $K$ of the interior of $\ffocps$; in particular, we shall stay away from all the other boundary hypersurfaces. So, in effect, we are dealing with a manifold with boundary. 

The manifold $\mathcal{M}$ is endowed with a canonical symplectic structure\footnote{In this article, we allow symplectic structures to blow up or degenerate at the boundary} $\omega$ from $\mathcal{M}_0$ by lifting the symplectic form on $\mathcal{M}_0$, which in turn is equipped with the product symplectic structure from its two factors. We are particularly interested in the symplectic/contact structures on $\ffocps$ induced by this symplectic structure in the interior. To prepare for this, we shall specify coordinates to use in a neighbourhood of $K \subset \ffocps$. These will be $\xoc$ (a boundary defining function) and $\yoc$, which are base coordinates on $\RR^n$ near base infinity;  $\xioc$, $\etaoc$, their one-cusp dual coordinates as defined in Section~\ref{sec: 1c PsiDO}; $z$, $t$, Euclidean space and time coordinates;  $\sigma$ as defined in \eqref{eq:sigma}; and fibre coordinates near fibre-infinity, which we take to be $\tilde \tau = \tau \xoc^2$, $\hat \zeta = \zeta / |\zeta|$. We remark that we are also mostly interested in $(z, t)$ near the perturbation of the metric, which by assumption is a compact set in spacetime. On the other hand, $(\zeta, \tau)$ will be near infinity, since $\rhops = 0$ at $\ffocps$. To summarize, our coordinates are
\begin{equation}\label{eq:coordinates near ffocps}
\xoc, \quad \yoc, \quad \xioc, \quad \etaoc, \quad z, \quad t, \quad \sigma, \quad \tilde \tau, \quad \hat \zeta.
\end{equation}

\begin{remark}\label{rem: ffocps coordinates}
Notice that if we define $\tilde \zeta = \xoc \zeta$, then we have an identity $1 = \sigma^{-4} (\tilde \tau^2 + |\tilde \zeta|^4)$, so we need to drop one of the $(\sigma, \tilde \tau, \tilde \zeta)$ variables to have linearly independent coordinates. We choose to drop $|\tilde \zeta|$ and retain $\hat \zeta$, which leads to a valid coordinate system provided that $\zeta$ is a dominant variable out of $(\zeta, \tau)$, in the parabolic sense, i.e. $|\zeta|^2 \geq c |\tau|$ for some $c > 0$ locally. Since we are primarily interested in a neighbourhood of the characteristic variety of $P$, where this is true near fibre-infinity, this choice is justified. Another equally valid choice, since $\tau$ is also dominant in a neighbourhood of the characteristic variety of $P$, is to drop $\tilde \tau$ and retain $\tilde \zeta$; we do this in \eqref{eq:fibresympform} below, for example. 
\end{remark}

\subsection{\texorpdfstring{Symplectic structures and Legendre submanifolds of $\ffocps$}{Symplectic structures and Legendre submanifolds of the 1c-ps front face}}\label{subsec:symplectic structures 1cps}
We begin our discussion by remarking on the invariance properties of the coordinates \eqref{eq:coordinates near ffocps}. Since we have the flat Schr\"odinger operator $D_t + \Delta_0$ at the boundary of spacetime infinity, we consider the symmetries of this operator. These consist of spatial and time translations, spatial rotations and Galilean boosts. First consider the effect such transformations have on the coordinate $Z = z/t$ on the interior of $\Rpm$. It is not hard to check that spatial and time translations act as the identity on $Z$, spatial rotations act by the `same' rotation on $Z$ and Galilean boosts act by translations in $Z$. The effect on the coordinate $\xoc = 1/|Z|$ is to map $\xoc \mapsto \xoc + O(\xoc^2)$. In particular, the differential $d\xoc$ at $\xoc = 0$ is invariant. 

The coordinate $t$ is canonically defined, up to translation, by these symmetries (note spatial rotations and Galilean boosts act trivially on $t$). It follows that $\tau$, its dual coordinate, is canonically defined. Since the differential $d\xoc$ at $\xoc = 0$ is invariantly defined (by definition in the 1c-calculus, see Section~\ref{subsec: 1c bundles}), it follows that $\tilde \tau$ is invariantly defined at $\ffocps$. Moreover, the invariance of $d\xoc$ at $\xoc = 0$ implies that the coordinate $\xioc$ is invariantly defined at $\ffocps$. 
It follows that there is a well-defined projection map 
\begin{equation}\label{eq:pi defn}
\pi : \ffocps \to \RR^3_{t, \tilde \tau, \xioc},
\end{equation}
up to translations in $t$, with $(4n-2)$-dimensional fibres. 

\begin{lmm}\label{lem:symplectic geom 1cps} The symplectic form $\omega$ in the interior of $\SM$, contracted with the vector field $-\xoc^3 \partial_{\xoc}$ (with respect to the coordinate system above), restricts to a smooth 1-form on $\ffocps$, which is the pullback, via $\pi$, of a contact 1-form on $\RR^3_{t, \tilde \tau, \xioc}$. Moreover, on each fibre of $\pi$, the symplectic form induces a symplectic structure by contracting with the vector field $-\xoc^2 \partial_{\xoc}$, restricting to $\ffocps$, and then taking the differential. 
\end{lmm}

\begin{proof}
Using our coordinates \eqref{eq:coordinates near ffocps},  the element of $\mathcal{M}$ at the point $(\xoc, \yoc, \xioc, \etaoc, z, t, \sigma, \tilde \tau, \hat \zeta)$, viewed as a cotangent vector, is 
$$
\xioc \frac{d\xoc}{\xoc^3} + \etaoc \frac{d\yoc}{\xoc} + \frac{\tilde \tau}{\xoc^2} dt + \frac{\sigma \hat\zeta}{\xoc} dz.
$$
This is then also the expression for the canonical one-form on $\mathcal{M}$. The symplectic form is thus given by taking the differential of this one-form: 
\begin{equation}\label{eq:omega1cps}
\omega = d\xioc \wedge \frac{d\xoc}{\xoc^3} + d(\frac{\etaoc}{\xoc}) \wedge d\yoc + d(\frac{\tilde \tau}{\xoc^2}) \wedge dt + d(\frac{\sigma \hat \zeta}{\xoc}) \wedge dz. 
\end{equation}
Contracting with $-\xoc^3 \partial_{\xoc}$ we obtain 
\begin{equation}\label{eq:3dcontactform}
d\xioc + 2 \tilde \tau dt + O(\xoc),
\end{equation}
which restricts to $\{ \xoc = 0 \}$ to the contact form $d\xioc +2 \tilde \tau dt$ on $\RR^3$ pulled back via $\pi$. Furthermore, if we restrict to a fibre, both $t$ and $\xioc$ are fixed, which kills the $(\xoc)^{-3} d\xoc$ terms and means that the contraction with the larger vector field $-\xoc^2 \partial_{\xoc}$ has a smooth limit at $\xoc = 0$. This gives the 1-form (writing $\sigma \hat \zeta = \tilde \zeta$, defined  in Remark~\ref{rem: ffocps coordinates})
\begin{equation}\label{eq:fibre1form}
\alpha = \etaoc \cdot d\yoc + \tilde \zeta \cdot dz.
\end{equation}
Taking the differential gives 
\begin{equation}\label{eq:fibresympform}
d\etaoc \wedge d\yoc + d\tilde \zeta \wedge dz, 
\end{equation}
which is a symplectic form. 
\end{proof}

%Let $\omega_{\pm,\oc-\ps}=\omega_{\oc,X}\pm\omega_{\ps,Y}$ (ps stands for parabolic scattering) be the symplectic forms obtained from the natural ones on $T^*X$ and $T^*Y$ by extension.

\begin{defn}[Admissible 1c-ps Lagrangian submanifold and 1c-ps fibred-Legendre submanifold]  \label{defn: admissible 1c-ps Lagrangian submanifold and 1c-ps fibred-Legendre submanifold}
We define an admissible 1c-ps Lagrangian submanifold of $\SM$ to be a $2n+1$-dimensional submanifold $\Lambda$ that is Lagrangian in the interior (the canonical symplectic form $\omega$ vanishes on it), such that 
\begin{itemize}
\item $\Lambda$ meets $\ffocps$ transversally, 
\item the differential $dt$ is non-vanishing on $\Lambda \cap \ffocps$, and 
\item  its closure is disjoint from all other boundary hypersurfaces of $\SM$ (other than $\ffocps$). 
\end{itemize}

We define a fibred-Legendre submanifold $L$ of $\ffocps$ to be the boundary of an admissible 1c-ps Lagrangian submanifold. In other words, there exists $\Lambda$ as above such that $L = \Lambda \cap \ffocps$. 
The justification for the term `fibred-Legendre' comes from the following geometric result. 
%Legendre submanifold of $\ffocps$ to be the intersection, with $\ffocps$, of a $2n+1$-dimensional Lagrangian submanifold $\Lambda \subset \SM$ that meets $\ffocps$ transversally (and hence in a submanifold of dimension $2n$). 
%We say that $\Lambda$, as well as $\Lambda \cap \ffocps$ is admissible if its closure is disjoint from all other boundary hypersurfaces of $\SM$. We shall only consider admissible Legendre submanifolds in this article, and usually omit the adjective `admissible' in the sequel.
\end{defn}

\begin{prop}\label{prop:1c-ps contact symplectic property} Let $L$ be a fibred-Legendre submanifold of $\ffocps$, and let $\pi_L$ be the restriction of $\pi$ to $L$, where $\pi$ is as in \eqref{eq:pi defn}. Then $\pi_L(L)$ is a Legendre submanifold (curve, in this case) of $\RR^3_{t, \tilde \tau, \xioc}$, and the fibres of $\pi_L$ are Lagrangian submanifolds for the symplectic structure on the fibres of $\pi$. 
\end{prop}

\begin{proof} By definition, there is an admissible 1c-ps Lagrangian submanifold $\Lambda$ such that $L = \partial \Lambda = \Lambda \cap \ffocps$. Since $\Lambda$ is Lagrangian, the 1-form $\omega \lrcorner  W$ vanishes on $\Lambda$, for any choice of vector field $W$ tangent to $\Lambda$. We choose $W$ to be a vector field transverse to $L = \partial \Lambda$, in fact, without loss of generality we take $W$ to have $\partial_{\xoc}$ component equal to $1$. This is possible due to the first and third conditions in Definition~\ref{defn: admissible 1c-ps Lagrangian submanifold and 1c-ps fibred-Legendre submanifold}. 

Thus we write $W = \partial_{\xoc} + \tilde W$, where $\tilde W$ has no $\partial_{\xoc}$ component. We compute, using \eqref{eq:3dcontactform}, 
\begin{multline}\label{eq:3dcontactform2}
0 = \omega \lrcorner (-\xoc^3 W) = \omega \lrcorner (-\xoc^3 \partial_{\xoc}) + \omega \lrcorner (-\xoc^3 \tilde W) \\ = 
d\xioc +2 \tilde \tau dt + \beta d\xoc + O(\xoc), \quad \beta \in C^\infty(\Lambda). 
\end{multline}
Restricting to $\xoc = 0$ we find that 
\begin{equation}\label{eq:3dcontactform3}
d\xioc +2 \tilde \tau dt = 0 \text{ on } L. 
\end{equation}

Let $p \in L$ be a point in $L$, and consider the rank of the map $d\pi_L$ at $p$. Clearly, by \eqref{eq:3dcontactform3}, the rank cannot be $3$ at any point of $L$. Suppose that the rank at $p$ is $2$. Then, the rank must be equal to $2$ in a neighbourhood of $p$ --- it cannot jump up to $3$, and it is at least $2$ in a neighbourhood by continuity. Therefore, by the constant rank theorem, $\pi_L$ would map $L$ to a 2-dimensional submanifold (locally) on which the form $d\xioc +2 \tilde \tau dt$ (now viewed as a form on $\RR^3$) vanishes. But this is impossible as this form is a contact form, and hence the maximal dimension of a submanifold on which it vanishes is $1$. It follows that the rank of $\pi_L$ is at most $1$ at every point of $L$.

On the other hand, the second assumption on $\Lambda$ in Definition~\ref{defn: admissible 1c-ps Lagrangian submanifold and 1c-ps fibred-Legendre submanifold} is that $dt$ is non-vanishing on $L$. It follows that $d\pi_L$ has rank at \emph{least} one at each point $p$. Combined with the previous paragraph, we see that the rank of $\pi_L$ is exactly $1$ at each $p \in L$. Applying the constant rank theorem, we find that the image $\pi_L(L)$ is a curve in $\RR^3_{t, \tilde \tau, \xioc}$ on which the contact form $d\xioc -2 \tilde \tau dt$ vanishes, i.e. a Legendre curve. Moreover, the fibres are smooth $(2n-1)$-dimensional submanifolds of the fibres of $\pi$. Let $F$ be an arbitrary fibre of $\pi_L$, given by $L \cap \{ t = t^*, \xioc = \xioc^*, \tilde \tau = {\tilde\tau}^* \}$. 

%We next consider the projection $\tilde \pi$ from $L$ to the remaining coordinates $\yoc$, $\etaoc$, $z$, $\sigma$, $\hat \zeta$; thus, the image of $d \tilde \pi$ is the tangent space to the fibres. Let $S$ be the image of $T_p L$ under $d \tilde \pi$ at a point $p \in L$. We claim that $\alpha$ vanishes on $S$, where $\alpha$ is the one-form in \eqref{eq:fibre1form}. To see this, 
%It follows that $S$ has dimension 

%Next suppose that the rank of $d\pi_L$ at $p$ is zero; we shall derive a contradiction from this assumption. In that case, we consider the intersection of $L$ with the fibre of $\pi$ through $p$. If we consider a coordinate projection $\tilde \pi$ from $L$ to the remaining coordinates $\yoc$, $\etaoc$, $z$, $\sigma$, $\hat \zeta$, then the sum of the ranks of $d\pi_L$ and $d\tilde \pi$ at $p$ is at least $2n$ (since their kernels intersect trivially), so when the rank of $d\pi_L$ vanishes at one point, then the rank of $d\tilde \pi$ must be $2n$, which is the maximum possible rank, since $L$ has dimension $2n$. The rank of $d\tilde \pi$ remains at $2n$ locally --- it cannot jump down, and as it is maximal, it cannot jump up either. Thus, the intersection of $L$ with a fibre is, locally, a submanifold $S$ of dimension $2n$. 

We claim that the 1-form $\alpha$ given by \eqref{eq:fibre1form}, restricted to any fibre $F$, is the differential of a function on that fibre. That implies that $d\alpha$ restricts to zero on $F$ (since the exterior derivative $d$ and restriction commute) and thus shows that $F$ is a Lagrangian submanifold of the corresponding fibre of $\ffocps$. 

To prove this claim, we consider the submanifold $\Lambda_{t^*} := \Lambda \cap \{ t = t^* \}$. This is a smooth submanifold of dimension $2n$ of $\Lambda$, since $dt$ is nonvanishing on $\Lambda$ at $p$, and hence in a neighbourhood.  Moreover, if we fix $t^*$ and let $\xoc$ vary, then since $d\xoc$ is nonvanishing at $p$, and hence in a neighbourhood, $\Lambda \cap \{ t = t^* \}$ fibres over $[0, \epsilon]_{\xoc}$ with smoothly varying fibre $\Lambda_{t^*, \xoc^*} = \Lambda \cap \{ t = t^* , \xoc = \xoc^* \}$, $\xoc^* \in [0, \epsilon]$, by the constant rank theorem. 

Given a smooth vector field  $V$ tangent to $F \subset \Lambda_{t^*, 0}$, we can extend it to a smooth vector field $\tilde V$ on  $\Lambda_{t^*}$, tangent to $\Lambda_{t^*, \xoc}$ for all sufficiently small $\xoc$. Then, we consider how $\xioc$  varies on $\Lambda_{t^*, \xoc}$ for $\xoc > 0$. We observe that $\xioc$ is constant on $\Lambda_{t^*, 0}$. This follows from \eqref{eq:3dcontactform3}, since $\Lambda_{t^*, 0}$ is contained in $L$ and so $d\xioc + 2 \tilde \tau dt = 0$ on it. Thus since $dt$ vanishes identically on $\Lambda_{t^*, 0}$, so does $d\xioc$, so it is constant, equal to  $\xioc^*$, say. Therefore on $\Lambda_{t^*}$, there is a smooth function $\Xi_{\oc}$ such that we can express 
\begin{equation}\label{eq:xioc-expand}
\xioc = \xioc^* + \xoc \Xi_{\oc}  \  \text{ on } \Lambda_{t^*}. 
\end{equation}

We next choose a vector field $W$ as before, but tangent to $\Lambda_{t^*}$ (and not merely $\Lambda$) with $\partial_{\xoc}$ component equal to $1$. This is possible because of the assumed transversality of $\Lambda$, and hence also $\Lambda_{t^*}$, to $\{ \xoc = 0 \}$. We write $W = \partial_{\xoc} + \tilde W$, where $\tilde W$ has no $\partial_{\xoc}$ component nor $\partial_t$ component. 

Since $\omega$ vanishes on $\Lambda$, both $\tilde V$ and $W$ are tangent to $\Lambda$, and neither $\tilde V$ nor $\tilde W$ have a $\partial_{\xoc}$ or $\partial_t$ component,  we have, using \eqref{eq:xioc-expand} and the formula \eqref{eq:fibre1form}, 
\begin{equation}\begin{gathered}
0 = \omega(\xoc^2 W, \tilde V) =  \omega(\xoc^2 \partial_{\xoc}, \tilde V) + \omega(\xoc^2 \tilde W, \tilde V)  \\
= \Big( -\frac{d\xioc}{\xoc} -  \alpha \Big)(\tilde V) + O(\xoc) \\
= (-d\Xi_{\oc} - \alpha)(\tilde V) + O(\xoc).
\end{gathered}\end{equation}
Taking the limit as $\xoc \to 0$ we find that 
$$
\alpha = -d\Xi_{\oc}  \ \text{ on } F. 
$$
Thus, $d\alpha$ vanishes on $F$ and $F$ has dimension $2n-1$, that is, $F$ is a Lagrangian submanifold of the corresponding fibre of $\ffocps$.

 \end{proof}

\subsection{Parametrization of 1c-ps fibred-Legendre submanifolds}\label{subsec:Parametrization of 1c-ps fibred-Legendre submanifolds}
%[TODO: (0,K') is a bit misleading here, any good way to express this?]
By a `parametrization' of a 1c-ps fibred Legendre submanifold, we mean a representation of it (locally) as the graph of the differential of a function, or an envelope of such graphs. To explain this we first make the simplifying assumption that 
$(\yoc, z)$ furnish coordinates on the fibres $F$ of $\pi_L$, locally. Since, as we saw above, $dt$ is non-vanishing on $\pi_L(L)$, we can index the fibre $F$ by the value of $t$, i.e. we write $F_{t^*}$ when $t=t^*$ on $F$. (Moreover, since $\pi_L(L)$ is one-dimensional, and $t$ is locally a coordinate on it, we can write $\xioc$ and $\tilde \tau$, restricted to $\pi_L(L)$,  as functions of $t$; in particular, $\xioc = -2\phi_0(t)$ for some smooth function $\phi_0(t)$.)   With this simplifying assumption, the function $\xioc$ restricted to $\Lambda_{t^*}$ can be expressed in terms of the coordinates $(\xoc, \yoc, z)$, and more generally, $\xioc$ restricted to $\Lambda$ can be expressed in terms of the coordinates $(\xoc, \yoc, t, z)$. We write 
$$
\xioc(\xoc, \yoc, t, z) = \Phi(\xoc, \yoc, t, z) = -2\phi_0(t) - \xoc \phi_1(\yoc, t, z) + O(\xoc^2).
$$
Thus $-\phi_1$, evaluated at $t=t^*$, coincides with the function $\Xi_{\oc}$ above. 
We claim that $L$ is given by the graph of the differential of the function 
$$
d\Big( \frac{\phi_0(t) + \xoc \phi_1(\yoc, t, z)}{\xoc^2} \Big) = d \big( \frac{\phi_0(t)}{\xoc^2} \big) + d \big( \frac{\phi_1(\yoc, t, z)}{\xoc} \big),
$$
at $\xoc = 0$, in the sense that 
\begin{equation}\begin{gathered}
L = \Big\{ (\xioc, \yoc, \etaoc; \tilde \tau, \sigma, \hat \zeta, z, t) \mid \xioc = -2 \phi_0(t) , \ \\
\etaoc = d_{\yoc} \phi_1(\yoc, t, z), \ \tilde \tau = d_t \phi_0(t), \ \tilde \zeta = d_z \phi_1(\yoc, t, z) \Big\}.
\end{gathered}\label{eq:Lparam}\end{equation}

We express \eqref{eq:Lparam} more succinctly by saying that 
$$
L = \Big\{ (K, d_{K} \big( \frac{\Phi(K)}{\xoc^2} \big) \mid \xoc = 0 \Big\}, \quad K = (\xoc, \yoc, t, z). 
%\text{ graph } d \Big( \frac{Phi}{\xoc^2} \Big) \Big|_{\xoc = 0 \}.
$$
We verify this claim using Proposition~\ref{prop:1c-ps contact symplectic property}. The first equation $\xioc = -2 \phi_0(t)$ is true by construction. Then, since $d\xioc +2 \tilde \tau dt$ vanishes on $L$, this implies that $\tilde \tau = d_t \phi_0(t)$. Finally, since we know that $\alpha = d\phi_1(\yoc, t, z)$ and we have the expression \eqref{eq:fibre1form} of it, it follows that $\etaoc = d_{\yoc} \phi_1$ and $\sigma \hat \zeta = d_z \phi_1$ on $L$. 

In the general case, where we do not assume that $(\yoc, z)$ furnish coordinates on the fibres $F$, we need to allow additional variables in our parametrizing function:

\begin{defn} \label{defn: 1c-ps parametrization}
Suppose $\Legps$ is an 1c-ps fibred-Legendre submanifold of $\ffocps \subset \mathcal{M}$, $q \in \Legps$ and $(\theta_0, \theta_1)$ is in a non-empty open subset $U$ of $\RR^{k_1+k_2}$, then we say that 
\begin{align}
\Phi_{\oc-\ps} = -2 \frac{ \varphi_0(t,\theta_0) }{x_{\oc}^2} - \frac{\varphi_1(\msf{K},\theta_0,\theta_1)}{x_{\oc}} 
\end{align}
gives a non-degenerate parametrization of $\Legps$ near $q$ if there is a point $q' \in \RR^{2n}_{\yoc, z, t} \times U_{\theta_0, \theta_1}$ such that the differentials
\begin{equation}\label{eq:varphi0}
d_{t,\theta_0} \frac{\partial \varphi_0}{\partial \theta_{0j}}, \quad 1 \leq j \leq k_1,
\end{equation}
and
\begin{equation}\label{eq:varphi1}
d_{\yoc, z, \theta_1} \frac{\partial \varphi_1}{\partial \theta_{1j}}, 1 \leq j \leq k_2,
\end{equation}
are linearly independent at $q' \in \RR^{2n}_{\yoc, z, t} \times U$, such that $\Legps$ is given locally by 
\begin{align}\label{eq:1c-ps param}
\Legps
=\{ (\msf{K},d_{\msf{K}}\Phi_{\oc-\ps}) \mid \; \xoc = 0, (\yoc, z, t,\theta_0,\theta_1) \in C_{\Phi_{\oc-\ps}}\},
\end{align}
where 
\begin{align}  \label{eq: 1c-ps critical set}
C_{\Phi_{\oc-\ps}} =  \{ (\yoc, z, t,\theta_0,\theta_1) \mid d_{\theta_0}\varphi_0 = 0, d_{\theta_1}\varphi_1 = 0 \}
\end{align}
with the point $q' \in C_{\Phi_{\oc-\ps}}$ corresponding under \eqref{eq:1c-ps param} to $q \in L$. We remark that either $\theta_0$ or $\theta_1$, or both, may be absent, in which case the linear independence conditions for the derivatives in \eqref{eq:varphi0}, \eqref{eq:varphi1} are dispensed with, as well as stationarity with respect to $\theta_0$ and/or $\theta_1$ in \eqref{eq: 1c-ps critical set}. 
\end{defn}

\begin{prop}\label{prop: 1c-ps nondeg param}
Every 1c-ps fibred-Legendre submanifold $L$ has a local non-degenerate parametrization near any point $q \in L$.
\end{prop}

\begin{proof}
In the case where $(\yoc, z)$ furnish coordinates on the fibres $F$, we have seen above that there is a parametrization of $L$ with no additional variables $\theta_0, \theta_1$. 

In general, one cannot assume that $(\yoc, z)$ furnish coordinates on the fibres $F$. It is, however, possible to choose coordinates $(\yoc, z) = (\yoc', \yoc'', z', z'')$ (after a linear change of variables), where $\yoc = (\yoc', \yoc'')$ is a splitting of the coordinates into two groups, so that, with $\tilde \zeta = \xoc \zeta$ as before, and writing the `dual' variables $(\etaoc', \etaoc'', \tilde\zeta', \tilde \zeta'')$ (these are not exactly dual variables because of the scaling in $\xoc$), the functions
$$
\yoc', \ \etaoc'', \ z', \ \tilde\zeta''
$$
furnish coordinates locally on the fibres of $\pi_L$. We omit the proof, as it is essentially contained in \cite[Theorem 21.2.17]{hormander2007analysis}. Therefore, $\xoc, \ \yoc', \ \etaoc'', \ z', \ \tilde\zeta''$ furnish coordinates locally on $\Lagps$, the ambient Lagrangian submanifold for $\Legps$.

We write $\xioc = - 2 \phi_0(t)$ as before on $\pi_{\Legps}(\Legps)$; thus we have $\xioc = - 2\phi_0(t) - \xoc \phi_1$ on $\Lambda$, for a suitable function $\phi_1$. 
We write the functions $\yoc'' = Y''(\yoc', \etaoc'', z', \zetaoc'')$ and $z'' = Z''(\yoc', \etaoc'', z', \zetaoc'')$ on $F$, and $\phi_1(0, \yoc', \etaoc'', z', \zetaoc'') = X(\yoc', \etaoc'', z', \zetaoc'')$. Then 
following the argument in \cite[Prop. 21.2.18]{hormander2007analysis}, we find that the function
\begin{equation}\label{eq:Phi-param-general}
\Phi(\msf{K}, \etaoc'', \zetaoc'') = \frac{\phi_0(t)}{\xoc^2} + \frac{ X(\yoc', \etaoc'', z', \zetaoc'') + (\yoc'' - Y'') \cdot \etaoc'' + (z'' - Z'') \cdot \zetaoc'' }{\xoc} 
\end{equation}
provides a local parametrization of $\Legps$, in the sense that 
\begin{equation}\label{eq:Legps param}
\Legps = \Big\{ (\msf{K}, d_{\msf{K}} \Phi(\msf{K}, \theta_1) ) \mid \xoc = 0, d_{\theta_1} \Phi = 0 \Big\}, \quad \msf{K} = (\xoc, \yoc, t, z), \ \theta_1 = (\etaoc'', \zetaoc'') .
\end{equation}
\end{proof}

\begin{remark} We proved the existence of a parametrization with the $\theta_0$ variables absent, so the reader might wonder why we bothered to include them. The answer is that although one can always parametrize a given Legendrian without the $\theta_0$ variables, in studying the composition of operators, such variables naturally appear, so it is better to include them from the outset.
\end{remark}

\begin{remark} One can wonder whether one should think of the phase function $\Phi$ as parametrizing not just $\Legps$, but the whole Lagrangian $\Lagps$. Indeed, one can parametrize $\Lagps$; to do so, then the function $X$ in \eqref{eq:Phi-param-general} would need to depend on $\xoc$, that is, we would need the full function $\phi_1$ and not just its restriction to $\xoc = 0$, as in \eqref{eq:Phi-param-general}. However, the $O(\xoc)$ part of $\phi_1$ would only contribute a smooth factor to the exponential and is more conveniently viewed as part of the amplitude, rather than the phase function. 

This highlights a difference between the Lagrangian $\Lagps$ and the Lagrangians of the original H\"ormander theory of Lagrangian distributions \cite{FIO1}: our Lagrangian $\Lagps$ is not conic, and therefore is not completely determined by its boundary $\Legps$. 
It is only the Legendrian $\Legps$ that is of microlocal significance; thus, the higher terms in the Taylor series of $\phi$, for example, are arbitrary and have no importance in the study of microlocal propagation. 

By contrast, there is a one-to-one correspondence between H\"ormander's Lagrangian submanifolds and their boundary at fibre-infinity, so it makes no difference whether one thinks of them as conic Lagrangian submanifolds or Legendre submanifolds of the induced contact structure at fibre-infinity. Here, due to the lack of such a one-to-one correspondence, we prefer to focus on the fibred-Legendre submanifold $L$, which is where the microlocal analysis (propagation) truly takes place. A similar choice was made in \cite{melrose1996scattering}, presumably for the same reason. 
\end{remark}

\begin{remark} There is a converse to Proposition~\ref{prop:1c-ps contact symplectic property}. That is, if a $2n$-dimensional subset $\Legps$ of $\ffocps$ has the properties deduced in Proposition~\ref{prop:1c-ps contact symplectic property}, then it is a 1c-ps fibred Legendre submanifold, i.e. has a (non-unique) extension to a $(2n+1)$-dimensional Lagrangian submanifold $\Lagps$. To obtain it, we can parametrize $\Legps$, and then define $\Lagps$ by \eqref{eq:Legps param} discarding the restriction to $\xoc = 0$. This requires knowing the linear term in the Taylor series of the function $\xioc$, but we observe that it is determined, up to a constant depending only on $t$, by the function $\Xi_{\oc}$ such that $\alpha = d\Xi_{\oc}$ on the fibres of $\pi_L$. The indeterminacy up to a constant reflects the non-uniqueness of the extension, and is to be expected. 
\end{remark}

\subsection{The forward and backward sojourn relations}
\label{sec: forward and backward sojourn relations}
It is shown in \cite[Section~3]{gell2022propagation} that the flow of $\msf{H}_p^{2,0}$ has a global source-sink structure with source denoted by $\mathcal{R}_-$ and sink denoted by $\mathcal{R}_+$; recall these radial sets $\Rpm$ are defined as the zeroes of the Hamilton vector field within the characteristic variety $\Sigma$, and are given explicitly by \eqref{eq:radial defn}. 

In this subsection we define and analyze the \emph{forward and backward sojourn relations} which we view as subsets of $\SM$. To define these,  recall that $W_\pm$ were defined in %Section~\ref{sec:1c-ps geometry} 
Section~\ref{subsec: 1c arises} to be the front face when  $\SR = \SR_+ \cup \SR_-$ is blown up inside $\Sigma$; $W_\pm$ was shown to naturally isomorphic to the one-cusp cotangent bundle of $\Rpm$. Earlier, in Section~\ref{sec: 1c PsiDO}, we found an expression for the rescaled Hamilton vector field $H_p^{2,0}$ in local coordinates near $\Rpm$. Away from the equatorial region, where we used coordinates $x_{\ps}, \tilde w, \zeta, \tau$, the Hamilton vector field took the form 
\begin{equation}
\frac{(-\sgn t)}{\ang{\zeta}}  \Big( \tilde w \cdot \partial_{\tilde w} + x_{\ps} \partial_{x_{\ps}} \Big),
\end{equation}
as in \eqref{eq:Hp polar}, while in the equatorial region, where we used coordinates $(\tilde s, \tilde v_j, x_{\ps}, \rho_{\ps}, \omega_j, \sigma)$, it takes the form 
\begin{equation}
-2\tilde s \partial_{\tilde s} -2  \sum_{j \geq 2} \tilde v_j \partial_{\tilde v_j} - 2  x_{\ps} \partial_{x_{\ps}},
\end{equation}
as in \eqref{eq:Hvf Rplus equator}. When we blow up $\Rpm$, the new boundary defining function $\rho$ could be taken to be $\sqrt{|\tilde w|^2 + x_{\ps}^2}$ in the first case and $\sqrt{\tilde s^2 + |\tilde v|^2 + x_{\ps}^2}$ in the second case. The remaining coordinates can be taken to be projective coordinates, i.e. suitable ratios of the coordinates participating in the blowup. 
In either case, the lifted  Hamilton vector field will take the form  
$$
\rho \partial_\rho
$$
in the new coordinates. We now further renormalize the Hamilton vector field $H_p^{2,0}$ on $[\Sigma; \SR]$ by dividing to be 
\begin{equation} \label{eq:further-rescaled-Hp}
    \rho^{-1} H_p^{2,0},
\end{equation}
where now $\rho$ is taken to be a boundary defining function for both $W_+$ and $W_-$ (the two front faces of this blowup). 

Our calculations show that $ H_p^{2,0}$ is a smooth vector field tangent to the boundary of $[\Sigma; \SR]$ except at $W_\pm$, where it is transverse: inward-pointing at $\Rm$ and outward-pointing at $\Rp$. Using the non-trapping assumption we see that each point of $W_\pm$ gives rise to a smooth integral curve of  $H_p^{2,0}$, that travels from $W_-$ to $W_+$ in finite parameter time $s$. Let $q_-$ be a point of $W_-$, and let  $\gamma_{q_-}(s)$ be the integral curve of $H_p^{2,0}$ emanating from $q_-$ at time $s=0$, and arriving at $W_+$ at time $T(q_-) > 0$. Similarly, let $q_+$ be a point of $W_+$, and let $\mu_{q_+}(s)$ be the integral curve of $H_p^{2,0}$ emanating from $q_+$ at time $s=0$, and arriving at time $T'(q_+) < 0$. 

\begin{defn}\label{defn:sojoun relations}
The forward sojourn relation is the subset of $\SM_0 =  \overline{^\ps{T^*\RR^{n+1}}} \times \overline{^\oc T^*{\Rm}} $ defined by 
\begin{equation}\label{eq:FSR def}
\FSR = \{ (\gamma_{q_-}(s), q_-) \mid q_- \in W_- , \ s \in [0, T(q_-)] \}. 
\end{equation}
Similarly, the backward sojourn relation is defined by 
\begin{equation}\label{eq:BSR def}
\BSR = \{ (\mu_{q_+}(s), q_+) \mid q_+ \in W_+ , \ s \in [T'(q_+), 0] \}. 
\end{equation}
\end{defn}

\begin{prop} The forward and backward sojourn relations are smooth Lagrangian submanifolds in the interior of $\SM_0$ for the symplectic form $\omega_{\ps} - \omega_{\oc}$. Moreover, in the case of the free Schr\"odinger operator $P_0$, this relations $\overline{(\Lambda_0)_\pm'}$ are the same for both choices of sign, and agree with the canonical relation of the free Poisson operator as computed in \eqref{eq: Poisson zero}. 
\end{prop}

\begin{proof}
Smoothness follows from the smoothness of $H_p^{2,0}$ and its non-vanishing everywhere on $[\Sigma; \SR]$. In the free case $p = p_0$, the Lagrangian property clearly follows if we show the equality with the canonical relation \eqref{eq: Poisson zero}, so we consider a point $(z, t, \zeta, -|\zeta|^2; Z, \mathfrak{Z})$ in this canonical relation; we may assume that $|Z| < \infty$ to show the Lagrangian property. We recall from the discussion in Section~\ref{subsec: 1c arises} that the 1c-frequency coordinate $\mathfrak{Z}$ is equal to the limiting value of $2t\zeta - z$ along any curve that tends to the corresponding point of $W_\pm$, where we have used the canonical identification of $W_\pm$ (the interior of the front face of the blowup $[\Sigma; \Rpm]$) with the cotangent bundle of $\Rpm$. We may choose this curve to be a bicharacteristic, in which case $2t\zeta - z$ is constant along the curve. 
Thus this bicharacteristic is of the form $(-\mathfrak{Z} + 2t \zeta, t, \zeta, -|\zeta|^2)$, and to reach the limiting point $Z \in \Rpm$, we require that $\zeta = Z$. Now if we relabel $-\mathfrak{Z} + 2t \zeta = z$, the bicharacteristic $(z, t, Z, -|Z|^2)$ corresponds to the point $(Z, 2tZ - z) \in W_\pm$. Now comparing with \eqref{eq: Poisson zero} we see that this means that  $(z, t, \zeta, -|\zeta|^2; Z, \mathfrak{Z})$ is in the canonical relation corresponding to the free Poisson operator if and only if $(z, t, \zeta, -|\zeta|^2)$ is on the bicharacteristic that corresponds to $(Z, \mathfrak{Z})$. 

To prove the Lagrangian property in the general case, observe that the forward sojourn relation agrees with the free forward sojourn relation in a neighbourhood of $W_-$. We can then take a cross section $T$ to the flow, lying over the interior of spacetime, and view the forward sojourn relation as the flowout from $T$ by $H_p$. Since $T$ is isotropic of dimension $2n$, one less than half the dimension of the total space, and is contained in $\{ p = 0 \}$, the flowout from $T$ by $H_p$ is Lagrangian. 
\end{proof}

We also define the Lagrangians $\overline{\Lambda_\pm}$ to be the forward/backward sojourn relations with $q_\pm$ replaced by $-q_\pm$, that is, with the 1-cusp fibre coordinate replaced by its negative. This is clearly Lagrangian in the interior of $\SM_0$ with respect to the \emph{sum} of the symplectic forms $\omega_{\ps} + \omega_{\oc}$. 

Our next task is to show that suitably microlocalized versions of $\overline{\Lambda_\pm}$ are admissible Lagrangian submanifolds in the sense of Definition~\ref{defn: admissible 1c-ps Lagrangian submanifold and 1c-ps fibred-Legendre submanifold}. Let $U_\pm \subset W_\pm \sim \overline{ ^{\oc}T^* \R^n} $ be open sets disjoint from fibre-infinity, and let $G_\pm \subset \overline{ ^{\ps}T^* \R^{n+1}}$ be open sets disjoint from spacetime infinity. Consider the microlocalized Lagrangians 
\begin{equation}\label{eq:microlocalized sojourn relns}
\Lambda_\pm = \beta_{\oc-\ps}^*\big(\{ (\gamma_{q_\pm}(s), -q_{\pm}) \in \overline{\Lambda_\pm} \mid q_{\pm} \in U_\pm, \ \gamma_{q_\pm}(s) \in G_\pm  \}\big)
\end{equation}
which we now view as being in $\SM$ rather than $\SM_0$, that is, we take these sets to be the closure, in $\SM$, of their interiors lifted to $\SM$ from $\SM_0$ via the blowdown map $\beta_{\oc-\ps}$. 

Obviously $\Lambda_\pm$ depend on the choice of $U_\pm$ and $G_\pm$ but we do not indicate this in the notation, regarding these choices of open sets as fixed.

\begin{prop}\label{prop:microlocalized FBSR} 
The microlocalized Lagrangians $\Lambda_\pm$ are admissible Lagrangian submanifolds in the sense of Definition~\ref{defn: admissible 1c-ps Lagrangian submanifold and 1c-ps fibred-Legendre submanifold}.
\end{prop}

\begin{proof} We have already shown that $\overline{\Lambda_\pm} \supset \Lambda_\pm$ are smooth Lagrangian submanifolds in the interior of $\SM_0$. It remains to check smoothness at the boundary, and to verify that $\Lambda_\pm$ satisfy the three conditions in Definition~\ref{defn: admissible 1c-ps Lagrangian submanifold and 1c-ps fibred-Legendre submanifold}. We treat the forward sojourn relation; the argument for the backward sojourn relation is similar. 

Due to the microlocalization to open sets $U_-$ and $G_-$, $\Lambda_-$ does not meet the boundary hypersurfaces at $\rho_{\oc} = 0$ or at $x_{\ps} = 0$. It remains to consider the boundaries at $\rho_{\ps} = 0$ or $\xoc = 0$. There are three such boundary hypersurfaces, since the codimension two hypersurface $\{ \rho_{\ps} = 0, \xoc = 0 \}$ was blown up to create $\ffocps \subset \SM$. 

We can parametrize the sojourn relation by coordinates on $W_-$, together with the flow parameter. We first consider the case that the starting point on $W_-$ is near the boundary of $\Rm$. In that case we use coordinates $\xoc, \yoc = \yoc(\omega), \xioc, \etaoc$ for $W_-$ and $\tilde s$ for the flow of the vector field $H_p^{2,0}$. Here $\omega$ is a point in $S^{n-1}$, and $\yoc$ are local coordinates on $S^{n-1}$ in a patch containing $\omega$. Abusing notation somewhat, we also view its `dual' variable $\etaoc$ as being in the hyperplane of $\R^n$ orthogonal to $\omega$. We note that due to the restriction to $U_-$, we only need to consider a bounded set in the $\xioc, \etaoc$ variables. On $\overline{{}^{\ps} T^* \R^{n+1}}$ we can use coordinates $(z, t, \rhops = |\zeta|^{-1}, \hat \zeta, \tau' = \tau/|\zeta|^2)$ which are valid in a neighbourhood of $\Sigma$ near fibre-infinity. We decompose $z = z^\perp + z^\parallel$ relative to the direction $\omega$. 
Then on $\SM$ we have coordinates $(\xoc, \yoc(\omega), \xioc, \etaoc, \sigma = \xoc/\rhops, z^\perp, z^\parallel, t, \hat \zeta, \tau')$. 

We first analyze $\overline{\Lambda_-}$ in a neighbourhood of $\Rm$, even though that is not contained in $\Lambda_-$ due to the open set $G_-$ not extending to spacetime infinity. In this region $\Lambda_-$ coincides with $\Lambda_-(P_0)$ for the free operator $P_0$, which we computed explicitly in Section~\ref{subsec: 1c arises}. In particular, in our coordinates it takes the form 
\begin{multline}\label{eq:Lambda form near Rminus}
(\xoc, \omega, \xioc, \etaoc, \tilde s) \mapsto (\xoc, \yoc(\omega), \xioc, \etaoc; \\ \sigma = 1, z^\perp = \etaoc, z^\parallel = 2\tilde s \omega, t =  -2\xioc + \tilde s \xoc \sigma^{-1}, \hat \zeta = \omega, \tau' = -1). 
\end{multline}
The image of this map is clearly a smooth submanifold that meets $\ffocps$ transversally, on which $dt \neq 0$, and, since $\sigma \equiv 1$, it is disjoint from all other boundary hypersurfaces of $\SM$. Of course, the explicit form \eqref{eq:Lambda form near Rminus} is no longer valid once the flowout reaches the perturbation. To deal with the perturbation, we again  view $\Lambda_-$ as the flowout from a cross-section $T$ as above by the vector field $H_p^{2,0}$. This vector field, on the ps-cotangent bundle $\overline{{}^{\ps} T^* \R^{n+1}}$, is smooth and tangent to the boundary. Lifted to $\SM$ it is smooth and tangent to $\ffocps$, with a component in the $\sigma$ direction (in the above coordinates) of the form $a \sigma \partial_\sigma$, where $a$ is smooth. Since the flow of $H_p^{2,0}$ crosses the open set $G_-$ in finite parameter time, this means that the coordinate $\sigma$ remains in a bounded set $[C^{-1}, C]$, and therefore the flowout remains disjoint from all boundary hypersurfaces of $\SM$ other than $\ffocps$. Finally, it is easy to compute that the Lie derivative of $dt$ under the flow of $H_p^{2,0}$ is $d\rhops$, which shows that $dt$ does not vanish under the flowout of this vector field (as the $-2 d\xioc$ component of $dt$ in the coordinates used in \eqref{eq:Lambda form near Rminus} is unchanged). 

We next consider the case that the starting point on $W_-$ is away from the boundary of $\Rm$. In this case, due to the localization $U_-$, we start in the interior of $W_-$, and propagate into the interior of $\overline{{}^{\ps} T^* \R^{n+1}}$. The microlocalized Lagrangian $\Lambda_-$ is therefore away from all boundary hypersurfaces of $\SM$, so the first and third conditions of Definition~\ref{defn: admissible 1c-ps Lagrangian submanifold and 1c-ps fibred-Legendre submanifold} are vacuously satisfied. Also, for these integral curves, $t$ itself can be taken as a parameter for the flow, so $dt$ is nonvanishing. This completes the proof of Proposition~\ref{prop:microlocalized FBSR}. 
\end{proof}

\begin{remark}\label{rem:phase function leading term}
In the case of the microlocalized forward and backward sojourn relations $\Lambda_\pm$, we can choose a more explicit form of parametrizing function. This arises from the observation that the Legendrian curve associated with these Lagrangians (as in Proposition~\ref{prop:1c-ps contact symplectic property}) coincides with that for the \emph{free} forward and backward sojourn relations. Indeed, the coordinates $\xioc$, $t$ and $\tilde \tau$ are all constant under the Hamiltonian flow at fibre-infinity, so we can evaluate them at any point along the flow. However, near spacetime infinity, the flow coincides with the free flow. We have already calculated for the free forward and backward sojourn relations that this Legendrian curve is given by 
\begin{equation}
    \xioc = 2 t, \quad \tilde \tau = -1.
\end{equation}
This therefore holds even for the perturbed flow. 
It follows that the $O(\xoc^{-2})$ part of the parametrizing function can be taken to be $-t/\xoc^2$. Thus we can always take our non-degenerate parametrization to take the form 
\begin{equation}
    \Phi_\pm = -\frac{t}{\xoc^2} + \frac{\phi_\pm(\msf{K},\theta_\pm)}{\xoc}, \quad \msf{K} = (\xoc, \yoc, t, z),
\end{equation}
with no $\theta_0$ variables and an explicit first term. 
\end{remark}

For completeness, we also give the definition of the forward and backward bicharacteristic relations, analogous to Definition~\ref{defn:sojoun relations}. 

\begin{defn}\label{defn:FBR} The forward and backward bicharacteristic relations are the subsets of $\overline{{}^{\ps} T^* \R^{n+1}} \times \overline{{}^{\ps} T^* \R^{n+1}}$ defined by 
\begin{equation}\label{eq:FBR def}
\FBR = \{ (\gamma_{q_-}(s), \gamma_{q_-}(s')) \mid q_- \in W_- , \ s, s' \in (0, T(q_-)), \ s \geq s' \}. 
\end{equation}
and 
\begin{equation}\label{eq:BBR def}
\BBR = \{ (\gamma_{q_-}(s), \gamma_{q_-}(s')) \mid q_- \in W_- , \ s, s' \in (0, T(q_-)), \ s \leq s' \}. 
\end{equation}
That is, the relation comprises pairs of points along the same bicharacteristic, with the first element being forward, respectively backward, of the second, with respect to Hamilton flow. 
\end{defn}
%%%%%%%%%%%%%%%%%%%%%%%%%%%%%%%%%%%%%%%%%%%%%%%%%%%%%%%%%%
%%%%%%%%%%%%%%%%%%%%%%%%%%%%%%%%%%%%%%%%%%%%%%%%%%%%%%%%%%
%%%%%%%%%%%%%%%%%%%%%%%%%%%%%%%%%%%%%%%%%%%%%%%%%%%%%%%%%%
%%%%%%%%%%%%%%%%%%%%%%%%%%%%%%%%%%%%%%%%%%%%%%%%%%%%%%%%%%
%%%%%%%%%%%%%%%%%%%%%%%%%%%%%%%%%%%%%%%%%%%%%%%%%%%%%%%%%%

\section{1c-ps Fourier Integral Operators}
\label{sec: 1c-ps FIO}
In this section we define a class of $\oc$-$\ps$ Fourier Integral Operators (FIOs) that contains suitably microlocalized Poisson operators, for a Schr\"odinger operator $P$ as specified in the Introduction. 
This is a class of operators whose Schwartz kernels are Legendre distributions associated to fibred-Legendre submanifolds of $\mathcal{M}$ as in Definition~\ref{defn: admissible 1c-ps Lagrangian submanifold and 1c-ps fibred-Legendre submanifold}.

\subsection{Definition of 1c-ps Legendre distributions and Fourier Integral Operators}

\begin{defn} \label{defn: 1c-ps FIO}
Let $\Legps$ be a fibred-Legendre submanifold of $\SM$ and $\Lagps$ be its Lagrangian extension in Definition~\ref{defn: admissible 1c-ps Lagrangian submanifold and 1c-ps fibred-Legendre submanifold}. 
A 1c-ps Legendre distribution associated to  $\Legps$
of order $m$ is a distributional half-density that can be written (modulo $\mathcal{S}(\R^{n+1} \times \R^n)$) as a finite sum of oscillatory integrals of the form 
\begin{multline} \label{eq: 1c-ps FIO definition}
u(\msf{K}) =  (2\pi)^{- ( \frac{2n+1}{4} ) - \frac{k_0+k_1}{2} }
\Big(\int e^{ i( \frac{ \varphi_0(t,\theta_0) }{x_{\oc}^2} + \frac{\varphi_1(\msf{K}',\theta_0,\theta_1)}{x_{\oc}} )} 
x_{\oc}^{-(m+\frac{2k_0+k_1}{2})-\frac{1}{4}}
\\  \times a(\msf{K},\theta_0,\theta_1) d\theta_0d\theta_1 \Big)
|dtdz|^{1/2}|\frac{dx_{\oc}d\yoc}{x_{\oc}^{n+2}}|^{1/2}, \quad \msf{K} = (\xoc, \yoc, t, z), \quad \msf{K}' = (\yoc, t, z),
\end{multline}
where $\Phi_{\oc-\ps} = \frac{ \varphi_0(t,\theta_0) }{x_{\oc}^2} + \frac{\varphi_1(\msf{K},\theta_0,\theta_1)}{x_{\oc}}$ is a local non-degenerate parametrization of $\Legps$ (see Definition~\ref{defn: 1c-ps parametrization}), with $\theta_0 \in \R^{k_0}$, $\theta_1 \in \R^{k_1}$, and 
 $a \in C_c^\infty([0,\epsilon)_{\xoc} \times \R^{n-1}_{\yoc} \times \R^{n+1}_{t,z} \times \R^{k_1 + k_2}_{\theta_0, \theta_1})$ is assumed to be supported in the region where $\Phi_{\oc-\ps}$ parametrizes $\Legps$. 
 The set of such Legendre distributions is denoted $I_{\oc-\ps}^m(\R^{n+1} \times \R^n, \Lagps; \Omega^{1/2})$, where $\Omega^{1/2}$ is the half-density bundle whose typical section is as in the expression above.
 A linear operator $A$, mapping half-densities on $\R^n$ to half-densities on $\R^{n+1}$ is called a 1c-ps Fourier Integral Operator of order $m$ associated to $\Legps$ if its Schwartz kernel is a Legendre distribution of order $m$ associated to $\Legps$. 
%\mathcal{A}_c^\infty(\R^{n+1}_{t,z} \times [0,x_0) \times \mathbb{S}^{n-1}_{y_{\oc}} \times \R^{N_0}_{\theta_0} \times \R^{N_1}_{\theta_1})$.
\end{defn}

\begin{remark}
We comment on the order convention in Definition~\ref{defn: 1c-ps FIO}, i.e. to explain the choice of exponent $-(m+\frac{2k_0+k_1}{2})-\frac{1}{4}$ of $\xoc$. The explanation for the term $-(k_0 + k_1/2)$ is quite standard; it ensures that the order is independent of which parametrizing function is used to represent it, cf. Proposition~\ref{prop:invariance under parametrization change}. It can be understood as follows: if one introduces additional integrated variables, then one could perform nondegenerate stationary phase in these variables and thereby gain a decay factor of $\xoc$ for each $\theta_0$-variable and $\xoc^{1/2}$ for each $\theta_1$-variable. 

More interesting is the explanation for how the `$-1/4$' in the exponent of $\xoc$ arises. The order convention for standard Lagrangian distributions is chosen so that, first, the distributions become more singular as the order increases; second, that the Lagrangian distributions of a fixed order live in a fixed Sobolev space whose (Sobolev) order does not depend on the geometry of the Lagrangian (such as the rank of the projection to the base, etc); third, if the Lagrangian distribution happens to be the Schwartz kernel of a pseudodifferential operator, then the order agrees with the order as a pseudodifferential operator. Because of the second condition, it makes sense to talk about the $L^2$-threshold order, that is, that order $m$ for which Lagrangian distributions are $L^2$ if they have order strictly less than $m$. Then, because of the third condition, in dimension $2n$ this order should be $-n/2$ since pseudodifferential Schwartz kernels on $\R^n$ are $L^2$ whenever their order is $< -n/2$, and $-n/2$ is the threshold for this property. More generally, in dimension $N$ we take the $L^2$-threshold order for Lagrangian distributions to be $-N/4$. 

We follow these order conventions in spirit, but they need to be suitably interpreted in both the ps- and 1c-calculus due to the anisotropy. A simple calculation involving scaling shows that the threshold orders for 1c-pseudodifferential operators having $L^2$ Schwartz kernels in dimension $n$ are $(-n/2, -(n+1)/2)$, where the anomalous spatial threshold order $-(n+1)/2$ is due to the 1c-density blowing up faster as $\xoc \to 0$ compared to a scattering density. Similarly, the $L^2$-threshold orders for ps-pseudodifferential operators in dimension $n+1$ are $(-(n+2)/2, -(n+1)/2)$ due to the parabolic scaling in frequency. Since the relevant orders for 1c-ps Fourier integral operators are the \emph{spatial} 1c order and the \emph{differential} ps order, the `effective' dimensions are actually the scaling dimensions $(n+1, n+2)$ as opposed to the manifold dimensions $(n, n+1)$. \emph{We have defined the order convention in Definition~\ref{defn: 1c-ps FIO}  so that the $L^2$-threshold order is $-N/4$ where $N = (n+1) + (n+2) = 2n+3$.} As a straightforward check, when there are no integrated variables $(\theta_0, \theta_1)$, and if $m = -(2n+3)/4$ then the overall power of $\xoc$ is $(n+1)/2$ which is at the threshold for square-integrability due to the denominator $\xoc^{(n+2)/2}$ in the half-density factor. 

This order convention leads to the expected numerology in Theorem~\ref{thm: 1c-ps composition to 1c-1c}, that is, under composition the orders add with an additional $e/2$ where $e=1$ is the excess in the clean composition of the canonical relations of the two factors $A_+^*$ and $A_-$. 
\end{remark}

From now on, we will specify the density bundle in which those operator kernels or symbols are taking values when we define them for the first time or when we want to emphasize the sub-principal part due to the choice of this bundle factor, but abbreviate them otherwise. 

\begin{example} \label{example:microlocalized-free-Poisson}
Consider the localized free Poisson operator $\chi \Poi_0$, where $\chi = \chi(z,t)$ is a multiplication operator with compact support in $\R^{n+1}$. We wish to view this as an operator from half-densities to half-densities. To do this, we multiply by $|dz \, dt \, dZ|^{1/2}$, i.e. we take it to have Schwartz kernel
\begin{equation}\label{eq:Poisson half density}
    (2\pi)^{-n} e^{i(z \cdot Z - t |Z|^2)} |dz \, dt \, dZ|^{1/2}, \quad (z, t) \in \R^{n+1}, \quad Z \in \R^n. 
\end{equation}
(One might wonder whether the half-density factor in $Z$ should be the Euclidean (scattering) half-density $|dZ|^{1/2}$ or a 1-cusp half-density, which in terms of the Euclidean variable $Z$ would take the form $|\ang{Z} dZ|^{1/2}$. That the Euclidean choice is the correct one may be seen by recalling the formula from \cite{gell2022propagation}:
\begin{equation}
    \Poi_0 \Poi_0^* = i (2\pi)^{n} \Big( (P_0)_+^{-1} - (P_0)_-^{-1} \Big), 
\end{equation}
where $P_0$ is the free Schr\"odinger operator. This identity determines the half-density factor in the $Z$ variables, which must be as in \eqref{eq:Poisson half density}.)

We claim that this operator $\chi \Poiz$, viewed as a half-density as above, is in $I_{\oc-\ps}^{-3/4}(\R^{n+1} \times \R^n, (\Lambda_0)_\pm; \Omega_{\oc-\ps}^{1/2})$. To see this, we note that the phase function has the required form as per Definition~\ref{defn: admissible 1c-ps Lagrangian submanifold and 1c-ps fibred-Legendre submanifold}. We use coordinates near $|Z| = \infty$ of the form $\xoc' = 1/|Z|$, $y'_j = y'_{\oc, j} = Z_j/|Z|$ for $2 \leq j \leq n$ where we assume that $Z_1$ is a dominant variable. In these coordinates, the phase function is 
\begin{equation}
 \Phi = \frac{z \cdot (y_1, \dots, y_{n-1}, \sqrt{1 - |y|^2})}{\xoc} - \frac{t}{(\xoc)^{2}} 
\end{equation}
and we note that this is a non-degenerate phase function in the sense of Definition~\ref{defn: admissible 1c-ps Lagrangian submanifold and 1c-ps fibred-Legendre submanifold} (with no phase variables). It parametrizes the Lagrangian submanifold $(\Lambda_0)_\pm$ associated to the free backward or forward sojourn relation (which coincide in the free case). The compactly supported cutoff $\chi$ ensures that the Lagrangian does not (on the support of the amplitude) meet any boundary of $\SM$ other than $\ffocps$, so that the third condition in Definition~\ref{defn: admissible 1c-ps Lagrangian submanifold and 1c-ps fibred-Legendre submanifold} is satisfied.  
\end{example}

\begin{prop}\label{prop:invariance under parametrization change} Assume that $u$ in \eqref{eq: 1c-ps FIO definition} has essential support in an open set $U \subset \Legps$, and that $\tilde \Phi_{\oc-\ps}$ is another parametrization of $\Legps$ in $U$. Then $u$ can be written, modulo $\mathcal{S}(\R^{n+1} \times \R^n)$, as as oscillatory integral with respect to the parametrizing function $\tilde \Phi_{\oc-\ps}$. 
\end{prop}

\begin{proof} The key step in this proof, following H\"ormander \cite{FIO1}, is to prove equivalence of phase functions in Proposition~\ref{prop:1c-ps-phase-equivalence}, which roughly says that when the signature of Hessians of two phase functions are the same, then one can apply a change of variables to reduce one to the other.
%This result is proved in Proposition~\ref{prop:1c-1c-phase-equivalence}. 

The proof of the Proposition, given equivalence of phase functions in Proposition~\ref{prop:1c-ps-phase-equivalence}, is standard. Given two different parametrizations, we first add additional quadratic factors $Q_0(\theta_0', \theta_0')$, $Q_1(\theta_1', \theta_1')$ to each function $\varphi_0, \varphi_1$ so as to make the total number of $\theta_i$-factors equal. This can be done in such a way as to preserve the signature conditions \eqref{eq: signature assumptions-main-prop}. Adding such quadratic factors to the phase function does not change the expression \eqref{eq: 1c-ps FIO definition}, due to the way that the powers of $\xoc$ and $2\pi$ depend on the number $(k_0, k_1)$ of $\theta$-variables. We then apply equivalence of phase functions to deduce that an expression for $u$ with respect to the first phase function (modified as above) can be transformed, via a change of variables, to an expression with respect to the (modified) second phase function. 
\end{proof}

\subsection{Principal symbol}
In this section we give a symbol calculus of 1c-ps Lagrangian distributions. 
We first define its principal symbol.
Suppose $\lambda = (\lambda_1,...,\lambda_{n_X+n_Y-1})$ are local coordinates of $\Legps \cap \partial \mathcal{M}$, such that $(x_{\oc},\lambda,\partial_{{\theta}_0}\varphi_0,\partial_{\theta_1}\varphi_1)$ gives a coordinate system of 
$\mathcal{M} \times \R^{N_0}_{\theta_0} \times \R^{N_1}_{\theta_1}$ near the critical set $C_{\Phi_{\oc-\ps}}$ in (\ref{eq: 1c-ps critical set}).
Then the symbol of $A \in I^{m}_{\oc-\ps}(\R^{n+1} \times \R^n, \Lagps)$ written as in (\ref{eq: 1c-ps FIO definition}) with respect to this coordinate system and this parametrization is a half density
\begin{align} \label{eq: 1c-ps principal symbol, half density form}
a(0,\lambda,0,0) |\det \frac{\partial (\lambda,\partial_{\theta_0} \varphi_0,\partial_{\theta_1}\varphi_1)}{\partial (\msf{K}',\theta_0,\theta_1)}|^{-1/2}
|d\lambda|^{1/2} \otimes |d\xoc|^{-m-n/2-5/4}.
\end{align}
%|\frac{dx_{\oc}dy_{\oc}}{x_{\oc}^{n+2}}|^{1/2}
This expression is designed to be invariant under changes of parametrization and scaling of the boundary defining function $\xoc$ (hence the inclusion of the  factor $|d\xoc|^{-m-n/2-5/4}$).
%\footnote{It can be made fully invariant under coordinate changes by including an `$E-$bundle' factor as in \cite[Section~3.1]{hassell2001resolvent}, but we shall not do that here.} 
As emphasized by H\"ormander, we need to view \eqref{eq: 1c-ps principal symbol, half density form} as a section of the Maslov bundle in order to make it invariant under changes of parametrization with different signatures \eqref{eq: signature assumptions}; this bundle accounts for the different powers of $e^{i\pi/4}$ that arise in applying the stationary phase formula to different phase functions. 

We thus define
\begin{align} \label{eq: S[m]bundle, 1c-ps case}
S^{[m]}_{\oc-\ps}(\Legps):=  |N^*(\ffocps)|^{-m-\frac{2n+5}{4} } \otimes M(\Legps) \otimes E(\Legps),
\end{align}
where $M(\Legps)$ is the Maslov bundle used in \cite[Section~3.2,3.3]{FIO1} and originates from \cite{maslov1972asymptotic};
and $|N^*(\ffocps)|^{-m-\frac{2n+5}{4} }$
%which originally only defined on $\partial X \times Y$ extends `by constant' to a collar neighborhood of it, which 
models a section that is homogeneous of degree $-m-\frac{2n+5}{4}$ in $x_{\oc}$. \footnote{This degree can be computed in the case without any extra parameters $\theta_0,\theta_1$, in which case \eqref{eq: 1c-ps FIO definition} has the same homogeneity in $\xoc$ as
$x_{\oc}^{-m-\frac{3}{4}} |\frac{dx_{\oc}d\yoc}{x_{\oc}^{n+2}}|^{1/2}$.}
Here $E(\Legps)$ is the line bundle arises in \cite[Section~3.1]{hassell2001resolvent}, whose transition coefficient depends on the choice of the boundary defining function (or more precisely, the zeroth and first order jet of it).
%%%%%%%%%%%%%%%%%%%%%%%%%%%%%%%%%

Using Proposition~\ref{prop:invariance under parametrization change} and track its proof, one can see that \eqref{eq: 1c-ps principal symbol, half density form} is invariant under change of parametrizations and is associated to the element in $I^m_{\oc-\ps}(\R^{n+1} \times \R^n ,\Lagps; \Omega^{1/2})$ itself,
thus the principal symbol is a map 
\begin{align} \label{eq: invariant, symbol map, 1c-ps}
\begin{split}
\sigma^m_{\oc-\ps}: \quad   I^m_{\oc-\ps}(\R^{n+1} \times \R^n ,\Lagps; \Omega^{1/2}) 
%\leftidx{^{\ps}}{\Omega}^{1/2} \otimes \leftidx{^{\oc}}{\Omega}^{1/2} ))
\rightarrow 
 C^\infty(\Legps \cap \partial \mathcal{M}; 
\Omega^{1/2}(\Legps) \otimes S_{\oc-\ps}^{[m]}(\Legps)).
\end{split}
\end{align}

For $A \in I^m_{\oc-\ps}(\R^{n+1} \times \R^n ,\Lagps; \Omega^{1/2})$, we say it is elliptic at $q \in \Legps$ if the amplitude in \eqref{eq: 1c-ps principal symbol, half density form} is non-vanishing at the point that is sent to $q$ under the parametrization map in Definition~\ref{defn: 1c-ps parametrization} This is equivalent to requiring the principal symbol to be invertible near $q$ as the section of $\Omega^{1/2}(\Legps) \otimes S_{\oc-\ps}^{[m]}$.

The fact that the symbol map in \eqref{eq: invariant, symbol map, 1c-ps} captures leading order singularity can be summarized as following statement: 
\begin{prop}
The principal symbol map $\sigma^m_{\oc-\ps}$ gives rise to following short exact sequence:
\begin{align}
\begin{split}
0 & \rightarrow I^{m-1}_{\oc-\ps}(\R^{n+1} \times \R^n,\Lagps)
\rightarrow I^{m}_{\oc-\ps}(\R^{n+1} \times \R^n, \Lagps)
\\ & \xrightarrow{\sigma^m_{\oc-\ps}} C^\infty(\Legps \cap \partial \mathcal{M}; 
\Omega^{1/2}(L) \otimes S_{\oc-\ps}^{[m]}(\Legps)) \rightarrow 0.
\end{split}
\end{align}
\end{prop}

The exactness at the left and right component follows from definition. The exactness in the middle follows from following lemma:

\begin{lmm}  \label{lemma: vanishing amplitude->lower order}
For $A \in I^m_{\oc-\ps}(\R^{n+1} \times \R^n, \Lagps)$, if $\sigma^m_{\oc-\ps}(A) = 0$, then $A \in I^{m-1}_{\oc-\ps}(\R^{n+1} \times \R^n, \Lagps)$.
\end{lmm}
\begin{proof}
We prove the result in a fixed parametrization. Hence the condition effectively says that the amplitude vanishes on $C_{\Phi_{\oc-\ps}}$, which means that we may write
\begin{align*}
a = a_0 \cdot(\partial_{\theta_0}\varphi_0+x_{\oc}\partial_{\theta_0}\varphi_1)
+ a_1 \cdot \partial_{\theta_1}\varphi_1.
\end{align*}
Hence 
\begin{align*}
A = \int e^{ i( \frac{\varphi_0(t,\theta)}{x_{\oc}^2} + \frac{\varphi_1(t,z,y_{\oc},\theta)}{x_{\oc}} )} x_{\oc}^{-(m+\frac{2N_0+N_1}{2}+\frac{3}{4})}(a_0 \cdot(\partial_{\theta_0}\varphi_0+x_{\oc}\partial_{\theta_0}\varphi_1)
+ a_1 \cdot \partial_{\theta_1}\varphi_1) d\theta_0d\theta_1
\end{align*}

Noticing that 
\begin{align*}
(a_0 \cdot(\partial_{\theta_0}\varphi_0 +x_{\oc}\partial_{\theta_0}\varphi_1)e^{ i( \frac{\varphi_0(t,\theta)}{x_{\oc}^2} + \frac{\varphi_1(t,z,y_{\oc},\theta)}{x_{\oc}} )} = x_{\oc}^2a_0 \cdot \partial_{\theta_0}e^{ i( \frac{\varphi_0(t,\theta)}{x_{\oc}^2} + \frac{\varphi_1(t,z,y_{\oc},\theta)}{x_{\oc}} )},
\end{align*}
and
\begin{align*}
(a_1 \cdot \partial_{\theta_1}\varphi_1) e^{ i( \frac{\varphi_0(t,\theta)}{x_{\oc}^2} + \frac{\varphi_1(t,z,y_{\oc},\theta)}{x_{\oc}})} 
=x_{\oc} a_1 \cdot \partial_{\theta_1}e^{ i( \frac{\varphi_0(t,\theta)}{x_{\oc}^2} + \frac{\varphi_1(t,z,y_{\oc},\theta)}{x_{\oc}})},
\end{align*}
an integration by parts gives
\begin{align*}
A = \int e^{ i( \frac{\varphi_0(t,\theta)}{x_{\oc}^2} + \frac{\varphi_1(t,z,y_{\oc},\theta)}{x_{\oc}} )} x_{\oc}^{-(m-1+\frac{2N_0+N_1}{2}+\frac{3}{4})}
 (x_{\oc}\mathrm{div}_{\theta_0}a_0 + \mathrm{div}_{\theta_1}a_1) d\theta_0d\theta_1,
\end{align*}
where 
\begin{align*}
\mathrm{div}_{\theta_0}a_0 = \sum_{j} \partial_{(\theta_0)_j}(a_0)_j,
\mathrm{div}_{\theta_1}a_1 = \sum_{j} \partial_{(\theta_1)_j}(a_1)_j.
\end{align*}
Thus $A \in I^{m-1}_{\oc-\ps}(\R^{n+1} \times \R^n, \Lagps)$.
\end{proof}

\begin{comment}
A repeated applciation of the lemma above gives:
\begin{coro}  \label{coro: modify the amplitude}
Two amplitudes gives the same leading order singularity, 
as long as their lowest order vanishing (in terms of $\partial_{\theta_0}\varphi_0+x_{\oc}\partial_{\theta_0}\varphi_1$ and $\partial_{\theta_1}\varphi_1$, with each power of the former counted as second order and each power of the latter counted as first order vanishing) parts on $\mathcal{L}$ are the same.
\end{coro}

%\subsection{Properties as a symbolic calculus} \label{sec:symbolic calculus}
\end{comment}

%Now we give a symbolic calculus to deal with when the Lagrangian is not necessarily a canonical graph, i.e., when its projection to the base degenerates. This happens when we reach conjugate points. In this section, 

\subsection{Composition with pseudodifferential operators}
\label{subsec:PsiDO-FIO-composition-1cps}

Now we characterize the composition of 1c-ps Lagrangian distributions from left with parabolic scattering pseudodifferential operators. Our first composition theorem is:
%%%%%%%%%%%%%%%%%%%%%%%

\begin{thm} \label{thm: PsiDO- 1c-ps FIO composition}
Suppose $Q \in \Psi_{\ps}^{m',0}(Y)$ has parabolically homogeneous principal symbol at fibre-infinity. If $A \in I_{\oc-\ps}^{m}(\R^{n+1} \times \R^n, \Lagps; \Omega^{1/2})$ %has amplitude $a(t,z,x_{\oc},y_{\oc},\theta)$ when we parametrize $\mathcal{L}$ by $\Phi$, with $\theta$ being extra parameters, 
then we have  
\begin{equation}
QA \in I^{m+m'}_{\oc-\ps}(\R^{n+1} \times \R^n,\Lagps),
\end{equation}
with principal symbol 
\begin{align}  \label{eq: QA principal symbol, non-vanishing q, density version}
\sigma^{m+m'}_{\oc-\ps}(QA) = \sigma^{m'}_{\ps}(Q)|_{\Legps} \otimes \sigma^m_{\oc-\ps}(A)
\end{align}
where we lift the principal symbol of $Q$ to $\SM$, and view as a section of $|N^*(\ffocps)|^{-m'}$ over $\Legps$. 
\end{thm}

%\begin{remark} We can also define the principal symbol of $A$ as a half density, so that it is invariant under change of parametrizations, and the multiplicativity of the principal symbol as above still holds. \end{remark}
%%%%%%%%%%%%%%%%%%%%%%%%%%%%%%%%%%%%%%%
\begin{proof} See Appendix~\ref{app:proofs}.
\end{proof}
%%%%%%%%%%%%%%%%%%%%%%%%%%%%%%%%%%%%%%%

%\subsection{Vanishing principal symbol composition}
%(partial projection, compactified version of \cite[Theorem~21.2.16]{hormander2007analysis}, FIO I Theorem 3.1.3)
In the next result, we consider the compositions of the form $PA$ when $P \in \Psi_{\ps}^{m', 0}$ has principal symbol vanishing on $\Legps$. 
Also, we specialize to the case with $\Lambda = \Lambda_\pm$ defined in \eqref{eq:microlocalized sojourn relns}.
Those requirements are satisfied when $P$ is our Schr\"odinger operator and $A$ is a microlocalized Poisson operator. 

%since $P \Poipm = 0$.)

Denote the full (left) symbol of $P$ by $p_{\full}$, and define
\begin{align*}
r=p_{\full}-p_{\mathrm{hom}}.
\end{align*}
As in \cite[Equation~(5.2.8)]{FIOII}, the invariant sub-principal symbol of $P$ is defined to be
\begin{align*}
p_{\sub} = r - (2i)^{-1} \sum_j \frac{\partial^2 p_{\full} }{\partial {z_j} \partial \zeta_j},
\end{align*}
which is independent of the choice of coordinates modulo $S^{m'-2,-1}_{\ps}-$terms.

\begin{thm} \label{thm: vanishing principal symbol product}
Suppose $P \in \Psi_{\ps}^{m',0}(Y)$, and its parabolically homogeneous principal symbol $p_{\hom}$ vanishes identically on the projection of $\Legps$ in $\leftidx{^\ps}T^*Y$. In addition, let $p$ be its left full symbol and assume that $\tilde{p}(t,z,\tilde{\tau},\tilde{\zeta})=x_{\oc}^{m'}p(t,z,\tau,\zeta)$ is smooth.\footnote{This in particular is satisfied by all differential operators.} 
If $A \in I_{\oc-\ps}^{m}(\R^{n+1} \times \R^n,\Lambda_\pm; \Omega^{1/2})$, then we have 
\begin{align*}
PA \in I^{m+m'-1}_{\oc-\ps}(\R^{n+1} \times \R^n,\Lambda_\pm; \Omega^{1/2}).
\end{align*}
The principal symbol is of $PA$ is as follows: let 
\begin{align*}
 \sigma^m_{\oc-\ps}(A) = \textbf{a} \otimes |dx_{\oc}|^{-m-\frac{2n+5}{4}},
\end{align*}
with $\textbf{a}$ being a section of $\Omega^{1/2}(\Legps) \otimes M(\Legps)$, then
\begin{align} \label{eq: 1c-ps vanishing principal composition, principal symbol} 
\begin{split}
\sigma_{\oc-\ps}^{m+m'-1}(PA) = & (-i\mathscr{L}_{ \rHp} + i (\frac{m'-1}{2} + m + \frac{2n+5}{4} )(x_{\oc}^{-1} \rHp x_{\oc}) + p_{\sub})\textbf{a}
\\& \otimes |dx_{\oc}|^{-m-m'+1-\frac{2n+5}{4}}.
\end{split}
\end{align}
\end{thm}

\begin{proof} See Appendix~\ref{app:proofs}.
\end{proof}

We now consider composition between 1c-ps Lagrangian distributions from right with 1c-pseudodifferential operators.

\begin{thm} \label{thm: PsiDO- 1c-ps FIO composition-2}
Suppose $Q' \in \Psi_{\oc}^{-\infty,m'}(\R^n)$. If $A \in I_{\oc-\ps}^{m}(\R^{n+1} \times \R^n, \Lagps)$ %has amplitude $a(t,z,x_{\oc},y_{\oc},\theta)$ when we parametrize $\mathcal{L}$ by $\Phi$, with $\theta$ being extra parameters, 
then we have  
\begin{equation}
AQ' \in I^{m+m'}_{\oc-\ps}(\R^{n+1} \times \R^n,\Lagps),
\end{equation}
with principal symbol 
\begin{align}  \label{eq: AQ' principal symbol, non-vanishing q, density version}
\sigma^{m+m'}_{\oc-\ps}(AQ') =   \sigma^m_{\oc-\ps}(A)
\otimes \sigma^{m'}_{\oc}(Q')|_{\Legps}
\end{align}
where we lift the principal symbol of $Q'$ to $\SM$, and view as a section of $|N^*(\ffocps)|^{-m'}$ over $\Legps$. 
\end{thm}

%\begin{remark} We can also define the principal symbol of $A$ as a half density, so that it is invariant under change of parametrizations, and the multiplicativity of the principal symbol as above still holds. \end{remark}
%%%%%%%%%%%%%%%%%%%%%%%%%%%%%%%%%%%%%%%
\begin{proof} 
The proof is completely analogous to the proof of Theorem~\ref{thm: PsiDO- 1c-ps FIO composition}, which is in Appendix~\ref{app:proofs}.
The only difference is that now the extra parameters in the composition are those left variables of $Q'$ (which are right variables of $A$ as well). 
Then we apply the stationary phase argument to them to see that it is again of the form \eqref{eq: 1c-ps FIO definition} with $m+m'$ being its order.

\end{proof}

%%%%%%%%%%%%%%%%%%%%%%%%%%%%%%%%%%%%
We now discuss the wavefront set of $\oc-\ps$ Legendre distributions.
Let $\mathcal{M}$ be as in \eqref{eq: 1c-ps cotangent bundle definition}, and denote the composition of the blow-down map $\beta_{\oc-\ps}$ with projections from $\mathcal{M}_0$ by
\begin{equation}
    \pi_{\oc}: \mathcal{M} \to \overline{{}^\oc{T^*\RR^n}},
    \quad \pi_{\ps}: \mathcal{M} \to \overline{{}^\ps{T^*\RR^{n+1}}}.
\end{equation}

Then for %$A \in I_{\oc-\ps}^{m}(\R^{n+1} \times \R^n, L; \Omega^{1/2})$,
$A \in \mathcal{S}'(\R^{n+1} \times \R^n)$, $\WF'_{\oc-\ps}(A) \subset 
\pi_{\oc}^{-1}(\partial \, \overline{{}^\oc{T^*\RR^n}}) \cap \pi_{\ps}^{-1}(\partial \,\overline{{}^\ps{T^*\RR^{n+1}}})$ is defined as follows.
Let $\msf{p} \in \pi_{\oc}^{-1}(\partial \, \overline{{}^\oc{T^*\RR^n}}) \cap \pi_{\ps}^{-1}(\partial \,\overline{{}^\ps{T^*\RR^{n+1}}})$, then we say $\mathsf{p} \notin \WF'_{\oc-\ps}(A)$ if and only if there exists $Q \in \Psi_{\ps}^{0,0}(\R^{n+1})$ that is elliptic at $\pi_{\ps}(\mathsf{p})$ and $Q' \in \Psi_{\oc}^{0,0}$ that is elliptic at $\pi_{\oc}(\msf{p})$, such that 
\begin{align} \label{eq:def-WF-1c-ps-1}
    QAQ' \in \mathcal{S}(\R^{n+1} \times \R^n).
\end{align}
Notice that $A$ by definition already have infinite order of decay in $(t,z)$ and infinite order of regularity in $Z \in \R^n$ except for potential growth in $x_{\oc}^{-1}$, so what we concern below is the infinite order regularity in $(t,z)$ and infinite order decay in $x_{\oc}$.
%We should point out that 

Notice that $\WF'_{\oc-\ps}(A)$ only depends on its projection to $\overline{{}^\oc{T^*\RR^n}}, \overline{{}^\ps{T^*\RR^{n+1}}}$, so effectively it can be identified with its image under $\beta_{\oc-\ps}$, which is a closed subset of $\partial \mathcal{M}_0$.
In this way, it becomes a product and one can see its analogy with standard operator wavefront sets.

\begin{prop} \label{prop:general-1cps-FIO-WF}
    For $A \in I_{\oc-\ps}^{m}(\R^{n+1} \times \R^n, \Lagps)$ with $\Legps = \partial \Lagps$, we have
    \begin{equation} \label{eq:1c-ps-FIO-WF-general}
        \WF_{\oc-\ps}'(A) \subset \pi_{\oc}^{-1}(\pi_{\oc}(\Legps)) \cap \pi_{\ps}^{-1}(\pi_{\ps}(\Legps)).
    \end{equation}
For $A$ above and $f \in \mathcal{S}'(\R^n)$, we have
    \begin{equation} \label{eq:1c-ps-FIO-WF-general-2}
        \WF_{\ps}(Af) \subset (\WF'_{\oc-\ps}(A))' \circ \WF_{\oc}(f),
    \end{equation}
    where $'$ means switching the sign of momentum on the right variable.
\end{prop}

%%%%%%%%%%%%%%%%%%%%%%%%%%%%%%%%%%%%%%%%
\begin{proof}
Suppose $A$ is given by the oscillatory integral in \eqref{eq: 1c-ps FIO definition} (since we can ignore the Schwartz term for the argument). 

By definition, for $q \in \pi_{\oc}^{-1}(\partial \, \overline{{}^\oc{T^*\RR^n}}) \cap \pi_{\ps}^{-1}(\partial \,\overline{{}^\ps{T^*\RR^{n+1}}})$ such that $q \notin \pi_{\oc}^{-1}(\pi_{\oc}(\Legps)) \cap \pi_{\ps}^{-1}(\pi_{\ps}(\Legps))$, we can find $Q_{\ps} \in \Psi_{\ps}^{0,0}(\R^{n+1})$ that is elliptic at $\pi_{\ps}(q)$ and $Q_{\oc} \in \Psi_{\oc}^{0,0}(\R^n)$ that is elliptic at $\pi_{\oc}(q)$ while
\begin{equation} \label{eq:WF-disjoint-4}
   \pi_{\ps}^{-1}(\WF'_{\ps}(Q_{\ps})) \cap 
   \pi_{\oc}^{-1}(\WF'_{\oc}(Q_{\oc})) \cap L = \emptyset.
\end{equation}

Then the kernel of $Q_{\ps}AQ_{\oc}$ is given by
\begin{align} \label{eq:QP0(Id-Q)-osc}
\begin{split}
Q_{\ps} A Q_{\oc} = &
(2\pi)^{-3n+1} \int
e^{i(t-t')\tau + (z-z') \cdot \zeta}
e^{ i( \frac{ \varphi_0(t',\theta_0) }{x_{\oc}^2} + \frac{\varphi_1(\msf{K}',\theta_0,\theta_1)}{x_{\oc}} )} 
x_{\oc}^{-(m+\frac{2k_0+k_1}{2})-\frac{1}{4}} 
e^{ i \xi_{\oc}\frac{x_{\oc}-x_{\oc}'}{x_{\oc}^3}+\eta_{\oc} \cdot \frac{y_{\oc}-y_{\oc}'}{x_{\oc}}}
\\   &  a(x_{\oc},y_{\oc},t',z',\theta_0,\theta_1) 
q_{\ps}(t,z,\tau,\zeta)
q_{\oc}(x_{\oc},y_{\oc},\xi_{\oc},\eta_{\oc}) d\theta_0d\theta_1 \frac{dx_{\oc}dy_{\oc}}{(x_{\oc})^{n+2}}dt'dz'd\tau d\zeta.
\end{split}
\end{align}

Let $\tilde{\tau},\tilde{\zeta}$ be as in \eqref{eq: Defition, tilde tau, tilde zeta 2} and we use the same decomposition as in \eqref{eq: PA decomposition}.
Then the contribution from the part with $\tilde{\tau},|\tilde{\zeta}| \gg 1$ is Schwartz by non-stationary argument in $(t',z')$. 
%%%%%%%%%%%%%%%%%%%%%%%%%%%%%%%%%%%%%%%%%
Now over the region $\tilde{\tau},|\tilde{\zeta}| \sim 1$, \eqref{eq:WF-disjoint-4} shows that near the critical set of the phase, which corresponds to $L$ under the parametrization map, either 
\begin{itemize}
    \item The full symbol of $Q$ vanishes to infinite order in $\rho_{\ps}$; 
    or 
    \item The full symbol of $Q'$ vanishes to infinite order in $x_{\oc}$.
\end{itemize}
Since this oscillatory integral, modulo a Schwartz term, is integrating over a region that is compact in terms of rescaled variables like $\tilde{\tau},\tilde{\zeta}$ in \eqref{eq: Defition, tilde tau, tilde zeta 2},
so we can decompose it into three parts: first a term with amplitude that is supported away from critical points of the phase, and a finite sum so that in each term one of the two situations above happens on the entire support of the amplitude.
In either case\footnote{The treatment to the first term with amplitude supported away from the critical points is the same as the proof of Lemma~\ref{lemma: vanishing amplitude->lower order}.}, we can integrate by parts to gain arbitrarily large order of decay in $x_{\oc}$, which in addition gives regularity in $(t,z,x_{\oc}',y_{\oc}')$ as well since differentiating in them loses at most polynomial growth in $x_{\oc}^{-1}$.

In addition, since the amplitude in $A$ is compactly supported in $(t',z')$, so via a non-stationary phase argument in $\tilde{\tau},\tilde{\zeta}$, we know that when $(t,z) \to \infty$ (which means $|(t-t')| \gtrsim |t|$ or $|z-z'| \gtrsim |z|$ when $(t',z')$ is in the support of $a(x_{\oc},y_{\oc},t',z',\theta_0,\theta_1)$), the kernel have arbitrary order of decay in $\la (t,z) \ra^{-1}$ as well.
This shows $q \notin \WF_{\oc-\ps}'(A)$ and proves \eqref{eq:1c-ps-FIO-WF-general}.
%so does $\tilde{A}$ and $Q\tilde{A}Q'$ modulo a Schwartz term since they are only acted by pseudodifferential operators from the left, which only give Schwartz contributions off the (lifted) diagonal (i.e.,  $(t',z')$ bounded but $(t,z) \to \infty$ in expressions like \eqref{eq:QP0(Id-Q)-osc} in this case). In particular, we 

Now we turn to prove \eqref{eq:1c-ps-FIO-WF-general-2}.
We only need to show that for $q_{\ps} \in \partial \overline{{}^{\ps}T^*\R^{n+1}}$ such that 
\begin{equation} \label{eq:qps-condition}
    q_{\ps} \notin \WF_{\oc-\ps}(A) \circ \WF_{\oc}(f),
\end{equation}
we can find $Q_{\ps} \in \Psi_{\ps}^{0,0}$ such that $QAf \in \mathcal{S}(\R^{n+1})$.
We decompose $f$ as follows. Let
\begin{equation}
    \Id = Q_1 + Q_2, \; Q_i \in \Psi_{\oc}^{0,0}(\R^n)
\end{equation}
be a microlocal partition of unity such that $q_{\ps} \notin \WF'_{\oc-\ps}(A) \circ \WF'_{\ps}(Q_1)$ and $\WF_{\ps}(f) \cap \WF'_{\ps}(Q_2) = \emptyset$.
In particular, the second condition is equivalent to requiring $Q_1$ to be microlocally identity on $\WF_{\ps}(f)$.
This is possible because of \eqref{eq:qps-condition} and Urysohn's lemma (applied to construct symbols of $Q_i$).

For the part $Q_2f$ we have $Q_2f \in \mathcal{S}(\R^n)$ and we observe that $A$ sends Schwartz functions to Schwartz functions for the following reason. Decay in $(t,z)$ is straightforward by the support condition on its amplitude, and the regularity in $(t,z)$ is because differentiating in them only introduces a factor that is a finite power of $x_{\oc}^{-1}$, which can be absorbed by our Schwartz (hence $O(x_{\oc}^\infty)$) function.

For the $Q_1$-term, the proof above for \eqref{eq:1c-ps-FIO-WF-general} applies to $Q_{\ps}AQ_1$ since now \eqref{eq:WF-disjoint-4} is satisfied with $Q_1$ in place of $Q_{\oc}$. Consequently, we have $Q_{\ps}AQ_1 \in \mathcal{S}(\R^{n+1} \times \R^n)$, which concludes the proof.

\end{proof}

\begin{remark}
Here we used the fact that $A$ sends $\mathcal{S}(\R^n)$ to $\mathcal{S}(\R^{n+1})$, which is not satisfied by general elements in $\mathcal{S}'(\R^{n+1} \times \R^n)$.
In particular, this is not true for the Poisson operator $\Poiz$ without microlocalizers: $\Poiz f$ will have $\ps$-wavefront set present at $\mathcal{R}_\pm$ even if $f \in \mathcal{S}(\R^n)$, which corresponds to the typical oscillation and decay of solutions to the Sch\"odinger equation.
As in \cite[Section~4]{FIO1}, such Schwartz to Schwartz property, or more directly mapping properties of wavefront sets like \eqref{eq:1c-ps-FIO-WF-general-2} could only hold if one can trade off the regularity or decay of the left and right variables. In the setting of \cite[Section~4]{FIO1}, H\"ormander required the canonical relation to be closed, conic and avoid zero sections so that momentum in the left and right variables are comparable. 
In the current setting, as pointed out in \cite[Equation~(7.7)]{gell2022propagation}\cite[Equation~(5.33)]{hassell2024final}, one can trade off between the module regularity in $(t,z)$ and the $\oc$-regularity in $x_{\oc},y_{\oc}$. However, operators defining this module regularity are precisely characteristic at $\mathcal{R}_\pm$ and $\Poiz$ (of course, this discussion applies to $\Poipm$ as well) and our canonical relation hits the zero section of ${}^{\oc}T^*\R^n$ in the right variable at those places and fails to have Schwartz to Schwartz property microlocally there.
\end{remark}

%%%%%%%%%%%%%%%%%%%%%%%%%%%%%%%%%%%%%%%%%%%%
\subsection{Generalized inverses and their properties}
In this and the next subsections, we turn to the specific case where $P$ is as in \eqref{eq:def-Schrodinger-op}. 

We first discuss inverses of the time dependent Sch\"odinger equation in this subsection and discuss their properties. These almost follows from main results of \cite{gell2022propagation} with minor adaptations.
We first briefly recall the set up of microlocal analysis, in particular the non-elliptic Fredholm theory for the Sch\"odinger equation in \cite{gell2022propagation}.

We choose variable decay orders $\mathsf{s}_\pm$ at fibre-infinity (i.e. the differential order) and $\mathsf{r}_\pm$ at spacetime infinity (the decay order), which are smooth on the compactified phase space $C^\infty(\overline{^{\ps}T^*\R^{n+1}})$, such that:
\begin{itemize}
    \item $\mathsf{r}_\pm < - \frac{1}{2}$ at $\mathcal{R}_\pm$.
    \item $\mathsf{r}_\pm > - \frac{1}{2}$ at $\mathcal{R}_\mp$.
    \item $\mathsf{r}_+$ and $\mathsf{s}_+$ are monotone nonincreasing along the (rescaled) $H_p$-flow in $\overline{^{\ps}T^*\R^{n+1}}$.
    \item $\mathsf{r}_-$ and $\mathsf{s}_-$ are monotone nondecreasing along the (rescaled) $H_p$-flow in $\overline{^{\ps}T^*\R^{n+1}}$.
\end{itemize}
Setting 
\begin{equation}
    \mathcal{X}^{\mathsf{s}_\pm,\mathsf{r}_\pm} = 
    \{ u \in H_{\ps}^{\mathsf{s}_\pm,\mathsf{r}_\pm}(\R^{n+1}) : \, Pu \in H_{\ps}^{\mathsf{s}_\pm-1,\msf{r}_\pm +1}. \},
\end{equation}
then \cite[Theorem~1.1]{gell2022propagation} asserts that\footnote{Actually \cite{gell2022propagation} only proved this result for constant $\mathsf{s}$, but the adaptation to variable differential order is straightforward.}
\begin{equation}
    \mathcal{X}^{\mathsf{s}_\pm,\mathsf{r}_\pm} \to H_{\ps}^{\mathsf{s}_\pm-1,\msf{r}_\pm +1}
\end{equation}
is invertible with bounded inverse, which means we have
\begin{equation} \label{eq:LS-main}
    \| u \|_{H_{\ps}^{\mathsf{s}_\pm,\mathsf{r}_\pm}(\R^{n+1})} \leq C \| Pu \|_{H_{\ps}^{\mathsf{s}_\pm-1,\msf{r}_\pm +1}}.
\end{equation}
Let 
\begin{equation}\label{eq:ps inverses}
P_{\pm}^{-1}: \;    H_{\ps}^{\mathsf{s}_\pm-1,\msf{r}_\pm +1} \to  \mathcal{X}^{\mathsf{s}_\pm,\mathsf{r}_\pm}.
\end{equation}
be the corresponding (generalized) inverses, then we have the following characterization of their outputs, phrased in terms of the forward/backward bicharacteristic relations $\FBR, \BBR$ from Definition~\ref{defn:FBR}. 

\begin{lmm} \label{lemma:non-forward-related-composition}  
For $f \in \mathcal{S}'(\R^{n+1})$ such that $\WF_{\ps}(f) \cap \Rm = \emptyset$, we have
\begin{equation} \label{eq:outgoing-WF-bound}
    \WF_{\ps}(P_+^{-1}f) \subset \WF_{\ps}(f) \cup \FBR \circ \WF_{\ps}(f) \cup \SR_+.
\end{equation}
Similar result holds with $\pm$ sign are switched and $\FBR$ replaced by $\BBR$.

In particular, suppose $Q_1,Q_2 \in \Psi_{\ps}^{0,0}(\R^{n+1})$ satisfy: the projections of $\WF'_{\ps}(Q_1), \, \WF'_{\ps}(Q_2)$ to $\R^{n+1}_{t,z}$ are compact and both $\WF'_{\ps}(Q_2)$ and $\FBR \circ \WF'_{\ps}(Q_2)$ are disjoint from $\WF'_{\ps}(Q_1)$, then the Schwartz kernel of $Q_1P_{+}^{-1}Q_2$ is in $\mathcal{S}(\R^{n+1} \times \R^{n+1})$.
\end{lmm}
\begin{proof}
%Without loss of generality, we may assume $Q_1,Q_2 \in \Psi_{\ps}^{0,0}(\R^{n+1})$ since one can compose invertible operators from the left and right to reduce to this case.
%We only need to show that it sends tempered distributions to Schwartz functions.
We only prove \eqref{eq:outgoing-WF-bound} since the case with $\pm$ sign switched and $\FBR$ replaced by $\BBR$ can be shown in the same way.
To prove \eqref{eq:outgoing-WF-bound}, we only need show that for $f \in \mathcal{S}'(\R^{n+1})$ such that $\WF_{\ps}(f) \cap \Rm = \emptyset$, and any $Q_1 \in \Psi_{\ps}^{0,0}(\R^{n+1})$ such that
\begin{equation} \label{eq:Q1-WF-condition}
\WF'_{\ps}(Q_1) \cap \big( \WF_{\ps}(f) \cup \FBR \circ \WF_{\ps}(f) \cup \SR_+ \big) = \emptyset,
\end{equation}
we have
\begin{equation}
   Q_1P_+^{-1}f \in \mathcal{S}(\R^{n+1}). 
\end{equation}
This follows if for every $S, R$ we have 
\begin{equation}\label{eq:QPinvf}
   Q_1P_+^{-1}f \in H_{\ps}^{S, R}(\R^{n+1}). 
\end{equation}
We show this by choosing variable orders $\mathsf{s}_+$ and $\mathsf{r}_+$ such that $\mathsf{s}_+ \geq S$ and $\mathsf{r}_+ \geq R$ on $\WF_{\ps}'(Q_1) \cup \WF_{\ps}'(G) \cup \WF_{\ps}'(B)$, while $\mathsf{s}_+ \leq s_0$ and $\mathsf{r}_+ \leq r_0$ on $\WF_{\ps}(f)$, where $(s_0, r_0)$ are chosen sufficiently negative so that $f \in H_{\ps}^{s_0-1, r_0+1}(\R^{n+1})$. This is possible due to the condition \eqref{eq:Q1-WF-condition} and the condition $\WF_{\ps}(f) \cap \Rm = \emptyset$, allowing the monotonicity and threshold conditions above to be satisfied. 
%Let $f \in H^{s_0,r_0}_{\ps}$ be a tempered distribution, and set
By construction,  $f \in H_{\ps}^{\mathsf{s}_+-1,\msf{r}_+ +1}$, hence by \eqref{eq:ps inverses}, we have $P_+^{-1} f \in H_{\ps}^{\mathsf{s}_+,\msf{r}_+}$. Then since $\mathsf{s}_+ = S$ and $\mathsf{r}_+ = R$ on $\WF_{\ps}'(Q_1)$, we have \eqref{eq:QPinvf}. The final claim follows from  \eqref{eq:outgoing-WF-bound} and the fact that $\WF_{\ps}(Q_2f) \subset \WF'_{\ps}(Q_2)$.

\end{proof}

Then we have the following corollary when $\WF'_{\ps}(f) = \emptyset$.

\begin{coro}
\label{coro:propagator WF mapping} 
%Wavefront set mapping properties of the propagators
Let $f \in \mathcal{S}(\R^{n+1})$, then
\begin{equation} \label{eq:propagator WF bound}
    \WF_{\ps}(P_\pm^{-1}f) \subset \mathcal{R}_\pm.
\end{equation}
%Let $P_+^{-1}$ be the generalized inverses of $P$ \cite[Theorem~1.1]{gell2022propagation}
\end{coro}

\begin{comment}
\begin{proof}
We only the case with $+$ sign for definiteness as the other case can be proved in the same manner.

For any $q \in \partial(\overline{^{\ps}T^*\R^{n+1}}) \cap \big(\mathcal{R}_+\big)^c$,
We can choose $\mathsf{r}_+$ satisfying conditions above such that $\msf{r}_+$ is arbitrarily large at $q$. Then \eqref{eq:LS-main} shows that $P_+^{-1}f$ is microlocally in $\ps$-Sobolev spaces of arbitrarily high differential and decay order  at all such $q$ since $f$ is is always in the space on the right hand side and \eqref{eq:propagator WF bound} follows.
\end{proof} 
\end{comment}

\subsection{Forward and backward Poisson operators}
Let $\Poipm$ be the Poisson operators for the operator $P$, which we recall are the operators mapping asymptotic data $f_\pm$ to the global solution $u$ to $Pu = 0$ with the asymptotic \eqref{eq:u expansion} as $t \to \pm \infty$. We also recall that for the free operator $P_0$, we have $\Poim = \Poip = \Poi_0$ where $\Poi_0$ is given explicitly by \eqref{eq: Poisson zero}. 
With the symbolic calculus established, we can now construct the (microlocalized) Poisson operators as fibred-Legendre distributions. 

Let $P$ be as in \eqref{eq:def-Schrodinger-op}, then $P \in \Psi_{\ps}^{2,0}(Y)$ has parabolically homogeneous principal symbol
\begin{align}
p_{\mathrm{hom}} = \tau + |\xi|_{g^{-1}}^2.
\end{align}
In particular, it fits into the framework in Section~\ref{subsec:PsiDO-FIO-composition-1cps}.

Recall the open sets $U_\pm$ and $G_{\pm}$ defined above Proposition~\ref{prop:microlocalized FBSR}.  We choose a one-cusp pseudodifferential operator $Q_{\oc} \in \Psi_{\oc}^{0,0}(\R^n)$ such that $\WF'(Q_{\oc})$ is contained in $U_-$, in particular is disjoint from fibre-infinity in ${}^{\oc} T^* \R^n$, and a parabolic scattering pseudodifferential operator $Q_{\ps} \in \Psi_{\ps}^{0,0}(\R^{n+1})$ such that $\WF'(Q_{\ps})$ is contained in $G_-$, in particular is disjoint from spacetime infinity. We further assume that 
$Q_{\oc}$ is microlocally equal to the identity on all points in ${}^{\oc} T^* \R^n$ whose corresponding bicharacteristics meet the perturbation, i.e. meet $\WF'_{\ps}(P - P_0)$ (here we make use of the canonical identification between ${}^{\oc} T^* \R^n$ and the front face of the blowup $[\Sigma; \Rm]$, as described in Section~\ref{subsec: 1c arises}), while $Q_{\ps} \in \Psi_{\ps}^{0,0}(\R^{n+1})$ is microlocally equal to the identity on $\WF'(P - P_0)$. 
%%%%%%%%%%%%%%%%%%%%%%%%%%%%%%%%%%%%%%%%%%%%

We first demonstrate that whenever we microlocalize from both sides of $\Poi_0$ by operators that are not connected via the sojourn relation, then its contribution is Schwartz.
 
\begin{prop} \label{prop:Poiz-membership}
Let $Q_{\ps} \in \Psi_{\ps}^{0,0}(\R^{n+1})$ and suppose the projection of $\WF'_{\ps}(Q_{\ps})$ in $\R^{n+1}$ is a bounded set, then
\begin{equation}  \label{eq:localized-Poiz-membership}
    Q_{\ps} \Poiz \in I_{\oc-\ps}^{-3/4}(\R^{n+1} \times \R^n , \Lambda_0).
\end{equation}

Here $\Lambda_{0} = (\Lambda_0)_\pm$ is the microlocalized Lagrangian for the free Schr\"odinger operator $P_0$, as in \eqref{eq:microlocalized sojourn relns} (observe that $(\Lambda_{0})_- = (\Lambda_{0})_+$ for the free operator). 

Moreover, if $Q_{\oc}$ as above is such that $WF'_{\oc}(\Id - Q_{\oc})$ is sufficiently close to fibre-infinity in ${}^{\oc} T^* \R^n$, then 
\begin{equation} \label{eq:5.5-3}
    Q_{\ps}\Poiz(\Id-Q_{\oc})f \in \mathcal{S}(\R^{n+1})
\end{equation}
for all $f \in \mathcal{S}'(\R^n)$: equivalently, $Q_{\ps}\Poiz(\Id-Q_{\oc})$ has a Schwartz kernel in $\mathcal{S}(\R^{n+1} \times \R^n)$. 

\end{prop}

\begin{proof}
In the case that $Q_{\ps}$ is a cut-off function $\chi(t,z) \in C_c^\infty(\R^{n+1})$, this has been discussed already in Example~\ref{example:microlocalized-free-Poisson}. 
For general $Q_{\ps}$ as in the proposition, we choose $\chi \in C_c^\infty(\R^{n+1})$ that is identically one on $\WF_{\ps}'(Q_{\ps})$, then we know $\chi \Poiz \in I_{\oc-\ps}^{-3/4}(\R^{n+1} \times \R^n , \overline{\Lambda_{0, -}'})$ from Example~\ref{example:microlocalized-free-Poisson}.
Then by Theorem~\ref{thm: PsiDO- 1c-ps FIO composition} we know $Q_{\ps} \chi \Poiz \in  I_{\oc-\ps}^{-3/4}(\R^{n+1} \times \R^n , \overline{\Lambda_{0, -}'})$ and it remains to show $Q_{\ps} (1-\chi) \Poiz \in  I_{\oc-\ps}^{-3/4}(\R^{n+1} \times \R^n , \overline{\Lambda_{0, -}'})$.
This term is in fact in $\mathcal{S}(\R^{n+1} \times \R^n)$.
This is because we know $\Poiz$ sends $\mathcal{S}'(\R^n)$ to $\mathcal{S}'(\R^{n+1})$ by \cite[Proposition~7.5]{gell2022propagation}; on the other hand, $Q_{\ps} (1-\chi) \Poiz$ sends $\mathcal{S}'(\R^n)$ to $\mathcal{S}(\R^{n+1})$ since $Q_{\ps} (1-\chi) \in \Psi_{\ps}^{-\infty,-\infty}(\R^{n+1})$. So $Q_{\ps} (1-\chi) \Poiz $ maps $\mathcal{S}'(\R^n)$ to $\mathcal{S}(\R^{n+1})$, which is equivalent to $Q_{\ps} (1-\chi) \Poiz \in \mathcal{S}(\R^{n+1} \times \R^n)$.  

To prove \eqref{eq:5.5-3} we combine \eqref{eq:localized-Poiz-membership} and Proposition~\ref{prop:general-1cps-FIO-WF}, and note that the composition of the sojourn relation $\Lambda_0'$ with $\WF'_{\oc}(\Id - Q_{\oc})$ will be empty if the latter is sufficiently close to fibre-infinity --- indeed, bicharacteristics with initial condition at 1c-fibre-infinity travel along the spacetime boundary.
\end{proof}

\begin{prop} \label{prop: Possion parametrix, K,-}
Suppose $Q_{\oc}$ is as above and $\tilde{Q}_{\ps} \in \Psi_{\ps}^{0,0}(\R^{n+1})$ satisfies:
\begin{multline}\label{eq:tilde Qps cond}
\text{no point of $\WF_{\ps}'(\tilde Q_{\ps})$ is forward, with respect to the Hamilton flow of $p$,}
\\ \text{of any point of $\WF_{\ps}'(P-P_0)$.} 
\end{multline}

Then there exists $K_- \in I^{-3/4}_{\oc-\ps}(\R^{n+1} \times \R^n, \Lambda_-)$  that is microlocally trivial near $\Rm$ and such that 
\begin{align} \label{eq: parametrix, infinity order}
PK_- = [P,\tilde{Q}_{\ps}]\mathcal{P}_0 Q_{\oc} \text{ modulo } \Schw(\R^{n+1} \times \R^n) 
%\in I^{-\infty}_{\oc-\ps}(\R^{n+1} \times \R^n, \Lambda_-),
\end{align}
microlocally away from $\Rp$. 

A similar statement holds with $-$ replaced by $+$.
\end{prop}

\begin{proof}
We claim that $[P,\tilde{Q}_{\ps}]$ can be decomposed as
\begin{equation}
    [P,\tilde{Q}_{\ps}] = Q_1+Q_2,
\end{equation}
where the projection of $Q_1$ to $\R^{n+1}$ is compact and $Q_2$ satisfies
\begin{equation}
    \pi_{\ps}^{-1}(\WF'_{\ps}(Q_2)) \cap \big( \Lambda_0' \circ \WF'(Q_{\oc}) \big) = \emptyset.
\end{equation}
Indeed, it suffices to ensure that the projection to spacetime of $\WF'_{\ps}(Q_2)$ is disjoint from a large ball; then any bicharacteristic that meets $\WF'_{\ps}(Q_2)$ must have initial point with large 1c-frequency, hence is disjoint from $\WF'(Q_{\oc})$. 
By \eqref{eq:localized-Poiz-membership} and Proposition~\ref{prop:general-1cps-FIO-WF}, then, the contribution from $Q_2\Poiz Q_{\oc}$ is Schwartz and we may replace \eqref{eq: parametrix, infinity order} by the equation $PK_- = Q_1\Poiz Q_{\oc} \mod \mathcal{S}(\R^{n+1} \times \R^n)$ below.

We construct $K_-$ as an asymptotic sum
\begin{align}
K_- = \sum_{j=0}^\infty K_{-,j}, 
\quad K_j  \in I^{-j -\frac{3}{4} }_{\oc-\ps}(\R^{n+1} \times \R^n, \Lambda_-),
\end{align}
such that
\begin{align} \label{eq: parametrix property, up to N}
\Big( P \sum_{j=0}^N K_{-,j} - Q_1\mathcal{P}_0 Q_{\oc} \Big) \in I^{-\frac{3}{4}-N}_{\oc-\ps}(\R^{n+1} \times \R^n, \Lambda_-).
\end{align}

By definition, the projection of $\WF_{\ps}'(Q_1)$ to $\R^{n+1}_{t,z}$ is compact and $Q_1\mathcal{P}_0 Q_{\oc} \in I^{\frac{1}{4}}_{\oc-\ps}(\R^{n+1} \times \R^n, \Lambda_-)$.
Here the order is follows from applying Theorem~\ref{thm: PsiDO- 1c-ps FIO composition} and noticing that $\mathcal{P}_0$ is in $I^{-\frac{3}{4}}_{\oc-\ps}(\R^{n+1} \times \R^n, \Lambda_-)$ after localization on both sides.
Here we have used the fact the support of derivatives of $\tilde{q}_{\ps}$ is disjoint from the metric perturbation,
so $\Lambda_-$ coincides, locally, with the \emph{free} forward sojourn relation obtained from the free operator $P_0$ on $\WF_{\ps}'(Q_1)$, deviating from it only once with bicharacteristics meet the perturbation $P - P_0$. (It is here the our assumption of compact support of the perturbation $P - P_0$ is very convenient.)

In order to satisfy \eqref{eq: parametrix property, up to N} with $N=0$, we only need 
\begin{equation} \label{eq:5.4-6}
    \sigma^{1/4}(PK_{-,0})
    = \sigma^{1/4}(Q_1\mathcal{P}_0 Q_{\oc}).
\end{equation}

\begin{comment}

Near the `starting part' of the flow out, or more precisely $(\mathcal{R}_- \times ^{\oc}T^*X) \cap \Lambda_-$, we set $K_{-,0} = \mathcal{P}_0$ and $K_{-,j}=0$ for $j \geq 1$, they will satisfy \eqref{eq: parametrix, infinity order}.
In addition, by Proposition \ref{prop: bicharacteristics stay in the unperturbed region}, under the identification of $\msf{I}_-$ from $\mathcal{R}_{-,\mathrm{res}}$ to $\overline{^{\oc}T^*X}$,
for bicharacteristic lines with starting at $|Z|<C$, or $|(\xi_{\oc},\eta_{\oc})|>C$ (with the same $C$ as in Proposition \ref{prop: bicharacteristics stay in the unperturbed region}), they (in fact there projection to the $Y$) will never enter $\supp \tilde{g}$, then we can set 
\begin{align}
K_{-,0} = \mathcal{P}_0, \quad K_{-,j}=0 \quad j \geq 1,
\end{align}
over the part of $\Lambda_-$ consists of these bicharacteristic lines.

[TODO: be more specific about regions]
Next we discuss how to continue them keeping \eqref{eq: parametrix, infinity order} to hold over bicharacteristic lines in the region near $x_{\oc}=|Z|^{-1}=0, |(\xi_{\oc},\eta_{\oc})|<C$.
This can be constructed inductively. For the first term, the requirement is
\begin{align*}
PK_{-,0}  \in I^{-\frac{3}{4}}_{\oc-\ps}(\R^{n+1} \times \R^n, \Lambda_-).
\end{align*}
\end{comment}

Since $p$ vanishes on $\Lambda_-$, hence applying Theorem \ref{thm: vanishing principal symbol product} (with $m'=2,m=-\frac{3}{4}$) shows that
 $PK_{-,0} \in I^{\frac{1}{4}}_{\oc-\ps}(\R^{n+1} \times \R^n, \Lambda_-)$ and \eqref{eq:5.4-6} is a transport equation along $\msf{H}^{2,0}_p-$flow with respect to $\sigma^{-3/4}_{\oc-\ps}(K_{-,0})$: 
\begin{align} \label{eq: transport ODE, a-0}
(-i\mathscr{L}_{ \msf{H}_p^{2,0} } 
+i (\frac{n}{2}+1)(x_{\oc}^{-1} \msf{H}_p^{2,0} x_{\oc}) 
+ p_{\sub})(a_{-,0}) = \sigma^{1/4}(Q_1\mathcal{P}_0 Q_{\oc}),
\end{align}
where $a_{-,0}$ is the amplitude of $K_{-,0}$.
This ODE admits a unique smooth solution that is zero in a neighbourhood of $\Rm$ since the support of the right hand side is away from $\beta_{\oc-\ps}^{-1}(\mathcal{R}_{-} \times \overline{^{\oc}T^*X})$; that is, we solve in the forward direction of Hamilton flow. The solution will become singular at $\Rp$ but this does not concern us as we only need to construct a solution microlocally away from $\Rp$. 

More generally, suppose $K_{-,0},...,K_{-,N}$ has been constructed, satisfying (\ref{eq: parametrix property, up to N}), then the equation
\begin{align*}
\sigma_{\oc-\ps}^{-\frac{3}{4}-N}(PK_{-,N+1}) = - \sigma_{\oc-\ps}^{-\frac{3}{4}-N}(P \sum_{j=0}^N K_{-,j} - Q_1\mathcal{P}_0 Q_{\oc})
\end{align*}
is a transport equation along $\msf{H}_p^{2,0}-$flow with respect to
$a_{-,N+1}$, the amplitude of $\sigma_{\oc-\ps}^{-\frac{3}{4}-N-1}(K_{-,N+1})$ after fixing the trivialization:
\begin{align}
(-i\mathscr{L}_{ \msf{H}_p^{2,0} } + i(\frac{n}{2}+1-N)(x_{\oc}^{-1} \msf{H}_p^{2,0} x_{\oc}) + p_{\sub}) a_{-,N+1} = c_{N+1},
\end{align}
with $c_{N+1} = - \sigma_{\oc-\ps}^{-\frac{3}{4}-N}(P \sum_{j=0}^N K_{-,j} - Q_1\mathcal{P}_0 Q_{\oc} )$, which admits a unique smooth solution that vanishes near $\Rm$.
%$\beta_{\oc-\ps}^{-1}(\overline{^{\oc}T^*X} \times \mathcal{R}_{\mp})$,

%Though the parametrix over two regions are constructed separately, but one can enlarge for example the region of the second part slightly so that these two regions have overlap and they together cover the entire $\Lambda_-$. We notice that when they overlap, the construction given by two approaches coincide up to a smoothing operator by the uniqueness of the Poisson operator (or, the uniqueness of the solution to the transport equations in the parametrix construction above). Thus we can use a partition of unity on $\mathcal{R}_{-,\mathrm{res}}$ to glue these two parts together to obtain $K_-$ satisfying \eqref{eq: parametrix, infinity order}.

The distribution $K_{+}$ is constructed similarly, using the backward flow of $\msf{H}^{2,0}_p$. 
%continuing $\mathcal{P}_0$ from a neighborhood of $\mathcal{R}_{+,\mathrm{res}}$.

\end{proof}

\begin{comment}
\begin{coro} \label{corollary:P+inverse-K--relation}
Suppose $Q_{\ps} \in \Psi_{\ps}^{0,0}(\R^{n+1})$ satisfy
$\WF'_{\ps}(Q_{\ps}) \cap \mathcal{R}_+ = \emptyset$,
then the $K_-$ constructed above satisfies
\begin{equation} \label{eq:parametrix-coro}
    Q_{\ps}K_- = Q_{\ps}P_{+}^{-1}(  [P,\tilde{Q}_{\ps}]\mathcal{P}_0 Q_{\oc} ) \mod \mathcal{S}(\R^{n+1} \times \R^n).
\end{equation}
Similar result holds with $-$ and $+$ switched.
\end{coro}

\begin{proof}
We only prove the conclusion for $K_-$ since the result for $K_+$ can be proved in the same way.

This is again verified by applying both sides of \eqref{eq:parametrix-coro} to $f \in \mathcal{S}'(\R^n)$ and then apply $P$ from the left.
Concretely, let
\begin{equation}
    u_1 = K_- f, \quad u_2 = P_+^{-1}[P,\tilde{Q}_{\ps}]\mathcal{P}_0 Q_{\oc}f
\end{equation}
then since $K_-$ is , we know $u = P_{+}^{-1}PK_- f$.
Consequently, we have
\begin{equation}
    u_1-u_2 = P_+^{-1}(PK_- -[P,\tilde{Q}_{\ps}]\mathcal{P}_0 Q_{\oc})f .
\end{equation}
By \eqref{eq: parametrix, infinity order}, we know $(PK_- -[P,\tilde{Q}_{\ps}]\mathcal{P}_0 Q_{\oc})f \in \mathcal{S}(\R^{n+1})$, applying Corollary~\ref{coro:propagator WF mapping}, we know $\WF_{\ps}'(u_1-u_2) \subset \mathcal{R}_+$, which is disjoint from $\WF'_{\ps}(Q_{\ps})$ and this leads to \eqref{eq:parametrix-coro}.

\end{proof}
\end{comment}

Next we show that microlocalized $\Poipm$ falls in our 1c-ps FIO class by justifying that we can view $K_\pm$ as `parametrices' for the operator $Q_{\ps} \Poipm Q_{\oc}$ in the perturbation region, where $P - P_0$ is non-vanishing.  

\begin{prop}\label{prop:microlocalized Poisson are 1c-ps FIOs} 
Let $Q_{\oc}$ and $Q_{\ps}$ be as above. Then $Q_{\ps} \Poipm Q_{\oc}$ are fibred-Legendre distributions associated to the microlocalized Lagrangians $\Lambda_\pm$. That is, these operators are 1c-ps FIOs associated to the Lagrangians $\Lambda_\pm$.
\end{prop} 

\begin{proof}
We only prove the conclusion for $Q_{\ps} \Poim Q_{\oc}$ since the conclusion for $Q_{\ps} \Poip Q_{\oc}$ can be proved in the same way.

First we notice that, because $K_-$ is microlocally trivial near $\Rm$ by construction, and satisfies 
\eqref{eq: parametrix, infinity order} microlocally away from $\Rp$, we have 
\begin{equation} \label{eq:parametrix-coro}
    Q_{\ps}K_- = Q_{\ps}P_{+}^{-1}(  [P,\tilde{Q}_{\ps}]\mathcal{P}_0 Q_{\oc} ) \mod \mathcal{S}(\R^{n+1} \times \R^n)
\end{equation}
for any $Q_{\ps}$ microlocally trivial near $\Rp$. 

Next, choose $\tilde Q_{\ps}$ to be microlocally equal to the identity near $\Rm$ so that \eqref{eq:tilde Qps cond} is satisfied. 
Similar reasoning shows that we can express 
\begin{align} \label{eq:minus-Poisson-rep-1}
\mathcal{P}_- = \tilde{Q}_{\ps} \mathcal{P}_- 
- P_+^{-1}[P,\tilde{Q}_{\ps}]\mathcal{P}_-.
\end{align}
Indeed, $(\Id - \tilde{Q}_{\ps}) \mathcal{P}_-$ is microlocally trivial near $\Rm$ due to the choice of $\tilde Q_{\ps}$, and satisfies
$$
P \big( (\Id - \tilde{Q}_{\ps}) \mathcal{P}_- \big) = [P,\tilde{Q}_{\ps}]\mathcal{P}_-,
$$
so we can apply $P_+^{-1}$ on the left to obtain \eqref{eq:minus-Poisson-rep-1}. 

\begin{comment}
One can verify this by applying both sides to a function on $\R^n$ and then apply $P$ from the left. Then we see
\begin{align*}
P(\tilde{Q}_{\ps} \mathcal{P}_- 
- P_+^{-1}[P,\tilde{Q}_{\ps}]\mathcal{P}_-)f
= P \tilde{Q}_{\ps} \mathcal{P}_-f - [P,\tilde{Q}_{\ps}]\mathcal{P}_-f = 0,
\end{align*}
where in the last equality we used $P\mathcal{P}_-f=0$.
In addition, $P_+^{-1}[P,\tilde{Q}_{\ps}]\mathcal{P}_-$ has no contribution to the incoming data, hence the entire operator $\tilde{Q}_{\ps} \mathcal{P}_- 
- P_+^{-1}[P,\tilde{Q}_{\ps}]\mathcal{P}_-$ coincide with $\mathcal{P}_-$.  See \cite[Proof of Proposition~7.5]{gell2022propagation} for more details.
\end{comment}

Multiplying both sides of \eqref{eq:minus-Poisson-rep-1} by $Q_{\ps}$ and $Q_{\oc}$ from the left and right, we have 
\begin{align} \label{eq:5-4-01}
Q_{\ps}\mathcal{P}_-Q_{\oc}
= Q_{\ps} \big( \tilde{Q}_{\ps} \mathcal{P}_- 
- P_+^{-1}[P,\tilde{Q}_{\ps}]\mathcal{P}_- \big) Q_{\oc}.
\end{align}

We can arrange $\WF_{\ps}'(Q_{\ps})$ and $\WF_{\ps}'(\tilde{Q}_{\ps})$ to be disjoint, which makes the first term on the right hand side above microlocally trivial (i.e., whose kernel is Schwartz).
Now we consider the second term.
Recall the expression (see \cite[Definition~6.4, Proposition~7.5]{gell2022propagation}) for the perturbed Poisson operator:
\begin{align} \label{eq:Poim-expression}
\mathcal{P}_\mp = \mathcal{P}_0 - P_\pm^{-1}(P-P_0)\mathcal{P}_0,
\end{align}
the second term of the right hand side of \eqref{eq:5-4-01} can be rewritten as 
\begin{align*}
-Q_{\ps}P_+^{-1}[P,\tilde{Q}_{\ps}](\mathcal{P}_0 - P_+^{-1}(P-P_0)\mathcal{P}_0 ) Q_{\oc}.
\end{align*}

For the term
\begin{align} \label{eq:5-5-3}
Q_{\ps}P_+^{-1}[P,\tilde{Q}_{\ps}]P_+^{-1}(P-P_0)\mathcal{P}_0,
\end{align}
by property \eqref{eq:tilde Qps cond} of $\tilde Q_{\ps}$ and Lemma~\ref{lemma:non-forward-related-composition}, we know the kernel of $[P,\tilde{Q}_{\ps}]P_+^{-1}(P-P_0)$ is a Schwartz function in $\mathcal{S}(\R^{n+1} \times \R^{n+1})$.
By the mapping property of $\mathcal{P}_0$ in \cite[Proposition~7.5]{gell2022propagation}, we know $\mathcal{P}_0$ sends a tempered distribution on $\R^n$ to a tempered distribution on $\R^{n+1}$.
So the composition $[P,\tilde{Q}_{\ps}]P_+^{-1}(P-P_0)\mathcal{P}_0$ sends tempered distributions (on $\R^n$) to Schwartz functions (on $\R^{n+1}$).
Then by Corollary~\ref{coro:propagator WF mapping}, we know the
$\ps$-wavefront set of the output of $P_+^{-1}[P,\tilde{Q}_{\ps}]P_+^{-1}(P-P_0)\mathcal{P}_0$ is contained in $\mathcal{R}_+$, which is disjoint from $\WF{\ps}'(Q_{\ps})$ and the final output of $Q_{\ps}P_+^{-1}[P,\tilde{Q}_{\ps}]P_+^{-1}(P-P_0)\mathcal{P}_0$ is Schwartz. Consequently, the Schwartz kernel of \eqref{eq:5-5-3} is in $\mathcal{S}(\R^{n+1} \times \R^n)$.

We are left with the `main term' $Q_{\ps}P_+^{-1}[P,\tilde{Q}_{\ps}]\mathcal{P}_0 Q_{\oc}$. By \eqref{eq:parametrix-coro}, and Propositions~\ref{prop: Possion parametrix, K,-} and \ref{thm: PsiDO- 1c-ps FIO composition}, this is in 
$I^{-3/4}_{\oc-\ps}(\R^{n+1} \times \R^n, \Lambda_-)$, which completes the proof. 
\end{proof}

Finally we recall \cite[Corollary~7.6]{gell2022propagation} here, for use in Section~\ref{sec: scattering map}:
\begin{prop}\label{prop:Poisson adjoint mapping} 
The adjoint of the Poisson operator, $\Poipm^*$, maps Schwartz functions to Schwartz functions. \end{prop}

%%%%%%%%%%%%%%%%%%%%%%%%%%%%%%%%%%%%%%%%%%%%%%%%%%%%%%%%%%
%%%%%%%%%%%%%%%%%%%%%%%%%%%%%%%%%%%%%%%%%%%%%%%%%%%%%%%%%%
%%%%%%%%%%%%%%%%%%%%%%%%%%%%%%%%%%%%%%%%%%%%%%%%%%%%%%%%%%
%%%%%%%%%%%%%%%%%%%%%%%%%%%%%%%%%%%%%%%%%%%%%%%%%%%%%%%%%%
%%%%%%%%%%%%%%%%%%%%%%%%%%%%%%%%%%%%%%%%%%%%%%%%%%%%%%%%%%

\section{Geometry of the 1c-1c phase space}\label{sec:1c-1c geometry}
%[TO DO] numerology: if we continue to use the spirit of H\"ormander, then a reasonable choice might be: let $A_2^*A_1$ coincide with the order of 1c-PsiDO, when the canonical relation is the antipodal one, hence `almost' is a PsiDO, modulo a antipodal action.

\subsection{The b-double space and the lifted 1c-1c cotangent bundle}

Let $X$ be a manifold with boundary $\partial X$ (in our case, $X=\overline{\R^n}$), the b-double space introduced by Melrose \cite{melrose1993atiyah}
is defined to be
\begin{align} \label{eq: defn  b-double space}
X_b^2 : = [ X \times X ; \partial X \times \partial X ],
\end{align}
and the blow down map $X_b^2 \rightarrow X^2$ is denoted $\beta_b$.
Then we define the b-lifted 1c-1c cotangent bundle to be
\begin{align} \label{eq:def-lifted-ococ-bundle}
\ococb:= \beta_b^*(\leftidx{^{\oc}}{T^*X} \times \leftidx{^{\oc}}{T^*X}),
\end{align}
where the right hand side is viewed pulling back a vector bundle over $X^2$ to $X_b^2$.
And denote the lift of $\partial X \times X, X \times \partial X, \partial X \times \partial X$ under $\beta_b$ by $\mathrm{lb}$ (`left boundary'), $\mathrm{rb}$ (`right boundary') and $\mathrm{bf}$ (`b-face') respectively. We denote the part of the bundle $\ococb$ lying over $\mathrm{bf}$ by $\ffococ$. 

It has more complicated fibered 1-cusp structure if we are concerned with the behaviour near the left and right boundaries; see \cite[Section~3]{hassell1999spectral} for the analogue for the scattering calculus. But away from these faces,  it has a relatively simple characterization.
Let $(x_{\oc,1}, y_1;x_{\oc,2}, y_2)$ be coordinates on $X^2$ when each of them are near the boundary. Let 
\begin{align*}
x_{\oc} = x_{\oc,1}, \sigma = \frac{x_{\oc,1}}{x_{\oc,2}},
\end{align*}
then $\msf{X}:=(x_{\oc},\sigma, y_1, y_2)$ forms a local coordinate system of $X_b^2$ on the region $\{ C^{-1} \leq \sigma \leq C \}$ for a fixed $C$. On each of the left and right factors of $^{\oc}{T^*X}$ we have coordinates $(\xoc, \yoc, \xioc, \etaoc)$. When lifted to $X^2_b$ by the left and right projections we will write these $(\xoco, \yoco, \xioco, \etaoco)$ and $(\xoct, \yoct, \xioct, \etaoct)$ respectively. An arbitrary element of  $\ococb$ is given by 
\begin{equation}\label{eq:1c-1c vector}
\xioco \frac{d\xoco}{\xoco^3} + \etaoco \frac{d\yoco}{\xoco} + \xioct \frac{d(\xoco/\sigma)}{(\xoco/\sigma)^3} + \etaoct \frac{d\yoct}{(\xoco/\sigma)},
\end{equation}
since $\xoct = \xoco/\sigma$. This is therefore the expression for the canonical one-form in these coordinates. The symplectic form is the differential of this, and takes the form 
\begin{equation}\label{eq:omega 1c1c}
\omega_{\oc-\oc} = d\xioco \wedge \frac{d\xoco}{\xoco^3} + d \big(\frac{\etaoco}{\xoco} \big) \wedge  d\yoco + d \big(\frac{\sigma^3 \xioct}{\xoco^3}\big) \wedge  d \big( \frac{\xoco}{\sigma} \big) + d \big( \frac{\sigma \etaoct }{\xoco} \big) \wedge d\yoct. 
\end{equation}

Notice that \eqref{eq:1c-1c vector} can be written as
\begin{align}
\label{eq:1c1c-1form-0}
\begin{split}
( \xioco + \sigma^2 \xioct) \frac{d\xoco}{\xoco^3} - (\sigma \xioct) \frac{d\sigma}{\xoco^2} + \etaoco \frac{d\yoco}{\xoco} + \sigma \etaoct \frac{d\yoct}{\xoco},
\end{split}
\end{align}
which we can further written as
\begin{equation} \label{eq:1c1c-1form-1}
 \nu_1 \frac{d\xoco}{\xoco^3} + \nu_2 \frac{d\sigma}{\xoco^2} + \etaoco \frac{d\yoco}{\xoco} +  \tilde{\etaoct} \frac{d\yoct}{\xoco},
\end{equation}
where coefficients above are defined via setting \eqref{eq:1c1c-1form-0} to be equal to \eqref{eq:1c1c-1form-0}. 

\subsection{\texorpdfstring{Symplectic structures and Legendrian submanifolds of the $\oc-\oc$ phase space}{Symplectic structures and Legendrian submanifolds of the 1c-1c phase space}}\label{subsec: symplectic structures 1c 1c}
The symplectic structure of $\ococb$, and the way it behaves at the front face $\ffococ$, is very similar to that of the $\oc-\ps$ phase space discussed in Section~\ref{sec:1c-ps geometry}. Recall from the discussion in Section~\ref{subsec:symplectic structures 1cps} that the coordinates $\xioco$, $\xioct$ and $\sigma$ are all well-defined (independent of choices) at $\ffococ$. We define a projection map with $(4n-4)$-dimensional fibres:
\begin{equation}\label{eq:pi 1c defn}
\pi_{\oc} : \ffococ \to \RR^3_{\xioco, \xioct, \sigma}
\end{equation}
by mapping to the coordinates $(\xioco, \xioct, \sigma)$ and forgetting the other coordinates.  

\begin{lmm}\label{lem:symplectic geom 1c1c} The symplectic form $\omega_{\oc-\oc}$ in the interior of $\ococb$, contracted with the vector field $-\xoc^3 \partial_{\xoc}$ (with respect to the coordinate system above), restricts to a smooth 1-form on $\ffococ$, which is the pullback, via $\pi_{\oc}$, of a contact 1-form on $\RR^3_{\xioco, \xioct, \sigma}$. Moreover, on each fibre of $\pi_{\oc}$, the symplectic form induces a symplectic structure by contracting with the vector field $-\xoc^2 \partial_{\xoc}$, restricting to $\ffococ$, and then taking the differential. 
\end{lmm}

\begin{proof} The proof is almost identical to the proof of Lemma~\ref{lem:symplectic geom 1cps}, so we shall be brief. 
A straightforward calculation shows that contracting \eqref{eq:omega 1c1c} with $-\xoc^3 \partial_{\xoc}$ gives  
\begin{equation}\label{eq:3dcontactform-3}
d\xioco  +  \sigma^2 d \xioct  + O(\xoc),
\end{equation}
which restricts to $\{ \xoc = 0 \}$ to the contact form 
\begin{equation}\label{eq:contact form 1c-1c}
d\xioco  +  \sigma^2 d \xioct
\end{equation}
 on $\RR^3$ pulled back via $\pi$. Furthermore, if we restrict to a fibre, then $\xioco$, $\xioct$ and $\sigma$ are all fixed, which kills the $(\xoco)^{-3} d\xoco$ terms in \eqref{eq:omega 1c1c} and means that contraction with the larger vector field $-\xoc^2 \partial_{\xoc}$ has a smooth limit at $\xoc = 0$. This gives the 1-form 
\begin{equation}\label{eq:fibre1form 1c}
\alpha_{\oc} = \etaoco \cdot d\yoco + \sigma \etaoct \cdot d\yoct. 
\end{equation}
Taking the differential gives (remembering $\sigma$ is constant on the fibre)
\begin{equation}\label{eq:fibresympform-1c1c}
d\etaoco \wedge d\yoco + \sigma d\etaoct \wedge d\yoct,
\end{equation}
on each fibre, which is a symplectic form. 
\end{proof}

We continue our discussion parallel to that of Section~\ref{subsec:symplectic structures 1cps}. The next task is to define Lagrangian and fibred-Legendre submanifolds.

\begin{defn}[Admissible 1c-1c Lagrangian submanifold and 1c-1c fibred-Legendre submanifold]  \label{defn: admissible 1c-1c Lagrangian submanifold and 1c-1c fibred-Legendre submanifold}
We define an admissible $\oc-\oc$ Lagrangian submanifold of $\ococb$ to be a $2n$-dimensional submanifold $\Uplambda$ that is Lagrangian in the interior (that, is the canonical symplectic form $\omega_{\oc-\oc}$ from \eqref{eq:omega 1c1c} vanishes on it), such that 
\begin{itemize}
\item $\Uplambda$ meets $\ffococ$ transversally, 
\item the differential $d\xioco$ is nonvanishing on $\Uplambda \cap \ffococ$, and 
\item  its closure is disjoint from all other boundary hypersurfaces of $\ococb$ (other than $\ffococ$). 
\end{itemize}

We define a fibred-Legendre submanifold $\SL$ of $\ffococ$ to be the boundary of an admissible 1c-1c Lagrangian submanifold. In other words, there exists $\Uplambda$ as above such that $\SL = \Uplambda \cap \ffococ$. 
%Legendre submanifold of $\ffocps$ to be the intersection, with $\ffocps$, of a $2n+1$-dimensional Lagrangian submanifold $\Lambda \subset \SM$ that meets $\ffocps$ transversally (and hence in a submanifold of dimension $2n$). 
%We say that $\Lambda$, as well as $\Lambda \cap \ffocps$ is admissible if its closure is disjoint from all other boundary hypersurfaces of $\SM$. We shall only consider admissible Legendre submanifolds in this article, and usually omit the adjective `admissible' in the sequel.
\end{defn}

\begin{prop}\label{prop:1c-1c contact symplectic property} Let $\SL$ be a fibred-Legendre submanifold of $\ffococ$, and let $\pi_{\SL}$ be the restriction of $\pi_{\oc}$ to $\SL$, where $\pi_{\oc}$ is as in \eqref{eq:pi 1c defn}. Then $\pi_{\SL}(\SL)$ is a Legendre curve of $\RR^3_{\xioco, \xioct, \sigma}$, and the fibres of $\pi_{\SL}$ are Lagrangian submanifolds for the symplectic structure on the fibres of $\pi_{\oc}$. 
\end{prop}

\begin{proof} We omit most of the proof, as it is closely analogous to the proof of Proposition~\ref{prop:1c-ps contact symplectic property}. In particular, the proof that $\SL$ fibres over a Legendre curve in $\RR^3$ is essentially identical, and thus omitted.  However, for later use, we repeat the latter part of the proof. 

We claim that the 1-form $\alpha_{\oc}$ given by \eqref{eq:fibre1form 1c}, restricted to any fibre $F$ of $\SL$, is the differential of a function on that fibre. That implies that $d\alpha_{\oc}$ restricts to zero on $F$ (since the exterior derivative $d$ and restriction commute) and thus shows that $F$ is a Lagrangian submanifold of the corresponding fibre of $\ffococ$. 

To prove this claim, we consider the submanifold $\Uplambda_{\xioct^*} := \Uplambda \cap \{ \xioct = \xioct^* \}$. This is a smooth submanifold of dimension $2n-1$ of $\Uplambda$, since $d\xioco$ is nonvanishing at $\partial \Uplambda$ by assumption, and hence, since we have the vanishing of the contact form \eqref{eq:contact form 1c-1c}, and $\sigma \in (0, \infty)$ using the third condition in Definition~\ref{defn: admissible 1c-1c Lagrangian submanifold and 1c-1c fibred-Legendre submanifold}, the differential $d\xioct$ is also nonvanishing at $\partial \Uplambda$.  Moreover, if we fix $\xioct^*$ and let $\xoco$ vary, then since $d\xoco$ is also nonvanishing at $\partial \Uplambda$, and hence in a neighbourhood, $\Uplambda \cap \{ \xioct = \xioct^* \}$ fibres over $[0, \epsilon]_{\xoco}$ with smoothly varying fibre $\Uplambda_{\xioct^*, \xoco^*} = \Uplambda \cap \{ \xioct = \xioct^* , \xoco = \xoco^* \}$, $\xoct^* \in [0, \epsilon]$, by the constant rank theorem. 

Given a smooth vector field  $V$ tangent to $F = \Uplambda_{\xioct^*, 0}$, we can extend it to a smooth vector field $\tilde V$ on  $\Uplambda_{\xioct^*}$, tangent to $\Uplambda_{\xioct^*, \xoco}$ for all sufficiently small $\xoco$, i.e. $\tilde V$ has no $\partial_{\xoco}$ nor $\partial_{\xioct}$ component. Then, we consider how $\xioco + \sigma^2 \xioct$ (which is the coefficient of $d\xoco/\xoco^3$ in \eqref{eq:1c-1c vector})  varies on $\Uplambda$ for $\xoco > 0$. Since $\xioco$ and $\sigma$ depend only on $\xioct$ at $\partial \Uplambda$, there are smooth functions $\phi_0 = \phi_0(\xioct)$ and $\phi_1$ such that we can express 
\begin{equation}\label{eq:xioc-expand-1c}
\xioco + \sigma^2 \xioct = -2 \phi_0(\xioct) - \xoco \phi_1  \  \text{ on } \Uplambda_{\xioct^*}. 
\end{equation}
(The strange coefficients of $-2$ and $-1$ are chosen to agree with the computation in the next subsection.) 
Notice that, since $d\xioco + \sigma^2 d\xioct$ vanishes on the Legendre curve at the base of the fibration $\pi_{\SL}$, and $\xioco + \sigma^2 \xioct = -2 \phi_0(\xioct)$ on this Legendre curve, we have 
\begin{equation}\label{eq:1-form identity base}
2d\phi_0 = -2\sigma \xioct d\sigma \text{ on } \SL. 
\end{equation}

We next choose a vector field $W$ as before,  tangent to $\Uplambda_{\xioct^*}$ (and not merely $\Uplambda$) with $\partial_{\xoco}$ component equal to $1$. We write $W = \partial_{\xoco} + \tilde W$, where $\tilde W$ has no $\partial_{\xoco}$ component nor $\partial_{\xioct}$ component. We calculate,  using \eqref{eq:xioc-expand-1c} to obtain the fourth line and \eqref{eq:1-form identity base} for the fifth, and noticing that $d\sigma$ and $d\xioct$ are $O(\xoco)$ restricted to any fibre $F$ of $\SL$, 
\begin{equation}\begin{gathered}
0 = \omega_{\oc-\oc}(\xoco^2 W, \tilde V) =  \omega_{\oc-\oc}(\xoco^2 \partial_{\xoco}, \tilde V) + \omega(\xoco^2 \tilde W, \tilde V)  \\
= \Big( -\frac{d\xioco + \sigma^2 d\xioct}{\xoco} -  \alpha_{\oc} \Big)(\tilde V) + O(\xoco) \\
= \Big( -\frac{d(\xioco + \sigma^2 \xioct) - 2 \sigma \xioct d\sigma}{\xoco} -  \alpha_{\oc} \Big)(\tilde V) + O(\xoco) \\
= \Big( -\frac{-2d\phi_0 - 2 \sigma \xioct d\sigma - \xoco d\phi_1}{\xoco} -  \alpha_{\oc} \Big)(\tilde V) + O(\xoco) \\
= (d\phi_1 - \alpha_{\oc})(\tilde V) + O(\xoco).
\end{gathered}\end{equation}
Taking the limit as $\xoco \to 0$ we find that 
$$
\alpha_{\oc} = d\phi_1 \ \text{ on } F. 
$$
Thus, $d\alpha_{\oc}$ vanishes on $F$ and $F$ has dimension $2n-2$, that is, $F$ is a Lagrangian submanifold of the corresponding fibre of $\ffococ$. 
\end{proof}

\subsection{Parametrization of 1c-1c fibred-Legendre submanifolds}
Analogously to Section~\ref{subsec:Parametrization of 1c-ps fibred-Legendre submanifolds}, we want to represent a given  1c-1c fibred Legendre submanifold as the graph of the differential of a function, or an envelope of such graphs. We first treat the simplest case of a graph, where $(\yoco, \yoct, \sigma)$ furnish coordinates on $\SL$, locally. This means that $(\xoco, \yoco, \yoct, \sigma)$ furnish coordinates on $\Uplambda$ locally. Thus $\xioco + \sigma^2 \xioct$ restricted to $\Uplambda$ can be expressed in terms of these coordinates, similarly to \eqref{eq:xioc-expand-1c}, except that we are using $\sigma$ as a coordinate along the base Legendre curve instead of $\xioct$: 
$$
\xioco + \sigma^2 \xioct = -2\phi_0(\sigma) - \xoco \phi_1(\yoco, \yoct, \sigma) + O(\xoc^2).
$$
%This form, with $\phi_0$ depending only on the base variable $\sigma$, follows from Proposition~\ref{prop:1c-1c contact symplectic property} (since the left hand side only depends on the base variables). 
%Thus $-\phi_1$, evaluated at $\sigma = \phi_0^{-1}(\xioco)$, coincides with the function $\Xi_{\oc}$ above. 
We claim that $L$ is given by the graph of the differential of the function 
$$
d\frac{\Phi(\xoco,\yoco, \yoct, \sigma)}{\xoco^2} :=  d\Big( \frac{\phi_0(\sigma) + \xoc \phi_1(\yoco, \yoct, \sigma)}{\xoc^2} \Big) = d \big( \frac{\phi_0(\sigma)}{\xoc^2} \big) + d \big( \frac{\phi_1(\yoco, \yoct, \sigma)}{\xoc} \big),
$$
at $\xoco = 0$. To check this, we use \eqref{eq:1c1c-1form-0}.
%we return to \eqref{eq:1c-1c vector}, and write in the form 
Thus we need to check that, on $\Uplambda$, 
\begin{equation}\label{eq:to check}\begin{gathered}
\xioco + \sigma^2 \xioct = -2 \phi_0(\sigma) + \xoco \phi_1(\yoco, \yoct, \sigma) + O(\xoco^2), \\
- \sigma \xioct = \phi_0'(\sigma)   + O(\xoco), \\
\etaoco = d_{\yoco} \phi_1, \\
\sigma \etaoct = d_{\yoct} \phi_1.
\end{gathered}\end{equation}
The first equation holds by construction. To verify the second, we take the identity $\xioco + \sigma^2 \xioct = -2\phi_0(\sigma) + O(\xoco)$ and take the differential. We obtain, using the vanishing of \eqref{eq:contact form 1c-1c} at $\SL$, 
\begin{equation}\label{eq:base var}
-2 d\phi_0 = d (\xioco + \sigma^2 \xioct) = 2\sigma \xioct d\sigma   \ \text{ on } \SL,
\end{equation}
which implies the second line of \eqref{eq:to check}. The third and fourth lines follow from the identity $\alpha_{\oc} = d\phi_1$ which was obtained above.

In the general case, where we do not assume that $(\yoc, z)$ furnish coordinates on the fibres $F$, we need to allow additional variables in our parametrizing function:

\begin{defn} \label{defn: 1c-1c parametrization}
Suppose $\Legps$ is an 1c-1c fibred-Legendre submanifold of $\ffococ \subset \mathcal{M}$, $q \in \SL$ and $(\theta_0, \theta_1)$ is in a non-empty open subset $U$ of $\RR^{k_1+k_2}$, then we say that 
\begin{align}\label{eq:Phi 1c-1c param}
\Phi_{\oc-\oc}(\msf{K}, \theta_0, \theta_1) =  \frac{ \varphi_0(\sigma,\theta_0) }{x_{\oc}^2} + \frac{\varphi_1(\msf{K},\theta_0,\theta_1)}{x_{\oc}} , \quad \msf{K} = (\xoco, \sigma, \yoco, \yoct), 
\end{align}
gives a non-degenerate parametrization of $\SL$ near $q$ if there is a point $q' \in \RR^{2n-1}_{\sigma, \yoco, \yoct} \times U_{\theta_0, \theta_1}$ such that the differentials
\begin{equation}\label{eq:varphi0 1c}
d_{\sigma,\theta_0} \frac{\partial \varphi_0}{\partial \theta_{0j}}, \quad 1 \leq j \leq k_1,
\end{equation}
and
\begin{equation}\label{eq:varphi1 1c}
d_{\sigma, \yoco, \yoct, \theta_1} \frac{\partial \varphi_1}{\partial \theta_{1j}}, 1 \leq j \leq k_2,
\end{equation}
are linearly independent at $q'$, such that $\SL$ is given locally by 
\begin{align}\label{eq:1c-1c param}
\SL
=\{ (\msf{K},d_{\msf{K}}\Phi_{\oc-\oc}(\msf{K},\theta_0,\theta_1)) \mid \; \xoc = 0, (\msf{K},\theta_0,\theta_1) \in C_{\Phi_{\oc-\oc}}\},
\end{align}
where 
\begin{align}  \label{eq: 1c-1c critical set}
C_{\Phi_{\oc-\oc}} =  \{ (\msf{K},\theta_0,\theta_1) \mid d_{\theta_0}\varphi_0 = 0, d_{\theta_1}\varphi_1 = 0 \}
\end{align}
with the point $q' \in C_{\Phi_{\oc-\oc}}$ corresponding under \eqref{eq:1c-1c param} to $q \in L$. We remark that either $\theta_0$ or $\theta_1$, or both, may be absent, in which case the linear independence conditions for the derivatives in \eqref{eq:varphi0 1c}, \eqref{eq:varphi1 1c} are dispensed with, as well as stationarity with respect to $\theta_0$ and/or $\theta_1$ in \eqref{eq: 1c-1c critical set}. 
\end{defn}

\begin{prop}
Every 1c-1c fibred-Legendre submanifold $L$ has a local non-degenerate parametrization near any point $q \in L$.
\end{prop}

\begin{proof}
In the case where $(\sigma, \yoco, \yoct)$ furnish coordinates on $\SL$, we have seen above that there is a parametrization of $\SL$ with no additional variables $\theta_0, \theta_1$. 

In general, one cannot assume that $(\sigma, \yoco, \yoct)$ furnish coordinates on $\SL$. It is, however, possible to choose coordinates $(\yoco, \yoct) = (\yoco', \yoco'', \yoct', \yoct'')$ (after a linear change of variables), where $\yoco = (\yoco', \yoco'')$ is a splitting of the coordinates into two groups, so that, writing the corresponding `dual' variables $(\etaoco', \etaoco'', \etaoct', \etaoct'')$ (these are not exactly dual variables because of the scaling in $\xoco$), the functions
$$
 \yoco'', \ \etaoco', \ \yoct'', \ \etaoct'
$$
furnish coordinates locally on each fibre of $\SL$  \cite[Theorem 21.2.17]{hormander2007analysis}. It follows that 
$$
\xioct, \  \yoco'', \ \etaoco', \ \yoct'', \ \etaoct'
$$
furnish coordinates locally on $\SL$. 

We write $\xioco + \sigma^2 \xioct = - 2 \phi_0(\xioct) - \xoco \phi_1(\xioct, \yoco'', \etaoco', \yoct'', \etaoct' ) + O(\xoco^2)$ as before on $\SL$. We write the functions $\sigma = \Sigma(\xioct)$, $\yoco' = Y'_1(\xioct, \yoco'', \etaoco', \yoct'', \etaoct' )$ and $\yoct' = Y'_2(\xioct, \yoco'', \etaoco', \yoct'', \etaoct' )$ on $\SL$. Then 
following the argument in \cite[Prop. 21.2.18]{hormander2007analysis}, we claim that the function
\begin{multline}\label{eq:Phi-param-general-1c1c}
\Phi(\msf{K}, \xioct, \etaoco', \etaoct') = \frac{\varphi_0}{\xoco^2} + \frac{\varphi_1}{\xoco}, \quad \msf{K} = (\xoco, \sigma, \yoco, \yoct), \\
\varphi_0(\sigma, \xioct) = \phi_0(\xioct) - \frac1{2} (\sigma^2 - \Sigma^2(\xioct)) \xioct,  \\ 
\varphi_1(\msf{K}, \xioct, \etaoco', \etaoct') = \phi_1(\xioct, \yoco'', \etaoco', \yoct'', \etaoct' ) + (\yoco' - Y'_1) \cdot \etaoco' +  (\yoct' - Y'_2) \cdot \etaoct'
\end{multline}
provides a local parametrization of $\SL$, in the sense that 
\begin{equation}\label{eq:Leg1c param}
\SL = \Big\{ (\msf{K}, d_{\msf{K}} \Phi(\msf{K}, \xioct, \etaoco', \etaoct') \mid \xoc = 0, d_{\xioct} \varphi_0 = 0, \quad d_{\etaoco', \etaoct'} \varphi_1 = 0 \Big\}  .
\end{equation}

Let us verify that the condition $d_{\xioct} \varphi_0 = 0$ implies that $\sigma = \Sigma(\xioct)$. On the Legendre curve at the base of the fibration $\pi_{\SL}$, we write $\xioco = \xioco(\xioct)$ and $\sigma = \Sigma(\xioct)$. Then, in terms of these functions, 
$$
\varphi_0(\sigma, \xioct) = - \frac1{2} \Big( \xioco(\xioct) + \Sigma^2 \xioct + (\sigma^2 - \Sigma^2) \xioct \Big). 
$$
Differentiating in $\xioct$ we find that 
$$
d_{\xioct} \varphi_0(\sigma, \xioct) = - \frac1{2} \Big( \frac{ d\xioco}{d \xioct}  + \Sigma^2  + (\sigma^2 - \Sigma^2)  \Big). 
$$
The vanishing of the contact 1-form implies that the first two terms on the RHS cancel, showing that 
$$
d_{\xioct} \varphi_0(\sigma, \xioct) = 0 \implies \sigma^2 = \Sigma^2.
$$
Then, since both $\sigma$ and $\Sigma$ are positive, this implies that $\sigma = \Sigma$. With $\xioct$ fixed, it is not hard to show that $\varphi_1$ parametrizes the fibre $F_{\xioct}$ indexed by $\xioct$; this is a standard calculation. This completes the proof. 
\end{proof}

We will also need to consider clean parametrizations of a fibred-Legendre submanifold. 

\begin{defn}\label{defn: 1c-1c parametrization clean}
Suppose $\Legps$ is an 1c-1c fibred-Legendre submanifold of $\ffococ \subset \mathcal{M}$, $q \in \SL$ and $(\ococparaone, \ococparatwo)$ is in a non-empty open subset $U$ of $\RR^{k_0+k_1}$. We say that 
\begin{align}\label{eq:Phi 1c-1c param clean}
\Phi_{\oc-\oc}(\msf{K}, \ococparaone, \ococparatwo) =  \frac{ \varphi_0(\sigma,\ococparaone) }{x_{\oc,1}^2} + \frac{\varphi_1(\msf{K},\ococparaone,\ococparatwo)}{x_{\oc,1}} , \quad \msf{K} = (\xoco, \sigma, \yoco, \yoct), 
\end{align} 
gives a clean parametrization of $\SL$ near $q$ with excess $e$ if there is a point $q' = (\msf{K}, \ococparaone', \ococparatwo')$ such that the differentials
\begin{equation}\label{eq:varphi0 1c clean}
d_{\sigma,\ococparaone} \frac{\partial \varphi_0}{\partial \theta_{0j}}, \quad 1 \leq j \leq k_0,
\end{equation}
are linearly independent at $(\sigma', \ococparaone')$, and the differentials 
\begin{equation}\label{eq:varphi1 1c clean}
d_{\yoco, \yoct, \ococparatwo} \frac{\partial \varphi_1}{\partial \theta_{1j}}, 1 \leq j \leq k_1,
\end{equation}
have a fixed rank $k_1 - e$ near $q'$, such that $\SL$ is given locally by 
\begin{align}\label{eq:1c-1c param clean}
\SL
=\{ (\msf{K},d_{\msf{K}}\Phi_{\oc-\oc}(\msf{K},\ococparaone,\ococparatwo)) \mid \; \xoc = 0, (\msf{K},\ococparaone,\ococparatwo) \in C_{\Phi_{\oc-\oc}}\},
\end{align}
where 
\begin{align}  \label{eq: 1c-1c critical set clean}
C_{\Phi_{\oc-\oc}} =  \{ (\msf{K},\ococparaone,\ococparatwo) \mid d_{\ococparaone}\varphi_0 = 0, d_{\ococparatwo}\varphi_1 = 0 \}.
\end{align}
As before, either $\ococparaone$ or $\ococparatwo$ may be absent, in which case conditions for the derivatives in \eqref{eq:varphi0 1c clean}, resp. \eqref{eq:varphi1 1c clean} are dispensed with, as well as stationarity with respect to $\ococparaone$, resp. $\ococparatwo$ in \eqref{eq: 1c-1c critical set}. 
\end{defn}

\begin{remark}
The set $C_{\Phi_{\oc-\oc}}$ is a submanifold of dimension $n+e$ due to the rank condition above, and the map from $C_{\Phi_{\oc-\oc}}$ to $\SL$ determined by \eqref{eq:1c-1c param clean} is a smooth fibration with $e$-dimensional fibres. A non-degenerate parametrization is a clean parametrization with $e= 0$. Clean parametrizations will show up `naturally' when we compose 1c-ps FIOs, as we shall see in the next section. 
We could allow for constant rank rather than linear independence in the $\ococparaone$-derivatives \eqref{eq:varphi0 1c clean} as for the $\ococparatwo$-derivatives, but for simplicity we shall not do so here, as we only encounter the drop of rank in the $\ococparatwo$-derivatives in this article. 
\end{remark}

\section{1c-1c Fourier integral operators}
\label{sec: 1c-1c Lagrangian distribution}

\subsection{Definition of 1c-1c fibred-Legendre distributions and Fourier integral operators}\label{subsec:defn 1c-1c distributions}
\begin{defn} \label{def:1c-1c-FIO}
Let $\Uplambda$ be an admissible Lagrangian submanifold of $\ococb$ as in Definition~\ref{defn: admissible 1c-1c Lagrangian submanifold and 1c-1c fibred-Legendre submanifold}, with boundary $\SL$. 
We define 
$I^{m}_{\oc-\oc}(X_b^2, \SL; \Omega_{\oc-\oc}^{1/2})$, i.e., the space of 1c-1c fibred-Legendre distributions of order $m$, to be the space of operators with Schwartz kernel given (modulo a Schwartz function) by a finite sum of terms of the form 
\begin{align} \label{eq: 1c-1c FIO, local form}
\begin{split}
 & (2\pi)^{-\frac{n+(k_0+k_1-e)}{2}} \Big(\int 
e^{i\Phi_{\oc-\oc}(\msf{K},v,w)} 
 a(\msf{K},v,w)
\\& x_{\oc,1}^{-m-\frac{2k_0+(k_1-e)}{2}+ \frac{n+1}{2}} dvdw \Big) 
|\frac{d\sigma dy_{\oc,2}}{x_{\oc,1}^{n+1}}|^{1/2}|\frac{dx_{\oc,1}dy_{\oc,1}}{x_{\oc,1}^{n+2}}|^{1/2},
\end{split}
\end{align}
where $\Phi_{\oc-\oc}$ is a phase function of the form \eqref{eq:Phi 1c-1c param clean} locally parametrizing $\mathcal{L}$ non-degenerately in the sense of Definition \ref{defn: 1c-1c parametrization}.
Moreover,  $a \in C^\infty_c([0,\infty)_{x_{\oc}} \times [C^{-1},C]_{\sigma} \times \mathbb{S}^{n-1} \times \mathbb{S}^{n-1} \times \R^{k_0} \times \R^{k_1})$, where $v \in \R^{k_0},w \in \R^{k_1}$. 
%In addition, we assume that $\varphi_{0i},\varphi_i$ are homogeneours of degree one in $w$. (in our actual application, $w=(t,z)$, and this fact is guaranteed by the fact that the phase function in the 1c-ps picture has the form $\frac{...}{x_{\ps}}$)
\end{defn}

\begin{remark}
    As demonstrated by Proposition~\ref{prop: phase equivalence, general parameter number, signature} below, the class of operators won't be enlarged if we allow clean phase functions in \eqref{eq: 1c-1c FIO, local form}.
    See also the proof of Proposition~\ref{prop: phase equivalence, general parameter number, signature} at the end of Appendix~\ref{app: phase equivalence}.
\end{remark}

The density factor $|\frac{d\sigma dy_{\oc,2}}{x_{\oc}^{n+1}}|^{1/2}$ here equals  
$|\frac{d(\sigma x_{\oc}) dy_{\oc,2}}{x_{\oc}^{n+2}}|^{1/2}$ to the leading order, which is the 1-cusp density and $\Omega_{\oc-\oc}^{1/2}$ denotes the density bundle with typical section as in the expression above.
One can view this as the product of 1-cusp half-densities on each factors lifted to $X_b^2$.

As usual, the space of fibred-Legendre distributions depends only on the fibred-Legendre submanifold $\SL$ (or more precisely, the 1-jet of its Lagrangian extension $\Uplambda$) and not on the particular choice of phase functions. This follows from the proposition below.

\begin{prop} \label{prop: phase equivalence, general parameter number, signature}
Let $\Phi_{\oc-\oc},\tilde{\Phi}_{\oc-\oc}$ be two clean parametrizations of $\mathcal{L}_{\oc}$ near $q \in \mathcal{L}_{\oc}$. Then any local expression
\begin{align*}
 \int e^{i \Phi_{\oc-\oc} } 
x_{\oc,1}^{-m - \frac{2k_0+(k_1-e)}{2} +\frac{n+1}{2} }
a(x_{\oc,1},\sigma,y_{\oc},y_{\oc}'',v,w)dvdw,
\end{align*}
with $a$ smooth, supported in the region where the parametrizations of $\Phi_{\oc-\oc},\tilde{\Phi}_{\oc-\oc}$ are both valid, can also be written as
\begin{align*}
A = u_0 + \int e^{i \tilde{\Phi}_{\oc-\oc} } 
x_{\oc,1}^{-m - \frac{2\tilde{k}_0+(\tilde{k}_1-e)}{2} +\frac{n+1}{2} }
\tilde{a}(x_{\oc,1},\sigma,y_{\oc},y_{\oc}'',\tilde{v},\tilde{w})d\tilde{v}d\tilde{w},
\end{align*}
with $u_0$ being Schwartz, $\tilde{a}$ smooth.
\end{prop}

\begin{proof} See Appendix~\ref{app: phase equivalence}. 
\end{proof}

\subsection{Principal symbol}
In order to define the symbol calculus of 1c-1c Lagrangian distributions, we need to define the principal symbol of $A \in I^m_{\oc-\oc}(X_b^2,\Uplambda)$.

For $A$ in the form (\ref{eq: 1c-1c FIO, local form}), let $\msf{X}'=(\sigma,y_{\oc},y_{\oc}'')$, and
let $\lambda$ be a set of functions of $(\msf{X}',v,w)$ such that $(x_{\oc},\lambda,d_{v}\varphi_0+x_{\oc}d_v\varphi_1,d_w\varphi_1)$ form a local coordinate system near $C_{\Phi}$.
When $e=0$ (that is, when the parametrization is non-degenerate), we define the principal symbol of $A$ to be the half density on $\mathcal{L}_{\oc} \cap \partial X_b^2$ given by:
\begin{align} \label{defn: 1c-1c principal symbol, non-degenerate}
\sigma_{\oc-\oc}^m(A) = a(0,\msf{X}',v,w)|_{C_{\Phi_{\oc-\oc}}} | \frac{\partial(d_v\varphi_0,d_{w}\varphi_1,\lambda)}{\partial(\msf{X}',v,w)}|^{-1/2} |d\lambda|^{1/2}.
\end{align}

%%%%%%%%%%%%%%%%%principal symbol in the clean parametrization setting

For general $e \geq 0$, let $\Phi_{\oc-\oc}$ be as in \eqref{eq:Phi 1c-1c param clean}, then there is a (local) splitting of $\ococparatwo$:
\begin{equation} \label{eq:ococpara-splitting}
   \ococparatwo = (\ococparatwo',\ococparatwo''), \; \ococparatwo' \in \R^{k_1-e} , \ococparatwo'' \in \R^e
\end{equation}
such that for each fixed $\ococparatwo''$, $\Phi_{\oc-\oc}$ gives a non-degenerate parametrization of $\Uplambda$, which means that there is a fibration of $C_{\Phi_{\oc-\oc}}$ (the critical set of $\Phi_{\oc-\oc}$) over $\Uplambda$ with $e$-dimensional fiber parametrized by $\ococparatwo''$.
From another perspective, we have a family of non-degenerate phase functions parametrized by $w''$, in the same way as in \cite[Equation~(7.6), Lemma~7.2]{duistermaat-guillemin1975spectrum}, integrating the amplitude along such fiber parametrized by $w''$ gives our half-density as in \eqref{defn: 1c-1c principal symbol, non-degenerate} when we parametrize using non-degenerate phase functions. More precisely, in this setting we have:
%we define the principal symbol of $A$ to be
\begin{align} \label{defn: 1c-1c principal symbol, clean}
\sigma_{\oc-\oc}^m(A) = \int_{C_q} a(0,\msf{X}',v,w) 
d\ococparatwo''
| \frac{\partial(d_v\varphi_0,d_{w'}\varphi_1,\lambda)}{\partial(\msf{X}',v,w')}|^{-1/2} |d\lambda|^{1/2}.
\end{align}
Here $q \in \Uplambda$ and $C_q$ is the fiber in $C_{\Phi_{\oc-\oc}}$ over $q$.
Then this is invariant (except for Maslov factors and $E(\mathcal{L}_{\oc})$ factors) in the leading order, which can be shown in the same way as in \cite[Section~25.1]{hormander2009analysis}. 
Notice that here we are only concerning the clean parametrization with excess in $\varphi_1$, hence the parabolic scaling between $\varphi_0$ and $\varphi_1$ part is not involved and the proof for invariance under change of coordinates in the homogeneous case, after converting to polar coordinates, almost applies here directly.
%Definition~\ref{defn: 1c-1c parametrization clean}

%%%%%%%%%%%%%%%%%%%%%%%%%%%
In order to make this concept of principal symbol invariant under change of parametrizations and scaling of boundary defining functions, we introduce
\begin{align}
S_{\oc-\oc}^{[m]}(\mathcal{L}_{\oc}):= |N^*\partial X_b^2|^{- m - \frac{n+1}{2} }
\otimes M(\mathcal{L}_{\oc}) \otimes E(\mathcal{L}_{\oc}),
\end{align}
where $M(\mathcal{L}_{\oc})$ is the Maslov bundle used in \cite[Section~3.2,3.3]{FIO1}, which in turn comes from \cite{maslov1972asymptotic}; and $|N^*\partial X_b^2|^{-m-\frac{n+1}{2}}$, 
%which originally only defined on $\partial X_b^2$ extends `by constant' to a collar neighborhood of it, which 
models a section that is homogeneous of degree $-m-\frac{n+1}{2}$ in $x_{\oc}$.
The order $-m-\frac{n+1}{2}$ comes from considering the homogeneous degree of:
\begin{align*}
  x_{\oc}^{-m + \frac{n+1}{2} }|\frac{dx_{\oc}dy_{\oc,1}}{x_{\oc}^{n+2}}|^{1/2}
|\frac{d\sigma dy_{\oc,2}}{x_{\oc}^{n+1}}|^{1/2},
\end{align*}
while the factor $x^{- \frac{2k_1+k_2}{2} }$ is encoded in $| \frac{\partial(d_v\varphi_0,d_w\varphi_1,\lambda)}{\partial(\msf{X},v,w)}|^{-1/2}$ in an invariant manner. Similar to the 1c-ps setting, the line bundle $E(\mathcal{L}_{\oc})$ here is introduced to make this definition invariant under change of boundary defining functions, see \cite[Section~3.1]{hassell2001resolvent} for more details.

%%%%%%%%%%%%%%%%%%%%%%%%%%%%%%%%
Combining discussions above, we can view the symbol map as
\begin{align} \label{eq: invariant, symbol map, 1c-1c}
\begin{split}
\sigma^m_{\oc-\oc}: \quad  I^m_{\oc-\oc}(X_b^2,\Uplambda;\Omega_{\oc-\oc}^{1/2}) 
%\beta_b^*(\leftidx{^{\oc}}{\Omega}^{1/2} \otimes \leftidx{^{\oc}}{\Omega}^{1/2} ))
\rightarrow 
 C^\infty(\mathcal{L}_{\oc} \cap \partial X_b^2; 
\Omega^{1/2}(\mathcal{L}_{\oc}) \otimes S_{\oc-\oc}^{[m]}(\mathcal{L}_{\oc})).
\end{split}
\end{align}
%where $\leftidx{^{\oc}}{\Omega}^{1/2} \otimes \leftidx{^{\oc}}{\Omega}^{1/2}$ is the half density bundle with 1-cusp density on both left and right factors of $X \times X$. 
%and we have abused the notation to use it to denote its lift to $X_b^2$.

We say $A \in I^m_{\oc-\oc}(X_b^2,\Uplambda;\Omega_{\oc-\oc}^{1/2})$ is elliptic at $q \in \mathcal{L}_{\oc}$ if the amplitude on right hand side of \eqref{defn: 1c-1c principal symbol, non-degenerate} or \eqref{defn: 1c-1c principal symbol, clean} is non-vanishing at the point that is sent to $q$ under the parametrization map in Definition~\ref{defn: 1c-1c parametrization clean}. This is equivalent to requiring the principal symbol to be invertible near $q$ as the section of $\Omega^{1/2}(\mathcal{L}_{\oc}) \otimes S_{\oc-\oc}^{[m]}(\mathcal{L}_{\oc})$.

The fact that this symbol map captures leading order singularity can be summarized as following statement, which follows from the definition.
\begin{prop}
The principal symbol map $\sigma^m_{\oc-\oc}$ gives rise to following short exact sequence:
\begin{align}
\begin{split}
0 & \rightarrow I^{m-1}_{\oc-\oc}(X_b^2,\mathcal{L}_{\oc})
\rightarrow I^{m}_{\oc-\oc}(X_b^2,\mathcal{L}_{\oc})
\\ & \xrightarrow{\sigma^m_{\oc-\oc}} C^\infty(\mathcal{L}_{\oc} \cap \partial X_b^2; 
\Omega^{1/2}(\mathcal{L}_{\oc}) \otimes S_{\oc-\oc}^{[m]}(\mathcal{L}_{\oc})) \rightarrow 0.
\end{split}
\end{align}
\end{prop}
The exactness at the first and last (non-zero) component follows from definition, while the exactness in the middle follows from an argument similar to Lemma \ref{lemma: vanishing amplitude->lower order}.

%[TODO: define ellipticity, using invertibility of density section]
Now we discuss the ellipticity. Let $q$ be the image of $(0,\msf{X}',v,w)$ as in \eqref{eq:1c-1c param}, 
%under \eqref{eq: critical set to Lagrangian}, 
then we say that $A \in I^{m}_{\oc-\oc}$ is elliptic at $q$ if $\sigma^m_{\oc-\oc}(A)$ is invertible at $q$ in the sense that there exists $B \in \Psi_{\oc}^{-m}$ such that
\begin{align}
\sigma^0_{\oc-\oc}(BA) =
|\frac{\partial(d_v\varphi_0,d_w\varphi_1,\lambda)}{\partial(\msf{X}',v,w)}|^{-1/2} |d\lambda|^{1/2},
\end{align}
where coordinates are as in \eqref{defn: 1c-1c principal symbol, non-degenerate}, and this is equivalent to $a(0,\msf{X}',v,w) \neq 0$ in \eqref{defn: 1c-1c principal symbol, non-degenerate}. In addition, this property is independent of the choice of the parametrization.

\subsection{Composition with one-cusp pseudodifferential operators}

Now we start to discuss the composition.
Let $\mathcal{L}_i = C_i' \cap \ffococ$, $i=1,2$ be admissible 1c-1c Legendre submanifolds with $C_i'$ being the corresponding Lagrangian submanifolds.
Suppose $C_1,C_2$ are lifts of canonical relations in $\leftidx{^{\oc}}{T^*X} \times \leftidx{^{\oc}}{T^*X}$ to $\ococb$, and $C_2 \times C_1$ 
intersect the lift to $\ococb \times \ococb$ of the diagonal in the second and third components of
\begin{align*}
\leftidx{^{\oc}}{T^*X} \times \leftidx{^{\oc}}{T^*X} \times \leftidx{^{\oc}}{T^*X} \times \leftidx{^{\oc}}{T^*X}
\end{align*}
transversally. 

Recall from \cite[Section~4.2]{FIO1} (see \cite[Theorem~21.6.7]{hormander2007analysis} for the clean intersection case) that there is a natural bilinear map
\begin{align}  \label{eq: symbol product, bundle maps}
\begin{split}
& \Omega^{1/2}(\mathcal{L}_2) \otimes S^{[m_2]}(\mathcal{L}_2)
 \times \Omega^{1/2}(\mathcal{L}_1) \otimes S^{[m_1]}(\mathcal{L}_1) 
 \\ & \rightarrow  \Omega^{1/2}((C_2 \circ C_1)' \cap \ffococ) \otimes S^{[m_2+m_1]}( (C_2 \circ C_1) ' \cap \ffococ).
\end{split}
\end{align}
Comparing the power of $|N^*\partial X_b^2|$ in $S^{[m_i]}-$factors and that in the $S^{[m_1+m_2]}-$factor, there is an extra $-\frac{n+1}{2}-$power of $|N^*\partial X_b^2|$, this is encoded in the half density bundle part. See also the counting of orders at the end of proof Proposition \ref{prop: 1c-1c transversal composition}  below.
Denoting this bilinear map by `$\times$', we have:
\begin{prop}  
\label{prop: 1c-1c transversal composition}
%Transversal composition only, we don't need clean composition in our setting
Suppose $C_1 \times C_2$ satisfies the transversal intersection condition above, and $A_1 \in I^{m_1}_{\oc-\oc}(X^2_b,C_1'), A_2 \in I^{m_2}_{\oc-\oc}(X^2_b,C_2')$, then we have
\begin{align*}
A_2A_1 \in I^{m_1+m_2}_{\oc-\oc}(X_b^2, (C_2 \circ C_1)').
\end{align*}
And when they have $a_1,a_2$ as their principal symbol respectively, then $A_2A_1$ has principal symbol
\begin{align*}
a_2 \times a_1.
\end{align*}
\end{prop}

\begin{proof}
Suppose $C_1'$ is locally parametrized by
\begin{align*}
\Phi_1 =  \frac{\varphi_{11}(v,\frac{x_{\oc}}{x_{\oc}'})}{x_{\oc}^2}
+ \frac{\varphi_{12}(x_{\oc},\frac{x_{\oc}}{x_{\oc}'},y_{\oc,1},y_{\oc,1}',v,w)}{x_{\oc}},
\end{align*}
and $C_2'$ is parametrized by
\begin{align*}
\Phi_2 = \frac{\varphi_{21}(v'',\frac{x_{\oc}''}{x_{\oc}'})}{(x_{\oc}'')^2} + \frac{\varphi_{22}(x_{\oc}'',\frac{x_{\oc}''}{x_{\oc}'},y_{\oc,1}'',y_{\oc,1}',v'',w'')}{x_{\oc}''}.
\end{align*}

Then we have:
\begin{align*}
A_2A_1
=  & \int e^{ i\big(\frac{\varphi_{11}(v,\frac{x_{\oc}}{x_{\oc}'})}{x_{\oc}^2} 
+ \frac{\varphi_{21}(v'',\frac{x_{\oc}''}{x_{\oc}'})}{(x_{\oc}'')^2}
+ \frac{\varphi_{12}(x_{\oc},\frac{x_{\oc}}{x_{\oc}'},y_{\oc,1},y_{\oc,1}',v,w)}{x_{\oc}}  + \frac{\varphi_{22}(x_{\oc}'',\frac{x_{\oc}''}{x_{\oc}'},y_{\oc,1}'',y_{\oc,1}',v'',w'')}{x_{\oc}''} \big)} 
 \\ & x_{\oc}^{-(m_1+m_2) - \frac{2(k_{11}+k_{21})+k_{12}+k_{22}}{2} + (n+1) }
 a_1(x_{\oc},\frac{x_{\oc}}{x_{\oc}'},y_{\oc,1},y_{\oc,1}',v,w) 
 \\& a_2(x_{\oc}'',\frac{x_{\oc}''}{x_{\oc}'},y_{\oc,1}'',y_{\oc,1}',v'',w'') dvdwdv''dw'' \frac{dx_{\oc}'dy_{\oc,1}'}{(x'_{\oc})^{n+2}}
 \\ & | \frac{dx_{\oc} dy_{\oc,1}}{x_{\oc}^{n+2}} |^{1/2} | \frac{dx_{\oc}'' dy_{\oc,1}''}{(x_{\oc}'')^{n+2}} |^{1/2}
\end{align*}
Let $\sigma' = \frac{x_{\oc}'}{x_{\oc}}$, then the leading order part can be written as
\begin{align*}
A_2A_1
=  & \int e^{ i\big(\frac{\varphi_{11}(v,\frac{x_{\oc}}{x_{\oc}'})}{x_{\oc}^2} 
+ \frac{\varphi_{21}(v'',\frac{x_{\oc}''}{x_{\oc}'})}{(x_{\oc}'')^2}
+ \frac{\varphi_{12}(x_{\oc},\frac{x_{\oc}}{x_{\oc}'},y_{\oc,1},y_{\oc,1}',v,w)}{x_{\oc}}  + \frac{\varphi_{22}(x_{\oc}'',\frac{x_{\oc}''}{x_{\oc}'},y_{\oc,1}'',y_{\oc,1}',v'',w'')}{x_{\oc}''} \big)} 
 \\ & x_{\oc}^{-(m_1+m_2) - \frac{2(k_{11}+k_{21}+1)+k_{12}+k_{22}+(n-1)}{2} + \frac{n+1}{2} }
 a_1(x_{\oc},\frac{x_{\oc}}{x_{\oc}'},y_{\oc,1},y_{\oc,1}',v,w) 
 \\& a_2(x_{\oc}'',\frac{x_{\oc}''}{x_{\oc}'},y_{\oc,1}'',y_{\oc,1}',v'',w'') dvdwdv''dw'' (\sigma')^{-(n+1)}d\sigma'dy_{\oc,1}'
 \\ & | \frac{dx_{\oc} dy_{\oc,1}}{x_{\oc}^{n+2}} |^{1/2} | \frac{dx_{\oc}'' dy_{\oc,1}''}{(x_{\oc}'')^{n+2}} |^{1/2},
\end{align*}
which gives an element of $I^{m_1+m_2}_{\oc-\oc}(X_b^2, (C_2 \circ C_1)')$ if the phase function parametrizes $(C_2 \circ C_1)'$.

Since
\begin{align*}
|\partial_{y_{\oc,1}'}\varphi_2| \leq C
\end{align*}
for some constant $C$,
hence on the critical point of 
\begin{align*}\varphi_{12}(x_{\oc},\frac{x_{\oc}}{x_{\oc}'},y_{\oc,1},y_{\oc,1}',v,w) +
\varphi_{22}(x_{\oc}'',\frac{x_{\oc}''}{x_{\oc}'},y_{\oc,1}'',y_{\oc,1}',v'',w'') ,
\end{align*}
we have
\begin{align*}
|\frac{x_{\oc}}{x_{\oc}''}| \leq C
\end{align*}
for some constant $C$. In fact we can even assume that on the support of $a$ and $\bar{a}$ factors, both of  $\frac{x_{\oc}}{x_{\oc}'},\frac{x_{\oc}''}{x_{\oc}'}$ are close to 1. And consequently $\frac{x_{\oc}}{x_{\oc}''}$
is close to 1.

For variables in $\Phi_1+\Phi_2$, we rewrite as 
\begin{align*}
\frac{x_{\oc}}{x'_{\oc}} = \frac{x_{\oc}}{x''_{\oc}} \cdot \frac{x''_{\oc}}{x'_{\oc}},
\end{align*}
and let $\bar{\sigma}=\frac{x_{\oc}}{x''_{\oc}}$ be the coordinate on the b-face while viewing $\frac{x''_{\oc}}{x'_{\oc}}$ as a parameter.

Then $\Phi_1+\Phi_2$ is parametrizing $(C_1 \circ C_2)'$, 
with $\frac{x''_{\oc}}{x'_{\oc}},y_{\oc}'$ viewed as parameters now,
and the amplitude has compact support in all variables.
In addition, applying the proof of \cite[Theorem~4.2.2]{FIO1} to $(\varphi_{11}+\varphi_{21})$ and $(\varphi_{12}+\varphi_{22})$ respectively, we know that $\Phi$ is non-degenerate if $\Phi_1,\Phi_2$ are so and the transversal intersection condition is satisfied.

Since on the support of the amplitude, all of $x_{\oc},x_{\oc}',x_{\oc}''$ are comparable to each other, 
The power of $x_{\oc}$ is matched because:
\begin{align*}
 & x_{\oc}^{-m_1 - \frac{2k_{11}+k_{12}}{2} + \frac{n+1}{2}}
(x_{\oc}'')^{-m_2 - \frac{2k_{21}+k_{22}}{2} + \frac{n+1}{2} }
(x_{\oc}')^{-(n+2)}dx_{\oc}' dy_{\oc}'
\\ \approx &  x_{\oc}^{-(m_1+m_2) + \frac{2(k_{11}+k_{21}+1)+(k_{12}+k_{22}+n-1)}{2} + \frac{n+1}{2} }    d(\frac{x''_{\oc}}{x'_{\oc}})dy_{\oc}'
\end{align*}
where $\approx$ means they are equal up to a smooth factor that is uniformly bounded from above by a constant and below away from,
hence $A_2A_1 \in I^{m_1+m_2}_{\oc-\oc}(X_b^2,(C_2 \circ C_1)')$. And the result about the principal symbol follows from how (\ref{eq: symbol product, bundle maps}) is defined in \cite[Section~4.2]{FIO1}.
This completes the proof.
\end{proof}

%\subsection{Clean composition between 1c-ps Lagrangian distributions}
%Composition of \texorpdfstring{$I^*_{\oc-\ps}(\Lambda_-)$ and adjoints of $I^*_{\oc-\ps}(\Lambda_+)$}{Lagrangian distributions and their adjoints}

%The effect of the exesses on the order should be:
%The order of the composition should be: $m_1+m_2+e_0+\frac{e_1}{2}$.
%[TO DO: define the FIO when excesses exists] 
%1c-1c- version of \cite[Lemma~7.1]{duistermaat-guillemin1975spectrum}:

\subsection{1c-1c operator wavefront set}
Before stating the propagation of singularities of the scattering map in terms of the refined definition of the wavefront set in \eqref{eq: definition of 1c WF, both decay and smoothness, with order}, we first show that the wavefront set of an 1c-1c Lagrangian distribution is contained in the Lagrangian submanifold that it is associated to, as in \cite[Proposition~2.5.7]{FIO1} for Lagrangian distributions on manifolds without boundary in terms of (classical) wavefront sets.

Let $\pi_{\oc,1},\pi_{1c,2}$ be the projection from $\ococb$ to $\overline{{}^{\oc}T^*\R^n}$ as the left and right factor before the b-blow up.
Then for $A \in \mathcal{S}'(\R^n \times \R^n)$, we define its 1c-1c (operator) wavefront set $\WF'_{\oc-\oc}(A) \subset  \partial(\ococb))$ by: for $\msf{q} \in \partial(\ococb))$, we say $\msf{q} \notin \WF'_{\oc-\oc}(A)$ if and only if there are $Q,Q' \in \Psi_{\oc}^{0,0}$ that are elliptic at $\pi_{\oc,1}(\msf{q}),\pi_{\oc,2}(\msf{q})$ respectively such that
\begin{equation}
QAQ' \in \mathcal{S}(\R^n \times \R^n).
\end{equation}

\begin{prop} 
\label{prop:1c-1c-WF}
Let $A \in I_{\oc-\oc}^{m}(X_b^2,\Uplambda)$ with $\SL = \partial \Uplambda$, then
\begin{align}
\WF'_{\oc-\oc}(A) \subset \pi_{\oc,1}^{-1}(\pi_{\oc,1}(\SL)) \cap \pi_{\oc,2}^{-1}(\pi_{\oc,2}(\SL)). 
\end{align}
\end{prop}

\begin{proof}
First we show that there are no wavefront set at the fiber infinity. One way to see this is notice that since
\begin{align*}
\varphi_1(\sigma,v), \; \varphi_2(\msf{X},v,w),
\end{align*}
are smooth functions over a compact region, hence uniformly bounded, thus under actions of
\begin{align*}
x_{\oc}^3\partial_{x_{\oc}},x_{\oc}\partial_{y_{\oc,1}},
x_{\oc}^2\partial_\sigma ,x_{\oc}\partial_{y_{\oc,2}},
\end{align*}
we know (\ref{eq: 1c-1c FIO, local form}) is smooth. Next we show that it vanishes to infinite order in $x_{\oc}$ when microlocalized away from $\mathcal{L}_{\oc}$.

Now we show that there is no wavefront set over $\mathrm{bf},\mathrm{lb},\mathrm{rb}$.
%(coordinates: $(x_{\oc,1},y_{\oc}_{\oc,1},\tau_{\oc,1},{Z}_{\oc,1},x_{\oc,2},y_{\oc}_{\oc,2},\tau_{\oc,2},{Z}_{\oc,2} )$)
For $q_0 in \partial(\ococb))$ such that $ q_0 \notin \pi_{\oc,1}^{-1}(\pi_{\oc,1}(\SL)) \cap \pi_{\oc,2}^{-1}(\pi_{\oc,2}(\SL))$, we have either
\begin{align} \label{eq: off Lagrangian, condition 1}
d_{v,w}\Phi \neq 0 ,\text{ or } d_{x_{\oc},y_{\oc,1}}\Phi \neq \tau_{\oc,1}\frac{dx_{\oc}}{x_{\oc}^3}+\mu_{\oc} \cdot \frac{dy_{\oc,1}}{x_{\oc}};
\end{align}
or
\begin{align} \label{eq: off Lagrangian, condition 2}
d_{v,w}\Phi \neq 0 ,\text{ or } d_{\sigma,y_{\oc,2}}\Phi \neq \tau_{\oc,2}\frac{d\sigma}{x_{\oc}^2}+\mu_{\oc,2} \cdot \frac{dy_{\oc,2}}{x_{\oc}},
\end{align}
at $q_0$ for every tuple of parameters $(v,w)$. Because if one of the condition is satisfied at certain $(v,w)$, then the same condition holds near by. 
So we can further localize in $(v,w)$ and only consider the case where one of the situation holds globally.
%Since the region overall is compact, we only need to consider the case for a fixed $(v,w)$, and hence one of the situation.

We consider the first case to show $A=O(x_{\oc}^N)$ locally there for any $N$.
If it were the second condition does not hold, then notice that using $x_{\oc,1}$ or $x_{\oc,2}$ as $x_{\oc}$ is an artificial choice. We can interchange $(x_{\oc,1},y_{\oc,1})$ and $(x_{\oc,2},y_{\oc,2})$
%\eqref{eq: 1c1c-fiber coordinates, over b-double space} similarly, 
and run the same argument; or consider $A^*$ and apply the same argument. 

%and the second case is treated simiarlly to show $A=O(x_{\oc,2}^N)$ locally there for any $N$. Since $A$ is supported in the region $x_{\oc,1}$ and $x_{\oc,2}$ are comparable, this finishes the proof.

Recall the definition how 1-cusp pseudodifferential operators acts in \cite[Equation~(2.5)]{zachos2022inverting} (and the formula on page 20 there),
microlocalize by $B$ with symbol $b$ means:
\begin{align*}
 (2\pi)^{-n} \int e^{i \frac{\tilde{x}_{\oc}-x_{\oc} }{\tilde{x}_{\oc}^3}\tau_{\oc}
  + \frac{\tilde{{Z}}_1-y_{\oc,1}}{ \tilde{x}_{\oc} } \cdot \mu_{\oc} }
  b(\tilde{x}_{\oc},\tilde{{Z}}_1,\tau_{\oc},\mu_{\oc}) \chi_1(\frac{\tilde{x}_{\oc}}{x_{\oc}})A(\msf{X}) d\tau_{\oc}d\mu_{\oc} \frac{dx_{\oc}dy_{\oc,1}}{x_{\oc}^{n+2}},
\end{align*}
where $\chi_1 \in C_c^\infty(\R)$ is supported near 1 (we will determine its support later), and inserting this cut-off does not affect the result since this effectively is multiplying the Schwartz kernel of a 1-cusp pseudodifferential operator by $\chi_1(\frac{\tilde{x}_{\oc}}{x_{\oc}})$, which does not change its pseudodifferential properties, and the fact that it is elliptic at $\pi_1(q_0)$.
%$B$ is microlocalized near the lifted diagonal in the 1-cusp double space.

Combining with (\ref{eq: 1c-1c FIO, local form}), we have
\begin{align}  \label{eq: localized A, expression}
\begin{split}
BA = 
 & (2\pi)^{-\frac{n_X+n_X+2(k_1+k_2)}{4} - n} \tilde{x}_{\oc}^{-(n+2)}\int 
 e^{i \frac{\tilde{x}_{\oc}-x_{\oc} }{\tilde{x}_{\oc}^3}\tau_{\oc}
  + \frac{\tilde{{Z}}_1-y_{\oc,1}}{ \tilde{x}_{\oc} } \cdot \mu_{\oc} }
  b(\tilde{x}_{\oc},\tilde{{Z}}_1,\tau_{\oc},\mu_{\oc}) 
  \chi_1(\frac{\tilde{x}_{\oc}}{x_{\oc}})
\\ &
e^{i( \frac{\varphi_1(\sigma,v)}{x_{\oc}^2} + \frac{\varphi_2(\msf{X},v,w)}{x_{\oc}} )} a(\msf{X},v,w)
\\& 
x_{\oc}^{-m - \frac{2k_1+k_2}{2} + \frac{n+1}{2} } (\frac{\tilde{x}_{\oc}}{x_{\oc}})^{n+2}dw_1dw_2dw d\tau_{\oc}d\mu_{\oc} dx_{\oc}dy_{\oc,1}.
\end{split}
\end{align}

Since
\begin{align}
d_{x_{\oc},y_{\oc,1}}(\frac{\varphi_{1}(\sigma,v,w)}{x_{\oc}^2}+\frac{\varphi_{2}(\msf{X},v,w)}{x_{\oc}})= (-2\varphi_1- x_{\oc}\partial_x\varphi_2) \frac{dx_{\oc}}{x^3} + \frac{ \partial_{y_{\oc,1}}\varphi_2 \cdot dy_{\oc,1}}{x_{\oc}}.
\end{align}

Since we only require that $B$ is elliptic at $\pi_1(q_0)$, so on the support of $b$ we may assume that for a constant $c>0$:
\begin{align*}
|\tau_{\oc}+2 \varphi_{1}(\sigma,v,w)+x_{\oc} \partial_{x_{\oc}}\varphi_{2}(\msf{X},v,w)|^2
+ |\mu_{\oc} - d_{\hat{{Z}_1}}\varphi_{2}(x_{\oc},y_{\oc,1},w_1,w)|^2  \geq c>0,
\end{align*}
whenever
\begin{align*}
d_{v,w}(\frac{\varphi_{1}(\sigma,v,w)}{x_{\oc}^2}+\frac{\varphi_{2}(\msf{X},v,w)}{x_{\oc}}) = 0.
\end{align*}
A direct computation shows that the first condition above means that the phase function has no critical point when differentiated with resect to $x_{\oc},y_{\oc,1}$, assuming that $\frac{\tilde{x}_{\oc}}{x_{\oc}}$ is close enough to 1, which is satisfied on $\supp \chi_1$.

This means that the oscillatory integral in (\ref{eq: localized A, expression}) as $\frac{1}{\tilde{x}_{\oc}} \rightarrow \infty$, there is no critical point with respect to $(x_{\oc},y_{\oc,1},v,w)$. 
This means it has infinite order decay with respect to $|(\frac{1}{\tilde{x}_{\oc}^2},\frac{1}{\tilde{x}_{\oc}})|$.
Although the large parameter in front of different parts has different powers of $\frac{1}{\tilde{x}_{\oc}}$, arbitrary polynomial order decay with respect to any one of them is equivalent to arbitrary order decay with respect to $\tilde{x}_{\oc}$.
\end{proof}

%%%%%%%%%%%%%%%%%%%%%%%%%%%%%%%%%%%%%%%%%%%%%%%%%%%%%%%%%%%%%%%

\section{Composition of 1c-ps operators}\label{sec:composition}

\subsection{Geometric setup}
In this section, we consider compositions of the form $A_+^*A_-,$,
where $A_+ \in I_{\oc-\ps}^{m_2}(\R^{n+1} \times \R^n,\Lambda_+)$, and $A_- \in I_{\oc-\ps}^{m_1}(\R^{n+1} \times \R^n,\Lambda_-)$. Recall that the scattering map can be represented in this way through Poisson operators --- see \eqref{eq:S formula}. 

We first consider the composition of the two canonical relations associated with $\Lambda_\pm$. We will denote by $C_+$ and $C_-$ the canonical relations associated with $\Lambda_+$ and $\Lambda_-$, obtained as usual by negating the frequency variable in the right factor. We also let $C_+^*$ and $\Lambda_+^*$ denote the canonical relation and Lagrangian corresponding to the adjoint operator, in the sense that
\begin{equation}
    (q,-q') \in \Lambda \Leftrightarrow (q', -q) \in \Lambda^* \Leftrightarrow (q, q') \in C \Leftrightarrow (q', q) \in C^*
\end{equation}
where the minus sign indicates negation in the fibre variables. 

Typically one looks at the product of the canonical relations in the product phase space, which in this case would be $\ocphase \times \psphase \times \psphase \times \ocphase$, intersected with the `partial diagonal' in the middle two factors, which in this case would be $\ocphase \times \Diag(\psphase) \times \ocphase$, and then projects to the outer factors, in this case $\ocphase \times \ocphase$. However, in our case, we already had to resolve the geometry for the 1c-ps phase space by blowing up $\{ \xoc = \rhops = 0 \}$, creating the resolved space $\SM$ --- see \eqref{eq: 1c-ps cotangent bundle definition}. That suggests replacing the overall phase space $\ocphase \times \psphase \times \psphase \times \ocphase$ by $\SM^* \times \SM$, where the star indicates that the 1c-coordinates are now the left coordinates, and the ps-coordinates the right coordinates (consistent with taking the adjoint of the operator).  In addition, the partial diagonal is only a p-submanifold if we blow up $\{ \rhops = \rhops' = 0 \}$, where the prime indicates a coordinate in the right factor of $\SM^* \times \SM$. Thus the analysis of the composition of canonical relations is done on the space 
\begin{equation}\label{eq:SM2*}
\SM^{2, *} = [ \SM^* \times \SM; \{ \rhops = \rhops' = 0 \}].
\end{equation}

We observe that there is a map (a b-fibration, in Melrose's terminology) from $\SM$ to $\ocphase$, given by the blowdown to $\SM_0$ followed by the projection to the right factor. Thus, there is a b-fibration 
\begin{equation}
    \SM^* \times \SM \to \ocphase \times \ocphase. 
\end{equation}
If we blow up the corner $\{\xoc = x_{\oc}' = 0 \}$ in $\ocphase \times \ocphase$, creating the 1c-1c phase space, then a b-fibration can be restored by blowing up those boundary hypersurfaces in $\SM^* \times \SM$ that map to this corner. Because we are only working near the hypersurfaces where $\xoc = x_{\oc}' = 0$ and $\rhops = \rhops' = 0$, this means that the blowup \eqref{eq:SM2*} is such that there is a b-fibration, that we will call $F$, to the 1c-1c phase space $(\ocphase)^2_b$.

\begin{lmm} \label{lemma:clean-composition-lagrangian}
Let $C_+^* * C_-$ denote the lift of $C_+^* \times C_- \subset \SM^*$ to $\SM^{2,*}$, and let $\Delta_{\ps}$ denote the partial diagonal $\ocphase \times \Diag(\psphase) \times \ocphase$ lifted to the same space. Then $C_+^* * C_-$ intersects $\Delta_{\ps}$ cleanly with excess $1$, in the sense that the intersection $\SI$ is a submanifold of both, and at every point of the intersection, the tangent space $T \SI$ is equal to the intersection $T(C_+^* * C_-) \cap T\Delta_{\ps}$. Moreover, the dimension of $\SI$ is $2n+1$ which is $1$ (the excess) greater than would be expected from a transverse intersection of two submanifolds of the same dimensions. 

The map $F$ restricts to $(C_+^* * C_-) \cap \Delta_{\ps}$ to a fibration to $\beta_b^* \mathrm{Gr} (\Sg)$, the graph of the classical scattering map defined in \eqref{eq:def-classical-sc-map}, lifted to the 1c-1c phase space $[\ocphase \times \ocphase; \{ \xoc = \xoc' = 0 \}]$ via the blowdown map $\beta_b$ for this blowup. The fibres are one-dimensional and can be identified with the bicharacteristic that is identified by the point in $\mathrm{Gr}(\Sg)$. 
\end{lmm}

\begin{proof}
 We first consider the situation in the interior of $\SM^{2, *}$. In this region, the blowups are irrelevant and we are genuinely in the product situation. So consider first the product, $C_+^* \times C_-$. This takes the form (in the notation of Definition~\ref{defn:sojoun relations}), locally in the interior, 
 \begin{equation}\label{eq: can rel product}
C_+^* \times C_- =  \{  (q, \mu_{q}(s), \gamma_{q'}(s'), q')  \mid q, q' \in \ocphase \}.
 \end{equation}
Let $B_\pm : \Sigma \to \ocphase$ be the map that takes a point on $\Sigma$ to its forward, resp. backward endpoint in $\ocphase$. Then a vector field tangent to $C_+^* \times C_-$ takes the form 
\begin{equation}\label{eq:Tproduct}
((B_+)_* V_1, V_1, V_2, (B_-)_* V_2), \text{ where } V_1, V_2 \text{ are tangent to } \Sigma. 
\end{equation}
As a set, $(C_+^* \times C_-) \cap \Delta_{\ps}$ is given by 
 \begin{equation}\label{eq:intersection product case}
(C_+^* \times C_-) \cap \Delta_{\ps} =  \{  (q_+, \gamma_{q'}(s), \gamma_{q'}(s), q')  \mid q' \in \ocphase \}
 \end{equation}
where $q_+$ is the forward endpoint of the bicharacteristic $\gamma_{q'}$, that is, $q_+ = B_+(\gamma_{q'}(s))$. Since both $q'$ and $q_+$ are smooth functions of $\gamma_{q'}(s)$, and since $\gamma_{q'}(s)$ is an arbitrary point in $\Sigma$ we see that this intersection is diffeomorphic to $\Sigma$ (recall that this is always working locally in the interior of $\Sigma$). This is a submanifold of dimension $2n+1$, as claimed. Next consider the intersection of the tangent spaces of $C_+^* \times C_-$ and of $\Delta_{\ps}$. The tangent space of the former is given by \eqref{eq:Tproduct}. For the vector field on the RHS of \eqref{eq:Tproduct} to be tangent to $\Delta_{\ps}$ we require that $V_1 = V_2$, while $V_1$ can be an arbitrary vector tangent to $\Sigma$. We see that this intersection also has dimension $2n+1$ and coincides with the tangent space of the intersection. 

It is clear from \eqref{eq:intersection product case} that if we project to the first and last factors, then we get the graph of the classical scattering relation, and that the fibres of this map are one-dimensional and are given by the bicharacteristic determined by a point on the graph. Another way to view the non-transversality of $C_+^* \times C_-$ and $\Delta_{\ps}$ is to observe that their conormal bundles have a nontrivial intersection, namely the one form $(0, dp, -dp, 0)$ is common to both, where $p$ is the symbol of $P$; moreover, this spans the intersection of their conormal bundles. 

Now we consider the situation near the boundary of $\SM^{2,*}$. Recall first that we are only interested in a region near the blowup of $\ffocps \cap \ffpsoc$ and away from all other boundary hypersurfaces. As a manifold with corners, $C_+^* * C_-$ is the blowup of $C_+^* \times C_-$ at the corner:
\begin{equation}\label{eq:can rel blowup}
C_+^* * C_- = [C_+^* \times C_-; \{ \rhops = \rhops' = 0 \}].
\end{equation}
(In case this is confusing, the set being blown up could equally well be written $\{ \xoc = \xoc' = 0 \}$, since $\xoc$ vanishes on $C_\pm$ precisely when $\rhops$ does.) 
However, we are only interested in $C_+^* * C_-$ near the blowup face and away from the (lifts of the) original boundary hypersurfaces $\{ \rhops = 0 \}$ and $\{ \rhops' = 0 \}$. So the coordinates on $C_+^* * C_-$ in the region of interest are the same as for $C_+^* \times C_-$, except that we replace $(\rhops, \rhops')$ by $(\rhops, \sigma_{\ps})$ where $\sigma_{\ps} = \rhops'/\rhops$. 

We can view \eqref{eq:can rel blowup} as being parametrized by $\Sigma^2_b = [\Sigma \times \Sigma; \{ \rhops = \rhops' = 0 \}]$ since $C_\pm$ are both parametrized by $\Sigma$. Moreover, the intersection with $\Delta_{\ps}$ is parametrized by the `diagonal' $\Sigma \subset \Sigma^2_b$, which is a p-submanifold of dimension $2n+1$. 

The description of the tangent space $T(C_+^* * C_-)$ is, however, a little different compared to \eqref{eq:Tproduct}. The reason for this is that, at a boundary point of $C_+$, say, a basis of the tangent space of $C_+$ includes (at least) one vector of the form $((B_+)_*V_n, V_n)$ transverse to the boundary, where $V_n$ is tangent to $\Sigma$, and transverse to the boundary at $\rhops = 0$. Here, the subscript $n$ stands for non-tangential, or normal. Lifted to $C_+^* * C_-$ however, this vector field will not be smooth, as the coefficient of $\partial_{\sigma_{\ps}}$ will blow up as $\rhops^{-1}$. However, a suitable linear combination of $V_n$ and $V_n'$, namely $V_n + \sigma_{\ps} V_n'$, where $V_n'$ is the same vector field in the primed coordinates, is smooth on $C_+^* * C_-$ as the singular coefficients of $\partial_{\sigma_{\ps}}$ cancel. Thus, a vector field tangent to $C_+^* * C_-$ is a linear combination of 
\begin{equation}\label{eq:Tproduct n}
((B_+)_* V_n, V_n, \sigma_{\ps} V_n',  (B_-)_* (\sigma_{\ps} V_n')), 
\end{equation}
and 
\begin{equation}\label{eq:Tproduct t}
((B_+)_* V_1, V_1, V_2, (B_-)_* V_2), \text{ where } V_1, V_2 \text{ are tangent to both $\Sigma$ and $\partial (\psphase)$.}
\end{equation}
We consider which of these vectors are also in $T(\Delta_{\ps})$. First, to lie in $\Delta_{\ps}$ we must have $\sigma_{\ps} = 1$ and 
the middle components are equal, i.e. in reference to \eqref{eq: can rel product}, we must have $\mu_{q}(s)= \gamma_{q'}(s')$. Vectors tangent to $\Delta_{\ps}$ take the form $V + V'$ where $V$ is tangent to $\Sigma$, and $V'$ is the same vector field in the primed coordinates. This means that \eqref{eq:Tproduct n} is in the intersection of tangent spaces, since $\sigma_{\ps} = 1$ at the intersection. Also, the vector fields in \eqref{eq:Tproduct t} are tangent when $V_2 = V_1$. This gives us a $2n+1$-dimensional space of vectors tangent to both submanifolds $C_+^* * C_-$ and $\Delta_{\ps}$, which agrees with the dimension of the intersection of the two submanifolds. This demonstrates that the intersection is clean, with excess $1$ uniformly to the boundary. As above, the composition of the canonical relations gives, locally, the graph of the classical scattering map, and when we project $C_+^* * C_- \cap\Delta_{\ps}$ to $(\ocphase)^2_b$ then the fibres of this map are 1-dimensional and can be identified, at $(q_-, \Sg(q_-))$, with the bicharacteristic that runs between $q_-$ and $\Sg(q_-)$. 
\end{proof}

Let $\Phi_\pm$ be local non-degenerate phase functions parametrizing $\Lambda_\pm$. Then $-\Phi_+$, which the order of variables reversed, is a phase function for $\Lambda_\pm^*$. We may assume without loss of generality, according to Remark~\ref{rem:phase function leading term}, that $\Phi_\pm$ take the form
\begin{equation}\label{eq:Phi explicit form}
    \Phi_\pm(\xoc,\yoc,t,z,\theta_\pm) = -\frac{t}{\xoc^2} + \frac{\varphi_\pm(\mathsf{K}, \theta_\pm)}{\xoc}. 
\end{equation}

\begin{prop}\label{prop: clean phase fn}
Let $\Phi_\pm$ be as in \eqref{eq:Phi explicit form}. Then 
\begin{equation}
\Phi(\xoct,\sigma,\yoco,\yoct,\theta_+,\theta_-) = -\Phi_+(\xoco=\sigma \xoct,\yoco,t,z,\theta_+) + \Phi_-(\xoct,\yoct,t,z,\theta_-)    
\end{equation}
is a clean phase function parametrizing the (twisted) graph of the classical scattering map $\Sg$ (which is the composition of the two corresponding canonical relations, according to the previous lemma), with excess $1$. Here we regard the `intermediate' variables $(z, t)$, over which we would integrate to compose two Fourier operators associated with $\Lambda_+^*$ and $\Lambda_-$, as additional integrated variables for the phase function $\Phi$. 
\end{prop}

\begin{proof}
This almost follow from the standard procedure in \cite[Proposition~21.2.19]{hormander2007analysis}, which still holds if the homogeneity is absent in both the assumption and conclusion, and we address minor differences caused by the b-blow up and parabolic scaling here.

To be consistent with Section~\ref{sec:1c-1c geometry}, we use $\xoc=\xoco$. We write $-\Phi_++\Phi_-$ as
\begin{equation}
   \frac{t(1-\sigma^2)}{\xoc^2} + \frac{-\varphi_+(\xoc,\yoco,t,z, \theta_+)+\sigma \varphi_-(\xoct,\yoct,t,z,\theta_-)}{\xoc}, 
\end{equation}
and now view $(t,z)$ as extra parameters in the 1c-1c parametrization. 
Then $\varphi_0 = t(1-\sigma^2)$ being critical in $t$ gives the condition $\sigma = 1$. And the critical condition in the second part (i.e., in $z,\theta_+,\theta_-$) is dealt in the same way as in \cite[Proposition~21.2.19]{hormander2007analysis}. 
In combination with Lemma~\ref{lemma:clean-composition-lagrangian}, this shows that $\Phi$ parametrizes $\beta_b^*(\mathrm{Gr}(\Sg))$ cleanly.

\end{proof}

\subsection{Composition}
We now give the main result of this section.

\begin{thm} \label{thm: 1c-ps composition to 1c-1c}
Suppose $A_- \in I_{\oc-\ps}^{m_-}(\R^{n+1} \times \R^n,\Lambda_-)$, $A_+ \in I_{\oc-\ps}^{m_+}(\R^{n+1} \times \R^n,\Lambda_+)$, we have
\begin{align}
A_+^*A_- \in I_{\oc-\oc}^{m_++m_-+\frac{1}{2}}(X_b^2,\beta_b^*(\mathrm{Gr}(\Sg))).
\end{align}
The principal symbol of $A_+^*A_-$ at $(q-, q_+)$, where $q_- \in W_-$ and $q_+ = \Sg(q_-)$, is 
\begin{align} \label{eq: principal symbol, 1c-ps composition to 1c-1c}
\int_{-\infty}^\infty \overline{\sigma_{1c-ps}^{m_+}(A_+)} \sigma_{1c-ps}^{m_-}(A_-) (q_-, \gamma_{q_-}(s)) ds ,
\end{align}
where the integral is a global non-vanishing section of\hfill\break $$\Omega^{1/2}(\beta_b^*(\mathrm{Gr}(\Sg))) \otimes S_{\oc-\oc}^{[m_-+m_++\frac{1}{2}]}(\beta_b^*(\mathrm{Gr}(\Sg))).$$ We remark that each of the principal symbols in \eqref{eq: principal symbol, 1c-ps composition to 1c-1c} has a half-density factor `along the flow', so the product is a density along the flow that can be invariantly integrated. Thus, the `$ds$' in \eqref{eq: principal symbol, 1c-ps composition to 1c-1c} is an abuse of notation, intended to indicate that the integral is along the bicharacteristic between $(q_-, q_+)$; the `$ds$' along the bicharacteristic is intrinsic to the product of principal symbols, and is not `added in later'. 
\end{thm}

\begin{proof}
Locally we can write, with $\msf{K}_\pm=(t,z,x_{\oc,\pm},y_{\oc, \pm)})$ and $\theta_\pm \in \R^{N_\pm}$:
\begin{align}  
\begin{split}
A_- = & (2\pi)^{\frac{-N_-}{2} }
\Big(\int e^{ i( \frac{-t}{x_{\oc,-}^2} + \frac{\varphi_{-}(\msf{K}_-,\theta_{-})}{x_{\oc,-}} )} 
x_{\oc,-}^{-(m_- +\frac{N_{-}}{2})-\frac{1}{4}}
\\& a_-(\msf{K}_-,\theta_{-}) d\theta_{-} \Big)
|dtdz|^{1/2}|\frac{dx_{\oc,-}dy_{\oc,-}}{x_{\oc,-}^{n+2}}|^{1/2}, 
\end{split}
\end{align}

\begin{align}  
\begin{split}
A_+ = & (2\pi)^{\frac{-N_+}{2} }
\Big(\int e^{ i( \frac{-t}{x_{\oc,+}^2} + \frac{\varphi_{+}(\msf{K}_+,\theta_{+})}{x_{\oc,+}} )} 
x_{\oc,+}^{-(m_+ +\frac{N_{+}}{2})-\frac{1}{4}}
\\& a_+(\msf{K}_+,\theta_{+}) d\theta_{+} \Big)
|dtdz|^{1/2}|\frac{dx_{\oc,+}dy_{\oc,+}}{x_{\oc,+}^{n+2}}|^{1/2}.
\end{split}
\end{align}

Then we have, writing $x_{\oc, -} = \sigma x_{\oc, +}$, and $\msf{K}'_\pm = (t, z, y_{\oc, \pm})$, and writing $\dvolh_{\oc, \pm}$ for a smooth nonvanishing 1c-half density, 
\begin{align*}
A_+^*A_-  = & (2\pi)^{-\frac{ N_- + N_+ }{2}} 
\Big( \int e^{i(  \frac{t(1 - \sigma^{-2})}{x_{\oc,1}^2}
+ \frac{\varphi_{+}(x_{\oc},\msf{K}'_+, \theta_+) - \sigma^{-1} \varphi_-(\sigma x_{\oc, _+}, \msf{K}'_-)}{x_{\oc,1}})
 }  \overline{a_+(\msf{K}_+, \theta_+)} a_-(\msf{K}_-, \theta_-)
\\& 
x_{\oc,1}^{-(m_+ + m_-) - (N_+ + N_-)/2 -1/2} \sigma^{-m_2-\frac{N_-}{2}+\frac{1}{4}}
 dt \, dz \, d\theta_+ \, d\theta_-  \Big) \dvolh_{\oc, _+} \dvolh_{\oc, -}.
\end{align*}

Using Proposition~\ref{prop: clean phase fn}, we can view the overall phase function as a phase function $\Phi$ with $k_0 = 1$ (with $t$ as the corresponding `$\ococparaone$' variable, as we see explicitly in the expression above) and $k_1 = N_- + N_+ + n$ other integrated variables. Writing the above expression as (with $e = 1$ the excess)
\begin{align*}
A_+^*A_-  = & (2\pi)^{-\frac{ N_- + N_+ }{2}} 
\Big( \int e^{i\Phi(x_{\oc, +}, \sigma, y_{\oc, -}, y_{\oc, +}, \theta_-, \theta_+, z, t)}
  \overline{a_+(\msf{K}_+, \theta_+)} a_-(\msf{K}_-, \theta_-)
\\& 
x_{\oc,1}^{-(m_+ + m_- + e/2) - (2k_0 + k_1 -e)/2 + (n+1)/2} \sigma^{-m_2-\frac{N_-}{2}-\frac{3}{4}}
 dt \, dz \, d\theta_+ \, d\theta_-  \Big) \dvolh_{\oc, _+} \dvolh_{\oc, -}.
\end{align*}
We note that $$a(x_{\oc, +}, \sigma, y_{\oc, -} y_{\oc, +}, t, z, \theta_-, \theta_+) := \overline{a_+(\msf{K}_+, \theta_+)} a_-(\msf{K}_-, \theta_-) \sigma^{-m_2-\frac{N_-}{2}+\frac{1}{4}}$$ is a smooth function down to $x_{\oc, +} = 0$. This expression is therefore a 1c-1c Fourier integral operator of order $m_+ + m_- + 1/2$. The phase function $\Phi$ parametrizes the classical scattering map according to Proposition~\ref{prop: clean phase fn}.

We can rewrite $A_+^*A_-$ as
\begin{align*}
A_+^*A_-  = & (2\pi)^{-\frac{ N_{10}+N_{11}+N_{20}+N_{21} }{2}} 
\Big( \int e^{i\tilde{\Phi}}  \msf{a}(s,\msf{M}_-,v,w)
\\ &  x_{\oc,1}^{-(m_-+m_++\frac{1}{2})-\frac{2(N_{10}+N_{20}+1)+(N_{11}+N_{21}+n-1)}{2}+\frac{n+1}{2}}
\\& 
dw dsd\msf{M}_- dv \Big) |\frac{dx_{\oc,2}dy_{\oc,2}}{x_{\oc,2}^{n+2}}|^{1/2}\frac{dx_{\oc,1}dy_{\oc,1}}{x_{\oc,1}^{n+2}}|^{1/2},
\end{align*}
where $\msf{a}(s,\msf{M}_-,v,w)$ is given by
\begin{align*}
\msf{a}(s,\msf{M}_-,v,w) = a_-(t,z;x_{\oc,1},y_{\oc,1},w) \bar{a}_+(t,z;x_{\oc,2},y_{\oc,2},v) (\frac{x_{\oc,1}}{x_{\oc,2}})^{m_++\frac{2N_{20}+N_{21}}{2}+\frac{3}{4}} \msf{J},
\end{align*}
and $\tilde{\Phi}$ is $- \frac{t}{x_{\oc,1}^2}
+ \frac{\varphi_{11}(x_{\oc},t,z,y_{\oc,1},w)}{x_{\oc,1}}
+\frac{t}{x_{\oc,2}^2} -\frac{\varphi_{21}(x_{\oc},t,z,y_{\oc,2},v)}{x_{\oc,2}}$ after this coordinate change, which parametrizes $\beta_b^*(\mathrm{Gr}(\Sg))$ in the sense of Definition~\ref{defn: 1c-1c parametrization clean}.
The final conclusion on the principal symbol in \eqref{eq: principal symbol, 1c-ps composition to 1c-1c} follows from the definition of the principal symbol in \eqref{defn: 1c-1c principal symbol, clean} and and one can verify that (after changing to coordinates $(s,\msf{M}_-,v,w)$) the geodesic $\gamma_{q_-}(\cdot)$ is the fiber over $(q_-,q_+)$ of $C_{\tilde{\Phi}}$.
In fact, in the same way as verifying that $\tilde{\Phi}$ parametrizes $(\Lambda_+' \circ \Lambda_-')' = \beta_b^*(\mathrm{Gr}(\Sg))$, critical points of $\tilde{\Phi}$ over a fixed point $(q_-,q_+)$ in $\beta_b^*(\mathrm{Gr}(\Sg))$ corresponds to the fiber in $\Lambda_+' \times \Lambda_-'$ intersecting the diagonal in the intermediate (lift of) $\overline{{}^{\ps}T^*\R^{n+1}} \times \overline{{}^{\ps}T^*\R^{n+1}}$ over it, which is precisely the bicharacteristic with $q_-,q_+$ as limiting point in the backward and forward direction respectively.
\end{proof}

%Notice that $s=z \cdot y_{\oc}$ is parametrizing fibers of the clean intersection $\Lambda_+^* \circ \Lambda_-$, and $\varphi_{01},%\varphi_{02}$ are independent of it, which means $e_0=0,e_1=1$ for this parametrization, thus recalling the numerology in \eqref{eq: %1c-1c FIO, local form} and summarizing discussion above, we have:

\begin{remark}
By our definition of $I_{\oc-\ps}^{m_-}(\R^{n+1} \times \R^n,\Lambda_\pm)$, $a_-,a_+$ have compact support in $t,z$. So the integral \eqref{eq: principal symbol, 1c-ps composition to 1c-1c} is always over a compact interval in $s$. 
\end{remark}

%%%%%%%%%%%%%%%%%%%%%%%%%%%%%%%%%%%%%%%%%%%%%%%%%%%%%%%%%%%%%%%
%%%%%%%%%%%%%%%%%%%%%%%%%%%%%%%%%%%%%%%%%%%%%%%%%%%%%%%%%%%%%%%
%%%%%%%%%%%%%%%%%%%%%%%%%%%%%%%%%%%%%%%%%%%%%%%%%%%%%%%%%%%%%%%

\section{The scattering map} \label{sec: scattering map}

\subsection{Proof of the main theorem}
We recall formula \eqref{eq:S formula} for the scattering map in terms of the Poisson operators:
\begin{equation}
 S = i(2\pi)^n \mathcal{P}_+^* [P,Q_+] \mathcal{P}_-.
\label{eq:S formula 2} \end{equation}
We microlocalize this formula to write $S$ as a sum of terms, and analyze each piece. First we choose a microlocal cutoff $Q_{\oc} \in \Psi_{\oc}^{0,0}(\RR^n)$ on the one-cusp side. This is chosen so that $Q_{\oc}$ is microlocally equal to the identity on all points of $\partial \overline{{}^{\oc} T^* \R^n}$ that are the initial point of a bicharacteristic that meets the perturbation, i.e. meets $\WF'_{\ps}(P - P_0)$,
and is microsupported away from fibre-infinity. (Recall that it is only points in $\partial \overline{{}^{\oc} T^* \R^n}$ at base-infinity, with finite one-cusp frequency, that produce bicharacteristics that travel over the interior of $\overline{\R^{n+1}}$.)
We break $S$ into $S Q_{\oc} + S (\Id - Q_{\oc})$. First we consider $S (\Id - Q_{\oc})$.

\begin{lmm}\label{lem:S near fibre infinity} The operator $S(\Id - Q_{\oc})$ is equal to $\Id - Q_{\oc}$, up to an operator mapping distributions to Schwartz functions, or equivalently, up to an operator whose Schwartz kernel is Schwartz, i.e. in $\Schw(\R^{2n})$.
\end{lmm}

\begin{proof} 
Using \eqref{eq:S formula 2} we write 
  \begin{equation}
 S(\Id - Q_{\oc}) = i(2\pi)^n \mathcal{P}_+^* [P,Q_+] \mathcal{P}_- (\Id - Q_{\oc}).
\label{eq:S formula 3} \end{equation}

We write this as 
  \begin{equation}
 S(\Id - Q_{\oc}) = i(2\pi)^n \Big( \mathcal{P}_+^* [P - P_0,Q_+] \mathcal{P}_- (\Id - Q_{\oc}) + \mathcal{P}_+^* [P_0,Q_+] \mathcal{P}_- (\Id - Q_{\oc}) \Big).
\label{eq:S formula 4} \end{equation}
For the first term on the right hand side, $[P - P_0,Q_+] \mathcal{P}_- (\Id - Q_{\oc})$ maps any tempered distribution to a Schwartz function by \eqref{eq:5.5-3}. 
Therefore, $\mathcal{P}_+^*[P - P_0,Q_+] \mathcal{P}_- (\Id - Q_{\oc})$ maps distributions to Schwartz functions, and then by Proposition~\ref{prop:Poisson adjoint mapping}. 

Using the formula \eqref{eq:Poim-expression} for the Poisson operator $\Poim$ in terms of $\Poiz$, we further decompose the second term on the RHS of \eqref{eq:S formula 4}, 
  \begin{equation}
 \mathcal{P}_+^* [P_0,Q_+] \mathcal{P}_- (\Id - Q_{\oc}) =  \mathcal{P}_+^* [P_0,Q_+] \mathcal{P}_0 (\Id - Q_{\oc}) - \mathcal{P}_+^* [P_0,Q_+] P_+^{-1} (P-P_0) \mathcal{P}_0 (\Id - Q_{\oc}).
\label{eq:S formula 5} \end{equation}
Consider the second term on the right hand side. 
By \eqref{eq:5.5-3} again, the operator $(P-P_0) \mathcal{P}_0 (\Id - Q_{\oc})$ maps distributions to Schwartz functions. Then by Corollary~\ref{coro:propagator WF mapping}, applying $P_+^{-1}$ to a Schwartz function yields a function with wavefront set contained in $\Rp$. But $[P_0, Q_+]$ has operator wavefront set disjoint from $\Rp$, so this maps back to Schwartz functions, 
and finally $\Poip^*$ maps to Schwartz functions by Proposition~\ref{prop:Poisson adjoint mapping}. Thus, modulo operators mapping distributions to Schwartz functions, $S(\Id - Q_{\oc})$ is given by 
$$
\mathcal{P}_+^* [P_0,Q_+] \mathcal{P}_0 (\Id - Q_{\oc}).
$$
Using the expression for $\mathcal{P}_+^*$ obtained by taking adjoint on both sides of \eqref{eq:Poim-expression} (for $\Poip$), we have
%We next write the adjoint of the Poisson operator on the left in terms of the free Poisson operator. This yields 
  \begin{equation}
\mathcal{P}_+^* [P_0,Q_+] \mathcal{P}_0 (\Id - Q_{\oc}) 
= \mathcal{P}_0^* [P_0,Q_+] \mathcal{P}_0 (\Id - Q_{\oc}) + \mathcal{P}_0^* (P - P_0)^*   (P_-^{-1})^* [P_0,Q_+] \mathcal{P}_0 (\Id - Q_{\oc}) .
\label{eq:S formula 6} \end{equation}
Similar to before, the second term on the RHS maps distributions to Schwartz functions. Thus up to such operators, we find that 
 \begin{equation}
 S(\Id - Q_{\oc}) =  i(2\pi)^n  \mathcal{P}_0^* [P_0,Q_+] \mathcal{P}_0 (\Id - Q_{\oc}) = \Id - Q_{\oc}, 
 \label{eq:S formula 7} \end{equation}
 where the second identity is obtained from the observation that for the free operator $P_0$, the scattering matrix is the identity, thus from \eqref{eq:S formula 2} we obtain  
 $$
 \Id = i(2\pi)^n \mathcal{P}_0^* [P_0,Q_+] \mathcal{P}_0.
 $$
 \end{proof}
 
 \begin{prop}\label{prop:S microlocalized}
 The operator $S Q_{\oc}$, where $Q_{\oc}$ is as described above Lemma~\ref{lem:S near fibre infinity}, is a one-cusp FIO associated to a fibred-Legendre submanifold as described in Section~\ref{sec: 1c-1c Lagrangian distribution}. 
 \end{prop}
 
 \begin{proof}
 We choose a microlocal cutoff $Q_{\ps}$ on the parabolic scattering side. We claim that $Q_{\ps}$ can be chosen so that
 \begin{itemize}
 \item $Q_{\ps}$ is microlocally equal to the identity on $\WF'([P, Q_+]) \cap \Lambda_-' \circ \WF'(Q_{\oc})$,  and 
% \item $\WF'(Q_{\ps}$ does not meet $\SjR_-(\WF'(\Id - Q_{\oc}))$, and 
 \item $\WF'(Q_{\ps})$ projects to a compact set in $\R^{n+1}_{z,t}$. 
 \end{itemize}
 These properties are possible due to the fact that $\WF'(Q_{\oc})$ is contained in a bounded set of one-cusp frequencies. One can therefore find a compact set $K \subset \R^{n+1}$ such that the bicharacteristics emanating from $\WF'(Q_{\oc})$ all meet the compact set $K$. Then, $\WF'([P, Q_+])$ is disjoint from the radial sets, hence, on each bicharacteristic emanating from $\WF'(Q_{\oc})$, the microsupport of $\WF'([P, Q_+])$ lies over a compact set in $(z,t)$-space. Since there are a compact family of such bicharacteristics, by enlarging $K$ if necessary we can assume that the projection of $\WF'([P, Q_+]) \cap \SjR_- \circ \WF'(Q_{\oc})$ to $\R^{n+1}_{z,t}$ is contained in $K$. Then we can take $Q_{\ps}$ to be a multiplication operator by a function that is identically $1$ on $K$, with compact support. 
 
 We then decompose $S Q_{\oc} = i(2\pi)^n \mathcal{P}_+^* [P,Q_+] \mathcal{P}_- Q_{\oc}$ as 
\begin{equation}
S Q_{\oc}  = i(2\pi)^n \mathcal{P}_+^* Q_{\ps} [P,Q_+] \mathcal{P}_- Q_{\oc}  + i(2\pi)^n \mathcal{P}_+^* (\Id - Q_{\ps}) [P,Q_+] \mathcal{P}_- Q_{\oc}.
\label{eq:S loc}\end{equation}

For the second term, by Proposition~\ref{prop:microlocalized Poisson are 1c-ps FIOs} we know $(\Id - Q_{\ps}) [P,Q_+] \mathcal{P}_- Q_{\oc}$ is a 1c-ps Lagrangian distribution.
By the first property of $Q_{\ps}$ above and  Proposition~\ref{eq:1c-ps-FIO-WF-general}, we know $\WF'_{\oc-\ps}\big( (\Id - Q_{\ps}) [P,Q_+] \mathcal{P}_- Q_{\oc} \big)$ is empty, which shows that $(\Id - Q_{\ps}) [P,Q_+] \mathcal{P}_- Q_{\oc}$ maps tempered distributions to Schwartz functions on $\R^{n+1}_{z,t}$.
Then applying Proposition~\ref{prop:Poisson adjoint mapping}, we see that the second term maps tempered distributions to Schwartz functions on $\R^n$. It remains to analyze the first term. To do this we introduce $\tilde Q_{\ps}$, a multiplication operator by a compactly supported function that is $1$ on the projection of $\WF'(Q_{\ps})$ in $\R^{n+1}_{z,t}$ (here we require the second condition on $Q_{\ps}$ above). We also 
introduce $\tilde Q_{\oc}$ that is microlocally the identity on $\WF'(Q_{\oc})$, such that 
$\SjR_+ \circ \WF'(\tilde Q_{\oc})$ is disjoint from $\WF'(\tilde Q_{\ps})$.  
The first term on the RHS of \eqref{eq:S loc} can be decomposed as 
\begin{equation}\begin{gathered}
\mathcal{P}_+^* Q_{\ps} [P,Q_+] \mathcal{P}_- Q_{\oc} = 
\mathcal{P}_+^* \tilde Q_{\ps} Q_{\ps} [P,Q_+] \mathcal{P}_- Q_{\oc} \\
=  \tilde Q_{\oc} \mathcal{P}_+^* \tilde Q_{\ps} Q_{\ps} [P,Q_+] \mathcal{P}_- Q_{\oc} 
+ (\Id - \tilde Q_{\oc})  \mathcal{P}_+^*  Q_{\ps}  \tilde Q_{\ps} [P,Q_+] \mathcal{P}_- Q_{\oc} .
\end{gathered}\end{equation}
 We claim that the last term on the RHS maps distributions to Schwartz functions. To show this, it suffices to prove the same property for the adjoint
 $$
 Q_{\oc}^* \mathcal{P}_-^* [Q_+^*, P^*] Q_{\ps}  \Poip (\Id - \tilde Q_{\oc})^*. 
 $$
 But the adjoint has the same form as the last term of \eqref{eq:S formula 5} above, and a similar argument applies: By the choice of $\tilde Q_{\oc}$, we find that $Q_{\ps}  \Poip (\Id - \tilde Q_{\oc})^*$ maps distributions to $C_c^\infty(\R^{n+1})$, so using Proposition~\ref{prop:Poisson adjoint mapping} we find this term maps distributions to Schwartz functions. 
 
 Thus, modulo operators mapping distributions to Schwartz functions, $S Q_{\oc}$ has the form 
 $$
 SQ_{\oc} =  \tilde Q_{\oc} \mathcal{P}_+^* \tilde Q_{\ps} Q_{\ps} [P,Q_+] \mathcal{P}_- Q_{\oc} \text{ modulo } \Schw(\R^{2n}). 
 $$
 This takes the form $A_+^* A_-$ where 
 \begin{equation}
 A_+ = \tilde Q_{\ps} \Poip \tilde Q_{\oc}, \qquad  A_- =  Q_{\ps} [P, Q_+] \Poim  Q_{\oc}.
 \end{equation}
 By Proposition~\ref{prop:microlocalized Poisson are 1c-ps FIOs}, these microlocalized Poisson operators are fibred-Legendre distributions associated to the (microlocalized) sojourn Lagrangians $\Lambda_\pm$ defined in \eqref{eq:microlocalized sojourn relns}. By the composition result of Theorem~\ref{thm: 1c-ps composition to 1c-1c}, this composition is a one-cusp FIO associated to the classical scattering map. This completes the proof. 
\end{proof}

\begin{prop} \label{prop: ellipticity of tilde S}
$\csmp = S Q_{\oc}$ is elliptic on the elliptic set of $Q_{\oc}$.
\end{prop}
\begin{proof}
Recalling the proof of Proposition \ref{prop: Possion parametrix, K,-}, we know that the principal symbol of $K_-$ is obtained by solving \eqref{eq: transport ODE, a-0} on $\mathcal{U}_-$ with $a_{-,0}=1$ outside the region with the metric perturbation. Thus the region over which we are actually solving this transport equation is compact and $a_{-,0}>0$ on it by the uniqueness theorem of linear homogeneous ODEs. So there is a constant $c>0$ such that $a_{-,0} \geq c>0$ on $\mathcal{U}_-$. Then by Theorem \ref{thm: PsiDO- 1c-ps FIO composition}, we know that the principal symbol of $i[P,Q_+]K_-$ is
\begin{equation}
 (\msf{H}_p^{2,0}q_+)a_{-,0} \cdot \nu_{\oc-\ps},
\end{equation}
where $\nu$ is the density factor of $\Omega^{1/2}(\Lambda_-) \otimes S_{\oc-\ps}^{[\frac{1}{4}]}(\Lambda_-)$. By our choice of $q_+$, we know $\msf{H}_p^{2,0}q_+ \geq 0$ and $q_+$ changes from $0$ to $1$ on $\mathcal{U}_-$ along the $\msf{H}_p^{2,0}$-flow.

By the same reason as above for $a_{-,0}$, we can assume the amplitude $a_{+,0} \geq c>0$ on $\mathcal{U}_+$. Applying the second part of Theorem \ref{thm: 1c-ps composition to 1c-1c}, the principal symbol of $\csmp$ is
\begin{align}
\big( (2\pi)^n\int  a_{+,0}(\msf{H}_p^{2,0}q_+)a_{-,0} ds \big) \cdot \nu_{\oc-\oc},
\end{align}
where $\nu_{\oc-\oc}$ is the density factor comes from $\Omega^{1/2}(\beta_b^*(\mathrm{Gr}(\Sg))) \otimes S_{\oc-\oc}^{[0]}(\beta_b^*(\mathrm{Gr}(\Sg)))$. By discussion above, the integral is strictly positive and finishes the proof.
%and the equality does not always h
\end{proof}

\begin{proof}[Proof of Theorem~\ref{thm: main}]
The main theorem follows from the combination of Lemma~\ref{lem:S near fibre infinity} and Propositions~\ref{prop:S microlocalized} and \ref{prop: ellipticity of tilde S}. In fact, by Lemma~\ref{lem:S near fibre infinity}, $S(\Id - Q_{\oc})$ is a pseudodifferential operator, therefore associated to the identity relation for large 1c-frequencies. This agrees with the classical scattering relation for such frequencies, as the corresponding bicharacteristics do not meet the perturbation and therefore scatter trivially. On the other hand, Proposition~\ref{prop:S microlocalized} shows that $S Q_{\oc}$ is an FIO associated to the classical scattering relation. 

To prove the ellipticity, first Lemma~\ref{lem:S near fibre infinity} demonstrates ellipticity for large 1c-frequency. For $q_- \in \partial \overline{{}^{\oc} T^* \R^n}$ with finite 1c-frequency, we can choose $Q_{\oc}$ so that it is microlocally equal to the identity near $q_-$. Then $S = S Q_{\oc}$ microlocally near $q_-$, and Proposition~\ref{prop: ellipticity of tilde S} shows the ellipticity of $S$ at $(\Sg(q_-), q_-)$. 

The last statement in Theorem~\ref{thm: main} is an immediate consequence of Lemma~\ref{lem:S near fibre infinity}. 
\end{proof}

\begin{remark} Another way to see the ellipticity of $S$ is to note that the Schr\"odinger equation is time-reversible, in the sense that the time-reversed equation is another Schr\"odinger equation with the $D_t$ term switched in sign, and with the perturbations replaced by their time-reversals. This time reversed equation has a scattering map which is a 1-cusp FIO associated with the inverse of the classical scattering map, and inverts $S$. The existence of such an inverse implies ellipticity of $S$. 
\end{remark}

\subsection{Egorov theorem}

Now we discuss mapping properties in terms of the 1c-wavefront set.
\begin{prop}  \label{prop: tilde S*S is 1c PsiDO}
For $\csmp = S Q_{\oc}$  as in Proposition~\ref{prop:S microlocalized}, we have 
\begin{align} \label{eq: tilde S*S is 1c PsiDO}
\csmp^*\csmp \in \Psi_{\oc}^{-\infty,0}(X).
\end{align} 
\end{prop}

\begin{proof}
Let $\leftidx{^{\oc}}{N^*\Delta_b}$ be the 1-cusp conormal bundle in $\ococb$ of the lifted diagonal in $X^2_b$.
Then by Proposition \ref{prop: 1c-1c transversal composition}, and the fact that a canonical graph composing with its own transpose gives the conormal bundle of the (lifted) diagonal, we know
\begin{align*}
\csmp^*\csmp \in I^{0}_{\oc-\oc}(X^2_b;\leftidx{^{\oc}}{N^*\Delta_b}).
\end{align*}

Notice that $\leftidx{^{\oc}}{N^*\Delta_b}$ is parametrized by the phase function:
\begin{align} \label{eq: phase function, 1c conormal bundle}
\Phi  =  \frac{\xi_{\oc}(\sigma-1)}{x_{\oc}^2} +  \frac{\eta_{\oc} \cdot (y_{\oc}-y_{\oc}'')}{x_{\oc}},
\end{align}
with $v = \tau_{\oc}, w = \mu_{\oc}$. By the phase equivalence in Proposition~\ref{prop: phase equivalence, general parameter number, signature},
%Corollary~\ref{coro: clean parametrized FIO as a single term}
all elements in $I^{0}_{\oc-\oc}(X^2_b;\leftidx{^{\oc}}{N^*\Delta_b})$ can be written as a Lagrangian distributions with phase function of the form \eqref{eq: phase function, 1c conormal bundle}.

By the construction in \cite[Section~2.3]{zachos2022inverting} (see also \cite[Section~2.1]{jia2022tensorial}) we know that in fact
$I^{m}_{\oc-\oc}(X_b^2,\leftidx{^{\oc}}{N^*\Delta_b})$ viewed as a Schwartz kernel on the 1-cusp double space defined there is exactly the pseudodifferential class
\begin{align*}
\Psi_{\oc}^{-\infty,m}(X),
\end{align*}
defined there.
\end{proof}

Next we prove the composition law of 1c-pseudodifferential operators and 1c-1c Lagrangian distributions.
\begin{prop}  \label{prop: PsiDO- 1c1c FIO, composition}
Suppose $B, B' \in \Psi_{\oc}^{-\infty,0}$ be classical 1c-pseudodifferential operators of order zero, and let $A \in I^{m}_{\oc-\oc}(X_b^2,\mathcal{L})$, with principal symbols $b, b', a$ respectively.
%And suppose $\mathcal{L}_{\oc}$ is parametrized by (\ref{eq: 1c-1c phase function, normal form 2}) and let $\nu$ denote the half-density bundle factor. 
Then
\begin{align*}
B'A, A B \in I^{m}_{\oc-\oc}(X_b^2,\mathcal{L}_{\oc}),
\end{align*}
with principal symbol $ba$, $a b'$ respectively, where the principal symbols $b', b$, being smooth functions on the 1c-cotangent space $\ocphase$, are lifted to $\ococb$ via the left, resp. right, stretched projections, (that is, composition of the blowdown map and the left, resp. right, projection,  see \eqref{eq: defn  b-double space}, \eqref{eq:def-lifted-ococ-bundle}) and then restricted to $\SL$. 
\end{prop}
\begin{proof}
This proposition follows from Proposition \ref{prop: 1c-1c transversal composition}, since 1-cusp pseudodifferential operators are 1c-1c Lagrangian distributions with the canonical relation being the identity.
\end{proof}

Next we prove an Egorov-type result for our class of Lagrangian distributions. 
Let $\pi_L, \pi_R$ denote the left and right stretched projections from $\ococb$ to $\ocphase$ as described in Proposition~\ref{prop: PsiDO- 1c1c FIO, composition}. 
\begin{prop} \label{prop: 1c-1c Egorov}
Suppose $B \in \Psi_{\oc}^{-\infty,r}$ is a classical 1c-pseudodifferential operator of order $r$, and $A \in I^{m}_{\oc-\oc}(X_b^2,\mathcal{L})$ is such that $A$ is elliptic at the points of $\SL$ corresponding to the lift of $\WF'(B)$ by $\pi_R$.   Then there exists $B' \in \Psi_{\oc}^{-\infty,r}$ such that
\begin{align}
 B'A = AB \text{ modulo } I^{-\infty}_{\oc-\oc}(X_b^2,\mathcal{L}_{\oc}). 
\end{align}
%modulo a term in $I^{-\infty}_{\oc-\oc}(X_b^2,\mathcal{L}_{\oc})$. 
Moreover, $B'$ is elliptic at $\Sg(q_-)$ whenever $B$ is elliptic at $q_- \in W_- \equiv \ocphase$. 
The same result holds when we interchange the roles of $B$ and $B'$. 
%where $B$ has principal symbol $b$ and $B'$ has principal symbol $\Sg^*b$.
\end{prop}
\begin{proof}
Consider the case where we are given $B$ and try to construct $B'$.
We can construct $B'$ as an asymptotic sum
\begin{align}
B' = \sum_{j=0}^\infty B'_j, \quad B'_j \in \Psi_{\oc}^{-\infty,r-j}, 
\end{align}
such that
\begin{align*}
(\sum_{j=0}^N B'_j)A - AB \in I^{m+r-1-N}_{\oc-\oc}(X_b^2,\mathcal{L}_{\oc}).
\end{align*}
For $j=0$, we only need 
\begin{align}
\sigma^{m+r}_{\oc-\oc}(B_0'A) = \sigma^{m+r}_{\oc-\oc}(AB).
\end{align}
This has the unique solution $\sigma^{m+r}_{\oc-\oc}(B_0')(\Sg(q_-)) = \sigma^{m+r}_{\oc-\oc}(B)(q_-)$ by Proposition \ref{prop: 1c-1c transversal composition}, using the ellipticity of $A$. 
Clearly $B_0'$ is elliptic at at $\Sg(q_-)$ precisely when $B$ is elliptic at $q_-$. Then $B_j'$ is defined inductively, using
\begin{align}
\sigma^{m+r-N}_{\oc-\oc}(B_{N}'A) = \sigma^{m+r-N}_{\oc-\oc}(-(\sum_{j=0}^{N-1} B'_j)A +AB)
\end{align}
and the ellipticity of $A$ again. 
The converse case is proved in the same manner.
\end{proof}

\subsection{Proof of Corollaries}\label{subsec:proof of corollaries}
We briefly sketch the proof of Corollaries~\ref{coro:bounded}, \ref{coro:wavefront} and \ref{coro:noncompact-difference}.

\begin{proof}[Proof of Corollary~\ref{coro:bounded}] We use the boundedness property \eqref{eq:1c boundedness} of the 1c-pseudodifferential algebra. We write $S = S Q_{\oc} + S (\Id - Q_{\oc})$, where $Q_{\oc}$ is as in Lemma~\ref{lem:S near fibre infinity}. Then, by this lemma, $S (\Id - Q_{\oc})$ has the desired boundedness property, so it remains to consider $\tilde S := S Q_{\oc}$. According to Proposition~\ref{prop: tilde S*S is 1c PsiDO}, $\tilde S^* \tilde S$ is a 1c-pseudodifferential operator of order $(0,0)$, hence bounded on $L^2$, and thus $\tilde S$ is itself bounded on $L^2$. Moreover, using Proposition~\ref{prop: 1c-1c transversal composition}, the same is true of $(B \tilde S B^{-1})^* B \tilde S B^{-1}$ for any elliptic, invertible $B \in \Psi_{\oc}^{s,l}$, showing that $\tilde S$ is bounded on $H_{\oc}^{s,l}$. 
\end{proof}

\begin{proof}[Proof of Corollary~\ref{coro:wavefront}] We have
\begin{align*}
q \in (\WF_{\oc}^{-\infty,r}(f))^c
\end{align*}
if and only if there exists $B \in \Psi_{\oc}^{-\infty,r}$ that is elliptic at $q$ such that
\begin{align*}
Bf \in L^2.
\end{align*}
By the $L^2-$boundedness of $\csmp$, we have
\begin{align*}
\csmp Bf \in L^2.
\end{align*}
On the other hand $\Sg(q) \in (\WF_{\oc}^{-\infty,r}(\csmp f))^c$ if and only if there exists $B' \in \Psi_{\oc}^{-\infty,r}$ that is elliptic at $\Sg(q)$ such that
\begin{align*}
B'\csmp f \in L^2.
\end{align*}
And the equivalence of these two statements follows from Proposition \ref{prop: ellipticity of tilde S} and Proposition \ref{prop: 1c-1c Egorov}.
\end{proof}

\begin{proof}[Proof of Corollary~\ref{coro:noncompact-difference}] 
Let $\beta$ be the projection $\ococb \to \overline{{}^{\oc}T^*\R^n} \times \overline{{}^{\oc}T^*\R^n}$ acting as the b-blow down map on base variables and identity on momentum variables.
By the given condition $\mathrm{Gr}(\Sg) \neq \mathrm{Gr}(\mathfrak{S}_{g_0})$, there exists $q \in \overline{{}^{\oc}T_{\partial \overline{\R^n}}^*\R^{n}}$ such that
\begin{equation}
     (q,q') \notin \mathrm{Gr}(\Sg),
\end{equation}
where $q'$ is $q$ with the sign of momentum variables switched.
This means we can find $Q \in \Psi_{\oc}^{0,0}(\R^{n})$ that is elliptic at $q$ while 
\begin{equation} \label{eq:9.3-1}
\{(q_1,q_2'): \, q_i \in \WF'_{\oc}(Q) \} \cap \mathrm{Gr}(\Sg) = \emptyset.
\end{equation}
Then applying Proposition~\ref{prop: 1c-1c transversal composition} twice, since $Q$ is a $\oc-\oc$ FIO associated to the part of $\leftidx{^{\oc}}{N^*\Delta_b}$ localized by $\WF'_{\oc}(Q)$, we know the wavefront set of $QSQ$ is contained in the left hand side of \eqref{eq:9.3-1},
%$\WF'_{\oc}(Q) \cap \beta(\mathrm{Gr}(\Sg)) = \emptyset$, 
hence empty.
So $Q(S-\Id)Q = Q^2$ modulo a Schwartz term and is an elliptic $\oc$-pseudodifferential operator, which is not compact from $H_{\oc}^{s,l}(\R^n)$ to itself. This implies that $S-\Id$ is not so either, since $Q$ is bounded on such space and otherwise the composition would be compact then. 
\end{proof}

\begin{comment}
We finally have:
\begin{thm}
 $q \in (\WF_{\oc}^{-\infty,r}(f))^c$, if and only if $\Sg(q) \in (\WF_{\oc}^{-\infty,r}(\csmp f))^c$.
\end{thm}
\begin{proof}
    
\end{proof}
\end{comment}

\appendix

\section{Equivalence of phase functions}
\label{app: phase equivalence}

In this appendix, we discuss the equivalence of phase functions in both the $\oc-\ps$ and $\oc-\oc$ setting.
Both of them are similar to \cite[Section~3.1]{FIO1} and \cite[Section~5]{duistermaat-guillemin1975spectrum}, with the only difference being that we have we have parabolic scaling for two parts of our phase functions now and this is similar in both settings, so we only give details in the 1c-1c setting.

We first state the phase equivalence in the 1c-ps setting, which is used in the proof of Proposition~\ref{prop:invariance under parametrization change}.

\begin{prop} \label{prop:1c-ps-phase-equivalence}
Let $u \in I_{\oc-\ps}^m(\R^{n+1} \times \R^n, \Lagps)$ be as in \eqref{eq: 1c-ps FIO definition} with phase function $\Phi_{\oc-\ps} = \frac{ \varphi_0(t,\theta_0) }{x_{\oc}^2} + \frac{\varphi_1(\msf{K},\theta_0,\theta_1)}{x_{\oc}}$.
Suppose the image of the support of the amplitude under the parametrization map as in \eqref{eq:Legps param} is contained in $U \subset \Legps$,
%has essential support in an open set 
and that 
\begin{equation}
    \tilde{\Phi}_{\oc-\ps} =   \frac{ \tilde{\varphi}_0(t,\tilde{\theta}_0) }{x_{\oc}^2} + \frac{\tilde{\varphi}_1(\msf{K},\tilde{\theta}_0,\tilde{\theta}_1)}{x_{\oc}}
\end{equation} 
is another parametrization of $\Legps$ in $U$. 

%%%%%%%%%%%%%%%%%%%%%
Namely if two phase functions parametrizing $L$ have the same number $(k_0, k_1)$ of $\tilde{\theta}_0$- and $\tilde{\theta}_1$-variables, and satisfy two signature conditions:
\begin{equation} \label{eq: signature assumptions-main-prop}
\sgn \; d_{\theta_0 \theta_0}^2\varphi_0 =  \sgn \; d_{\tilde{\theta}_0 \tilde{\theta}_0}^2 \tilde{\varphi}_0, \qquad 
\sgn \; d_{\theta_1 \theta_1}^2\varphi_1 =  \sgn \; d_{\tilde{\theta}_1 \tilde{\theta}_1}^2\tilde{\varphi}_0,
\end{equation}
at one point of $C_{\Phi_{\oc-\ps}}$,
then the two parametrizations are related by a coordinate change in the sense that: there are functions $\Theta_0(t,\tilde{\theta}_0,\tilde{\theta}_1),\Theta_1(\msf{K},\tilde{\theta}_0,\tilde{\theta}_1)$ such that 
\begin{equation}
    (\msf{K},\tilde{\theta}_0,\tilde{\theta}_1) \to (\msf{K},\Theta_0,\Theta_1)
\end{equation}
is a local diffeomorphism and $\Phi_{\oc-\ps}(\msf{K},\Theta_0,\Theta_1) = \tilde{\Phi}_{\oc-\ps}(\msf{K},\tilde{\theta}_0,\tilde{\theta}_1)$.
Consequently, $u$ can be written, modulo $\mathcal{S}(\R^{n+1} \times \R^n)$, as an oscillatory integral with respect to the parametrizing function $\tilde \Phi_{\oc-\ps}$. 
\end{prop}

Now we turn to the 1c-1c case, or more precisely the proof of
Proposition~\ref{prop: phase equivalence, general parameter number, signature}.
The main part of the proof of this equivalence, is the following lemma, which deals with the non-degenerate phase functions with fixed signature and follows closely to \cite[Lemma~4.5]{hassell1999spectral} and \cite[Theorem~3.1.6]{FIO1}.
\begin{lmm}  \label{lemma: phase equivalence, same parameter number, signature}
Suppose 
\begin{align*}
\Phi_{\oc-\oc}(x_{\oc},\sigma,y_{\oc},y_{\oc}'',v,w) = \frac{\varphi_0(\sigma,v)}{x_{\oc}^2} + \frac{\varphi_1(x_{\oc},\sigma,y_{\oc},y_{\oc}'',v,w)}{x_{\oc}} ,
\end{align*}
and
\begin{align*}
\tilde{\Phi}_{\oc-\oc}(x_{\oc},\sigma,y_{\oc},y_{\oc}'',\tilde{v},\tilde{w}) = \frac{\tilde{\varphi}_0(\sigma,\tilde{v})}{x_{\oc}^2} + 
\frac{\tilde{\varphi}_1(x_{\oc},\sigma,y_{\oc},y_{\oc}'',\tilde{v},\tilde{w})}{x_{\oc}},
\end{align*}
both parametrize $\mathcal{L}_{\oc}$ non-degenerately in the sense of Definition~\ref{defn: 1c-1c parametrization} near 
\begin{align*}
q = (0,\sigma_0,y_{\oc,0},y_{\oc,0}'', d_{x_{\oc},\sigma,y_{\oc},y_{\oc}''}\Phi_{\oc-\oc}(q')) =
(0,\sigma_0,y_{\oc,0},y_{\oc,0}'', d_{x_{\oc},\sigma,y_{\oc},y_{\oc}''}\tilde{\Phi}_{\oc-\oc}(\tilde{q}')),
\end{align*}
where
\begin{align*}
q'= (0,\sigma_0,y_{\oc,0},y_{\oc,0}'',v_0,w_0) \in C_{\Phi_{\oc-\oc}},
\; q'= (0,\sigma_0,y_{\oc,0},y_{\oc,0}'',\tilde{v}_0,\tilde{w}_0) \in C_{\tilde{\Phi}_{\oc-\oc}}.
\end{align*}
If in addition 
\begin{align} \label{eq: v tilde v, w tilde w, dimension equal}
\dim v = \dim \tilde{v} =k_0, \quad \dim w = \dim \tilde{w} = k_1,
\end{align}
and 
\begin{align} \label{eq: signature assumptions}
\sgn \; d_{vv}^2\varphi_0(\sigma_0,v_0) =  \sgn \; d_{\tilde{v}\tilde{v}}^2\tilde{\varphi}_0(\sigma_0,\tilde{v}_0),
\; \sgn \; d_{ww}^2\varphi_1(q') =  \sgn \; d_{\tilde{w}\tilde{w}}^2\tilde{\varphi}_1(\tilde{q}'),
\end{align}
where $\sgn$ denotes the signature of quadratic forms. Then $\Phi,\tilde{\Phi}_{\oc-\oc}$ are equivalent near $q$, in the sense that there exists a coordinate change 
\begin{align} \label{eq:lemma-A-2-coordinate-change}
\tilde{V}(x_{\oc},\sigma,y_{\oc},y_{\oc}'',v,w),
\tilde{W}(x_{\oc},\sigma,y_{\oc},y_{\oc}'',v,w),
\end{align}
sending $q'$ to $\tilde{q}'$, such that
\begin{align*}
\tilde{\Phi}_{\oc-\oc}(x_{\oc},\sigma,y_{\oc},y_{\oc}'',\tilde{V}(x_{\oc},\sigma,y_{\oc},y_{\oc}'',v,w),\tilde{W}(x_{\oc},\sigma,y_{\oc},y_{\oc}'',v,w))
=\Phi(x_{\oc},\sigma,y_{\oc},y_{\oc}'',v,w).
\end{align*}
\end{lmm}

\begin{proof}
Similar to \cite[Lemma~4.5]{hassell1999spectral} and \cite[Theorem~3.1.6]{FIO1}, the first step is to pull-back $\tilde{\Phi}_{\oc-\oc}$ by a fiber preserving diffeomorphism to make it equivalent to a phase function that coincide with $\Phi$ up to second order on $\mathcal{L}_{\oc}$.

We denote base variables by:
\begin{align*}
\mathsf{X} = (x_{\oc},\sigma,y_{\oc},y_{\oc}'').
\end{align*}
%%%%%%%%%%%%%%%%%%%%%%% The first part, second order approximation.
Then we consider the map
\begin{align*}
(\mathsf{X},\tilde{v},{\tilde{w}}) \rightarrow 
(\mathsf{X},d_{\sigma}(\tilde{\varphi}_0+x_{\oc}\tilde{\varphi}_1),d_{\msf{X}}\tilde{\varphi}_1,\partial_{\tilde{v}}\tilde{\varphi}_0(\sigma,{\tilde{v}}),\partial_{\tilde{w}}\tilde{\varphi}_1(\mathsf{X},{\tilde{v}},{\tilde{w}})) ,
\end{align*}
where \eqref{eq:1c1c-1form-1}

By the non-degenerate condition at $q'$, the derivative of $(d_{\sigma}\tilde{\varphi}_0,\partial_{\tilde{v}}\tilde{\varphi}_0(\sigma,{\tilde{v}}))$ with respect to ${\tilde{v}}$ 
and the derivative of $(d_{\mathsf{X}}\tilde{\varphi}_1,\partial_{\tilde{w}}\tilde{\varphi}_1(\mathsf{X},{\tilde{v}},{\tilde{w}}))$ with respect to ${\tilde{w}}$ have full rank, hence by the implicit function theorem (applied to two parts respectively), there is are maps $(\tilde{V},\tilde{W})$ (abusing the same notation as in \eqref{eq:lemma-A-2-coordinate-change}) such that
\begin{align} \label{eq:tildeV-tildeW-def2}
{\tilde{v}} = \tilde{V}(\mathsf{X},d_{\sigma}\tilde{\varphi}_0(\sigma,{\tilde{v}}),\partial_{\tilde{v}}\tilde{\varphi}_0(\sigma,{\tilde{v}})),
\; {\tilde{w}} = \tilde{W}(\mathsf{X},d_{\mathsf{X}}\tilde{\varphi}_1(\msf{X},{\tilde{v}},{\tilde{w}}),\partial_{\tilde{w}}\tilde{\varphi}_1(\mathsf{X},{\tilde{v}},{\tilde{w}})).
\end{align}

Now consider following map:
\begin{align*}
\msf{G}: (\msf{X},v,w) \rightarrow &
(\msf{X},\tilde{V}(\mathsf{X},d_{\sigma}\varphi_0(\sigma,v),\partial_{v}\varphi_0(\sigma,v)) + 
A_1 \partial_v\varphi_0(\sigma,v),
\\ & \tilde{W}(\mathsf{X},d_{\mathsf{X}}\varphi_1(\msf{X},v,w),\partial_{w}\tilde{\varphi}_1(\mathsf{X},v,w)) + A_2 \partial_{w}\varphi_1(\msf{X},v,w) ),
\end{align*}
with $A_1,A_2$ being linear maps that we determine next.
It is straightforward that $\msf{G}$ is fiber preserving. For it to be a diffeomorphism, we only need to choose $A_1,A_2$ so that its differential in $(v,w)$ for fixed $\msf{X}$ to be non-degenerate. Such $A_1,A_2$ can be chosen in the same way as in the argument after \cite[Equation~(3.1.6)]{FIO1}.
%Here we are substituting parameters on $\Phi$ into $\tilde{\Phi}$ and this is justified by the assumption \eqref{eq: v tilde v, w tilde w, dimension equal}.
Then
\begin{align*}
\Psi_{\oc-\oc}:=\msf{G}^*\tilde{\Phi}_{\oc-\oc}= \frac{\psi_0(\sigma,v)}{x_{\oc}^2}+\frac{\psi_1(\msf{X},v,w)}{x_{\oc}},
\end{align*}
which is equivalent to $\tilde{\Phi}_{\oc-\oc}$ by definition,
is the desired phase function parametrizing $\mathcal{L}_{\oc}$, and approximates $\Phi_{\oc-\oc}$ to second order (in terms of $|\partial_v \varphi_0|,|\partial_w \varphi_1|$) on $C_{\Phi}$.

%%%%%%%%%%%%%%%%%%%%%%%%%%%%%%%%%%% The second part, full equivalence
Let 
\begin{align*}
\psi = x_{\oc}^2\Psi_{\oc-\oc}= \psi_0(\sigma,v)+x_{\oc}\psi_1(\msf{X},v,w),
\quad \varphi = x_{\oc}^2\Phi_{\oc-\oc} = \varphi_0(\sigma,v)+x_{\oc}\varphi_1(\msf{X},v,w).
\end{align*}
By \cite[Proposition~5]{melrose1996scattering} i.e., the equivalence between phase functions of Legendrians on manifolds with boundary, which applies the proof of \cite[Theorem~3.1.6]{FIO1}, but without the homogeneity in $v$ in both the assumption and the conclusion, and the first part of \eqref{eq: signature assumptions}, we know there is a change of variable $(\sigma,v) \to (\sigma,V'(\sigma,v))$ such that
\begin{align*}
\psi_0(\sigma,V'(\sigma,v)) = \varphi_0(\sigma,v),
\end{align*}
and
\begin{align} \label{eq: psi=phi1+O(x)}
\psi(\msf{X},V'(\sigma,v),w) = \varphi_0(y',v) + O(x_{\oc}).
\end{align}

The goal of the rest of the proof is to further find a fiber preserving diffeomorphism so that after a pull-back, the $\varphi_1,\psi_1-$part are equal as well. By (\ref{eq: psi=phi1+O(x)}) and $\varphi-\psi$ vanishes to second order on $C_{\Phi}$, we have
\begin{align} \label{eq: psi - phi = O(x|dphi|^2) }
\psi - \varphi 
= \frac{x_{\oc}}{2}
\big(\sum_{1 \leq j,k \leq k_0} \frac{\partial \varphi_0}{\partial v_j}\frac{\partial \varphi_0}{\partial v_k}b^{00}_{jk}
+ 2 \sum_{ \substack{1 \leq j \leq k_0
\\1 \leq k \leq k_1} } \frac{\partial \varphi_0}{\partial v_j}\frac{\partial \varphi_1}{\partial w_k}b^{01}_{jk}
+ \sum_{1 \leq j,k \leq k_1} \frac{\partial \varphi_1}{\partial v_j}\frac{\partial \varphi_1}{\partial v_k}b^{11}_{jk} \big),
\end{align}
with $b^{ab}_{jk}$ being smooth functions of $(\msf{X},v,w)$.

We consider the change of variables of the form:
\begin{align} \label{eq: tilde v,w change}
\begin{split}
& \tilde{v}_j = v_j + x_{\oc} \sum_{1 \leq k \leq k_0} \frac{ \partial \varphi_0 }{\partial v_k}a_{jk}^{00} + x_{\oc} \sum_{1 \leq k \leq k_1} \frac{ \partial \varphi_1}{\partial w_k}a_{jk}^{01},\\
& \tilde{w}_j = w_j 
+  \sum_{1 \leq k \leq k_0} \frac{ \partial \varphi_0 }{\partial v_k}a_{jk}^{10} + \sum_{1 \leq k \leq k_1} \frac{ \partial \varphi_1}{\partial w_k}a_{jk}^{11}.
\end{split}
\end{align}
with $a^{ab}_{ij}$ to be determined. Then by Taylor expansion, we have
\begin{align} \label{eq: phi(tilde v, tilde w)-phi, expansion}
\begin{split}
& \varphi(\msf{X},\tilde{v},\tilde{w})-\varphi(\msf{X},v,w)
=  \sum_{1 \leq j \leq k_0} (\tilde{v}_j-v_j) \frac{\partial \varphi_0}{\partial v_j} + x_{\oc} \sum_{1 \leq j \leq k_1} (\tilde{w}_j-w_j) \frac{\partial \varphi_1}{\partial w_k}
\\ &+  \sum_{1 \leq j,k \leq k_0}(\tilde{v}_j-v_j)(\tilde{v}_k-v_k)c^{00}_{jk}
 + \sum_{ \substack{1 \leq j \leq k_0
\\1 \leq k \leq k_1} } (\tilde{v}_j-v_j)(\tilde{w}_k-w_k) c^{01}_{jk}
\\ & + \sum_{1 \leq j,k \leq k_1}(\tilde{w}_j-w_j)(\tilde{w}_k-w_k)c^{11}_{jk},
\end{split}
\end{align}
for some smooth function $c^{ab}_{jk}$. Now substituting (\ref{eq: tilde v,w change}) into (\ref{eq: phi(tilde v, tilde w)-phi, expansion}), and 
combining with (\ref{eq: psi - phi = O(x|dphi|^2) }), we know that requiring 
\begin{align*}
\psi(\msf{X},V'(v,\sigma),w) = \varphi(\msf{X},\tilde{v},\tilde{w}),
\end{align*}
is equivalent to matching coefficients of $\partial_v\varphi_0,\partial_w\varphi_1$ after this substitution, which turns out to be a equation of the form
\begin{align} \label{eq: matrix equation, A,B,C}
\msf{A} = \msf{B} + \msf{Q}(\msf{A},\msf{C},x_{\oc}),
\end{align}
where 
\begin{align*}
\msf{A} = \begin{pmatrix} 
a^{00}  & a^{01} \\
a^{10} & a^{11} 
\end{pmatrix}, 
\msf{B} = \begin{pmatrix} 
b^{00}  & b^{01} \\
b^{10} & b^{11} 
\end{pmatrix}, \msf{C} = \begin{pmatrix} 
c^{00}  & c^{01} \\
c^{10} & c^{11} 
\end{pmatrix}.
\end{align*}
with $a^{ab}$ being matrices with entries $a^{ab}_{ij}$ in (\ref{eq: tilde v,w change}), similarly for $\msf{B}$ with $b^{ab}_{ij}$ from (\ref{eq: psi - phi = O(x|dphi|^2) }), and $\msf{C}$ with entries from (\ref{eq: phi(tilde v, tilde w)-phi, expansion}). In addition, by direct computation, $\msf{Q}(\msf{A},\msf{C},x_{\oc})$ is polynomial in all variables and quadratic in entries of $\msf{A}$.

Consider the map
\begin{align*}
\msf{A} \rightarrow \msf{A} - \msf{Q}(\msf{A},\msf{C},x_{\oc}).
\end{align*}
It has non-degenerate derivative at $\msf{A}=0$. Hence by the inverse function theorem, for $\msf{B}$ sufficiently small, (\ref{eq: matrix equation, A,B,C}) has a unique solution $\msf{A}$. Thus if we can connect $\psi$ and $\varphi$ by a path of non-degenerate phase functions, then we can move along the path each time with small enough step to keep equivalence, and reach $\psi$ starting from $\varphi$ in finite step, hence proving the result. 

Now we consider a path of
\begin{align*}
\msf{B}_s = \begin{pmatrix} 
b^{00}_s  & b^{01}_s \\
b^{10}_s & b^{11}_s 
\end{pmatrix}, s \in [0,1],
\end{align*}
and define $\psi_s$ by (\ref{eq: psi - phi = O(x|dphi|^2) }) with $b^{ab}_{ij}$ replaced by entries of $\msf{B}_s$. Now we check that we can select $\msf{B}_s$ making $\psi_s$ is kept to be non-degenerate while connecting $\varphi$ and $\psi$, which is equivalent to keeping $\psi_s$ non-degenerate while letting the path $\msf{B}_s$ connecting $0$ and $\msf{B}$. 

Decomposing $\psi_s$ as (since the $\varphi_0-$part is not changed through this process)
\begin{align*}
\psi_s(\msf{X},V'(\sigma,v),w) = \varphi_0(\sigma,v) + x_{\oc} \psi_{s,1}(\msf{X},v,w),
\end{align*}
then the non-degeneracy condition is
\begin{align*}
d_{\sigma,v} \frac{\partial \varphi_0 }{\partial v_j}, 
\end{align*}
are linearly independent at $(\sigma_0,v_0)$;
and
\begin{align} \label{eq: psi s2 non-degenerate 2}
d_{y_{\oc},y_{\oc}'',w} \frac{\partial \psi_{s,1}}{\partial w_j},
\end{align}
are linearly independent at $q'$.

The first condition is always satisfied by the non-degeneracy of $\varphi$. For the second condition, we compute that on $C_{\Phi}$:
\begin{align*}
d_{ww}\psi_{s,1} = d^2_{ww} \varphi_1 + (d^2_{ww}\varphi_{1})(b^{11}_s)(d^{2}_{ww}\varphi_{1}),
\end{align*}
where we used the fact that $\partial_w\varphi_1 = 0$ on $C_{\Phi}$, hence only terms without any factors like $\frac{\partial \varphi_1}{\partial w_k}$ will have contribution. Similarly we have
\begin{align*}
d_{(y_{\oc},y_{\oc}'')w} \psi_{s,1} = d_{(y_{\oc},y_{\oc}'')w}\varphi_1
+ (d_{(y_{\oc},y_{\oc}'')w}\varphi_1)(b^{11}_s)(d_{(y_{\oc},y_{\oc}'')w}\varphi_1).
\end{align*}

So by the non-degeneracy condition on $\varphi_1$, we know that (\ref{eq: psi s2 non-degenerate 2})  has full rank if and only if
\begin{align*}
\Id + b^{11}_s d_{ww}^2\varphi_1
\end{align*}
is non-singular. As shown in \cite[Lemma~3.1.7]{FIO1}, we can connect $\msf{B}_0=0$ (corresponding to $\varphi$) and $\msf{B}_1=\msf{B}$ (corresponding to $\psi$) by such a path of $\msf{B}_s$ if and only if 
\begin{align*}
\sgn \; d^2_{ww}\varphi_1 = \sgn \; d^2_{ww}\psi_1,
\end{align*}
which is what we assumed, and the proof is completed.
\end{proof}

Another ingredient needed in the proof of Proposition \ref{prop: phase equivalence, general parameter number, signature} is the fact that the `degeneracy' in the Hessian $d_{vv}^2\varphi_0$ and $d_{ww}^2\varphi_1$ is an intrinsic quantity, which is essentially \cite[Theorem~3.1.4]{FIO1} in our setting of manifolds with boundary:
\begin{prop} \label{prop: N-rank Hessian = rank drop of the projection}
Let $\Phi_{\oc-\oc} = \frac{\varphi_0}{x_{\oc}^2}+\frac{\varphi_1}{x_{\oc}}$ be a non-degenerate phase function parametrizing $\mathcal{L}_{\oc}$ near $q$, with $v,w,k_0,k_1$ defined as above,  then we have
\begin{align} \label{eq: N-rank hessian = rank drop, 1} 
k_0 - \mathrm{rank} \; d^2_{vv} \varphi_0 = 1 - \mathrm{rank} \; d\pi_0,
\end{align}
where $\pi_0$ is restriction to $\mathcal{L}_{\oc}$ of the projection 
%$(\msf{X},\msf{X}^*) \rightarrow \sigma$, with $\msf{X}^*$ being the fiber part of $\ococb$.
\begin{align*}
\pi_0: \; (\sigma,d_{\sigma}(\Phi_{\oc-\oc})) \rightarrow \sigma.
\end{align*}
%where $\tau_\sigma\frac{d\sigma}{x_{\oc}^2}$ is the part in the canonical one form that is dual to $\sigma$ on $\ococb$, and $\tau_\sigma= \partial_{\sigma}(\Phi)$ on $\mathcal{L}_{\oc}$.

\begin{align} \label{eq: N-rank hessian = rank drop, 2} 
k_1 - \mathrm{rank} \; d^2_{ww} \varphi_1 = (2n-2) - \mathrm{rank} \; d\pi_1,
\end{align}
where $\pi_1$ is restriction to $\mathcal{L}_{\oc}$ of the projection 
\begin{align*}
\pi_1: \; (y_{\oc},y_{\oc}'',d_{y_{\oc},y_{\oc}''}(\frac{\varphi_1}{x_{\oc}})) \rightarrow (y_{\oc},y_{\oc}'').
\end{align*}

\end{prop}

\begin{proof}
%Let $F_{\oc}$ be the map from $X^2_b \times \R^{k_0} \times \R^{k_1}$ to $\mathcal{L}_{\oc}$ defined by
%\begin{align*}F_{\oc}: (\msf{X},v,w) \rightarrow (\msf{X},d_{\msf{X}}( \frac{\Phi}{x_{\oc}^2} )).\end{align*}

Let $F_0$ be the map defined by
\begin{align*}
F_0(\sigma,v) = (\sigma,d_\sigma(\frac{\varphi_0(\sigma,v)}{x_{\oc}^2})),
\end{align*}
then the image of its restriction to $C_{\varphi_0}$ is one dimensional.

Consider \eqref{eq: N-rank hessian = rank drop, 1} first. 
The right hand side of (\ref{eq: N-rank hessian = rank drop, 1}) is the dimension of the kernel of 
\begin{align}
d(\pi_0 \circ F_1|_{C_{\varphi_0}}).
\end{align}
This consists of vectors $V$ such that
\begin{align*}
d\sigma(V) = 0, \;  d_{\sigma,v}(\partial_v\varphi_0)(V) = 0,
\end{align*}
Or equivalently:
\begin{align*}
d\sigma(V) = 0, \; (\partial^2_{vv}\varphi_0)dv(V) = 0.
\end{align*}
The dimension of such forms is given by the left hand side of (\ref{eq: N-rank hessian = rank drop, 1}).

For the proof of (\ref{eq: N-rank hessian = rank drop, 2}), let $F_1$ be defined by 
\begin{align*}
F_1(y_{\oc},y_{\oc}'',w) =  (y_{\oc},y_{\oc}'',d_{y_{\oc},y_{\oc}''}(\frac{\varphi_1}{x_{\oc}})).
\end{align*}
Then its image restricted to $C_{\varphi_1}$ has dimension $2n-2$, and the right hand side of (\ref{eq: N-rank hessian = rank drop, 2}) is the dimension of the kernel of
\begin{align*}
d(\pi_1 \circ F_2|_{C_{\varphi_1}}),
\end{align*}
which consists of vectors such that
\begin{align*}
dy_{\oc}(V)=0, dy_{\oc}''(V)=0, \; d_{y_{\oc},y_{\oc}'',w}(\partial_w\varphi_1)(V) = 0.
\end{align*}
Or equivalently
\begin{align*}
dy_{\oc}(V)=0, dy_{\oc}''(V)=0, \; (\partial^2_{ww}\varphi_1)dw(V) = 0.
\end{align*}
This shows that this dimension coincides with the left hand side of \eqref{eq: N-rank hessian = rank drop, 2}, completing the proof.

\end{proof}

Now we return to the proof of Proposition \ref{prop: phase equivalence, general parameter number, signature}.

\begin{proof}[Proof of Proposition \ref{prop: phase equivalence, general parameter number, signature}]

First of all, we only need to prove it for non-degenerate phase functions.
For a clean parametrization, we have a splitting of parameters as in \eqref{eq:ococpara-splitting} and it can be viewed as a family of non-degenerate phase functions. If the conclusion for non-degenerate phase functions is already proven, then we can apply it to this family to reduce to an oscillatory integral using a fixed non-degenerate phase function. In addition, as one can check following the proof of Lemma~\ref{lemma: phase equivalence, same parameter number, signature} and the proof below, the diffeomorphism that pulls back a family of phase functions to a fixed one has smooth dependence on the parameter (parametrizing this family, i.e. the $\ococparatwo''$ part in \eqref{eq:ococpara-splitting}) if the original family depends on $\ococparatwo''$ smoothly.
Then for any two clean phase functions, we can apply such reduction respectively and then apply the result for non-degenerate phase functions again to conclude the proof.
%as the discussion after shows, we can view 

So it remains to justify the setting where both phase functions are non-degenerate.
If the dimension of parameters and signatures in $\Phi,\tilde{\Phi}_{\oc-\oc}$ already match as in Lemma \ref{lemma: phase equivalence, same parameter number, signature}, then Lemma \ref{lemma: phase equivalence, same parameter number, signature} shows that we can rewrite $A \in I^{m}_{\oc-\oc}$ with phase function $\Phi$ as a oscillatory integral with phase function $\tilde{\Phi}_{\oc-\oc}$, up to a Schwartz error $u_0$.

For the general case, we may modify $\Phi$ by adding terms of the form
\begin{align*}
\msf{Q}_1(\bar{v},\bar{v})+x_{\oc}\msf{Q}_2(\bar{w},\bar{w}),
\end{align*}
where $\msf{Q}_1,\msf{Q}_2$ are non-degenerate quadratic forms, and replace $v$ by $(v,\bar{v})$, $w$ by $(w,\bar{w})$; and similarly for $\tilde{\Phi}_{\oc-\oc}$.
Then they can be made to satisfy the matching conditions in Lemma \ref{lemma: phase equivalence, same parameter number, signature} if and only if the difference of the dimension of $v,w$ respective and the rank of Hessian of $\varphi_0,\varphi_1$ with respect to $v,w$ are the same as the corresponding quantities of $\tilde{\Phi}_{\oc-\oc}$ (or equivalently, the difference of the dimension of $v,w$ and the signatures are the same mod 2). But this equals to the rank drop of corresponding projections in Proposition \ref{prop: N-rank Hessian = rank drop of the projection}, which are independent of the choice of parametrizations.

When we take the effect of these extra quadratic forms into account, by a direct computation using Gaussian integral (extended to the complex case, with the same formula), each dimension of $\bar{v}$ gives a factor of $x_{\oc}^{-1}$, while each dimension of $\bar{w}$ gives a factor of $x_{\oc}^{-1/2}$. These are collected in the factor $x_{\oc}^{- \frac{2k_0+k_1}{2} }$.
\end{proof}

\section{Proof of results from Section~\ref{sec: 1c-ps FIO}}
\label{app:proofs}

We give details of some proofs in Section~\ref{sec: 1c-ps FIO} in this appendix.

\begin{proof}[Proof of Theorem~\ref{thm: PsiDO- 1c-ps FIO composition}]

By the phase equivalence proven in Proposition~\ref{prop:1c-ps-phase-equivalence} and discussion in Remark~\ref{rem:phase function leading term}, $A$ can locally be written as
\begin{align*}
A = (2\pi)^{-\frac{N}{2} } (\int e^{ i( \frac{-t}{x_{\oc}^2} + \frac{\varphi_1(t,z,y_{\oc},\theta)}{x_{\oc}} )} x_{\oc}^{-(m+\frac{2N_0+N_1}{2}-\frac{3}{4})}a(t,z,x_{\oc},y_{\oc},\theta) d\theta).
\end{align*}
Then we have, with $p_{\mathrm{full}}$ being the left full symbol of $P$, 
\begin{align*}
PA = & (2\pi)^{-\frac{N+2(n+1)}{2} }\int e^{i( (t-t')\tau + (z-z') \cdot \zeta  )} p_{\mathrm{full}}(t,z,\tau,\zeta) (\int e^{ i( \frac{-t'}{x_{\oc}^2} + \frac{\varphi_1(t',z',y_{\oc},\theta)}{x_{\oc}} )} 
\\ &x_{\oc}^{-(m+\frac{2N_0+N_1}{2}-\frac{3}{4})}
a(t',z',x_{\oc},y_{\oc},\theta) d\theta) dt'dz'd\tau d\zeta.
\end{align*}
Choose $\chi \in C_c^\infty(\R)$ such that $\chi = 1$ on $[-C,C]$ for some large $C$, with 
\begin{align}
C > \max |\varphi_1|.
\end{align}

decompose $PA$ as
\begin{align} \label{eq: PA decomposition}
\begin{split}
PA = &P\tilde{A}+R,
\end{split}
\end{align} 
where 
\begin{align*}
P\tilde{A}= & (2\pi)^{-\frac{2n+1+N+2(n+1)}{2} }\int e^{i( (t-t')\tau + (z-z') \cdot \zeta  )} p_{\mathrm{full}}(t,z,\tau,\zeta) (\int e^{ i( \frac{-t'}{x_{\oc}^2} + \frac{\varphi_1(t',z',y_{\oc},\theta)}{x_{\oc}} )} 
\\ & x_{\oc}^{-(m+\frac{2N_0+N_1}{2}-\frac{3}{4})}\chi(|\tau| x_{\oc}^2)\chi(|\zeta|x_{\oc})a(t',z',x_{\oc},y_{\oc},\theta) d\theta) dt'dz'd\tau d\zeta,\\
R = &(2\pi)^{-\frac{2n+1+N+2(n+1)}{2} } \int e^{i( (t-t')\tau + (z-z') \cdot \zeta  )} p_{\mathrm{full}}(t,z,\tau,\zeta) (\int e^{ i( \frac{-t'}{x_{\oc}^2} + \frac{\varphi_1(t',z',y_{\oc},\theta)}{x_{\oc}} )} 
\\ & x_{\oc}^{-(m+\frac{2N_0+N_1}{2}-\frac{3}{4})}(1-\chi(\tau x_{\oc}^2)\chi(|\zeta|x_{\oc}))a(t',z',x_{\oc},y_{\oc},\theta) d\theta) dt'dz'd\tau d\zeta.
\end{align*}

Consider the part of the phase involving $(t',z')$, which is 
\begin{align*}
-t'(1+\tau x_{\oc}^2)x_{\oc}^{-2}+ (\varphi_1(t',z',y_{\oc},\theta)-z' \cdot (x_{\oc}\zeta) )x_{\oc}^{-1}.
\end{align*}
By the choice of $\chi$, we know that on $\supp(1-\chi(\tau x_{\oc}^2)\chi(|\zeta|x_{\oc}))$ there is no critical point with respect to $(t',z')$, hence a stationary phase argument shows that we can find constant $C_N$ such that
\begin{align}
|R| \leq C_Nx_{\oc}^N,
\end{align}
for any $N$. Hence up to an residual error, we only need to consider the first term on the right hand side of (\ref{eq: PA decomposition}). 
Introducing
\begin{align}  \label{eq: Defition, tilde tau, tilde zeta 2}
\tilde{{\tau}} = x_{\oc}^2 \tau , \tilde{{\zeta}} = x_{\oc}\zeta ,
\end{align} 
%which is equivalent to \eqref{eq: Definition, tilde tau, tilde zeta 1} and they 
which are in a compact region on the support of the amplitude in $P\tilde{A}$, and we may rewrite $P\tilde{A}$ as
\begin{align*}
P\tilde{A}= &(2\pi)^{-\frac{2n+1+N+2(n+1)}{2} }  \int e^{i( (t-t')\frac{\tilde{{\tau}}}{x_{\oc}^2} + (z-z') \cdot \frac{\tilde{{\zeta}}}{x_{\oc}}  )} (\int e^{ i( \frac{-t'}{x_{\oc}^2} + \frac{\varphi_1(t',z',y_{\oc},\theta)}{x_{\oc}} )} x_{\oc}^{-(m+\frac{2N_0+N_1}{2}-\frac{3}{4})}
\\ & \chi(\tilde{{\tau}})\chi(|\tilde{{\zeta}}|)p_{\mathrm{full}}(t,z,\tau,\zeta)  a(t',z',x_{\oc},y_{\oc},\theta) d\theta) dt'dz' x_{\oc}^{-(n+2)}d\tilde{{\tau}} d\tilde{{\zeta}},
\end{align*}
which equals to
\begin{align} \label{eq: PtildeA, standard form}
\begin{split}
P\tilde{A}= &(2\pi)^{-\frac{2n+1+N+2(n+1)}{2} }  \int e^{i( \frac{(t\tilde{{\tau}} - t'(1+\tilde{{\tau}}) }{x_{\oc}^2} +  \frac{(z-z') \cdot \tilde{{\zeta}}+\varphi_1(t',z',y_{\oc},\theta)}{x_{\oc}})}  x_{\oc}^{-(m+\frac{2(N_0+2)+ (N_1+2n) }{2}-\frac{3}{4})}
\\ & \chi(\tilde{{\tau}})\chi(|\tilde{{\zeta}}|)p_{\mathrm{full}}(t,z,x_{\oc}^{-2}\tilde{{\tau}},x_{\oc}^{-1}\tilde{{\zeta}})  a(t',z',x_{\oc},y_{\oc},\theta) d\theta) dt'dz' d\tilde{{\tau}} d\tilde{{\zeta}},
\end{split}
\end{align}
which is in the form of (\ref{eq: 1c-ps FIO definition}), with
\begin{align} \label{eq: parameter tilde theta0, theta1}
\tilde{\theta}_0 = (t',\tilde{{\tau}},\theta_0), \tilde{\theta}_1=(z',\tilde{{\zeta}},\theta_1),
\end{align}
except that in the variables of $p$ we still have $x_{\oc}^{-1},x_{\oc}^{-2}-$factors, which will be dealt with below. And the phase function is
\begin{align} \label{eq: phase function, in PA composition}
\tilde{\Phi}(t,z,x_{\oc},y_{\oc},\tilde{\theta}_0,\tilde{\theta}_1) = 
\frac{(t\tilde{{\tau}} - t'(1+\tilde{{\tau}})) }{x_{\oc}^2} +  \frac{(z-z') \cdot \tilde{{\zeta}}+\varphi_1(t',z',y_{\oc},\theta)}{x_{\oc}} ,
\end{align}
with $\mathcal{L}$ defined by
\begin{align*}
\partial_{\tilde{\theta}_0,\tilde{\theta}_1}\tilde{\Phi}=0, \quad \tau dt+\zeta\cdot dz = d_{(t,z)}\tilde{\Phi}, \quad
\xi_{\oc}\frac{dx_{\oc}}{x_{\oc}^3}+\eta_{\oc}\cdot \frac{dy_{\oc}}{x_{\oc}} = d_{x_{\oc},y_{\oc}}\tilde{\Phi},
\end{align*}
or more concretely:
\begin{align} \label{eq: definition of the Lagrangian, in terms of phase in PA}
\begin{split}
& t=t',z=z', d_{\theta}\varphi_1(t,z,y_{\oc},\theta)= 0,\quad -(1+\tilde{{\tau}})+x_{\oc}\partial_{t}\varphi_1(t,z,y_{\oc},\theta) = 0, 
\tilde{{\zeta}} = \partial_{z}\varphi_1(t,z,y_{\oc},\theta),\\
&   \xi_{\oc} = (-2t+x_{\oc}\partial_{\oc}\varphi_1(t,z,y_{\oc},\theta)), \eta_{\oc} = \partial_{y_{\oc}}\varphi_1(t,z,y_{\oc},\theta).
\end{split}
\end{align}

Now write
\begin{align*}
p_{\mathrm{full}}(t,z,x_{\oc}^{-2}\tilde{{\tau}},x_{\oc}^{-1}\tilde{{\zeta}}) = p_{\mathrm{hom}}(t,z,x_{\oc}^{-2}\tilde{{\tau}},x_{\oc}^{-1}\tilde{{\zeta}})+r(t,z,x_{\oc}^{-1}\tilde{{\zeta}}),
\end{align*}
%=\tau+|\zeta|^2_{g^{-1}}
where $p_{\mathrm{hom}}(t,z,\tau,\zeta)$ is the parabolically homogeneous part. We have
\begin{align*}
p_{\mathrm{full}}(t,z,x_{\oc}^{-2}\tilde{{\tau}},x_{\oc}^{-1}\tilde{{\zeta}}) = x_{\oc}^{-m'}p_{\mathrm{hom}}(t,z,\tilde{{\tau}},\tilde{{\zeta}}) + x_{\oc}^{-(m'-1)} (x_{\oc}^{m'-1}r(t,z,x_{\oc}^{-1}\tilde{{\zeta}})),
\end{align*}
and
\begin{align*}
\tilde{r}(t,z,x_{\oc},\tilde{{\zeta}}):= x_{\oc}^{m'-1}r(t,z,x_{\oc}^{-1}\tilde{{\zeta}}).
\end{align*}
By the symbolic property of $r$ and since on the support of the amplitude $|\zeta|$ is comparable to $x_{\oc}^{-1}$, we have
\begin{align*}
|\partial_{\tilde{{\zeta}}}^\alpha \tilde{r}(t,z,x_{\oc},\tilde{{\zeta}})| =  &
 |x_{\oc}^{m'-1-|\alpha|}\partial^\alpha_\zeta r(t,z,x_{\oc}^{-1}\tilde{{\zeta}})|
 \\ \leq & Cx_{\oc}^{m'-1-|\alpha|} |x_{\oc}^{-1}\tilde{{\zeta}}|^{m'-1-|\alpha|}
 \\ \leq & C,
\end{align*}
for some constant $C$ depending on $\alpha$, which proves smoothness with respect to $\tilde{{\zeta}}$. Next we verify the conormal property in $x_{\oc}$:
\begin{align*}
x_{\oc}\partial_{x_{\oc}}\tilde{r} = x_{\oc}^{m'-1}r(t,z,x_{\oc}^{-1}\tilde{{\zeta}})
- x_{\oc}^{m'} \times x_{\oc}^{-2}\tilde{{\zeta}} \cdot \partial_{\zeta}r(t,z,x_{\oc}^{-1}\tilde{{\zeta}}).
\end{align*}
Since $r$ is a symbol of order $m'-1$, the first term is bounded by a constant multiple of 
\begin{align*}
x_{\oc}^{m'-1} |x_{\oc}^{-1}\tilde{{\zeta}}|^{m'-1} = O(1).
\end{align*}
The second term can be bounded by 
\begin{align*}
x_{\oc}^{m'-2} |\tilde{{\zeta}} \cdot \partial_{\zeta}r(t,z,x_{\oc}^{-1}\tilde{{\zeta}})|
\leq C x_{\oc}^{m'-2} |x_{\oc}^{-1}\tilde{{\zeta}}|^{m'-1-1} \leq C,
\end{align*}
and similarly for higher powers of $x_{\oc}\partial_{x_{\oc}}$. To summarize, (\ref{eq: PtildeA, standard form}) can be written as

\begin{align} 
\begin{split}
P\tilde{A}= I_1+I_2,
\end{split}
\end{align}
where

\begin{align} \label{eq: I1, I2}
\begin{split}
I_1 = & (2\pi)^{-\frac{2n+1+N+2(n+1)}{2} } \int e^{i( \frac{(t\tilde{{\tau}} - t'(1+\tilde{{\tau}}) }{x_{\oc}^2} +  \frac{(z-z') \cdot \tilde{{\zeta}}+\varphi_1(t',z',y_{\oc},\theta)}{x_{\oc}}  )}  x_{\oc}^{-(m+m'+\frac{2(N_0+2)+ (N_1+2n) }{2}-\frac{3}{4})}
\\ & \chi(\tilde{{\tau}})\chi(|\tilde{{\zeta}}|)p_{\mathrm{hom}}(t,z,\tilde{{\tau}},\tilde{{\zeta}})  a(t',z',x_{\oc},y_{\oc},\theta) d\theta) dt'dz' d\tilde{{\tau}} d\tilde{{\zeta}} \in 
I_{\oc-\ps}^{m+m'}
,
\\ I_2 = & (2\pi)^{-\frac{2n+1+N+2(n+1)}{2} }  \int e^{i( \frac{(t\tilde{{\tau}} - t'(1+\tilde{{\tau}}) }{x_{\oc}^2} +  \frac{(z-z') \cdot \tilde{{\zeta}}+\varphi_1(t',z',y_{\oc},\theta)}{x_{\oc}}  )}  x_{\oc}^{-(m+m'-1+\frac{2(N_0+2)+ (N_1+2n) }{2}-\frac{3}{4})}
\\ & \chi(\tilde{{\tau}})\chi(|\tilde{{\zeta}}|)
\tilde{r}(t,z,x_{\oc},\tilde{{\zeta}})
  a(t',z',x_{\oc},y_{\oc},\theta) d\theta) dt'dz' d\tilde{{\tau}} d\tilde{{\zeta}} \in I_{\oc-\ps}^{m+m'-1}.
\end{split}
\end{align}

%Applying Corollary \ref{coro: modify the amplitude}, at 
Notice that
\begin{align*}
e^{i( \frac{(t\tilde{{\tau}} - t'(1+\tilde{{\tau}}) }{x_{\oc}^2} +  \frac{(z-z') \cdot \tilde{{\zeta}}+\varphi_1(t',z',y_{\oc},\theta)}{x_{\oc}}  )} 
= e^{i ( -\frac{t}{x_{\oc}^2}+ \frac{\varphi_1(t,z,y_{\oc},\theta)}{x_{\oc}})}
e^{i( \frac{(t - t')(1+\tilde{{\tau}}) }{x_{\oc}^2} +  \frac{(z-z') \cdot \tilde{{\zeta}}-
(\varphi_1(t,z,y_{\oc},\theta)-\varphi_1(t',z',y_{\oc},\theta))}{x_{\oc}})} .
\end{align*}
So we may rewrite $I_1$ as
\begin{align*}
I_1 = & (2\pi)^{-\frac{2n+1+N+2(n+1)}{2} } \int e^{i ( -\frac{t}{x_{\oc}^2}+ \frac{\varphi_1(t,z,y_{\oc},\theta)}{x_{\oc}})}
   x_{\oc}^{-(m+2+\frac{2(N_0+2)+ (N_1+2n) }{2}-\frac{3}{4})}
\\ & 
e^{i( \frac{(t - t')(1+\tilde{{\tau}}) }{x_{\oc}^2} +  \frac{(z-z') \cdot \tilde{{\zeta}}-
(\varphi_1(t,z,y_{\oc},\theta)-\varphi_1(t',z',y_{\oc},\theta))}{x_{\oc}})}
\chi(\tilde{{\tau}})\chi(|\tilde{{\zeta}}|)p_{\mathrm{hom}}(t,z,\tilde{{\tau}},\tilde{{\zeta}})  
\\&a(t',z',x_{\oc},y_{\oc},\theta) d\theta) dt'dz' d\tilde{{\tau}} d\tilde{{\zeta}}.
\end{align*}
Consider the integral in $t',z',\tilde{{\tau}},\tilde{{\zeta}}$, and apply the stationary phase lemma (see \cite[Theorem~7.7.5]{hormanderbookvolI}) to the integral:
\begin{align} \label{eq: PA, before t tau reduction}
\begin{split}
& \int e^{i( \frac{(t - t')(1+\tilde{{\tau}}) }{x_{\oc}^2} +  \frac{(z-z') \cdot \tilde{{\zeta}}-(\varphi_1(t,z,y_{\oc},\theta)-\varphi_1(t',z',y_{\oc},\theta))}{x_{\oc}})}
\chi(\tilde{{\tau}})\chi(|\tilde{{\zeta}}|)p_{\mathrm{hom}}(t,z,\tilde{{\tau}},\tilde{{\zeta}})  
\\&a(t',z',x_{\oc},y_{\oc},\theta) dt'dz' d\tilde{{\tau}} d\tilde{{\zeta}}.
\end{split}
\end{align}
Because of the parabolic nature in the phase, we consider the $t',\tilde{{\tau}}-$integral first, with $x_{\oc}^{-2}$ being the large parameters and then consider the $z',\tilde{{\zeta}}-$integral with $x_{\oc}^{-1}$ being the large parameter. 
The critical point in terms of $t',\tilde{{\tau}}$ is given by 
\begin{align} \label{eq: PA, t tau critical point}
t'=t,\tilde{{\tau}}=-1+x_{\oc}\partial_t\varphi_1(t,z',y_{\oc},\theta).
\end{align}

The Hessian in terms of $t',\tilde{{\tau}}$ is
\begin{align}
\mathcal{H}_1:=\begin{pmatrix}
\partial_{t'}^2\varphi_1 & -1\\
-1 & 0
\end{pmatrix},
\end{align}
with determinant $-1$. Hence the leading order contribution is
\begin{align*}
& ( - (x_{\oc}^{-2}/2\pi i)^2 )^{-1/2}  \int e^{i(\frac{(z-z') \cdot \tilde{{\zeta}}-(\varphi_1(t,z,y_{\oc},\theta)-\varphi_1(t,z',y_{\oc},\theta))}{x_{\oc}})}
\chi(-1+x_{\oc}\partial_t\varphi_1(t)\chi(|\tilde{{\zeta}}|)
\\&p_{\mathrm{hom}}(t,z,-1+x_{\oc}\partial_t\varphi_1(t,z',y_{\oc},\theta),\tilde{{\zeta}})  a(t,z',x_{\oc},y_{\oc},\theta) dz'  d\tilde{{\zeta}}.
\end{align*}
Now apply the stationary phase lemma in $(z',\tilde{{\zeta}})$. The critical point is given by
\begin{align*}
z'=z, \tilde{{\zeta}} = \partial_{z}\varphi(t,z,y_{\oc},\theta).
\end{align*}
The Hessian in $(z',\tilde{{\zeta}})$ is
\begin{align}  \label{eq: PA, z,zeta Hessian}
\mathcal{H}_2:=
\begin{pmatrix}
\partial^2_{z'z'}\varphi_1 & -I_n\\
-I_n & 0
\end{pmatrix},
\end{align}
where $I_n$ is the $n-$th order identity matrix, and this has determinant $(-1)^{2n(2n-1)/2}=(-1)^{n}$. Hence the leading order term is
\begin{align*}
& ( 2\pi )^{2/2} x_{\oc}^{2}  \times  
( (-1)^n (x_{\oc}^{-1}/2\pi i)^{2n} )^{-1/2}
\chi(-1+x_{\oc}\partial_t\varphi_1(t,z,y_{\oc},\theta))\chi(|\partial_z\varphi_1|)
\\&p_{\mathrm{hom}}(t,z,-1+x_{\oc}\partial_t\varphi_1(t,z,y_{\oc},\theta),\tilde{{\zeta}})  a(t,z,x_{\oc},y_{\oc},\theta) .
\end{align*}
Notice that
\begin{align*}
\chi(-1+x_{\oc}\partial_t\varphi_1(t,z,y_{\oc},\theta))=\chi(|\partial_z\varphi_1(t,z,y_{\oc},\theta)|)=1
\end{align*}
on the region we are concerning, hence this equals to
\begin{align*}
(2\pi)^{ \frac{2+2n}{2} }x_{\oc}^{ \frac{2\times2+2n}{2} }
p_{\mathrm{hom}}(t,z,-1+x_{\oc}\partial_t\varphi_1(t,z,y_{\oc},\theta),\tilde{{\zeta}})  a(t,z,x_{\oc},y_{\oc},\theta).
\end{align*}
Substitute this back to \eqref{eq: I1, I2} and we have
\begin{align} \label{eq: I1 final expression}
\begin{split}
I_1 = & (2\pi)^{-\frac{2n+1+N}{2} } \int e^{i( \frac{-t}{x_{\oc}^2} +  \frac{\varphi_1(t,z,x_{\oc},y_{\oc},\theta)}{x_{\oc}})}  x_{\oc}^{-(m+2+\frac{2N_0+N_1}{2}-\frac{3}{4})}
\\ &p_{\mathrm{hom}}(t,z,-1+x_{\oc}\partial_t\varphi_1(t,z,y_{\oc},\theta),\partial_z\varphi_1(t,z,y_{\oc},\theta))  a(t,z,x_{\oc},y_{\oc},\theta) d\theta.
\end{split}
\end{align}
Since $I_2$ has no contribution to the $I^{m+m'}_{\oc-\ps}$ order, this completes the proof.
\end{proof}

We now turn to the proof of Theorem~\ref{thm: vanishing principal symbol product}. We begin with a few preparatory results that help us to write the phase function in a convenient form. The first result shows that when the rank of the projection from $\Legps$ to $M$ drops, we can select some $z-$components, and duals of other $z-$components, together with $t,x_{\oc},y_{\oc}$, which are not involved in this rank drop, to parametrize $\Legps$. This allows us to put the phase function into a normal form, which will simplify the composition.

\begin{lmm}  \label{lemma: full rank projection}
Let $\tilde{\tau},\tilde{\zeta}$ be as in \eqref{eq: Defition, tilde tau, tilde zeta 2} and $L = \partial \Lambda_\pm$. There exists a linear change of coordinates in $z$ (with $\tilde \zeta$ transforming correspondingly) and $k \in \{ 1, \dots, n\}$ so that at $q_0 \in \Legps$ the projection
 \begin{align}\label{eq:fullrank}
    \Lambda_\pm \ni (t,z,\tilde{{\tau}},\tilde{\zeta},x_{\oc},y_{\oc},\xi_{\oc},\eta_{\oc}) \rightarrow      (t,z',\tilde{{\zeta}}'',x_{\oc},y_{\oc})
 \end{align}
   has full rank $2n+1$, where $z'=(z_1,...,z_k),\ \tilde{\zeta}''=(\tilde{\zeta}_{k+1},...,\tilde{\zeta}_n)$.
   Consequently, the projection has full rank in a neighbourhood of $q_0$.
\end{lmm}

\begin{proof}
Write $q_0$ in local coordinates as
\begin{align*}
q_0 = (t_0,z_0,\tilde{{\tau}}_0,|\zeta_0|^{-1}=0,\tilde{{\zeta}}_0,x_{\oc}=0,y_{\oc,0},\xi_{\oc,0},\eta_{\oc,0}),
\end{align*}
and use the notation
\begin{align}
L_{t_0}:=\Legps \cap  \{ t=t_0\} = \Lambda_\pm \cap \{ t = t_0, \xoc = 0 \}.
\end{align}
%Consider the conic extension of $\tilde{\mathcal{L}}_{t_0}$.

By assumption, $d\xoc$ and $dt$ are linearly independent on $\Lambda_\pm$ at $q_0$ 
(see Definition~\ref{defn: admissible 1c-ps Lagrangian submanifold and 1c-ps fibred-Legendre submanifold}). We need to find $2n-1$ coordinates that restrict to a coordinate system on $L_{t_0}$; then these together with $\xoc$ and $t$ will form a local coordinate system on $\Lambda_\pm$. 

From Proposition~\ref{prop:1c-ps contact symplectic property}, $L_{t_0}$ is a fibre of $\pi_L$, and is therefore a Lagrangian submanifold in $(\yoc, \etaoc, z, \tilde \zeta)$-space relative to the standard symplectic form $d\etaoc \wedge d\yoc + d\tilde \zeta \wedge dz$. As shown in the proof of \cite[Theorem 21.2.17]{hormander2007analysis}, one can find a splitting of the base coordinates $(\yoc', \yoc'', z', z'')$ so that, if $(\etaoc', \etaoc'', \tilde \zeta', \tilde \zeta'')$ are the dual coordinates, then $(\yoc', \etaoc'', z', \tilde \zeta'')$ furnish coordinates on $L_{t_0}$. 

On the other hand, we know that the differentials $dy_{\oc, 1}, \dots dy_{\oc, n-1}$ are linearly independent on $\Lambda_\pm$. In fact,  if we return to the definition \eqref{eq:microlocalized sojourn relns}, then we see that local coordinates can be chosen to be local coordinates on $W_\pm$ (where $q_\pm$ lives), together with the parameter along the rescaled flow. Since $dy_{\oc, 1}, \dots dy_{\oc, n-1}$ are linearly independent on $W_\pm$, they are therefore also linearly independent on $\Lambda_\pm$. This means that we can take $\yoc' = \yoc$ and can dispense with any $\etaoc''$ coordinates. This establishes the full rank of the map \eqref{eq:fullrank} at $q_0$. 

\end{proof}

%[two things to prove]: 1. invariance under change of phase function.  2. change a(t',z',...) to a(t,z,) gives another vanishing order. (t-t') ,(z-z')-> second order vanishing-> 2 order lower FIO.
Next we give a normal form of the phase function when we have the full rank projection obained above. This is analogous to \cite[Theorem~21.2.18]{hormander2007analysis}, which considered the case where $\Legps$ is conic.
\begin{lmm} \label{lemma: normal form of phase function}
   Let $\Lambda_\pm \subset \mathcal{M}$ be the (microlocalized) sojourn relations as in \eqref{eq:microlocalized sojourn relns}, and let $q_0 \in L = \partial \Lambda_\pm$. 
   a Lagrangian submanifold. Suppose near $q_0 \in \Legps$, coordinates have been chosen as in Lemma~\ref{lemma: full rank projection} such that the projection
   %and suppose that the coordinates on $\R^{n+1} \times \R^n$ is devided into two groups of the form $q=(q',q'')$, with their dual variables written as $k=(k',k'')$.  Suppose near $(q_0,k_1)$ the projection 
   \begin{align}
    \Lambda_\pm \ni (t,z,\tilde{\tau},\tilde{\zeta},x_{\oc},y_{\oc},\xi_{\oc},\eta_{\oc}) \rightarrow
     (t,z',\tilde{\zeta}'',x_{\oc},y_{\oc})
   \end{align}
   has full rank $2n+1$. Then we can take a phase function of the form
   \begin{align}
   \Phi = -\frac{t}{x_{\oc}^2}+\frac{z'' \cdot \tilde{\zeta}''-\tilde{\varphi}_1(t,z',\tilde{{\zeta}}'',x_{\oc},y_{\oc}) }{x_{\oc}},
   \end{align}
   to parametrize it near $q_0$, with $\tilde{\varphi}_1$ smooth in all of its arguments.  
\end{lmm}

\begin{proof} This is a special case of Proposition~\ref{prop: 1c-ps nondeg param} (with no $\etaoc''$ variables thanks to Lemma~\ref{lemma: full rank projection}), combined with Remark~\ref{rem:phase function leading term} for the $O(\xoc^{-2})$ term. 
\end{proof}

From now on, we assume that we order our components as in Lemma \ref{lemma: full rank projection} and choose phase function as in Lemma \ref{lemma: normal form of phase function}. 
By \eqref{eq:fullrank}, we may choose $(t,z',\tilde{{\zeta}}'',x_{\oc},y_{\oc})$ as coordinates on $\Lambda_\pm$, hence $(t,z',\tilde{{\zeta}}'',y_{\oc})$ as coordinates on $\Legps$. In terms of those coordinates, we have the following formula for the restriction of the rescaled Hamilton vector field to $\Legps$.
\begin{prop} \label{prop:Hp-restriction}
Suppose $p$ be the left full symbol of $P \in \Psi^{m',0}_{\ps}(\R^{n+1})$ and $\tilde{p}(t,z,\tilde{\tau},\tilde{\zeta}) = x_{\oc}^{m'}p(t,z,\tau,\zeta)$ is smooth,
then the restriction of $H_p$ to $\Lambda_\pm$ is given by
\begin{align} \label{eq: Hp, parameter form-interior}
\rHp  = \sum_{j'=1}^k\frac{\partial \tilde{p}}{\partial \tilde{\zeta}'_{j'}} \frac{\partial}{\partial z'_{j'}} 
- \sum_{j''=1}^{n-k}\frac{\partial \tilde{p}}{\partial z''_{j''}} \frac{\partial}{\partial \tilde{\zeta}''_{j''}} + x_{\oc}V,
\end{align}
where $V$ is a smooth vector field on $\Lambda_\pm$.
Consequently, the restriction of $H_p$ to $\Legps$ is given by
\begin{align} \label{eq: Hp, parameter form}
H_p^{2,0}  = \sum_{j'=1}^k\frac{\partial \tilde{p}}{\partial \tilde{\zeta}'_{j'}} \frac{\partial}{\partial z'_{j'}} 
- \sum_{j''=1}^{n-k}\frac{\partial \tilde{p}}{\partial z''_{j''}} \frac{\partial}{\partial \tilde{\zeta}''_{j''}}.
%= \sum_{j'=1}^k \tilde{q}_{j'} \frac{\partial}{\partial z'_{j'}} - \sum_{j''=1}^{n-k} \frac{\partial \tilde{p}}{\partial z''_{j''}} \frac{\partial}{\partial \tilde{\zeta}''_{j''}} .
\end{align}
\end{prop}

\begin{proof}
For $a(t,z,\tilde{\tau},\tilde{\zeta})$, restriction to $\mathcal{L}_{\tilde{\varphi}_1}$ is 
\begin{align*}
a(t,z',z''=\partial_{\tilde{\zeta}''}\tilde{\varphi}_1,\tilde{\tau} = -1-x_{\oc}\partial_t\tilde{\varphi}_1
,\tilde{\zeta}'=-\partial_{z'}\tilde{\varphi}_1,\tilde{\zeta}'').
\end{align*}

We use $\frac{d \cdot}{d \cdot}$ to indicate the total derivative under the dependence 
\begin{align*}
z''=\partial_{\tilde{\zeta}''}\tilde{\varphi}_1,
\quad \tilde{\zeta}'=-\partial_{z'}\tilde{\varphi}_1, \quad \tilde{\tau} = -1-x_{\oc}\partial_t\tilde{\varphi}_1.
\end{align*}
Notice that any term introduced by the chain rule due to the dependence $\tilde{\tau} = -1-x_{\oc}\partial_t\tilde{\varphi}_1$ is $O(x_{\oc})$.
Then \eqref{eq: Hp, parameter form-interior} without the $x_{\oc}V$-term applied to $a$ is:
\begin{align}  \label{eq: Hpa with dependence}
\begin{split}
 & \sum_{j'=1}^k \frac{\partial \tilde{p}}{\partial \tilde{\zeta}'_{j'}} \frac{d a}{d z'_{j'}} 
- \sum_{j''=1}^{n-k} \frac{\partial \tilde{p}}{\partial z''_{j''}} \frac{da}{d \tilde{\zeta}''_{j''}}
\\  = & \sum_{j'=1}^k \frac{\partial \tilde{p}}{\partial \tilde{\zeta}'_{j'}} 
(\frac{\partial a}{\partial z'_{j'}} + \sum_{l=1}^{n-k} \frac{\partial a}{\partial z''_l} \frac{\partial z''_l}{\partial z'_{j'}} + \sum_{l=1}^k \frac{\partial a}{\partial \tilde{\zeta}'_l} \frac{\partial \tilde{\zeta}'_l}{\partial z'_{j'} })
\\& - \sum_{j''=1}^{n-k} \frac{\partial \tilde{p}}{\partial z''_{j''}} (\frac{\partial a}{\partial \tilde{\zeta}''_{j''}}
+ \sum_{l=1}^{n-k} \frac{\partial a}{\partial z''_l} \frac{\partial z''_l}{\tilde{\zeta}''_{j''}} + \sum_{l=1}^k \frac{\partial a}{\partial \tilde{\zeta}'_l} \frac{\partial \tilde{\zeta}'_l}{\tilde{\zeta}''_{j''}}) + O(x_{\oc})
\\  = & \sum_{j'=1}^k \frac{\partial \tilde{p}}{\partial \tilde{\zeta}'_{j'}} 
(\frac{\partial a}{\partial z'_{j'}} + \sum_{l=1}^{n-k} \frac{\partial a}{\partial z''_l} \frac{\partial^2 \tilde{\varphi}_1}{\partial\tilde{\zeta}''_l\partial z'_{j'}} - \sum_{l=1}^k \frac{\partial a}{\partial \tilde{\zeta}'_l} \frac{\partial \tilde{\varphi}_1}{\partial z'_l \partial z'_{j'} })
\\ & - \sum_{j''=1}^{n-k} \frac{\partial \tilde{p}}{\partial z''_{j''}} (\frac{\partial a}{\partial \tilde{\zeta}''_{j''}}
+ \sum_{l=1}^{n-k} \frac{\partial a}{\partial z''_l} \frac{\partial \tilde{\varphi}_1}{ \partial \tilde{\zeta}''_l \partial \tilde{\zeta}''_{j''} } - \sum_{l=1}^k \frac{\partial a}{\partial \tilde{\zeta}'_l} \frac{\partial \tilde{\varphi}_1}{\tilde{z'}_l\tilde{\zeta}''_{j''}}) + O(x_{\oc}).
\end{split}
\end{align}

On the other hand, consider (without dependence of $z'',\tilde{\zeta}'$ in $z',\tilde{\zeta}''$ in differentiation)
\begin{align} \label{eq: Hpa, without dependence}
\begin{split}
H_p^{2,0}a = x_{\oc}\partial_ta + \sum_{j'=1}^k\frac{\partial \tilde{p}}{\partial \tilde{\zeta}'_{j'}} \frac{\partial a}{\partial z'_{j'}} 
+ \sum_{j''=1}^{n-k}\frac{\partial \tilde{p}}{\partial \tilde{\zeta}''_{j''}} \frac{\partial a}{\partial z''_{j''}}
-\sum_{j'=1}^k\frac{\partial \tilde{p}}{\partial z'_{j'}} \frac{\partial a}{\partial \tilde{\zeta}'_{j'}} 
- \sum_{j''=1}^{n-k}\frac{\partial \tilde{p}}{\partial z''_{j''}} \frac{\partial a}{\partial \tilde{\zeta}''_{j''}}.
\end{split}
\end{align}
Write $\tilde{p}$ as 
\begin{align} \label{eq: p pj qj decompostion}
\begin{split}
\tilde{p}(t,z,\tilde{\tau},\tilde{\zeta}) = & 
\tilde{p}_0(t,z,\tilde{\tau},\tilde{\zeta})(\tilde{\tau}+1+x_{\oc}\partial_t\varphi_1)+
\sum_{j''=1}^{n-k}\tilde{p}_{j''}(t,z,,\tilde{\tau},\tilde{\zeta}) (z_j''-\partial_{\tilde{\zeta}''_j}\tilde{\varphi}_1)
\\& + \sum_{j'=1}^k \tilde{q}_{j'}(t,z,\tilde{\tau},\tilde{\zeta})(\tilde{\zeta}'_{j'}+\partial_{z'_{j'}}\tilde{\varphi}_1).
\end{split}
\end{align}
Since factors of the form $(\tilde{\tau}+1+x_{\oc}\partial_t\varphi_1),(z_j''-\partial_{\tilde{\zeta}''_j}\tilde{\varphi}_1),(\tilde{\zeta}'_{j'}+\partial_{z'_{j'}}\tilde{\varphi}_1)$ vanishes on $\Lambda_\pm$, we only need to consider terms with derivatives hitting them instead of $\tilde{p}_0,\tilde{p}_{j''},\tilde{q}_{j'}$.
In addition, they only enter as $\tilde{\zeta}_{j''}''$ or $z'_{j'}$ derivatives, hence the contribution from the first term on the right hand side of \eqref{eq: p pj qj decompostion} is $O(x_{\oc})$.
%Since $\partial_{\tilde{\tau}}\tilde{p}$ is not involved, all 

Then restricted to $\Lambda_\pm$, we have
\begin{align*}
& \frac{\partial \tilde{p}}{\partial \zeta''_l} =
-\sum_{j''=1}^{n-k}\tilde{p}_{j''}\frac{\partial^2\tilde{\varphi}_1}{\partial \tilde{\zeta}''_j \partial \tilde{\zeta}''_l}
+ \sum_{j'=1}^kq_{j'}\frac{\partial^2 \tilde{\varphi}_1}{\partial z_j'\partial \tilde{\zeta}_l''}+O(x_{\oc}),  \quad  \frac{\partial \tilde{p}}{\partial z''_l} = p_l;
\\ &\frac{\partial \tilde{p}}{\partial z_l'} = -\sum_{j''=1}^{n-k}\frac{\partial \tilde{p}}{\partial z''_{j''}}\frac{\partial^2\tilde{\varphi}_1}{\partial \tilde{\zeta}_{j''} \partial z_l'}
+ \sum_{j'=1}^kq_{j'}\frac{\partial^2\tilde{\varphi}_1}{\partial z'_{j'} \partial z_l'} +O(x_{\oc}) , \quad \frac{\partial \tilde{p}}{\partial \tilde{\zeta}'_l} = q_l. 
\end{align*}
Substituting in (\ref{eq: Hpa, without dependence}) gives the last expression of (\ref{eq: Hpa with dependence}), which justifies \eqref{eq: Hp, parameter form-interior} and all $O(x_{\oc})$-terms are collected in $x_{\oc}V$ there.
\end{proof}

We are ready to prove Theorem~\ref{thm: vanishing principal symbol product}.

\begin{proof}[Proof of Theorem~\ref{thm: vanishing principal symbol product}]
Using the parametrization given in Lemma (\ref{lemma: normal form of phase function}), we have
  \begin{align*}
   \Phi = -\frac{t}{x_{\oc}^2}+\frac{z'' \cdot \tilde{\zeta}''-\tilde{\varphi}_1(t,z',\tilde{{\zeta}}'',x_{\oc},y_{\oc}) }{x_{\oc}},
   \end{align*}

When there is no confusion with indices of components, we denote svariables of $P$ by $t_1,z_1',z_1'',\tilde{\zeta}_1',\tilde{\zeta}_1''$, and variables of $A$ by $t_2,z_2',z_2'',\tilde{\zeta}_2',\tilde{\zeta}_2''$.
Then we have
\begin{align*}
A =&  (2\pi)^{ -\frac{2n+1+(n-k)}{2} }\int e^{i(-\frac{t_2}{x_{\oc}^2}+\frac{z''_2 \cdot \tilde{\zeta}''_2-\tilde{\varphi}_1(t_2,z'_2,\tilde{{\zeta}}''_2,x_{\oc},y_{\oc}) }{x_{\oc}})}
x_{\oc}^{-(m+\frac{n-k}{2}-\frac{3}{4})}
\\ &  a(t_2,z_2',z_2'',x_{\oc},y_{\oc},\tilde{\zeta}''_2)d\tilde{\zeta}''_2.
\end{align*}

Then we have (as in the proof of Theorem \ref{thm: PsiDO- 1c-ps FIO composition}, we only integrate over the region where $\tilde{\tau},|\tilde{\zeta}|$ are close to 1, modulo residual error)
\begin{align*}
PA = & (2\pi)^{-\frac{2n+1+(n-k)+2(n+1)}{2} }\int e^{i( \frac{(t_1-t_2)\tilde{\tau}_1 }{x_{\oc}^2} + \frac{(z_1-z_2) \cdot \tilde{\zeta}_1}{x_{\oc}} )} 
p(t_1,z_1,x_{\oc}^{-2}\tilde{\tau}_1,x_{\oc}^{-1}\tilde{\zeta}_1) 
\\&+ %r(t_1,z_1,x_{\oc}^{-2}\tilde{\tau}_1,x_{\oc}^{-1}\tilde{\zeta}_1) ) 
(\int e^{i(-\frac{t_2}{x_{\oc}^2}+\frac{z''_2 \cdot \tilde{\zeta}''_2-\tilde{\varphi}_1(t_2,z'_2,\tilde{{\zeta}}''_2,x_{\oc},y_{\oc}) }{x_{\oc}})}
x_{\oc}^{-(m+\frac{n-k}{2}-\frac{3}{4})}
a(t_2,z_2',z_2'',x_{\oc},y_{\oc},\tilde{\zeta}''_2)d\tilde{\zeta}''_2) dt_2dz_2\frac{d\tilde{\tau}_1 d\tilde{\zeta}_1}{x_{\oc}^{n+2}}.
\end{align*}
Recall that, because of the parabolic nature, when we applied the stationary phase method to (\ref{eq: PA, before t tau reduction}), $x_{\oc}^{-2}$ plays the role of large parameter and the lower order term are of order $O(x_{\oc}^2)$ compared with the leading order contribution from (\ref{eq: PA, t tau critical point}), hence modulo terms in $I^{m+m'-2}_{\oc-\ps}$, we may only consider the contribution from 
\begin{align} \label{eq: PA, t tau critical point, vanishing p case}
t_2=t_1,
\quad \tilde{\tau}_1=-1+x_{\oc}\partial_{t_1}\varphi_1(t_1,z_2,y_{\oc},\tilde{\zeta}_2''),
\end{align} 
and have (including the $(2\pi)x_{\oc}^2-$factor from the stationary phase expansion)
\begin{align*}
PA = & (2\pi)^{-\frac{2n+1+(n-k)+2n}{2} }\int e^{i(\frac{(z_1-z_2) \cdot \tilde{\zeta}_1}{x_{\oc}} )}  
 e^{i(-\frac{t_1}{x_{\oc}^2}+\frac{z''_2 \cdot \tilde{\zeta}''_2-\tilde{\varphi}_1(t_1,z'_2,\tilde{{\zeta}}''_2,x_{\oc},y_{\oc}) }{x_{\oc}})}
\\&x_{\oc}^{-(m+\frac{n-k}{2}-\frac{3}{4})}
p(t_1,z_1,x_{\oc}^{-2}(-1-x_{\oc}\partial_{t_1}\tilde{\varphi}_1),x_{\oc}^{-1}\tilde{\zeta}_1) 
\\&  
%+r(t_1,z_1,x_{\oc}^{-2}(-1-x_{\oc}\partial_{t_1}\tilde{\varphi}_1),x_{\oc}^{-1}\tilde{\zeta}) )
a(t_1,z_2',z_2'',x_{\oc},y_{\oc},\tilde{\zeta}''_2)d\tilde{\zeta}''_2 dz_2\frac{d\tilde{\zeta}_1}{x_{\oc}^{n}}.
\end{align*}

Recalling Lemma \ref{lemma: vanishing amplitude->lower order}, we may replace $a(t_2,z_2',z_2'',x_{\oc},y_{\oc},\tilde{\zeta}''_2)$ by
\begin{align*}
a_0(t_2,z'_2,x_{\oc},y_{\oc},\tilde{\zeta}_2''):=a(t_2,z'_2,\partial_{\tilde{\zeta}''}\tilde{\varphi}_1,x_{\oc},y_{\oc},\tilde{\zeta}_2''),
\end{align*}
without changing the leading order behaviour of $PA$, even in the case where $p$ has vanishing principal symbol on $\Legps$, since this vanishing of $a_0-a$ on $\Legps$ is on the top of vanishing of $p$. Thus we may write (modulo terms in $I^{m+m'-2}_{\oc-\ps}$):
\begin{align*}
PA = & (2\pi)^{-\frac{2n+1+(n-k)+2n}{2} }\int e^{i(\frac{(z_1-z_2) \cdot \tilde{\zeta}_1}{x_{\oc}} )}  
 e^{i(-\frac{t_1}{x_{\oc}^2}+\frac{z''_2 \cdot \tilde{\zeta}''_2-\tilde{\varphi}_1(t_1,z'_2,\tilde{{\zeta}}''_2,x_{\oc},y_{\oc}) }{x_{\oc}})}
\\&x_{\oc}^{-(m+\frac{n-k}{2}-\frac{3}{4})}
p(t_1,z_1,x_{\oc}^{-2}(-1-x_{\oc}\partial_{t_1}\tilde{\varphi}_1),x_{\oc}^{-1}\tilde{\zeta}_1) 
\\& 
%+ r(t_1,z_1,x_{\oc}^{-2}(-1-x_{\oc}\partial_{t_1}\tilde{\varphi}_1),x_{\oc}^{-1}\tilde{\zeta}_1) )
a_0(t_1,z'_2,x_{\oc},y_{\oc},\tilde{\zeta}_2'')d\tilde{\zeta}''_2 dz_2\frac{d\tilde{\zeta}_1}{x_{\oc}^{n}}.
\end{align*}
Notice that
\begin{align*}
e^{i(\frac{(z_1-z_2) \cdot \tilde{\zeta}_1 }{x_{\oc}} )} 
= e^{i(\frac{z_1 \cdot \tilde{\zeta}_1 - z_2' \cdot \tilde{\zeta}_1' - z_2'' \cdot \tilde{\zeta}_1'' }{x_{\oc}}) } ,
\end{align*}
the integral over $z_2''$ gives 
\begin{align*}
(2\pi)^{n-k}x_{\oc}^{n-k}\delta(\tilde{\zeta}_2''-\tilde{\zeta}_1''),
\end{align*}
and then we further integrate over $\tilde{\zeta}_1''$ (effectively evaluating $\tilde{\zeta}_1''$ to be $\tilde{\zeta}_2''$) to obtain
\begin{align*}
PA = & (2\pi)^{-\frac{2n+1+(n-k)+2k}{2} }\int 
 e^{i(-\frac{t_1}{x_{\oc}^2}+\frac{z_1'' \cdot \tilde{\zeta}_2''+z_1' \cdot \tilde{\zeta}_1'- z_2' \cdot \tilde{\zeta}_1' -\tilde{\varphi}_1(t_1,z'_2,\tilde{{\zeta}}''_2,x_{\oc},y_{\oc}) }{x_{\oc}})}
\\&x_{\oc}^{-(m+\frac{n-k}{2}-\frac{3}{4})}
p(t_1,z_1,x_{\oc}^{-2}(-1-x_{\oc}\partial_{t_1}\tilde{\varphi}_1),x_{\oc}^{-1}\tilde{\zeta}_1',x_{\oc}^{-1}\tilde{\zeta}_2'') 
\\& 
%+ r(t_1,z_1,x_{\oc}^{-2}(-1-x_{\oc}\partial_{t_1}\tilde{\varphi}_1),x_{\oc}^{-1}\tilde{\zeta}_1',x_{\oc}^{-1}\tilde{\zeta}_2'') )
a_0(t_1,z'_2,x_{\oc},y_{\oc},\tilde{\zeta}_2'')d\tilde{\zeta}''_2 dz_2'\frac{d\tilde{\zeta}_1'}{x_{\oc}^{k}}.
\end{align*}

Using the coordinate system in Lemma \ref{lemma: full rank projection} and 
Lemma \ref{lemma: normal form of phase function} means that we can write 
\begin{align}  \label{eq: Z'', Xi' definition}
z'' = Z'':= \partial_{\tilde{\zeta}''}\tilde{\varphi}_1, \; \tilde{\zeta}' = \Xi':= -\partial_{z'}\tilde{\varphi}_1 , \; \tilde{\tau} = \tilde{T}:= -1 - x_{\oc} \partial_t \tilde{\varphi}_1
\end{align}
locally on $\Lambda_\pm$. In addition, in terms of the notation in Proposition~\ref{prop:Hp-restriction}, we denote $x_{\oc}^{m'}p(t_1,z_1,x_{\oc}^{-2}(-1-x_{\oc}\partial_{t_1}\tilde{\varphi}_1),x_{\oc}^{-1}\tilde{\zeta}_1',x_{\oc}^{-1}\tilde{\zeta}_2'')$ by $\tilde{p}(t_1,z_1,\tilde{T},\tilde{\zeta}_1',\tilde{\zeta}_2'')$ below.
%%%%%%%%%%%%%%%%%%%%%%%%%%%%%%%%%%%%translate Andrew's calculation to the 1c setting 

Consider $\tilde{I}_1$ first. We now apply the stationary phase method in the $(z_2', \tilde{\zeta}_1')$-variables. The stationary points occur at 
\begin{equation}
- \tilde{\zeta}_{1,j}' -    \frac{\partial \tilde{\varphi}_1}{\partial z'_{1, j}}  =0 \text{ and } z_1' = z_2'.
\end{equation}
 At the critical point, the Hessian $H$ of the phase function $\Phi$ and its inverse $H^{-1}$  are given in block-diagonal form by (replacing $z_2'$ in $\tilde{\varphi}_1$ by $z_1'$)
 \begin{equation}
 H =  \begin{pmatrix} - \frac{\partial^2 \tilde{\varphi}_1}{\partial z'_{1}\partial z'_{1}} & -I \\ -I &  0   \end{pmatrix} \text{ and }
 H^{-1} = \begin{pmatrix} 0 & -I \\ -I &    \frac{\partial^2 \tilde{\varphi}_1}{\partial z'_{1}\partial z'_{1}} \end{pmatrix} .
 \end{equation}
 
 We want to compute the principal symbol of $PA$ when $p$ vanishes, to leading order, on $L$. This will be at order $m'+m -1$ if $m$ is the order of $A$ and $m'$ is the order of $P$. 
To do this, we apply the stationary phase expansion, using \cite[Theorem~7.7.5]{hormanderbookvolI} with $\omega = 1$ (it is an expansion in the sense of decay as the frequency variables tend to infinity). We require the first two terms in this expansion, as the remaining terms decay faster than the term we are interested in. We observe that the $g$ function (in the notation of \cite[Theorem~7.7.5]{hormanderbookvolI}) in this theorem, which is the difference between the phase function and its quadratic approximation at the critical point, is independent of $\tilde{\zeta}_1'$. This means that we need three $z_2'$-derivatives to hit $g$ before it produces a nonzero term in the expansion. Now notice that the quadratic expression $\langle H^{-1} D, D \rangle$ there is at most one $z_2'$-derivative, as the top left corner of the matrix $H^{-1}$ vanishes.   It follows that in first two terms of the stationary phase expansion, the $g$ factor can be ignored. 

The top two terms therefore are the leading term, (we drop the lower indices for $z$ and $\tilde{\zeta}$ where convenient from now on, and collect the part of the phase function that is independent of $\tilde{\zeta}_2''$ as $\tilde{\Phi}$):
\begin{equation}\label{spe-leading}
 (2\pi)^{- \frac{2n+1+2k}{2} } \int e^{i\tilde{\Phi}}
 e^{i \frac{z''_1 \cdot \tilde{\zeta}''_2 - \tilde{\varphi}_1(t,z'_1, \tilde{\zeta}''_2,x_{\oc},y_{\oc})}{x_{\oc}} )  } x^{-m-m'+1-\frac{n-k}{2} -\frac{3}{4} } 
 \tilde{p}(t_1,z'_1, z''_1, \tilde{T}, \Xi', \tilde{\zeta}''_2) a_0(t,z'_1, \tilde{\zeta}''_2,x_{\oc},y_{\oc}) \, d\tilde{\zeta}_2'',
 \end{equation}
 and the second term (note that $\det H^{-1} = 1$)
 \begin{multline}\label{spe-second}
  (2\pi)^{-k} \int e^{i\tilde{\Phi}}
 e^{i \frac{z''_1\cdot \tilde{\zeta}''_2 - \tilde{\varphi}_1(z'_1, \tilde{\zeta}''_2)}{x_{\oc}} )  } x^{-m-m'+1-\frac{n-k}{2} -\frac{3}{4} } \Big( i D_{z_2'} \cdot D_{\tilde{\zeta}_1'} \big( \tilde{p}(t_1,z'_1, z''_1, \tilde{\tau}, \zeta_1', \tilde{\zeta}''_2) a_0(t_1,z_2', \tilde{\zeta}_1'',x_{\oc},y_{\oc}) \big) \\
- \frac{i}{2}    \frac{\partial^2 \tilde{\varphi}_1}{\partial z'_{1,l}\partial z'_{1,m}} D_{\tilde{\zeta}_{1, l}'} D_{\tilde{\zeta}_{1, m}'} \big( \tilde{p}(t_1,z'_1, z''_1, \tilde{\tau}, \zeta_1', \tilde{\zeta}''_2) a_0(t_1,z_2', \tilde{\zeta}_1'',x_{\oc},y_{\oc}) \big) \Big)\Big|_{z_2' = z_1', \tilde{\zeta}_1' = \Xi'(z_1', \tilde{\zeta}_2''),\tilde{\tau}=\tilde{T}} \, d\tilde{\zeta}_2''  .
\end{multline}

%Since we only have $z_1',z_1'',\zeta''_2$ after this step, we will drop the indices $1,2$ and denote them by $z',z'',\zeta''$.
From now on, we denote $z_1',z_1'',\zeta''_2$ by $z',z'',\zeta''$.
We also abbreviate the $t,x_{\oc},y_{\oc}$ variables in functions $\tilde{\varphi}_1,\Xi',Z''$, since they only play the role as parameters in the following procedure, and don't affect the argument.

We work first on the leading term \eqref{spe-leading}. We expand $\tilde{p}$ around $z'' = Z''(z',\tilde{\zeta}'')$:
\begin{multline}
\tilde{p}(z', z'',\tilde{T}(z',\tilde{\zeta}''), \Xi'(z', \tilde{\zeta}''), \tilde{\zeta}'') = \tilde{p}(z', Z''(z', \tilde{\zeta}''), \tilde{T}, \Xi'(z', \tilde{\zeta}''), \tilde{\zeta}'') + (z'' - Z'')_j \frac{\partial \tilde{p}}{\partial z''_j}(z', Z'', \tilde{T}, \Xi', \tilde{\zeta}'') \\ + \frac1{2} (z'' - Z'')_j  (z'' - Z'')_k \frac{\partial^2 \tilde{p}}{\partial z''_j \partial z''_k}(z', Z'', \tilde{T}, \Xi', \tilde{\zeta}'') + O((z'' - Z'')^3).
\end{multline}

We now notice that $(z'' - Z'')_j = \partial_{\tilde{\zeta}''_j}(z''\cdot \tilde{\zeta}'' - \tilde{\varphi}_1(z', \tilde{\zeta}''))$. 
%where we use the second identity in \eqref{conic-identities} here. 
We can use this to integrate by parts to remove the $(z'' - Z'')_j$ factors. Notice that, to obtain terms at order $m'+m-1$, we need to do a Taylor expansion to second order here, which is a bit subtle. It is because, on integrating by parts the $D_{\tilde{\zeta}''}$, can hit the remaining $z'' - Z''$ factor, removing it as a vanishing factor. However, we can ignore terms of third order and higher. 
Because $p_{\hom}$ vanishes on (the projection of) $\Legps$, we can write
$$
\tilde{p}(z', Z''(z', \tilde{\zeta}''),\tilde{T}(z',\tilde{\zeta}''), \Xi'(z', \tilde{\zeta}''), \tilde{\zeta}'') = r(z', \tilde{\zeta}'')
$$
and note it is order $m-1$ and forms part of the subprincipal symbol of $P$. Putting this all together, from \eqref{spe-leading} we obtain, modulo Lagrangian distributions of order $m'+m-2$, 
\begin{align*}
 \frac1{(2\pi)^{k}} \int e^{i\tilde{\Phi}}
 e^{i \frac{z''\cdot \tilde{\zeta}'' - \tilde{\varphi}_1(z', \tilde{\zeta}'')}{x_{\oc}} )  } x^{-m-m'+1-\frac{n-k}{2} -\frac{3}{4} } \Big( r(z', \tilde{\zeta}'')  a(z', \tilde{\zeta}'') 
 \\ + i \frac{\partial}{\partial \tilde{\zeta}''_j} \big( \frac{\partial \tilde{p}}{\partial z''_j}(z', Z'',\tilde{T}, \Xi', \tilde{\zeta}'') a(z', \tilde{\zeta}'') \big)   - \frac{i}{2} \frac{\partial Z''_j}{\partial \tilde{\zeta}''_k} \frac{\partial^2 \tilde{p}}{\partial z''_j \partial z''_k}(z', Z'', \tilde{T}, \Xi', \tilde{\zeta}'') a(z', \tilde{\zeta}'') \Big)   \, d\tilde{\zeta}''.
\end{align*} 
 We use the product and chain rules to apply the $\tilde{\zeta}''$-derivatives, and we obtain 
  \begin{align}
      \label{spe-leading2}
      \begin{split}
 \frac1{(2\pi)^{k}} \int e^{i\tilde{\Phi}}
 e^{i \frac{z''\cdot \tilde{\zeta}'' - \tilde{\varphi}_1(z', \tilde{\zeta}'')}{x_{\oc}} )  } x^{-m-m'+1-\frac{n-k}{2} -\frac{3}{4} }
 \Big( r(z', \tilde{\zeta}'')  a(z', \tilde{\zeta}'')  + \\ i   \frac{\partial^2 \tilde{p}}{\partial \tilde{\zeta}''_j \partial z''_j}(z', Z'', \tilde{T}, \Xi', \tilde{\zeta}'') a(z', \tilde{\zeta}'')  
 + \frac{i}{2} \frac{\partial Z''_k}{\partial \tilde{\zeta}''_j} \frac{\partial^2 \tilde{p}}{\partial z''_j \partial z''_k}(z', Z'', \tilde{T}, \Xi', \tilde{\zeta}'') a(z', \tilde{\zeta}'') \\
 + i \frac{\partial \Xi'_k}{\partial \tilde{\zeta}''_j} \frac{\partial^2 \tilde{p}}{\partial \tilde{\zeta}'_k\partial z''_j}(z', Z'', \tilde{T}, \Xi', \tilde{\zeta}'') a(z', \tilde{\zeta}'')
 + i \frac{\partial \tilde{p}}{\partial z''_j}(z', Z'', \tilde{T}, \Xi', \tilde{\zeta}'') \frac{ \partial a}{\partial \tilde{\zeta}''_j} (z', \tilde{\zeta}'')
 \Big)  \, d\tilde{\zeta}'',          
      \end{split}
 \end{align} 
 where we used \eqref{eq: Z'', Xi' definition} to combine two terms involving derivatives of $Z''$.  
 
 The second term can be expressed, using the second identity in \eqref{eq: Z'', Xi' definition}, or more precisely, its partial derivative with respect to $z'$,  
 \begin{equation}
\label{spe-second2}
\begin{gathered}
  (2\pi)^{-k} \int e^{i\tilde{\Phi}}
 e^{i \frac{z''\cdot \tilde{\zeta}'' - \tilde{\varphi}_1(z', \tilde{\zeta}'')}{x_{\oc}} )  } x^{-m-m'+1-\frac{n-k}{2} -\frac{3}{4} }
  \\ \Big( - i   \frac{\partial \tilde{p}}{\partial \tilde{\zeta}'} (z', Z'', \tilde{T}, \Xi', \tilde{\zeta}'') \frac{\partial a}{\partial z'}(z', \tilde{\zeta}'') \big) 
- \frac{i}{2} \frac{\partial \Xi'_l}{\partial z'_{m}} \frac{\partial^2 \tilde{p}}{\partial \tilde{\zeta}_l' \partial \tilde{\zeta}_m'} (z', Z'', \tilde{T}, \Xi', \tilde{\zeta}'') a(z_2', \tilde{\zeta}_1'') \big) \Big) \, d\tilde{\zeta}_1''.     
\end{gathered}    
 \end{equation}

 We therefore find, combining \eqref{spe-leading2} and \eqref{spe-second2}, that the principal symbol of $PA$ is, as a multiple of the half-density $|dz' d\tilde{\zeta}''|^{1/2}$,  
 \begin{multline}
 r  a + i \big( \frac{\partial \tilde{p}}{\partial z''_j} \frac{ \partial a}{\partial \tilde{\zeta}''_j} - \frac{\partial \tilde{p}}{\partial \tilde{\zeta}'} \frac{\partial a}{\partial z'} \big) 
 + i \frac{\partial^2 \tilde{p}}{\partial \tilde{\zeta}''_j \partial z''_j} a \\ + \frac{i}{2} \Big( \frac{\partial Z''_k}{\partial \tilde{\zeta}''_j} \frac{\partial^2 \tilde{p}}{\partial z''_j \partial z''_k} +  2\frac{\partial \Xi'_k}{\partial \tilde{\zeta}''_j} \frac{\partial^2 \tilde{p}}{\partial \tilde{\zeta}'_k\partial z''_j}
 - \frac{\partial \Xi'_l}{\partial z'_{m}} \frac{\partial^2 \tilde{p}}{\partial \tilde{\zeta}_l' \partial \tilde{\zeta}_m'} \Big) a.
 \end{multline}
    
 We now write this in invariant terms. First, the (rescaled) subprincipal symbol of $P$ is 
 $$
 p_{\mathrm{sub}} = r + \frac{i}{2} \frac{\partial^2 \tilde{p}}{\partial \tilde{\zeta}_j \partial z_j}. 
 $$
 Next, the vector field appearing above applied to $a$ is $-i$ times the Hamilton vector field $H_p$ on the Legendre submanifold $\Legps$. 
 We can therefore rewrite the symbol, as  a multiple of the half-density $|dz' d\tilde{\zeta}''|^{1/2}$, as 
  \begin{multline}\label{prsy}
 p_{\mathrm{sub}}  a - i H_p a  + \frac{i}{2} \big( \frac{\partial^2 \tilde{p}}{\partial \tilde{\zeta}''_j \partial z''_j} -  \frac{\partial^2 \tilde{p}}{\partial \tilde{\zeta}'_j \partial z'_j} \big) a \\ + \frac{i}{2} \Big( \frac{\partial Z''_k}{\partial \tilde{\zeta}''_j} \frac{\partial^2 \tilde{p}}{\partial z''_j \partial z''_k} +  2\frac{\partial \Xi'_k}{\partial \tilde{\zeta}''_j} \frac{\partial^2 \tilde{p}}{\partial \tilde{\zeta}'_k\partial z''_j}
 -  \frac{\partial \Xi'_l}{\partial z'_{m}} \frac{\partial^2 \tilde{p}}{\partial \tilde{\zeta}_l' \partial \tilde{\zeta}_m'} \Big) a.
 \end{multline}
 
 Now we compute the divergence of the vector field $ H_p$ restricted to the Lagrangian. As we have just mentioned, $H_p$ on $L$ takes the form (using the chain rule for the dependence of $Z''$ and $\Xi'$ on $(z', \tilde{\zeta}'')$)
 $$
- \frac{\partial \tilde{p}}{\partial z''_j}(z', Z'', \Xi', \tilde{\zeta}'') \frac{ \partial }{\partial \tilde{\zeta}''_j} + \frac{\partial \tilde{p}}{\partial \tilde{\zeta}'}(z', Z'', \Xi', \tilde{\zeta}'') \frac{\partial }{\partial z'},
 $$
 so its divergence in these coordinates is 
 \begin{multline}
 - \frac{\partial^2 \tilde{p}}{\partial \tilde{\zeta}''_j \partial z''_j} - \frac{\partial Z''_k}{\partial \tilde{\zeta}''_j} \frac{\partial^2 \tilde{p}}{\partial z''_k \partial z''_j} - 
    \frac{\partial \Xi'_k}{\partial \tilde{\zeta}''_j} \frac{\partial^2 \tilde{p}}{\partial \tilde{\zeta}'_k \partial z''_j} \\
    +  \frac{\partial^2 \tilde{p}}{\partial z'_j \partial \tilde{\zeta}'_j } + \frac{\partial Z''_k}{\partial z'_j} \frac{\partial^2 \tilde{p}}{\partial z''_k \partial \tilde{\zeta}'_j} + 
    \frac{\partial \Xi'_k}{\partial z'_j} \frac{\partial^2 \tilde{p}}{\partial \tilde{\zeta}'_k \partial \tilde{\zeta}'_j} .
    \end{multline}
   Using  \eqref{eq: Z'', Xi' definition}, we can combine two terms, and write this as 
  \begin{equation}\label{div}
\big(  - \frac{\partial^2 \tilde{p}}{\partial \tilde{\zeta}''_j \partial z''_j}  +  \frac{\partial^2 \tilde{p}}{\partial z'_j \partial \tilde{\zeta}'_j } \big) 
 - \frac{\partial Z''_k}{\partial \tilde{\zeta}''_j} \frac{\partial^2 \tilde{p}}{\partial z''_k \partial z''_j} - 
    2\frac{\partial \Xi'_k}{\partial \tilde{\zeta}''_j} \frac{\partial^2 \tilde{p}}{\partial \tilde{\zeta}'_k \partial z''_j} + 
    \frac{\partial \Xi'_k}{\partial z'_j} \frac{\partial^2 \tilde{p}}{\partial \tilde{\zeta}'_k \partial \tilde{\zeta}'_j} .
    \end{equation}   
 Now using \eqref{div}, \eqref{prsy} can be expressed
   \begin{equation}\label{prsy2}
 p_{\mathrm{sub}}   a - i \rHp a  - \frac{i}{2} \mathrm{div} (\rHp|_{\Legps}) a. 
  \end{equation}  
 Writing $\mathbf{a} = a \cdot \nu$, with $\nu$ being the half-density factor as in (\ref{eq: 1c-ps principal symbol, half density form}) tensored with $M(\Legps)$, the previous equation can be invariantly written as
 \begin{equation}\label{prsy3}
(- i \mathscr{L}_{H_p}+p_{\mathrm{sub}}) \big(\textbf{a} \otimes |dx_{\oc}|^{-m-\frac{2n+5}{4}} \big). 
  \end{equation}  
Now notice that $H_p = x_{\oc}^{-(m'-1)}\rHp$, we know when acting on half densities we have:
\begin{align*}
\mathscr{L}_{H_p}  = &  x_{\oc}^{-(m'-1)}\mathscr{L}_{\rHp} + 
\frac{1}{2} \rHp x_{\oc}^{-(m'-1)}
\\ & = x_{\oc}^{-(m'-1)}\mathscr{L}_{\rHp} - 
\frac{m'-1}{2}
x_{\oc}^{-(m'-1)} x_{\oc}^{-1} \rHp x_{\oc},
\end{align*}
completing the proof.
\end{proof}

\bibliographystyle{plain}
\bibliography{bib_schrodinger}

\end{document}